\documentclass[11pt,leqno]{amsart}
\usepackage[total={6.5in,9.3in},hmarginratio=1:1,vmarginratio=1:1,includefoot]{geometry}
\usepackage{amsthm,amsfonts,amssymb,amsmath,amscd,latexsym}
\usepackage{graphicx,exscale,epsfig}
\usepackage{color}
\usepackage{bbm}
\usepackage{eucal}
\newcommand{\Z}{\mathbb{Z}}
\numberwithin{equation}{section}

\newtheorem{theorem}{Theorem}[section]

\newtheorem{definition}[theorem]{Definition}

\newtheorem{corollary}[theorem]{Corollary}

\newtheorem{lemma}[theorem]{Lemma}

\newtheorem{proposition}[theorem]{Proposition}

\newtheorem{remark}[theorem]{Remark}

\renewcommand{\H}{d\mathcal{H}^{n-1}(x)}

\newcommand{\C}{\mathbb{C}}

\newcommand{\R}{\mathbb{R}}

\newcommand{\F}{\mathcal{F}}
\newcommand{\beq}{\begin{equation}}
\newcommand{\eeq}{\end{equation}}
\newcommand{\bea}{\begin{eqnarray}}
\newcommand{\eea}{\end{eqnarray}}
\newcommand{\beaa}{\begin{eqnarray*}}
\newcommand{\eeaa}{\end{eqnarray*}}
\newcommand{\ben}{\begin{enumerate}}
\newcommand{\een}{\end{enumerate}}
\newcommand{\bi}{\begin{itemize}}
\newcommand{\ei}{\end{itemize}}

\def\bar{\overline}

\def\intav#1{\mathchoice
   {\mathop{\vrule width 6pt height 3 pt depth -2.5pt
   \kern -9pt \intop}\nolimits_{\kern -6pt#1}}%
   {\mathop{\vrule width 5pt height 3 pt depth -2.6pt
   \kern -6pt \intop}\nolimits_{#1}}%
   {\mathop{\vrule width 5pt height 3 pt depth -2.6pt
   \kern -6pt \intop}\nolimits_{#1}}%
   {\mathop{\vrule width 5pt height 3 pt depth -2.6pt
   \kern -6pt \intop}\nolimits_{#1}}}

\newcommand{\bpm}{\begin{pmatrix}}
\newcommand{\epm}{\end{pmatrix}}

\renewcommand\d{\partial}

\def\eps{\varepsilon }

\renewcommand\d{\partial}

\def\eps{\varepsilon}

\newcommand\br{\begin{remark}}
\newcommand\er{\end{remark}}
\newcommand\bp{\begin{pmatrix}}
\newcommand\ep{\end{pmatrix}}
\newcommand{\be}{\begin{equation}}
\newcommand{\ee}{\end{equation}}
\newcommand\ba{\begin{equation}\begin{aligned}}
\newcommand\ea{\end{aligned}\end{equation}}

\newcommand{\bap}{\begin{app}}
\newcommand{\eap}{\end{app}}
\newcommand{\begs}{\begin{exams}}
\newcommand{\eegs}{\end{exams}}
\newcommand{\beg}{\begin{example}}
\newcommand{\eeg}{\end{exaplem}}
\newcommand{\bpr}{\begin{proposition}}
\newcommand{\epr}{\end{proposition}}
\newcommand{\bt}{\begin{theorem}}
\newcommand{\et}{\end{theorem}}
\newcommand{\bc}{\begin{corollary}}
\newcommand{\ec}{\end{corollary}}
\newcommand{\bl}{\begin{lemma}}
\newcommand{\el}{\end{lemma}}
\newcommand{\bd}{\begin{definition}}
\newcommand{\ed}{\end{definition}}
\newcommand{\brs}{\begin{remarks}}
\newcommand{\ers}{\end{remarks}}

\newcommand{\const}{\text{\rm constant}}
\newcommand{\Id}{{\rm Id }}

\theoremstyle{definition}

\newtheorem{example}[theorem]{Example}

\title{Multidimensional stability and transverse bifurcation of 
hydraulic shocks and roll waves in open channel flow}
\author{Zhao Yang}
\address{Academy of Mathematics and Systems Science, Chinese Academy of Sciences, Beijing 100190 China}
\email{yangzhao@amss.ac.cn}
\author{Kevin Zumbrun}
\address{Indiana University, Bloomington, IN 47405}
\email{kzumbrun@iu.edu}
\thanks{Research of K.Z. was partially supported under NSF grants no. DMS-2154387 and DMS-2206105} 

\begin{document}

\begin{abstract}
We study by a combination of analytical and numerical methods
multidimensional stability and transverse bifurcation of planar hydraulic shock and roll wave solutions 
of the inviscid Saint Venant equations for inclined shallow-water flow, both in the whole space and in a channel 
of finite width, obtaining complete stability diagrams across the full parameter range of existence.
Technical advances include development of efficient multi-d Evans solvers, low- and high-frequency asymptotics,
explicit/semi-explicit
computation of stability boundaries, and rigorous treatment of channel flow with wall-type physical
boundary.  Notable behavioral phenomena are a novel essential transverse bifurcation of hydraulic shocks
to invading planar periodic roll-wave or doubly-transverse periodic herringbone patterns,
with associated metastable behavior driven by mixed roll- and herringbone-type waves
initiating from localized perturbation of an unstable constant state; and Floquet-type transverse
``flapping'' bifurcation of roll wave patterns.
\end{abstract}

\date{\today}
\maketitle
\tableofcontents

\section{Introduction}\label{s:intro}
In this paper, we initiate the systematic study of multidimensional stability and behavior for planar shock 
and roll wave solutions of the inviscid Saint Venant equations of inclined shallow water flow, in the whole space 
or a channel of finite width.
One-dimensional stability of these waves has been resolved analytically and numerically in \cite{JNRYZ,YZ,SYZ,FRYZ}.
Here, we achieve similarly definitive results on multidimensional stability by a combination of analytical
investigation and numerical Evans function computations. Moreover, we go beyond basic stability to investigate
1- and 2d bifurcation from instability, in particular a novel transverse essential bifurcation 
of hydraulic shocks to invading periodic roll wave or doubly periodic 
``herringbone'' patterns and a Floquet-type transverse ``flapping'' bifurcation of roll waves at instability.

Roll waves are of interest in hydraulic engineering as unwanted ``mini-rogue wave'' patterns forming 
in canals from unstable laminar flow. Of particular interest are the questions when such waves can occur- i.e.,
when they exist as stable solutions- and the maximum fluid height arising thereby.
We show that 2d stability is a much stronger constraint for open channel flow than 1d
stability, greatly refining the 1d stability diagram found in \cite{JNRYZ}; we compare our results, suitably
normalized, to experimental findings of Brock \cite{Brock,Br1,Br2}.

Hydraulic shocks are familiar as stable structures, used to dissipate energy in dam spillways, for example.
And, in the hydrodynamically stable case, i.e., the scenario of stable laminar flow, or Froude number
$F<2$, we find similarly as in the 1d analysis \cite{YZ} that they are universally stable. However, in the hydrodynamically unstable, or pattern formation regime $F>2$, we find that 1d stable hydraulic
shocks can in some circumstances exhibit transverse instabilities, leading to bifurcation. Remarkably, these bifurcation points may be explicitly computed; indeed, they correspond to novel
``essential'' bifurcations corresponding to failure of convective stability at the endstates of the shock.
These bifurcations lead to appearance of both roll waves and herringbone patterns on the upstream side of the shock.
Moreover, we find interesting ``metastable'' behavior with respect to highly localized (Gaussian) perturbations for which we find no counterpart in the literature, related to, but substantially extending/elaborating on recent observations of \cite{FHSS}.

Roll waves, as inherently more unstable structures, seem a more likely place to find transverse
instabilities than hydraulic shocks.
And, indeed, they do exhibit frequent transverse instabilities for 1d stable waves.  However, these
are of a standard type involving ``near-field'' rather than essential instability as in the shock case, that is, corresponding to solutions of a generalized Floquet eigenvalue system that may be studied via a periodic Evans-Lopatinsky determinant (described below).
Associated bifurcations are seen to be Floquet-type ``flapping'' bifurcations not changing the topological structure
of the waves.
Moreover, roll wave instabilities are seen to be of {\it low-frequency} type, determinable from the 
Taylor series expansion of the periodic Evans-Lopatinsky determinant about the origin.
This somewhat opposite (from the shock) scenario
allows us again to numerically compute the 2d stability diagram. 

These results involve a number of technical innovations of general interest, 
several of them importing techniques from the more-developed area of multi-d detonation stability
for the ZND equations, a system of balance laws with features similar to relaxation.
These include:
1. development of efficient numerical Evans solvers for multi-d relaxation systems,
including cases involving singular ODE or ``sonic'' points.
2. a low-frequency equivalence theorem for shocks similar to that noted for detonations in \cite{JLW},
3. a high-frequency turning-point analysis for shocks similar to that carried out for detonations in \cite{LWZ,Er4},
4. low-frequency stability analysis by a combination of contour integral-based Taylor series expansion as in \cite{JNRZ3} 
together with Weierstrass preparation manipulations as in \cite{HY}, and
5. systematic treatment of channel flow and relation to flow in the whole plane. 

The latter treatment is somewhat similar to analyses of detonation \cite{KS} and MHD \cite{BT,FT,BMZ},
reducing the study of eigenvalues under additional constraints to eigenvalues for the whole-plane problem without
constraints.
Namely, similarly as for, say, the scalar heat equation with Dirichlet boundary conditions, admitting half-sine
expansion in place of the standard Fourier series for periodic boundary conditions, we find that
the channel problem admits a complete basis of mixed half-sine/half-cosine functions corresponding to periodic
boundary conditions on a channel of double the width. This means that we may study the finite-width problem
by simple restriction of the full Fourier transform treatment of the whole-plane case to the lattice of frequencies
compatible with the doubled channel, allowing us to use the many analytical and Evans function
tools developed for study of the whole plane.
At the same time, we can compare to numerical time-evolution simulations (using CLAWPACK \cite{Cl})
carried out in the computationally more friendly domain of the channel with finite width.

A related side-note is that the cross-propagating herringbone patterns that we observe in 
bifurcating time-evolution studies in the channel with wall boundary conditions appear to
be restrictions to the channel of doubly-periodic patterns on the whole plane, with vertically symmetric
periods summing to a vertical period of double the width of the channel.
That is, just as in the shock and roll wave case, exact herringbone solution in the plane if known would yield 
exact solutions on the channel, with no modification due to boundary layers, a nice advantage of the {\it inviscid} 
formulation of shallow-water flow.
On the other hand, the existence of such doubly periodic discontinuous solutions does not seem to readily 
fit standard analytical frameworks.  By contrast, their (smooth) viscous cousins could presumably be investigated
on the whole plane by Turing-type bifurcation from the constant state; see \cite{RW} or \cite{WZ1,WZ2,WZ3}
in the one-dimensional case, and \cite{MR} for related analysis in the reaction diffusion case.
On the other hand, these do not furnish exact solutions on the channel with no-slip or Navier-slip boundary conditions.
\subsection{Equations}\label{s:eqs}
The (inviscid) Saint-Venant equations\footnote{The system \eqref{sv_intro} is dimensionless. See Appendix~\ref{s:brock} for equivalence between its 1d version \eqref{sv-1d} and the dimensionless shallow water model \eqref{brock_table} used by Brock \cite[eqs. (3.8) and (3.9)]{Brock} under the scale invariance \eqref{scaling}.} (SV) for inclined shallow-water flow are \cite{BaM,Ro}
\ba\label{sv_intro}
\d_th+\d_x (hu) +\d_y(hv) &=0, \\
\d_t (hu) +\d_{x}\left(hu^2 +h^2/2F^2 \right) + \d_y(huv) & =h - \sqrt{u^2+v^2} \,u , 
\\
\d_t (hv) + \d_x( huv) + \d_{y}\left(hv^2 +h^2/2F^2 \right) & = -\sqrt{u^2+v^2} \,v , 
\ea
where $h$ is fluid height, $u$ and $v$ are velocities in $x$ and $y$ directions, and
$F$ is a nondimensional {\it Froude number} depending on angle of incline, gravitational force, and
units for $(h,u,v)$.

We will consider these equations either on the whole plane $(x,y)\in \R^2$ or else in a channel,
$x\in \R$, $0\leq y\leq L$, with wall type boundary conditions 
\be\label{wallBC}
\hbox{\rm $v(y)=0$ at $y=0,L$.}
\ee
The latter is called ``open channel flow,'' referring to the fact that the
fluid surface forms a free boundary, and applies to a typical canal, stream, or spillway.
``Closed-channel flow'' refers, rather to the case of a channel with a top lid as well as sides,
such as flow through a rectangular pipe.

Derived from a free-boundary problem for the 3d {incompressible} irrotational Euler equations
in the limit as vertical scale goes to zero with respect to horizontal,
these are reminiscent of {compressible} 2d flow.
Indeed, the first-order, lefthand side, part of \eqref{sv_intro} corresponds to isentropic $\gamma$-law gas dynamics
with $h=$ density and $\gamma=2$ \cite{Ba,Sm},
with turbulent friction $-(\sqrt{u^2+v^2}u, \sqrt{u^2+v^2}v)$ opposing gravitational force $(h,0)$ 
in the zeroth-order, righthand side of \eqref{sv_intro}(ii)-(iii).
Balancing zeroth order terms and applying to \eqref{sv_intro}(i) gives the formal (scalar) equilibrium equation
\be\label{CE}
\d_t h+ \d_x h^{3/2}= 0, \quad (u,v)=(h^{1/2},0).
\ee

Equations \eqref{sv_intro} are of interest both mathematically, as a canonical example of 
relaxation/pattern formation \cite{Wh,Li,JK},
and physically, as the industry standard for hydraulic engineering and related applications \cite{BaM,Co,D,Brock}.
In the ``subcharacteristic'', or hydrodynamically stable case $F<2$ for which constant equilibrium
solutions are linearly stable, the system is expected to move toward equilibrium $r(w)=0$, being
governed time-asymptotically by formal equilibrium equation \eqref{CE} \cite{Je,Wh,Li}. 
In the ``supercharacteristic'' or hydrodynamically unstable case $F>2$, it is expected rather to display
complex behavior/pattern formation \cite{JK,BN}. 

\subsection{Hydraulic shocks and roll waves, and stability in 1d}\label{s:planar}
\begin{figure}[htbp]
\begin{center}
\includegraphics[scale=0.32]{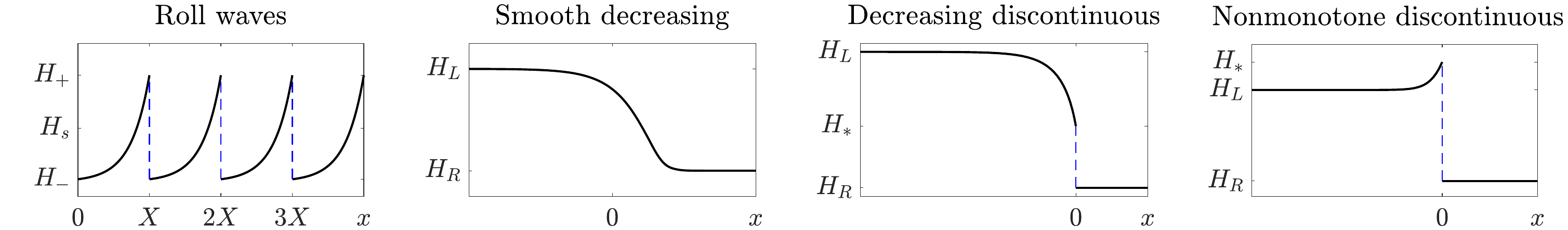}
\end{center}
\caption{First panel: profile of roll waves; Second, third, and fourth panel: profiles of hydraulic shocks of smooth decreasing, decreasing discontinuous, and nonmonotone discontinuous types. }
\label{fig_profiles}
\end{figure}
In the hydrodynamically stable case $F \leq 2$, there exist planar hydraulic shock fronts propagating in the $x$
direction connecting any two equilibria at $\pm\infty$ corresponding to entropy admissible shocks of \eqref{CE},
which are monotone decreasing in fluid height $h$. Denote by $H_L$ and $H_R$ the limits at $-$ and $+\infty$,
and 
\be \label{defnu}
\nu:=\sqrt{H_L/H_R}(>1),
\ee so that $\nu^2-1$ measures the strength of the equilibrium shock.
Then, for each fixed $\nu>1$ there is a characteristic Froude number \footnote{We refer to \eqref{characteristicF} as the characteristic Froude number since at $F=F_{\text{char}}(\nu)$ there are characteristic points on the profile. See \cite[Fig. 2. (b)]{YZ}}
\be \label{characteristicF}
F_{\text{char}}(\nu):=\frac{\nu+1}{\nu^2}(<2)
\ee 
such that profiles are smooth for $F<F_{\text{char}}(\nu)$
and contain an entropy admissible discontinuity, or ``subshock'' for $F$ above. See FIGURE~\ref{fig_profiles} second and third panels.
To put this another way, for each fixed $F<2$, profiles are smooth for shock strengths below a critical
value $\nu_{\text{char}}^2(F)-1$, and discontinuous for shock strengths above.
These waves have been shown analytically to be 1d stable \cite{YZ,SYZ}.

For $F>2$, there appear also periodic ``roll waves'' \cite{D,H03} consisting of
an array of shock discontinuities separated by smooth wave profiles. See FIGURE~\ref{fig_profiles} first panel. Stability of these waves selects period and amplitude, 
important in hydro-engineering applications such as canal and spillway design (see FIGURE~ \ref{fig_phy}),
for which typically $3.5\leq F\leq 10 $, and amplitudes can be sufficiently large to affect safety concerns \cite{Brock}.
A success in the 1d theory was the development \cite{JNRYZ} of a rigorous and practical spectral
stability framework for these waves, along with an efficient numerical hybrid method for periodic Evans-Lopatinsky determinant computation and a remarkable explicit analytic computation of the low-frequency (``sideband'') stability boundary I.
These yield the surprisingly simple {\it global} 1d stability diagram of FIGURE \ref{fig_roll_1d}, expressed in terms of the Froude number $F$ versus $H_-/H_s$ and $F$ versus $X/H_s$. Here, as depicted in FIGURE~\ref{fig_profiles} first panel, $H_-$ stands for minimum fluid height, $H_s$ stands for fluid height at the sonic point, and $X$ stands for the period of the roll waves. 
\begin{figure}[htbp]
\begin{center}
$
\begin{array}{lll}
\includegraphics[scale=0.408]{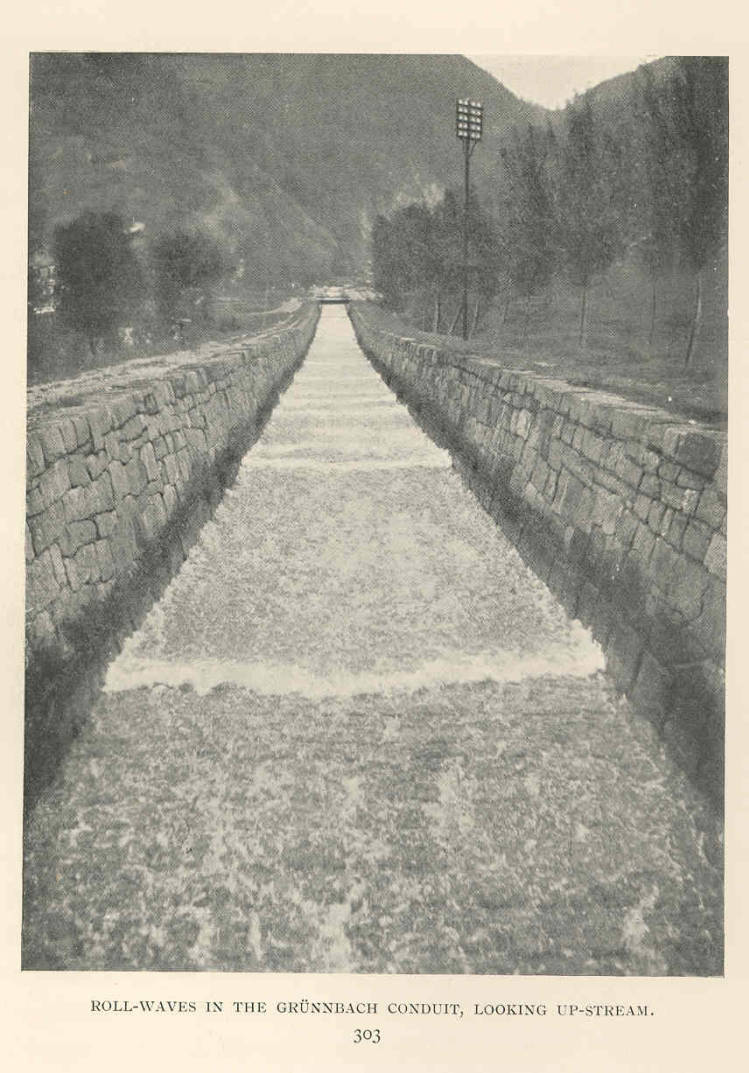}\;
\includegraphics[scale=0.335]{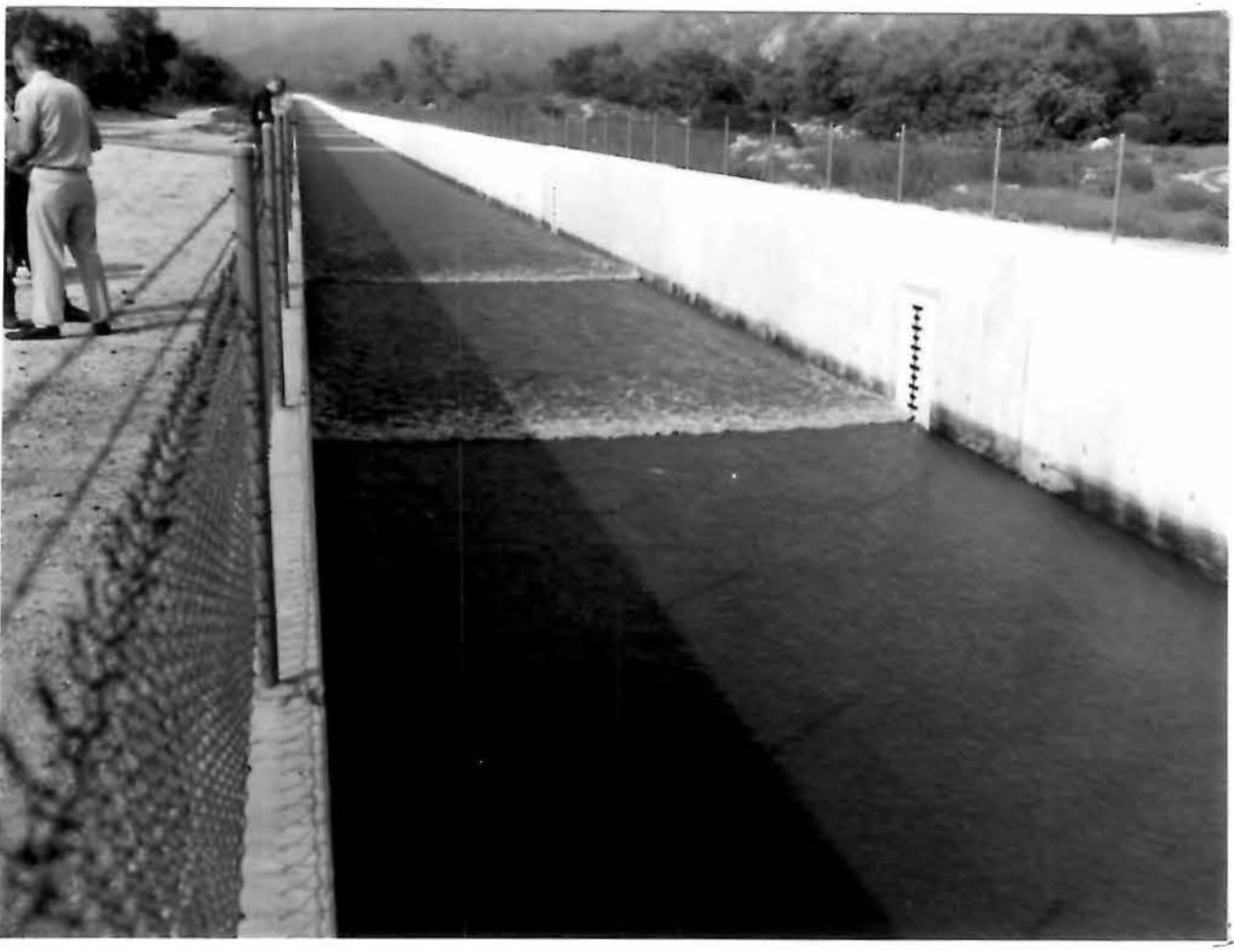}\;
\includegraphics[scale=0.261]{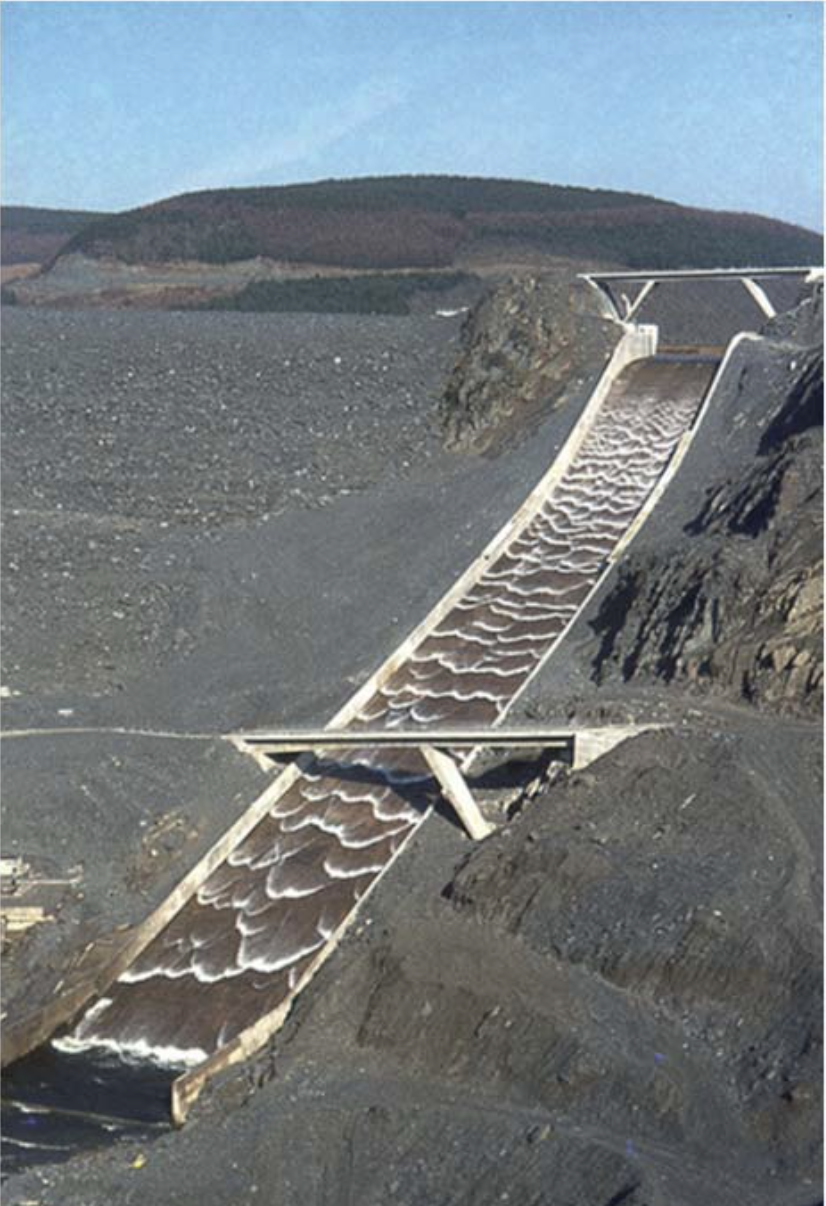}
\end{array}
$
\end{center}
\caption{Examples of roll waves. Left: Grunnbach conduit, Germany, V. Cornish (1910); 
Middle: Santa Anita Wash, California, R. Brock, (1969-1970);
Right: Llyn Brianne dam, Wales, G. Richard \& S. Gavrilyuk (2012, 2013).}
\label{fig_phy}\end{figure}

\begin{figure}[htbp]
\begin{center}
\includegraphics[scale=.42]{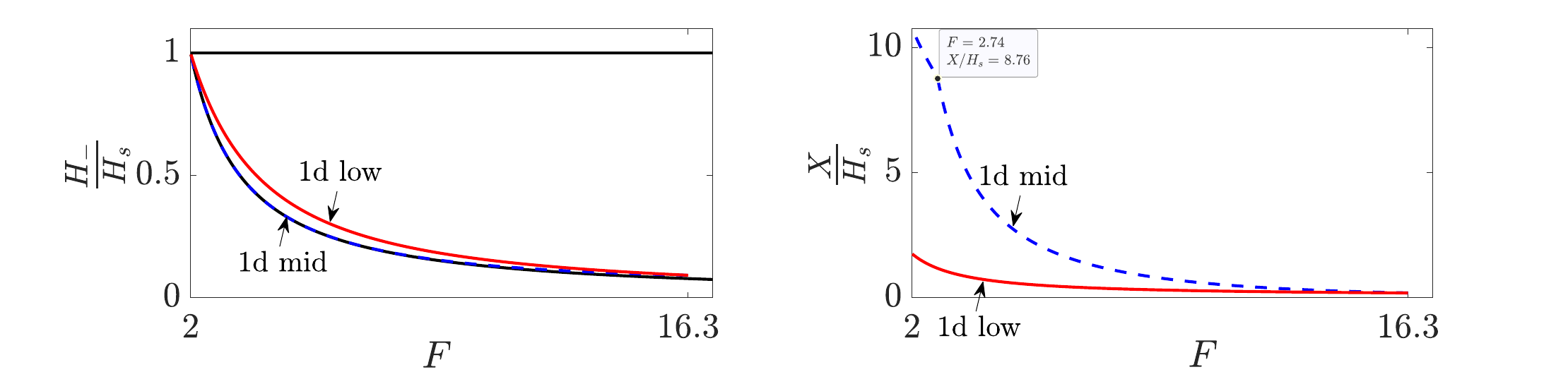}
\end{center}
	\caption{The 1d stability diagrams with respect to parameters $(F,H_-/H_s)$ in the left panel and $(F,X/H_s)$ in the right panel. In both panels, the 1d stable region is between the 1d low-frequency stability boundary I (red solid curve with label ``1d low'') and 1d medium-frequency stability boundary (blue dash curve with label ``1d mid'') which crosses the former curve at $F=16.3\ldots$, making the 1d stable region bounded. In the left panel, we find the ``1d mid'' boundary sits very close to the lower boundary of domain of existence corresponding to the homoclinic limit. In the right panel, we clearly see there is a sharp edge on the ``1d mid'' boundary at $F=2.73\ldots$. See \cite[Fig. 11.]{JNRYZ} for cause of the sharp edge. Also, readers may refer to \cite[Fig. 2., Fig. 3., and Fig. 7.]{JNRYZ} for a more complete understanding of 1d stability boundaries and regions.  }
\label{fig_roll_1d}\end{figure}


\smallskip

Along with periodic roll waves, there appear also  hydrodynamically unstable hydraulic shock profiles 
connecting entropy-admissible shock endstates of \eqref{CE}, which for fixed height ratio
$\nu^2$ exist for Froude numbers 
\be \label{defF_existence}
2<F<F_{\text{exist}}(\nu):=\nu(\nu+1).  
\ee 
See \cite[Proposition 2.1]{FRYZ} for further classification of the profiles into (ii) nonmonotone discontinuous, (iii) Riemann, and (iv) decreasing discontinuous types. See FIGURE~\ref{fig_profiles} fourth panel for a typical nonmonotone discontinuous hydraulic shock profile.
Put another way, for fixed $F>2$, hydrodynamically unstable profiles
exist for shock strength greater than a certain value $\nu_{\text{existence}}^2(F)-1$. All such shock profiles contain subshock
discontinuities.
These waves have been shown analytically to be 1d {\it convectively} stable \cite{FRYZ},
that is, stable with respect to sufficiently rapidly exponentially decaying perturbations, or, equivalently,
stable in an appropriate exponentially weighted norm, {so long there exists such a norm stabilizing essential
spectra of the limiting endstates of the shock.}
The latter condition translates to 
\be \label{F_1d}
F<F_{\text{1d}}(\nu):= \sqrt{2\nu(\nu+1)},
\ee
where it can be shown that $F_{\text{1d}}(\nu)<F_{\text{exist}}(\nu)$ for all $\nu>1$.
For $F>F_{\text{1d}}(\nu)$, numerical time evolution experiments \cite{FRYZ} indicate instability even with respect to
superexponentially localized (Gaussian, or smooth compactly-supported ``bump function'') perturbations, resulting in transition to a composite solution
consisting of an invading front of roll waves into a constant state on the right; see \cite[\S 1.1.4 and \S 6.1]{FRYZ}.
More generally, 1d time-asymptotic behavior for $F>2$ appears to consist of combinations of
roll waves and (stable) shock fronts \cite{FRYZ}.

We note \cite[Proposition 2.1]{FRYZ} that the nonmonotone hydraulic shocks appearing for $F>(\nu+1)\sqrt{2(\nu^2+1)}/(2\nu)$
are rather similar to single-cell profiles
for roll waves, having from the standpoint of hydraulic engineering the similar undesirable property that waves may reach a height far above the initiating upstream height of dam or spillway, leading to potential water overflow 
issues if this is not accounted for in design.

\subsection{Results in Multi-d}\label{s:new}
We now briefly describe our main results, to be presented in detail in the rest of the paper.
This work has two main motivations. The first is mathematical: namely, the 
systematic extension to multi-d of numerical and analytical tools for the study of 
1d stability of roll waves and hydraulic shocks developed recently in \cite{JNRYZ,YZ,SYZ,FRYZ} and elsewhere.

In this direction, our main results include: (i) a low-frequency equivalence theorem
similar to a result of \cite{JLW} for detonations, stating that the low-frequency expansion of neutral
spectra for hydraulic shock fronts agrees to lowest order with that for Lax shocks of the associated equlibrium
system \eqref{CE}, {\it even for large-amplitude shock profiles extending far from equilibrium}, possibly
including subshock discontinuities (Section~\ref{s:lf}),
(ii) low-frequency expansion of critical spectra for roll waves using
contour integral-based Taylor series expansion as in \cite{JNRZ3} together with Weierstrass preparation theorem,
yielding efficient numerical techniques for determining first- and second-order behavior 
(Sections ~\ref{s:rollLF} and~\ref{s:rollformal}),
(iii) high-frequency spectral asymptotics for hydraulic shocks based on Airy-type turning point analysis
like that done for detonations in \cite{Er4,LWZ} (Section \ref{s:hf}),
(iv) characterization of spectra on a channel as the restriction to an appropriate lattice of spectra on 
the whole plane (Section \ref{s:channel_flow_hydraulic}),
and
(v) efficient multi-d numerical Evans function solvers for hydraulic shocks and roll waves 
(Sections \ref{s:shockMF} and \ref{s:rollMF}, respectively).
Together, these give a complete ensemble of tools generalizing those of the 1d case, suitable for
the study of multi-d hydraulic shocks and roll waves.

The second, and to us more compelling, is pheomenological: to understand new 
waves and behavior not present in 1d.
In particular, 1d hydraulic shocks and roll waves
may be observed ``in the wild'' in many types of open channel flow, for example, in canals,
small streams, or rough-bottomed street gutters on a rainy day.
But, a more frequently-observed pattern is a doubly periodic ``crosshatch'' or ``herringbone'' flow, 
consisting roughly of superimposed planar periodic fronts moving obliquely 
to and symmetrically about the $x$ axis, or downhill grade.
One might guess- as indeed we initially did- that herringbone patterns would be associated with 
transverse instability of roll waves, in the same way (see \cite{FRYZ}) that roll waves are associated with 
longitudinal instability of hydraulic shocks.
Given the frequent analogies to detonation theory, one might guess, further, that this occurs
through evolution of the front, by a mechanism like that 
proposed by Majda-Rosales \cite{MR} for Mach stem formation in gas dynamical shock and detonation waves
based on weakly nonlinear geometric optics expansion, starting with smooth transverse bifurcation
and ending with a kink-type singularity.

That is, one might imagine a cascade of instabilities, from 1d hydraulic shock instability to roll waves,
and then by 2d roll wave instability to herringbone patterns.
{\bf But, this is not in fact what we see.}
Rather, thinking of shocks as parametrized by Froude number
$F$, herringbone patterns emerge {\it directly} through 2d shock instability under highly localized 
(Gaussian or compact support) perturbation as $F$ is increased to a point of instability, 
{\it before} the appearance of 1d shock instability and roll waves.
And, this phenomenon has nothing to do with front evolution, but corresponds to ``essential'' bifurcation/instability
of the (constant) far-field upstream limit.

\subsubsection{Hydraulic shocks, metastability, and bifurcation}\label{ss:shocks}
We discuss in more detail the above-described shock bifurcation and its relation to spectra and stability.
In simplest terms, one could view these candidate asymptotic states as arising through a (discontinuous) 
Turing-type bifurcation at $F=2$ of {\it either} doubly periodic herringbone waves or singly 
periodic roll waves emerging from the constant limiting state, 
connected to the constant right endstate by a modulating front.
But, this does not address the question of convective stability for $F>2$, or behavior with respect to localized
perturbations: in particular, when and how instability will result in convergence toward one or the other pattern
under localized perturbation.

The initial discussion above concerns observation of numerical time-evolution studies.
A fuller understanding is given by combining numerical Evans function studies with convective stability
conditions on the limiting end states of the shock. The Evans-Lopatinsky 
function detects ``near-field'' instabilities,
or eigenfunctions of the eigenvalue system Fourier transformed in the invariant $y$ direction.
An exhaustive numerical study yields that for multi-d stability, as seen for 1d in \cite{YZ,FRYZ},
there exist no unstable roots of the Evans-Lopatinsky function, hence no near-field instabilities.

Meanwhile, a detailed algebraic study of the Fourier symbol about the constant end-states yields
explicit boundaries for 1 and 2d convective instability of the end-states, 
thus completely deciding stability.
The result is that there exists between $F=2$ and the 1d convective stability boundary $F_{\text{1d}}(\nu)$ \eqref{F_1d}
a 2d convective stability boundary 
\be
\label{F_2d}
F_{\text{2d}}(\nu):=\frac{\nu(\nu+1)}{\sqrt{\nu^2+\nu-1}},
\ee 
explicitly computable.
Thus, for $F_{\text{2d}}(\nu)<F<F_{\text{1d}}(\nu)$ and $\nu$ held fixed, hydraulic shocks are 1d convectively 
stable but 2d convectively {\it unstable}, with a transverse essential bifurcation at $F=F_{\text{2d}}$.
In this range, perturbed shocks, being stable in 1d, cannot converge to roll waves, so must converge if
they do at all, to some genuinely multidimensional state, such as the herringbone patterns we see
in time-evolution simulations.

However, though satisfyingly explicit, this still does not completely describe what we see in
time-evolution experiments under highly localized perturbations. For, there, we find another value
$F_{\text{1d}}< \tilde F < F_{\text{2d}}$ such that for $F_{\text{1d}}<F<\tilde F$, a shock perturbed by a highly localized
(i.e., Gaussian decay or compact support) perturbation will still, after a brief $\mathcal{O}(1)$ excursion, converge
back to an undisturbed hydraulic shock, in original unweighted norm.
This interesting ``metastability'' phenomenon is an example of a more general phenomenon
studied in \cite{FHSS}, where perturbations in weighted norm give stability in a mixed weighted/$L^\infty$
norm applied to different pieces of the solution, and the latter property should be rigorously verifiable
by the techniques there introduced.

Moreover, it appears to be also an example of something more novel, but not yet verified.
Namely, if in a channel one perturbs {\it any} constant state with $F>F_{\text{2d}}$ by a Gaussian-decaying perturbation,
one sees emergence of a ubiquitous pattern consisting of herringbone to the left and roll waves to the right,
each moving at distinct speeds, separated by a modulating front with yet a third speed, and further modulating
fronts at left and right edges separating the entire pattern from the constant state.
This sort of pattern resembles a multi-d version of the ``unstable diffusion wave'' described for 1d in 
\cite[Appendix B]{OZ2}, \cite{AMPZ4}, namely, a nonlinear saturation/coarsening of growing oscillations 
created by linear instability of the constant state that after a sufficiently long time stabilizes at 
a fixed $\mathcal{O}(1)$ amplitude bounded above and below.

If the leftmost modulating front has speed greater than the hydraulic shock, then this pattern will eventually
be absorbed by the shock; however, it may take arbitrarily long, depending how far the left the initial
perturbation is located, yielding metastability rather than stability. (Imposing an exponentially weighted
norm of sufficiently rapid decay then yields stability in mixed norm as in \cite{FHSS}.)
If the middle modulating front but not the leftmost one 
has speed greater than the shock, then one sees convergence to a herringbone/shock pattern.
If neither of these has greater speed, then one sees convergence to a roll wave pattern on any finite domain, 
with a piece of herringbone pattern moving away to negative infinity.
And, indeed, we find that this heuristic effectively predicts the metastability boundary $\tilde F$ seen
in time-evolution simulations.  
This completes the discussion of shock stability and bifurcation, rigorous and nonrigorous.
See Section \ref{s:hydro_bif} for further discussion.

\subsubsection{Roll waves}\label{ss:rollwaves}
As regards roll waves, the main practical and mathematical issues are still best framed by Brock 
in the foundational work described in his Caltech doctoral thesis \cite{Brock}:

(\cite[p. 1 - p. 3]{Brock})
``{\it The maximum depth of flow in a roll wave train must necessarily
be greater than the normal or undisturbed depth. Thus a prismatic
channel designed to convey a discharge at normal depth, may not be
capable of conveying this same discharge with roll waves present
\cite{Holmes} in a channel
that had water overtopping the 8 ft side walls when the discharge was
estimated to be less than 25 per cent of the design discharge. Thus
the practical need for understanding the mechanics of roll waves is
clear.

In general one would like to be able to predict whether a channel
will exhibit roll waves for any particular discharge. Furthermore,
if roll waves are to be present, it is desirable to know where they will
start, and how high they will be at any section of the channel...}'' 
\begin{figure}[htbp]
\begin{center}
$
\begin{array}{lll}
\includegraphics[scale=0.13]{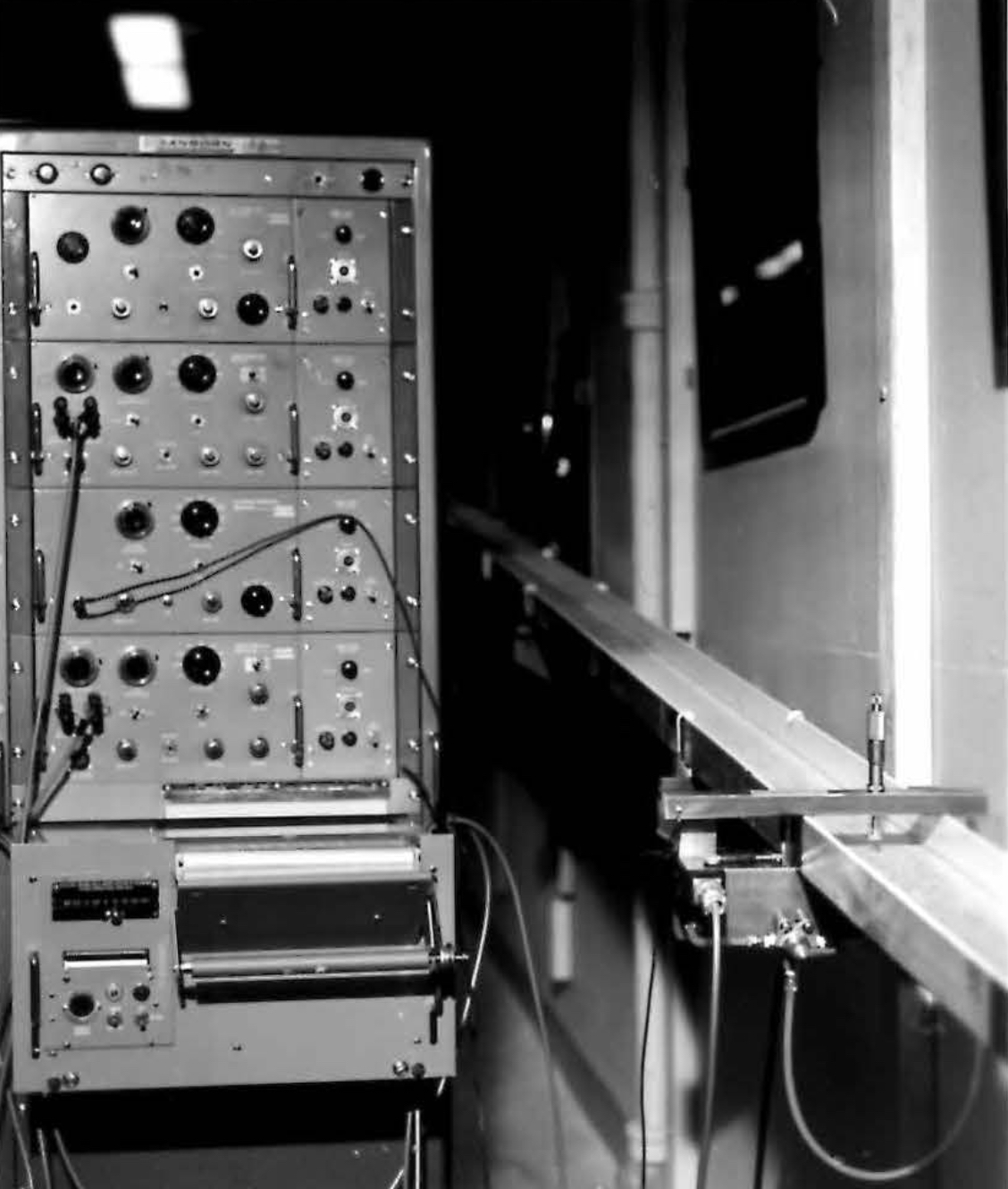}\;
\includegraphics[scale=0.13]{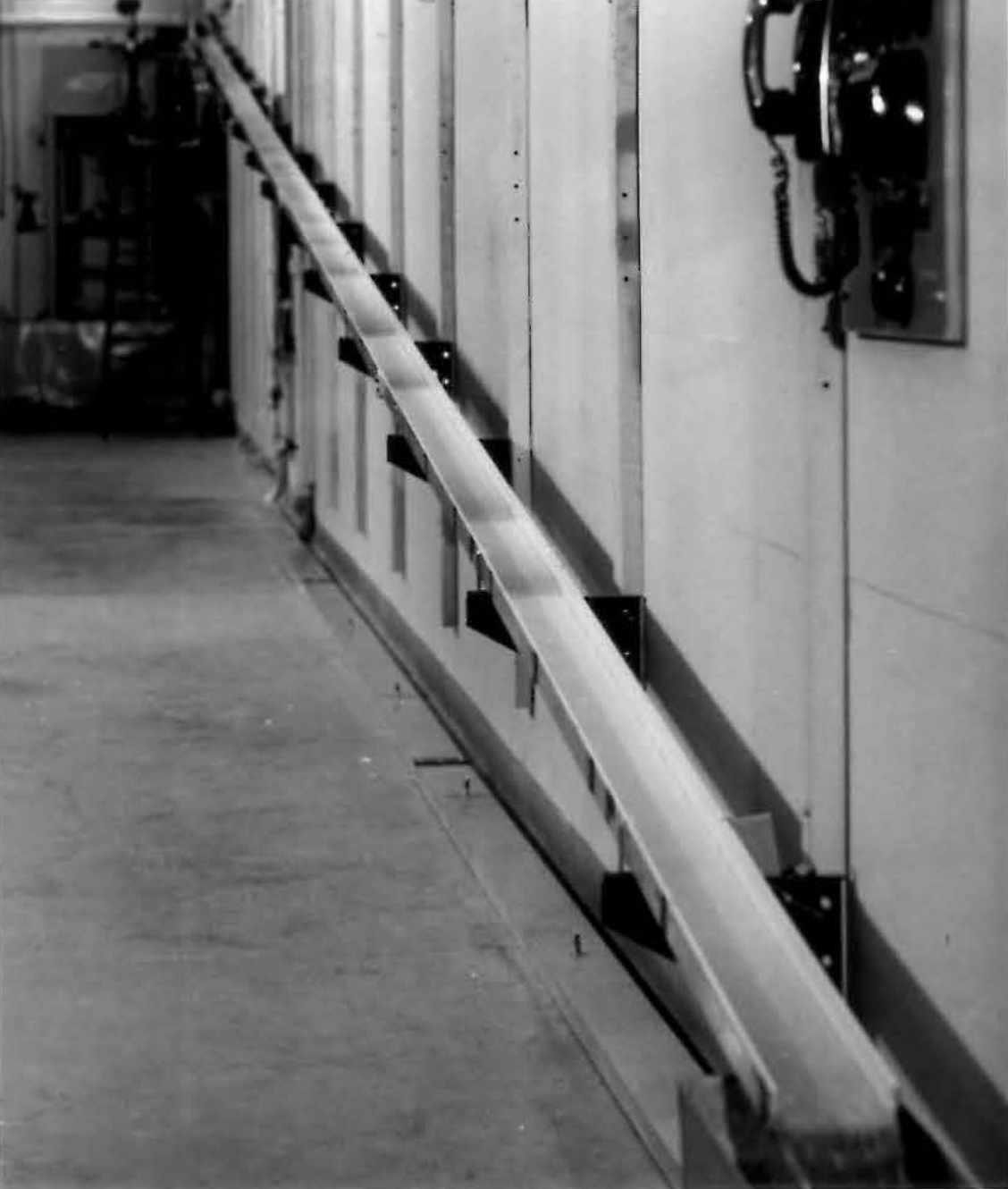}\;
\includegraphics[scale=0.13]{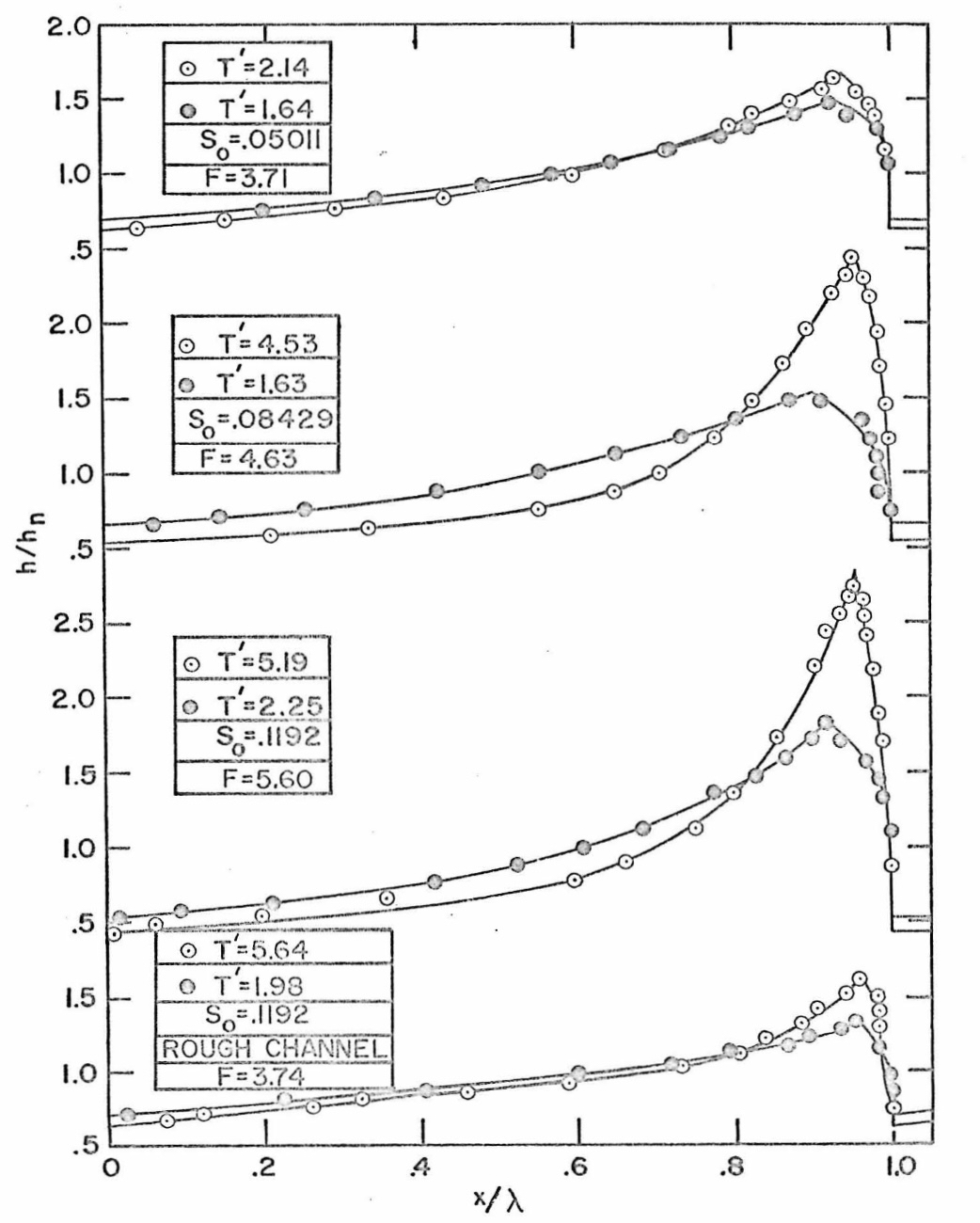}\;
\end{array}
$
\end{center}
\caption{Experimental apparati and results, reproduced from Brock's thesis. 
Left: \cite[Fig. 16.]{Brock}; 
Middle: \cite[Fig. 24.]{Brock};
Right: \cite[Fig. 45.]{Brock}}
\label{fig:brockpic_fig}
\end{figure}

\begin{figure}[htbp]
\vspace{-0.1cm}
\begin{center}
\includegraphics[scale=.42]{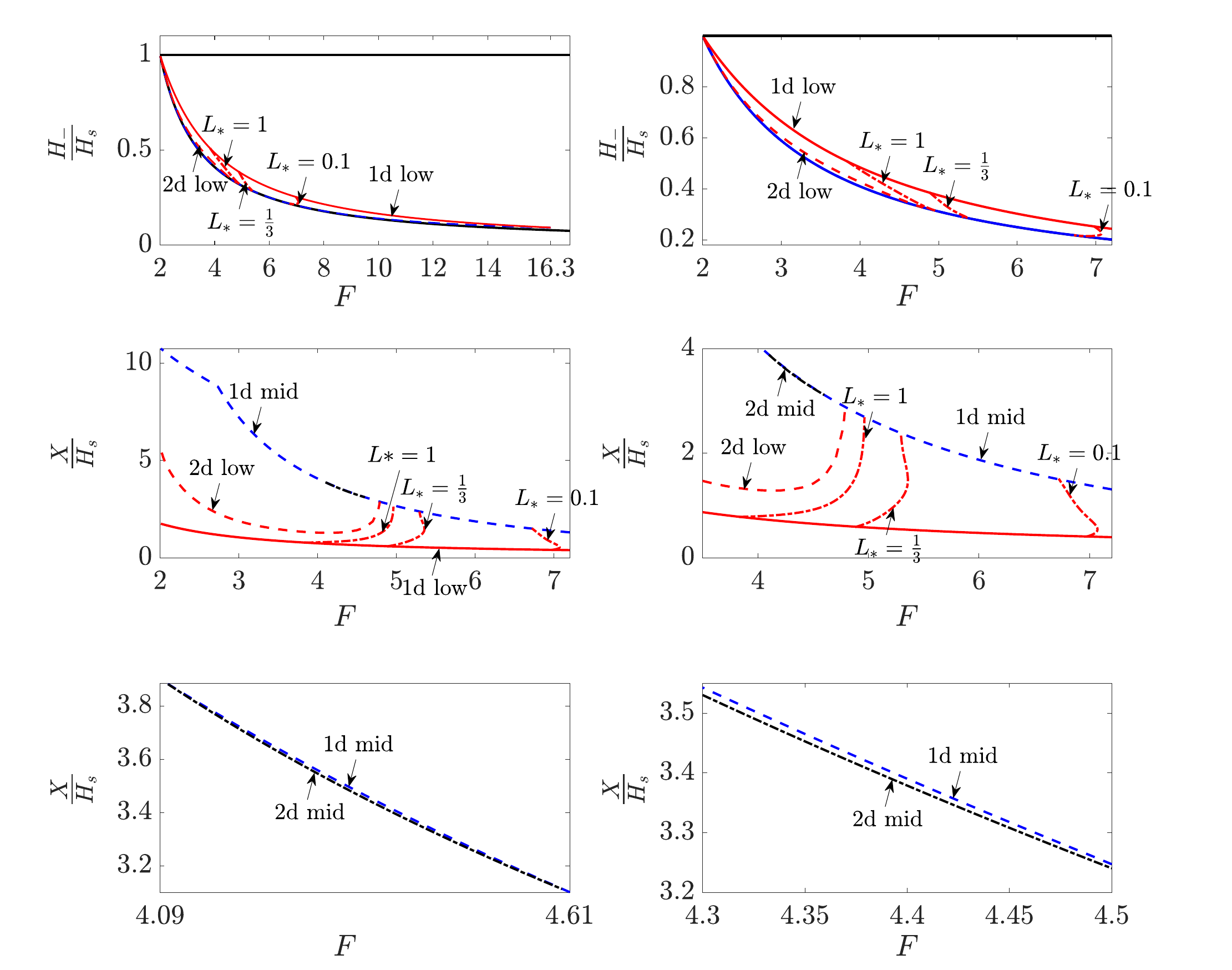}
\end{center}
\vspace{-0.4cm}
\caption{2d versus 1d stability diagrams with respect to parameters $(F,H_-/H_s)$ in the top two panels and $(F,X/H_s)$ in the middle and bottom four panels. The right panels are blowup of the left ones. The 2d stable regions, both in the whole space and a channel, are subsets of the 1d stable region (see FIGURE~\ref{fig_roll_1d}). Specifically, the stable region for roll waves in the whole space is between the 2d low-frequency stability boundary (red dash curve with label ``2d low'') and the 2d medium-frequency stability boundary (black dash-dotted curve with label ``2d mid''). The latter boundary coincides with the 1d medium-frequency stability boundary (blue dash curve with label ``1d mid'') for $F\in(2,4.09\ldots]\cup [4.61\ldots,16.3\ldots)$ and sits lower than the ``1d mid'' boundary for $F\in(4.09\ldots,4.61\ldots)$. See blowup in the bottom two panels. The ``2d low'' curve crosses the ``1d mid'' curve at $F=4.78\ldots$. The interior of the complement of the 2d stable region corresponds to the unstable region for roll waves in the whole space. As for roll waves in a channel, we compute, for $c=1$ or $\frac{1}{3}$ or $0.1$, a stability boundary for channel roll waves with critical channel width $L_*=c$ (red dash-dotted curve with label ``$L_*=c$''). The curves intersect the ``1d low'' boundary at $F=3.83\ldots$, $4.88\ldots$, and $6.97\ldots$ and the ``1d mid'' boundary at $F=4.96\ldots$, $F=5.29\ldots$, and $6.72\ldots$, for $c=0.1$, $\frac{1}{3}$, and $1$ respectively. The curve ``$L_*=c$'' sub-divides the 1d stable region into a left sub-region and a right sub-region.  Consider roll waves in a channel with width $L$. The condition $L<c$ implies stability in the left sub-region and the condition $L>c$ implies instability in the right sub-region.}
\label{fig:stabdiag_fig}
\end{figure}

\begin{figure}[htbp]
\begin{center}
\includegraphics[scale=.42]{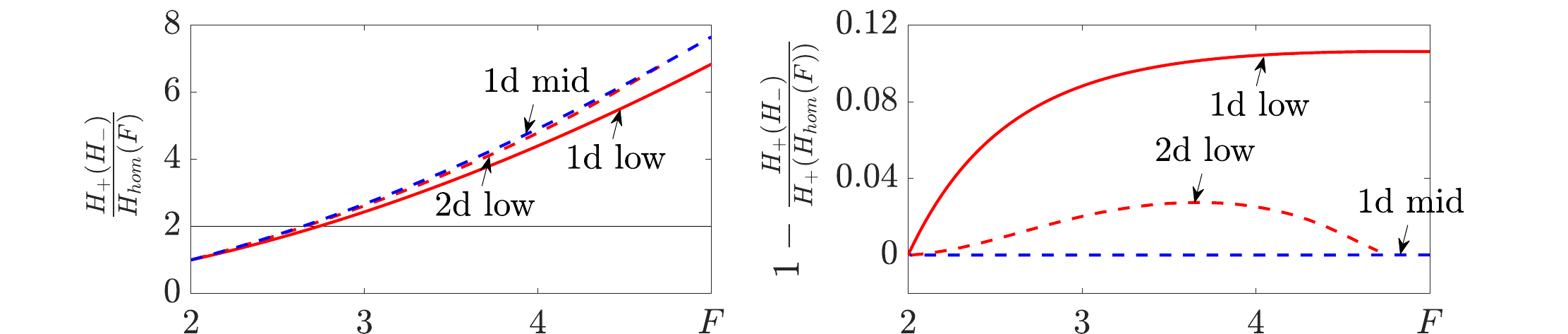}
\end{center}
\caption{Left: translations of stability boundaries depicted in FIGURE~\ref{fig:stabdiag_fig} to the parameter $H_+(H_-)/H_{hom}(F)$ versus $F$. Right: plot of the relative errors $\Big(\frac{H_+(H_{hom}(F))}{H_{hom}(F)}-\frac{H_+(H_-)}{H_{hom}}\Big)/\Big(\frac{H_+(H_{hom}(F))}{H_{hom}(F)}\Big)=1-\frac{H_+(H_-)}{H_+(H_{hom}(F))}$ versus $F$ of various stability boundaries.}
\label{fig_schon}\end{figure}

(\cite[p. 36]{Brock})
``{\it As the wave length approaches infinity the wave shape and
velocity approach a finite solution which gives rather substantial
values of $h_{max}/h_n$. This corresponds to one wave of infinite length
in a channel of infinite length.
For this limiting case $h_{min}$ approaches $h_n$.
Schonfeld \cite{Schonfeld} has claimed to have found that only the solution
with $h_{max}/h_n = 2.07$ is a stable one. However, in his work relations
were derived by assuming that there was a discontinuity in the water
surface at $h = h_n$ (in addition to the one at the shock). The above
solution has no discontinuities of this type, and therefore Schonfeld's result is doubtful. A stability analysis of the periodic permanent solution may lead to some interesting results. This remains to be done.}''

\medskip

In FIGURE~\ref{fig:brockpic_fig} left and middle panels, we reproduce from Brock's thesis pictures of his
experimental apparatus, apparently in a Caltech hallway, showing the formation of roll waves numerically recorded in the right panel.
We highly recommend the description of his innovative homemade fluids lab and instruments in \cite{Brock},
along with discussion of the issues and obstacles in obtaining such waves.
On this topic we repeat a possibly apocryphal anecdote shared by Jerry Bona when asked by the second author where to find such waves in nature, wherein a pair of British experimentalists, after an unsuccessful
day trying to produce roll waves on a nicely polished stainless steel water ramp, retreat through
the British rain to a pub to regroup, only to find beautifully formed roll waves marching down the cobblestone
street out the window.

Indeed, a rough bottom, produced by Brock by a mixture of sand and paint, appears to be critical for
the spontaneous formation of roll waves (the ``natural'' type of \cite{Brock}, as opposed to those
produced by a wave-generator in the form of oscillating upstream paddle).
And direct observation reveals that roll waves appear in force during a rain on any rough and 
gently inclined road. 

\medskip

Let us start then with the two key questions raised by Brock, above, and what our investigations 
add to the state of the art, starting at the end with Schonfeld's conjectured stability criterion.

In FIGURE~\ref{fig:stabdiag_fig}, we present graphically our main results on roll wave stability, superimposing
the 2d stability diagram obtained here upon the 1d stability diagram  (see FIGURE~\ref{fig_roll_1d}) obtained previously in \cite{JNRYZ}. In FIGURE~\ref{fig_schon}, we display the range of $h_{max}/h_{n}$, or $H_+/H_{hom}$, on various stability boundaries $(F,\frac{H_-}{H_s}(F))$ of 
the 1 and 2d stability regions, for various $F$, including the range $F\in [2.5,6]$ relevant to
hydrodynamic engineering and the experiments of Brock.
Contrary to the conjectured stability criterion $h_{max}/h_{n}\approx 2.07$ of Schonfeld \cite{Schonfeld}, 
the values of $h_{max}/h_{n}$ in the stable region are seen to depend substantially on $F$, and
are not close to $2.07$ except for $F$ in a small subinterval of $(2.7,2.8)$.
And, indeed, within the 1d stability region, they vary substantially even for fixed $F$.
On the other hand, within the smaller 2d stability region, the values for fixed $F$ appear 
to be nearly constant, recovering something like an $F$-dependent version of the Schonfeld criterion, this time based on rigorous analysis.

Moreover, from the right panel of FIGURE~\ref{fig_schon}, we see that the value of $h_{max}/h_{n}$ 
for stable (in the whole space) planar waves is well-approximated by its value in the homoclinic limit. 
Thus, for practical purposes it may be computed explicitly \cite{JNRYZ}.
Continuing this line of reasoning, one may conjecture that for a canal with fixed fluid injection rate at
one end, the height of any resulting stable roll waves can be essentially determined as that of the associated homoclinic wave of the same fluid flux; this is yet to be determined.
The period of these roll waves is a more delicate question; however,
it, too, may be read off from the stability diagram of FIGURE~\ref{fig:stabdiag_fig}.
The infinite-period homoclinic limit originally discussed by Shonfeld is always {\it unstable}, it should
be noted.\footnote{As mentioned in the quote from Brock above, the actual homoclinic limit
involves only a single shock, perhaps explaining this discrepancy with the heuristic investigation of Shonfeld based on a two-shock configuration.}

\medskip

A new aspect of our analysis, so far as we know not appearing in the previous literature, is the consideration of channel width and its effect on behavior. In standard references, e.g., \cite{D,Brock}, it is usually mentioned only that the channel should be ``wide enough'' for appearance of roll waves.
Our guess is that what is meant is wide enough that viscous boundary-layer effects are negligible on the side walls of the channel.

However, our more detailed treatment of channel boundary conditions and spectra show that the effect of 
channel width $L$ is to sample from continuous Fourier wave numbers $\eta\in \R$ the discrete
lattice $\pi n/L$, where $n$ runs through the integers, restricting Fourier-Floquet spectra 
$$
(\eta,\xi,\lambda) \in \R \times \R \times \C 
$$
for the whole-plane stability problem, where $\eta$ is Fourier frequency in the $y$ direction and
$\xi$ Floquet/Bloch number, to these values of $\eta$.
Thus, 2d stability in the channel is {\it less restrictive} than 2d stability in the whole space,
with the two conditions agreeing in the infinite-width limit $L\to \infty$.

In the opposite, zero-width limit $L\to 0$, ignoring non-modeled effects of viscosity, we find that
2d stability in the channel reduces to 1d stability.
In reality, the flow would likely be ``choked'' for $L$ small enough by viscous effects not modeled in (SV);
hence the physical rule of thumb that channel width should not be too small.
The argument just given explains why, nonetheless, we may expect essentially 1d behavior for channel
width that is ``not too small,'' but also ``not too big.''

An interesting consequence, as discussed in \cite{Z7} for the related issue of periodic channel flow,
is that waves that are 1d stable but 2d unstable in the whole plane may be expected to undergo
a series of transverse bifurcations as width $L$ increases from $0$, here considering $L$ as
a bifurcation parameter. 

In the traveling front case \cite{Z7}, sampling changes the spectrum from continuous to discrete,
making the bifurcation a classical finite-eigenvalue bifurcation.
Here, due to periodicity in $x$, one expects rather a Floquet-type bifurcation to quasiperiodic flow,
as we indeed observe numerically, in the form of an amusing ``flapping'' bifurcation, where roll wave
shock fronts change inclination time-periodically, with slightly shifted phase:
a roll wave ``doing the wave.''
As discussed in Section \ref{s:channel_flow_roll}, this may be deduced as the generic transition to
instability for roll waves in a channel, analogous to cellular or spinning instabilies for shock
and detonation waves in a rectangular or cylindrical duct \cite{TZ2,TZ4,Z5,KS}.

From these considerations, it is clear that channel width can be an important parameter in open channel flow.
In Appendix \ref{s:brock}, we compute 
dimensionless values of channel width for Brock's experimental data,
conveniently tabulating these useful data for the first time; see Tables \ref{tway1}-\ref{tway6}.
The majority of these rescaled widths are approximately $0.1$ or $1/3$.
Consulting the stability boundary for $L_*=1/3$ in the top right panel of FIGURE~\ref{fig:stabdiag_fig},
we see that within the range of Froude numbers $F\leq 5$ for which most of Brock's experiments take place,
transverse instabilities do not come into play; for $L_*=0.1$, they are irrelevant on the 
entire range $F\leq 6$ of the experiments.
This leaves just a 
single experiments in the range $L_*\approx 1/3$, $5\leq F\leq 6$ for which
transverse instabilities could in principle play a role, namely case 9 as displayed in TABLE~\ref{brock_table},
for which $F=5.60$ and the rescaled width is $\tilde L\approx 0.27$.
However, the wave in question, under the standard physical parametrization displayed in TABLE~\ref{tway2},
lies in the 1d unstable region, so that transverse instability is again irrelevant.
Thus, in all cases, {\it the experiments carried out by Brock lie in the essentially 1d regime where
1d and 2d stability coincide.}
Of course, these experiments were limited in width by the nature of Brock's experimental apparatus, consisting
of off-the-shelf lengths of aluminum rain gutter.
It would be very interesting to carry out experiments in a larger rescaled-width regime for which 2d
effects genuinely come into play, in particular to look for physical examples of theoretically
predicted flapping instabilities.

\subsubsection{Global asymptotics}\label{ss:global}
An interesting feature of the 1d analysis, as seen in numerical experiments of \cite[\S 6]{FRYZ},
is that 1d stability of hydraulic shocks, even though an inherently local condition, in fact
determined global asymptotics with respect to ``dam-break'', or Riemann-type initial data,
the latter consisting of a single hydraulic shock in the stable case, and in the unstable case of an 
``invading front'' consisting of a roll-wave pattern impinging from the left upon a constant righthand state.

In the present 2d setting, as described above in Section \ref{ss:shocks}, 
the situation is a bit more subtle, involving metastability considerations along with local shock stability.
However, aside from this subtlety, the situation is much as before, with the addition of the new herringbone
patterns arising when planar symmetry is relaxed: namely, when shocks are meta-stable, the asymptotic behavior
is a single hydraulic shock, while otherwise there results a pattern of either an invading roll-wave or invading
herringbone front.

Notably, the invading herringbone pattern arises first as Froude number $F$, considered as a bifurcation parameter,
is increased, with invading roll-waves appearing last.  This perhaps explains the more ubiquitous appearance
in nature of herringbone as compared to roll-wave flow, at least in the informal observations of the authors.

\subsection{Discussion and open problems}\label{s:disc}
In this paper, we have opened the study of multi-d hydraulic shock and roll wave stability, both
in the whole plane and in the physical setting of an open channel.
The results seem quite interesting, both in terms of the new tools required, and of the new phenomena observed.
Particularly striking we feel are: the introduction to shallow water flow of Airy-type WKB analysis
of high-frequency stability following the groundbreaking but not widely-known work of Erpenbeck in detonation;
the careful treatment of wall-type boundary conditions and associated bases in $L^2$;
and the first numerical multi-d Evans solver for a physical relaxation system. We note that the latter
is only one of two or three such analyses carried out for viscous or relaxation models in continuum mechanics,
where they have proved quite resistant to computation \cite{HLyZ,BMZ}.
Our observations on bifurcation, meta-stability, and transition to herringbone flow
yield both new physical insights and a set of new and interesting problems for future rigorous analysis.

Here, our investigations combined spectral stability and numerical evolution to yield insight into
nonlinear behavior.
A central open problem on the theoretical side is to establish also a Lyapunov-type theorem as done in 
the 1d case in \cite{YZ,FRYZ}.
We note that these, along with the very recent \cite{FR1,FR2,DR1,DR2} are the first time-asymptotic stability
results to our knowledge in the setting of hyperbolic balance laws concerning traveling waves containing
a shock-type discontinuity.
Recall that the fundamental results of Kreiss, Majda, M\'etivier, and others \cite{K,M1,M2,Met3}
on 1- and multi-d stability of shock and boundary problems concern, rather, {\it short-time} stability, 
or well-posedness.

Such a result would be highly desirable both in the whole space and, as emphasized in the present context,
in a channel with periodic or wall-type boundary conditions.
The latter should be much easier, since, as noted in the smooth setting in \cite{TZ4}, only the zero Fourier 
mode lacks a spectral gap, so that behavior should be exponentially convergent to that of the 1d case.
The main technical issue in either case seems to be generation of a ``nonlinear damping estimate'' controlling
regularity; see \cite{YZ,FRYZ,RZ2,WZ} for discussion of such estimates, and relevant recent results.

As noted earlier, another very interesting direction for further investigation is the rigorous treatment of the
various bifurcations observed numerically, by the development of new theory accomodating discontinuous 
background waves.
These include both essential bifurcations and the eigenvalue-type ``flapping'' bifurcations observed for roll waves.
At the same time it would be quite interesting to carry out physical roll wave experiments like those of 
\cite{Brock} for larger channel widths in which transverse instability/flapping bifurcation may be expected to arise.
Tabulation of rescaled channel widths for other existing hydrodynamic data would likewise be of interest.

An equally fundamental theoretical problem seems to be the pointwise description of the 
Green's function for the linearized
system about a constant equilibrium state, identifying distinct ``proto-herringbone'' and ``proto-roll wave''
regions of instability.
For, based on our experiments, this mechanism appears to be the true driver for the herringbone bifurcation
from hydraulic shocks, which in turn appeared to be the main mechanism for formation of herringbone flow in
our experiments.
This is not to say that roll waves could not eventually transform via Majda-Rosales or other type 
evolution to final herringbone pattern.
But, at least the initial (``flapping'') bifurcation is not of this type.
And, we did not see evidence of such singularity formation in our time evolution studies.
These were hardly exhaustive, though, and this question perhaps deserves further, more systematic investigation.

Note that the Green function description should be done on a channel with periodic or wall-type boundary conditions,
to match experiments.
It is an interesting question whether this changes behavior of the Green's function substantially as compared to
the corresponding computation on the whole plane.
Certainly one may expect that channel width plays a role in selecting the period of a herringbone pattern, but
is it also essential in the development of distinct regions/types of instability?
A related problem is to study behavior of unstable shocks in the whole plane, rather than in a channel,
and also to perform numerical time-evolution studies of such.
We note that numerical time evolution on the whole plane is substantially more computation-intensive,
including also new issues of boundary conditions/truncation.

Two interesting side topics are the appearance in the whole plane problem for \eqref{sv_intro}
of new ``oblique roll waves'' and their role in proposed formal Whitham modulation equations, discussed 
in Appendix \ref{s:oblique} and Section \ref{s:rollformal}.
These represent two more interesting directions for further investigation.

Finally, we return to the original problem posed by Brock, of determining the height of persistent (stable)
roll waves.
As noted by Brock, there is a substantial overshoot in peak height for the mathematical ``Dressler wave''
solutions (SV) as compared to experimentally observed profiles.
This was addressed in \cite[Ch. VII-C]{Brock} with excellent physical agreement,
by the introduction of a turbulent boundary layer at the wave crest in place of the sharp peak of the Dressler theory.
Quite recently, Richard and Gavrilyuk have introduced in \cite{R,RG1,RG2} an enlarged shallow water model
including effects of turbulent vorticity along the channel bottom and at the shock, recovering from 
first principle the same remarkable agreement obtained by the phenomenological adjustment of Brock.
We view this new model as one of the most exciting developments in the area for many years,
and the extension of our ideas/investigations to this larger context as a particularly
exciting direction for future study. See \cite{RYZ1,RYZ2,RYZ3} for investigations in the 1d case.

\smallskip

\noindent{\bf Acknowledgement:} 
Thanks to G. Faye and L.M. Rodrigues for useful conversations and insights in the course of our
concurrent investigations \cite{FRYZ} on nonmonotone hydraulic shocks for (SV) in the 1d case.
The numerical computations in this paper were carried out in the CLAWPACK and MATLAB environments; 
analytical calculations were checked with the aid of MATLAB Symbolic Math Toolbox. Thanks to Indiana Universities University Information Technology Services (UITS) division for providing the Quartz \& Big Red 200 supercomputer environment in which most of our computations were carried out. This research was supported in part by Lilly Endowment, Inc., through its support for the Indiana University Pervasive Technology Institute. 

\smallskip

\noindent{\bf Data availability:} Matlab codes for Evans solvers and Python codes for time-evolution solvers are available at \cite{YZcode}. Time-evolution movies corresponding to snapshots in presented figures are available at \cite{YZmovie}.

\smallskip

\noindent{\bf Glossary of terms:} By a planar wave, we mean a wave that depends only on a single spatial
variable, without loss of generality $x$. In this paper, we consider planar waves both in the whole
space and in a channel. We refer to stability of planar waves in the whole space simply as 2d
stability, following this convention also in labeling of stability boundaries, e.g., ``2d low-frequency boundary''. We indicate
stability in a channel by additional comment where it occurs.
Likewise, unless otherwise indicated, spectrum of the linearized operator about the wave refers to
spectrum in the whole space, with spectrum in the channel indicated by further descriptions.

\section{Multi-d spectral stability of hydraulic shocks}\label{s:mult_hydro} 
We begin with the most accessible case of spectral stability of a hydraulic shock: more generally, a solution 
\ba\label{rshock}
U(x,t)&=\bar U(x-st), \quad x\gtrless st,\quad [F_1(\bar U)-sF_0(\bar{U})]_{x=st}=0,
\ea
with limits $U_\pm$ as $x\to\pm \infty$ of a 2d general relaxation system
\be\label{req}
F_0(U)_t+F_1(U)_x+ F_2(U)_y=R(U)=(0,r(U)^T)^T, \quad U\in \R^N
\ee
augmented (see, e.g., \cite{Er1,Er2,M1,M2,ZS}) with Rankine-Hugoniot jump conditions 
\be\label{RH}
[\psi_t F_0(U) + \psi_y F_2(U)- F_1(U)]_{x=\psi(y,t)}=0,
\ee
at curves of discontinuity $x=\psi(y,t)$, where $[h]_{x=\psi(y,t)}:=h(\psi(y,t)^+,y,t)-h(\psi(y,t)^-,y,t)$.

Changing to comoving coordinates $x\to (x-st)$, $F_1\to (F_1-sF_0)$, we have the profile equations
\be\label{rels}
\bar \psi\equiv 0, \quad  (d/dx)F_1(\bar U(x))= R(\bar U).
\ee

For planar hydraulic shocks, either in the whole space or a channel, taking without of generality $v\equiv 0$, we quickly find that the profile equations \eqref{rels} 
reduce to the one-dimensional version of \cite{YZ}. Hence, profiles can be set either continuous or discontinuous in two dimensions as $(\bar{h}(x,y),\bar{u}(x,y),\bar{v}(x,y))_{\text{2d}}=(\bar{h}(x),\bar{u}(x),0))_{\text{1d}}$, where the right hand side are the 1d hydrodynamically stable hydraulic shock profiles constructed in \cite[\S 2]{YZ} for $F<2$ and the 1d hydrodynamically unstable hydraulic shock profiles (including non-monotone ones) constructed in \cite[Proposition 2.1]{FRYZ} for $F\geq 2$.

Considering a nearby solution $\tilde U=\bar U +V$, $\tilde{\psi}=\bar \psi+\psi$, and recalling $\bar \psi=0$,
we obtain by formal linearization of \eqref{req}-\eqref{RH} the interior and boundary equations
\be\label{lineqs_n}
A_0 V_t+\psi_tA_0(\bar{U}_x)+(A_1V)_x+ (A_2 V)_y=EV, \qquad
[\psi_t F_0(\bar U)+ A_1V   + \psi_y F_2(\bar U)]_{x=0}=0,
\ee
where $A_j=dF_j(\bar U)$, $E=dR(\bar U)$, or in ``good unknown" $V\to V+\psi\bar{U}_x$ \cite{JLW,Z1,Z2,JNRYZ,YZ},
\be\label{lineqs}
A_0 V_t+(A_1V)_x+ (A_2 V)_y=EV, \qquad
[\psi_t F_0(\bar U) + A_1V+ \psi_y F_2(\bar U)  -\psi A_1\bar{U}_x]_{x=0}=0.
\ee
Denoting by $\hat V$ and $\hat \psi$ the Laplace/Fourier transforms in $t$/$y$ of $V$ and $\psi$,
and using \eqref{rels}, we reduce to the generalized eigenvalue system
\be\label{geig}
\lambda A_0 \hat V +(A_1\hat V)_x+ i\eta A_2\hat V =E\hat V, \qquad
\hat \psi[\lambda F_0(\bar U) +i\eta  F_2(\bar U) - R(\bar U)]+[A_1 \hat V]=0.
\ee
Since profiles are non-characteristic, i.e., $\det(A_1)\neq 0$, \eqref{geig} becomes for $Z:=A_1\hat{V}$
\be 
\label{zgeig}
Z_x=(E-\lambda A_0-i\eta A_2)A_1^{-1}Z=:G(x;\lambda,\eta)Z, \qquad
\hat \psi[\lambda F_0(\bar U) +i\eta  F_2(\bar U) - R(\bar U)]+[Z]=0.
\ee 
In flux coordinates $U=(h,q=hu,p=hv)$,  $A_0=\Id_3$, $A_1$, $A_2$, and $E$ explicitly read
\ba \label{A1A2E}
A_1=&\left[\begin{array}{rrr}-s&1&0\\H/F^2-Q^2/H^2&2Q/H-s&0\\0&0& Q/H - s\end{array}\right],\\
A_2=&\left[\begin{array}{rrr}0&0&1\\0&0&Q/H\\H/F^2&0&0\end{array}\right],\quad\text{and}\quad E=\left[\begin{array}{rrr}0&0&0\\2Q^2/H^3 + 1& -2Q/H^2&0\\0&0&-Q/H^2\end{array}\right],
\ea 
where $(H,Q,0)=\bar{U}$ denotes a hydraulic shock profile.

\begin{lemma}[Transverse symmetry]\label{transverse_sym_hydraulic} A solution to the generalized eigenvalue system \eqref{geig} is invariant under the coordinate change
$\big(\lambda,\eta,\hat{\psi},[\hat{V}_1,\;\hat{V}_2,\;\hat{V}_3]\big)\rightarrow \big(\lambda,-\eta,\hat{\psi},[\hat{V}_1,\;\hat{V}_2,\;-\hat{V}_3]\big)$.
\end{lemma}

\begin{proof}
By \eqref{A1A2E}, \eqref{geig}(i) explicitly reads
$$
\begin{aligned}
\lambda \hat{V}_1+(-s\hat{V}_1+\hat{V}_2)_x+i\eta\hat{V}_3&=0,\\
\hat{V}_2+((H/F^2-Q^2/H^2)\hat{V}_1+(2Q/H-s)\hat{V}_2)_x+i\eta(Q/H)\hat{V}_3&=(2Q^2/H^3+1)\hat{V}_1-2(Q/H^2)\hat{V}_2,\\
\lambda \hat{V}_3+((Q/H-s)\hat{V}_3)_x+i\eta(H/F^2)\hat{V}_1&=-(Q/H^2)\hat{V}_3,
\end{aligned}
$$
and \eqref{geig}(ii) explicitly reads
$$
\begin{aligned}
\hat{\psi}[\lambda H]+[-s\hat{V}_1+\hat{V}_2]&=0,\\
\hat{\psi}[\lambda Q-H+Q^2/H^2]+[(H/F^2-Q^2/H^2)\hat{V}_1+(2Q/H-s)\hat{V}_2]&=0,\\
\hat{\psi}[i\eta H^2/2F^2]+[(Q/H-s)\hat{V}_3]&=0.
\end{aligned}
$$
It is then easy to see that \eqref{geig} is invariant under the proposed symmetry. 
\end{proof}

\smallskip

\noindent\textbf{Scale-invariance, hydraulic 
shocks.} In the rest of Section~\ref{s:mult_hydro}, we adopt the useful scale-invariance \cite[Observation 2.4]{YZ}, which reduces all conclusions, up to scale-invariance, to the hydraulic shocks with $H_L=1$ and $H_R=1/\nu^2$. Here $H_L$ ($H_R$) denotes the limiting fluid height at $-\infty$ ($+\infty$).

\subsection{Essential Spectra.}\label{essential}
In the unweighted $L^2$ space, the essential spectra of a non-characteristic planar hydraulic shock profile sit to the left of the rightmost dispersion relations of the limiting matrices $G(\pm\infty;\lambda,\eta)$,\footnote{
	More precisely, the essential spectrum of the Fourier transformed operator with respect to $\eta$
	of the linearized operator about the wave, with point spectra
of the Fourier transformed operator corresponding to (a different type of) 
essential spectra of the linearized operator about the wave.}
the latter given by algebraic curves $\lambda_{j,\pm}(k;\eta)$, $j=1,2,3$ consisting of zeros of the cubic polynomial $\det(G(\pm\infty;\lambda,\eta)-ik\Id)=0$ for $k,\eta\in \mathbb{R}$. In the one-dimensional case \cite{YZ} and \cite{FRYZ}, it is proved that stability of the unweighted essential spectra of a non-characteristic hydraulic shock profile is equivalent to the hydrodynamic stability condition $F\leq 2$. In the current two-dimensional case, we will establish the same result in the lemma below.

\begin{lemma}
Consider a non-characteristic planar hydraulic shock profile in the whole space. If the hydrodynamic stability condition holds, i.e., $F\leq 2$, then the dispersion relations $\lambda_{j,\pm}(k;\eta)$ for $j=1,2,3$, 
defined as roots of  $\det(G(\pm\infty;\cdot,\eta)-ik)=0$, satisfy $\Re\lambda_{j,\pm}(k;\eta)<0$, for $j=2,3$ and all $k,\eta\in \mathbb{R}$ and $\Re \lambda_{1,\pm}(k;\eta)<0$ for all $k,\eta\in \mathbb{R}$ except at $k,\eta=0$ where $\lambda_{1,\pm}(0;0)=0$. 

If the hydrodynamic stability condition fails, i.e., $F>2$, then $\sup_{k\in \mathbb{R}}\Re\lambda_{1,\pm}(k;0)>0$, implying existence of unstable essential spectra.
\end{lemma}

\begin{proof}
We make computations to find where and when dispersion relations intersect the imaginary axis, or equivalently, to investigate the solvablity of $\det(G(\pm\infty;ib,\eta)-ik)=0$ for $b,k,\eta\in\mathbb{R}$. 

At $x=+\infty$, $\det(G(+\infty;ib,\eta)-ik)=0$ is equivalent to
$$
\begin{aligned}
&\mathrm{Re}\det(G(+\infty;ib,\eta)-ik)=\frac{\nu + 1}{\nu^4(F_{\text{char}}^2-F^2)}\Big(3F^2\nu^2(\nu + 1)^2b^2\\&+F^2k\nu(- 6\nu^3 - 5\nu^2 + 2\nu + 1)b-F^2k^2\nu^2(- 3\nu^2 + \nu + 1) -(2\eta^2 + k^2)(\nu + 1)^2\Big)=0
\end{aligned}
$$
and
$$
\begin{aligned}
&\mathrm{Im}\det(G(+\infty;ib,\eta)-ik)=\frac{1}{\nu^5(F_{\text{char}}^2-F^2)}\Big(F^2\nu^2(\nu + 1)^3b^3\\
&-3F^2k\nu^3(\nu + 1)^2b^2-(\nu + 1)(\nu^4(- 3k^2 + 2\nu^2 + 4\nu + 2)F^2+(\nu + 1)^2(\eta^2 + k^2))b\\
&-k\nu^3(k^2\nu^2 - 2\nu^4 - 3\nu^3 + \nu^2 + 3\nu + 1)F^2 +k\nu(\nu + 1)^2(\eta^2 + k^2)\Big)=0,
\end{aligned}
$$
where $F_{\text{char}}$ is given by \eqref{characteristicF} and, for a non-characteristic hydraulic shock profile, there holds $F\neq F_{\text{char}}$. By long division, the two equations amount to
$$
\begin{aligned}
&b=\frac{-\nu^2(k^2 - 18\nu^4 + 9\nu^3 + 9\nu^2)F^2+(3\nu^2 - 2\nu - 2)\eta^2+(6\nu^2 - \nu - 1)k^2}{\nu(\nu + 1)((6 - F^2)k^2+18F^2\nu^4 + 3\eta^2)}k
\end{aligned}
$$
and
$$
\begin{aligned}
&(4 - F^2)k^6+(4-F^2)(F^4\nu^4 + 6F^2\nu^4 + 3\eta^2)k^4\\
&+\big(\eta^4(9-2F^2) + 9F^4\nu^8(4-F^2 ) + F^2\eta^2\nu^4(60-13F^2 )\big)k^2+2\eta^2(6F^2\nu^4 + \eta^2)^2=0.
\end{aligned}
$$
If $\eta\neq 0$, the left hand side of the latter equation above is great than $0$ for $0<F\leq 2$ and all $k\in\mathbb{R}$, whence the dispersion relation $\lambda_{j,+}(k;\eta\neq 0)$, for $j=1,2,3$, never intersects the imaginary axis. 

At $x=-\infty$, $\det(G(-\infty;ib,\eta)-ik)=0$ is equivalent to
$$
\begin{aligned}
&\mathrm{Re}\det(G(-\infty;ib,\eta)-ik)=\frac{\nu(\nu + 1)}{F_{\text{exist}}^2-F^2}\Big(3F^2\nu^2(\nu + 1)^2b^2-F^2k\nu(- \nu^3 - 2\nu^2 + 5\nu + 6)b\\&
-F^2k^2(\nu^2 + \nu - 3) -\nu^2(2\eta^2 + k^2)(\nu + 1)^2\Big)=0
\end{aligned}
$$
and
$$
\begin{aligned}
&\mathrm{Im}\det(G(-\infty;ib,\eta)-ik)=\frac{1}{F_{\text{exist}}^2-F^2}\Big(F^2\nu^3(\nu + 1)^3b^3-3F^2k\nu^2(\nu + 1)^2b^2\\&-\nu(\nu + 1)\big(F^2(- 3k^2 + 2\nu^4 + 4\nu^3 + 2\nu^2)+ \nu^2(\nu + 1)^2(\eta^2 + k^2)\big)b\\
&-F^2k(k^2 + \nu^6 + 3\nu^5 + \nu^4 - 3\nu^3 - 2\nu^2)+ k\nu^2(\nu + 1)^2(\eta^2 + k^2)\Big)=0,
\end{aligned}
$$
where, by \eqref{defF_existence}, $F<F_{\text{exist}}$.
The two equations amount to
$$
\begin{aligned}
&b=\frac{F^2(- k^2 - 9\nu^2 - 9\nu + 18) - \eta^2(2\nu^2 +2\nu - 3)-k^2(\nu^2 + \nu - 6)}{\nu(\nu + 1)\big((6- F^2)k^2 + 18F^2 + 3\eta^2\big)}k
\end{aligned}
$$
and
$$
\begin{aligned}
&(F^2 - 4)k^6+(F^2 - 4)(F^4 + 6F^2 + 3\eta^2)k^4\\&+\big(9F^4(F^2-4) + (13F^2-60)F^2\eta^2  + (2F^2-9)\eta^4\big)k^2-2\eta^2(6F^2 + \eta^2)^2=0.
\end{aligned}
$$
If $\eta\neq 0$, the left hand side of the latter equation above is less than $0$ for $0<F\leq 2$ and all $k\in\mathbb{R}$, whence the dispersion relation $\lambda_{j,-}(k;\eta\neq 0)$, for $j=1,2,3$, never intersects the imaginary axis. 

Therefore, if $\eta\neq 0$, the dispersion relation $\lambda_{j,\pm}(k;\eta)$, for $j=1,2,3$, stays either on the left or the right side of the imaginary axis. We make further computations and obtain that, for $F=1$, $\nu=2$, $k=0$, and $\eta=\pm 1$,
$
\det(G(+\infty;\lambda,\pm 1))=\frac{54}{7}\left(\lambda^3 + 6\lambda^2 + \frac{33}{4}\lambda + 1\right)$ and
$\det(G(-\infty;\lambda,\pm 1))=-\frac{216}{35}\left(\lambda^3 + 3\lambda^2 + 3\lambda + 2\right)$.
By the  Routh-Hurwitz criterion, a cubic polynomial $p(x)=x^3+a_2x^2+a_1x+a_0$ has all roots in the left half-plane if and only if $a_2$, $a_1$, and $a_0$ are positive and $a_2a_1>a_0$. We readily check $6\times \frac{33}{4}>1$ and $3\times 3>2$. Hence $\Re\lambda(0;\pm 1)<0$ when $F=1$ and $\nu=2$. As $(F,\nu,k,\eta)$ is continuously varied in $(0,2]\times(1,\infty)\times \mathbb{R}\times \mathbb{R}^-$ and in $(0,2]\times(1,\infty)\times \mathbb{R}\times \mathbb{R}^+$, the dispersion relation $\lambda_{j,\pm}(k;\eta)$, for $j=1,2,3$, stays in the left hand side of the imaginary axis. 

If $\eta=0$, explicit computations show that 
$\lambda_{3,+}(k;0)=-\nu+\frac{ik\nu}{\nu+1}$ and 
$$
\lambda_{j,+}(k;0)=-\nu+\frac{ik\nu}{\nu+1}-(-1)^j\frac{\sqrt{F^2\nu^4- F^2k\nu^2i - k^2}}{F\nu},\quad \text{for $j=1,2$.}
$$
If $0<F\leq 2$, we appeal to Lemma \ref{real_part:parabola} and compute
$$
\sup_{k\in\mathbb{R}}\Re\lambda_{1,+}(k;0)=-\nu+\sup_{k\in\mathbb{R}}\Re\frac{\sqrt{F^2\nu^4- F^2k\nu^2i - k^2}}{F\nu}=\lambda_{1,+}(0;0)=-\nu+\nu=0
$$
and 
$$\begin{aligned}
\sup_{k\in\mathbb{R}}\Re\lambda_{2,+}(k;0)=&-\nu-\inf_{k\in\mathbb{R}}\Re\frac{\sqrt{F^2\nu^4- F^2k\nu^2i - k^2}}{F\nu}\\
=&-\nu-\lim_{k\to \infty}\Re\frac{\sqrt{F^2\nu^4- F^2k\nu^2i - k^2}}{F\nu}=-\nu\left(1+\frac{F}{2}\right)<0.
\end{aligned}
$$
Similarly, we compute to obtain $\Re\lambda_{3,-}(k;0)\equiv-1$, $\sup_{k\in\mathbb{R}}\Re\lambda_{1,-}(k;0)=\lambda_{1,-}(0;0)=0$, and $\sup_{k\in\mathbb{R}}\Re\lambda_{2,-}(k;0)=-\left(1+\frac{F}{2}\right)<0$. In the course of the
proof of Lemma \ref{real_part:parabola}, we also find the supremum of $\Re\lambda_{1,\pm}(k;0)$ is achieved only at $k=0$. 

If $F> 2$, there hold
$$\begin{aligned}
\sup_{k\in\mathbb{R}}\Re\lambda_{1,+}(k;0)=&-\nu+\sup_{k\in\mathbb{R}}\Re\frac{\sqrt{F^2\nu^4- F^2k\nu^2i - k^2}}{F\nu}\\
=&-\nu+\lim_{k\to\infty}\Re\frac{\sqrt{F^2\nu^4- F^2k\nu^2i - k^2}}{F\nu}=\nu\left(\frac{F}{2}-1\right)>0
\end{aligned}
$$
and 
$\sup_{k\in\mathbb{R}}\Re\lambda_{1,-}(k;0)=\frac{F}{2}-1>0$.
\end{proof}
\begin{lemma}\label{real_part:parabola}
Consider the parabola $z(k)=-ak^2+ik+c$, $k\in\mathbb{R}$ with $a,c>0$. Then,
$$
\inf_{k\in\mathbb{R}}\Re\sqrt{z(k)}=\left\{\begin{aligned}&\sqrt{z(0)}&4ac\leq 1,\\
&\sqrt{z(\infty)}&4ac>1,\end{aligned}\right. \quad \text{and}\quad \sup_{k\in\mathbb{R}}\Re\sqrt{z(k)}=\left\{\begin{aligned}&\sqrt{z(\infty)}&4ac< 1,\\
&\sqrt{z(0)}&4ac\geq 1,\end{aligned}\right.
$$
where $\sqrt{z(0)}=\sqrt{c}$ and $\sqrt{z(\infty)}=\frac{1}{2\sqrt{a}}$.
\end{lemma}
\begin{proof}
For $c>0$, we readily find the parabola $\{z(k):k\in \mathbb{R}\}$ lies in the analytic branch of the square root function $\mathbb{C}\setminus(\infty,0]$. By setting $\tilde{k}:=k/c$, we obtain $z(k)=c(-ac\tilde{k}^2+i\tilde{k}+1)$ and hence $\Re\sqrt{z(k)}=\sqrt{c}\Re\sqrt{-ac\tilde{k}^2+\tilde{k}i+1}$. It then suffices to consider $\Re\sqrt{-b k^2+ik+1}$ where $b:=ac>0$. Now, denoting $\alpha=\Re\sqrt{-b k^2+ik+1}$ and $\beta=\Im\sqrt{-b k^2+ik+1}$ gives the equation $\alpha+\beta i=\sqrt{-b k^2+ik+1}$, or
$$
\alpha^2-\beta^2=-bk^2+1\quad 2\alpha\beta=k.
$$
Eliminating $\beta$ from the system above yields a quadratic equation of $\alpha^2$ given by 
$$
\alpha^4-(-bk^2+1)\alpha^2-\frac{k^2}{4}=0,
$$
which always has a positive zero and a non-positive zero, both of which depend smoothly on $(k^2,b)\in [0,\infty)\times(0,\infty)$. To find the infimum/supremum of $\alpha$, or equivalently the infimum/supremum of the positive zero, taking partial derivative of the quadratic equation against $k^2$ yields
$$
2\alpha^2\frac{\partial\alpha^2}{\partial k^2}+b\alpha^2-(-bk^2+1)\frac{\partial\alpha^2}{\partial k^2}-\frac{1}{4}=0, \quad \text{and hence} \quad \frac{\partial\alpha^2}{\partial k^2}=\frac{\frac{1}{4}-b\alpha^2}{2\alpha^2+bk^2-1}.
$$
Setting $\frac{\partial\alpha^2}{\partial k^2}=0$, we find $b\alpha^2=\frac{1}{4}$. Plugging it back into the quadratic equation gives $\alpha=1$ and $b=\frac{1}{4}$. Indeed, for $b=\frac{1}{4}$, $\alpha\equiv 1$ for all $k\in \mathbb{R}$. For $b\neq \frac{1}{4}$, our analysis above shows that there is no critical point of the function $\alpha^2(k^2;b)$. Therefore, $\alpha^2(k^2;b\neq \frac{1}{4})$ must be either strictly increasing or strictly decreasing on $k^2\in[0,\infty)$. For $k^2=8$ and $b=\frac{1}{8}$, we directly compute $\alpha^2(8;\frac{1}{8})=\sqrt{2}$ and $\frac{\partial \alpha^2}{\partial k^2}(8;\frac{1}{8})=\frac{\sqrt{2}-1}{16}>0$. For $k^2=2$ and $b=\frac{1}{2}$, we directly compute $\alpha^2(2;\frac{1}{2})=\frac{1}{\sqrt{2}}$ and $\frac{\partial \alpha^2}{\partial k^2}(2;\frac{1}{2})=\frac{\sqrt{2}-2}{8}<0$. Therefore $\alpha^2(k^2;b)$ must be strictly increasing on $k^2\in[0,\infty)$ for all $0<b<\frac{1}{4}$ and be strictly decreasing on $k^2\in[0,\infty)$ for all $b>\frac{1}{4}$, giving the claimed result.
\end{proof}

To recover stability of a non-characteristic planar hydraulic shock profile in the whole space with $F>2$, one can use exponential weights that stabilize the dispersion relations. As in \cite{FRYZ}, we introduce, for $(\mu_L,\mu_R)\in [0,\infty)\times (-\infty,0]$, the weighted spaces
\be \label{weighted_space}
L^2_{\mu_L,\mu_R}(\mathbb{R}^*\times\mathbb{R},\mathbb{C}^3):=\left\{V(x,y):\left.\begin{aligned}&e^{-\mu_L x}V(x,y)|_{\mathbb{R}^-\times \mathbb{R}}\in L^2(\mathbb{R}^-\times\mathbb{R},\mathbb{C}^3)\quad \text{and}\\
&e^{-\mu_R x}V(x,y)|_{\mathbb{R}^+\times \mathbb{R}}\in L^2(\mathbb{R}^+\times\mathbb{R},\mathbb{C}^3) \end{aligned}\right.\right\},
\ee 
which in short will be referred to as $(\mu_L,\mu_R)$--weighted space. In a $(\mu_L,\mu_R)$--weighted space, the weighted essential spectra of a non-characteristic planar hydraulic shock profile in the whole space sit to the left of the rightmost weighted dispersion relations of the matrix $G(\pm\infty;\lambda,\eta)$ that are given by algebraic curves $\lambda_{j,\pm}(k;\eta)$, $j=1,2,3$ consisting of zeros of the cubic polynomial $\det(G(\pm\infty;\cdot,\eta)-(ik+\mu_{L,R})\Id)=0$ for $k,\eta\in \mathbb{R}$. Our setting motivates the following definition of stabilization of weighted dispersion relations (essential spectra) in a $(\mu_L,\mu_R)$--weighted space.

\begin{definition}\label{stable_ess}
Consider a discontinuous planar hydraulic shock  $(F,H_L,H_R)$ in the whole space with $2<F<{\normalfont F_{\text{exist}}}$ that admits unstable essential spectra in the unweighted $L^2$ space. For $(\mu_L,\mu_R)\in[0,\infty)\times (-\infty,0]$, it is said to have $(\mu_L,\mu_R)$--weightedly stable dispersion relations {\normalfont (}essential spectra{\normalfont )} if $$\sup_{k,\eta\in \mathbb{R},j=1,2,3}\Re\lambda_{j,\pm}(k;\eta)\leq 0,$$
where $\lambda_{j,\pm}(k;\eta)$, $j=1,2,3$ are zeros of the characteristic polynomial $\det(G(\pm;\cdot,\eta)-ik-\mu_{L,R})=0$ for $k,\eta\in\mathbb{R}$. It is said to have $(\mu_L,\mu_R)$--weightedly exponentially stable dispersion relations {\normalfont (}essential spectra{\normalfont )} if $$\sup_{k,\eta\in \mathbb{R},j=1,2,3}\Re\lambda_{j,\pm}(k;\eta)< 0.$$
\end{definition}
The stabilization of weighted dispersion relations (essential spectra) of a discontinuous planar hydraulic shock in the whole space can be characterized by $\gamma_{j,\pm}(\lambda,\eta)$, the eigenvalues of $G(\pm\infty,\lambda,\eta)$ \eqref{zgeig}.
\begin{lemma}
Let $\gamma_{i,\pm}(\lambda,\eta)$, $i=1,2,3$ be eigenvalues of $G(\pm\infty,\lambda,\eta)$ \eqref{zgeig} ordered by $$\Re\gamma_{1,\pm}(\lambda,\eta)\leq \Re\gamma_{2,\pm}(\lambda,\eta)\leq \Re\gamma_{3,\pm}(\lambda,\eta).$$
Let $(\mu_L,\mu_R)\in[0,\infty)\times (-\infty,0]$. A discontinuous planar hydraulic shock $(F,H_R)$ in the whole space
with $2<F< {\normalfont F_{\text{exist}}}$ has $(\mu_L,\mu_R)$--weightedly stable dispersion relations {\normalfont (}essential spectra{\normalfont )} if and only if 
$$
\sup_{\Re\lambda\geq 0,\eta\in\mathbb{R}}\Re\gamma_{1,-}(\lambda,\eta)\leq \mu_L\leq \inf_{\Re\lambda\geq 0,\eta\in\mathbb{R}}\Re\gamma_{2,-}(\lambda,\eta)\quad\text{and} \quad\mu_R\leq \inf_{\Re\lambda\geq 0,\eta\in\mathbb{R}}\Re\gamma_{1,+}(\lambda,\eta).
$$
It has $(\mu_L,\mu_R)$--weightedly exponentially stable dispersion relations {\normalfont (}essential spectra{\normalfont )} if and only if 
$$
\sup_{\Re\lambda\geq 0,\eta\in\mathbb{R}}\Re\gamma_{1,-}(\lambda,\eta)< \mu_L< \inf_{\Re\lambda\geq 0,\eta\in\mathbb{R}}\Re\gamma_{2,-}(\lambda,\eta)\quad\text{and} \quad\mu_R< \inf_{\Re\lambda\geq 0,\eta\in\mathbb{R}}\Re\gamma_{1,+}(\lambda,\eta).
$$
\end{lemma}
\begin{proof}
On the right side of the rightmost $(\mu_L,\mu_R)$--weighted dispersion relations, referred to as domain of extended consistent splitting in the literature, the number of eigenvalues of $G(-\infty;\lambda,\eta)$ ($G(+\infty;\lambda,\eta)$) with real part greater than $\mu_L$  ($\mu_R$) must be constant, which we now determine. When $|\lambda|\rightarrow \infty$ and $\eta$ stays bounded, the leading terms of eigenvalues of $G(\pm\infty;\lambda,\eta)$ are connected to characteristic speeds of \eqref{lineqs} as $x\to \pm \infty$. Specifically, eigenvalues of $G(-\infty;\lambda,\eta)$ ($G(+\infty;\lambda,\eta)$) expand as
$$
\begin{aligned}
&\lambda F_{\text{exist}}+\mathcal{O}(1),\quad \lambda\frac{F_{\text{exist}}F}{F\pm F_{\text{exist}}}+\mathcal{O}(1)\\
\Big(&\lambda F_{\text{char}}\nu+\mathcal{O}(1),\quad \lambda\frac{F_{\text{char}}F}{F\pm F_{\text{char}}}+\mathcal{O}(1)\Big).
\end{aligned}
$$
For $2<F< F_{\text{exist}}$ and bounded $\eta$, as $\Re\lambda \to \infty$, the constant is $2$ at $-\infty$ ($3$ at $+\infty$). Nonpresence of $\mu_L$($\mu_R$)--weighted dispersion relations $\lambda_{j,-}(k;\eta)$ ($\lambda_{j,+}(k;\eta)$), $j=1,2,3$, on the right half plane is equivalent to the foregoing (latter) condition on $\mu_L$ ($\mu_R$).
\end{proof}
In the one-dimensional case, \cite[\S 3.4]{FRYZ} shows that one can always choose a $\mu_R<0$ so that the $\mu_R$--weighted dispersion relations $\lambda_{j,+}(k)$, $j=1,2$, are stable. We assert the same result in the two-dimensional case by claiming that $\Re \gamma_{1,+}(\lambda,\eta)$ has a lower bound.
\begin{lemma}\label{existencemuR} For $2<F$, $\Re\gamma_{1,+}(\lambda,\eta)$ is bounded from below for $\Re\lambda\geq 0$ and $\eta\in \mathbb{R}$.
\end{lemma}
\begin{proof}
On any compact subset of $\{\lambda:\Re\lambda\geq 0\}\times \mathbb{R}$, $\Re\gamma_{1,+}(\lambda,\eta)$ is surely bounded from below by continuity. It suffices to show that $\Re\gamma_{1,+}(\lambda,\eta)$ is also bounded from below as $|\lambda|^2+|\eta|^2\rightarrow \infty$. For $\theta\in[0,\pi]$, $r\gg 1$, and $\tilde{\lambda}=e^{i\tilde{\theta}}$ where $\tilde{\theta}\in[-\frac{\pi}{2},\frac{\pi}{2}]$, let $\lambda=\tilde{\lambda}r\sin(\theta)$ and $\eta=r\cos(\theta)$ so that
$$
\sigma(G(+\infty;\lambda,\eta))=r\sigma\big((-\tilde{\lambda}\sin(\theta)A_0-i\cos(\theta)A_2+\eps E)A_1^{-1}(+\infty)\big)=:r\sigma(M(\eps)),
$$
where $\eps:=1/r\ll 1$. The eigenvalues of the leading matrix $M(0)$ are $\tilde{\lambda}\sin(\theta)(\nu^{-1}+ 1)$ and
$$
\frac{(\nu + 1)\left(F^2\tilde{\lambda}\sin(\theta)\nu\pm\sqrt{\cos(\theta)^2(F_{\text{char}}^2-F^2) + F^2\tilde{\lambda}^2\sin(\theta)^2F_{\text{char}}^2\nu^2}\right)}{\nu^2(F^2 -F_{\text{char}}^2)}.
$$
Clearly, $\Re(\tilde{\lambda}\sin(\theta)(\sqrt{H_R}+1))=\sin(\theta)(\sqrt{H_R}+1)\Re(\tilde{\lambda})\geq 0$ and 
$$
\Re\left(F^2\tilde{\lambda}\sin(\theta)\nu+\sqrt{\cos(\theta)^2(F_{\text{char}}^2-F^2) + F^2\tilde{\lambda}^2\sin(\theta)^2F_{\text{char}}^2\nu^2}\right)\geq 0.
$$
To verify 
$$\Re\left(F^2\tilde{\lambda}\sin(\theta)\nu-\sqrt{\cos(\theta)^2(F_{\text{char}}^2-F^2) + F^2\tilde{\lambda}^2\sin(\theta)^2F_{\text{char}}^2\nu^2}\right)\geq 0,
$$
when $\sin(\theta)=0$, it follows from $F_{\text{char}}<2<F$, and when $\sin(\theta)>0$, factoring out $\sin(\theta)$ yields
$$\sin(\theta)\Re\left(F^2\tilde{\lambda}\nu-\sqrt{\cot(\theta)^2(F_{\text{char}}^2-F^2) + F^2\tilde{\lambda}^2F_{\text{char}}^2\nu^2}\right).
$$
Since $\cot(\theta)^2$ varies in $[0,\infty)$ and $F_{\text{char}}^2-F^2<0$, we make computations by the formula $\Re\sqrt{z}=\sqrt{\frac{\Re z+|z|}{2}}$ to obtain
$$
\begin{aligned}
&\sup_{M\leq 0}\Re\sqrt{M+F^2\tilde{\lambda}^2F_{\text{char}}^2\nu^2}\\
=&\sqrt{\sup_{M\leq 0}\frac{\Re(M+F^2\tilde{\lambda}^2F_{\text{char}}^2\nu^2)+|M+F^2\tilde{\lambda}^2F_{\text{char}}^2\nu^2| }{2}}\\
=&\Re\sqrt{M+F^2\tilde{\lambda}^2F_{\text{char}}^2\nu^2}\Bigg|_{M=0}=FF_{\text{char}}\nu\Re{\tilde\lambda}
\end{aligned}
$$
and hence
$$
\begin{aligned}
&\Re\left(F^2\tilde{\lambda}\nu-\sqrt{\cot(\theta)^2(F_{\text{char}}^2-F^2) + F^2\tilde{\lambda}^2F_{\text{char}}^2\nu^2}\right)\geq &  F(F-F_{\text{char}})\nu\Re\tilde{\lambda}\geq 0.
\end{aligned}
$$
\end{proof}

As for the stabilization of $\mu_L$--weighted dispersion relations $\lambda_{j,-}(k;\eta)$, $j=1,2,3$, in the one-dimensional case, \cite{FRYZ} shows that it suffices to stabilize them as $k\to \infty$ for the most unstable essential spectra are located in the $|\Im\lambda|\to \infty$ limit. This motivates in the 2d case to search a $\mu_L$ weight that stabilizes the weighted dispersion relations in the $k^2+\eta^2\to \infty$ limit.

\begin{lemma}\label{stability_1d2d}

If ${\normalfont F_{\text{2d}}}< F<{\normalfont F_{\text{exist}}}$, there exists no $\mu_L$--weight that can stabilize the weighted dispersion relations $\lambda_{j,-}(k;\eta)$. If $2<F< {\normalfont F_{\text{2d}}} $, for
\be \label{w_range}
\mu_L\in\left(\frac{F^2}{2}-\frac{F^2}{\normalfont{F_{\text{exist}}}}-\frac{F}{{\normalfont F_{\text{2d}}}}\sqrt{{\normalfont F_{\text{2d}}}^2-F^2},\frac{F^2}{2}-\frac{F^2}{\normalfont{F_{\text{exist}}}}+\frac{F}{{\normalfont F_{\text{2d}}}}\sqrt{{\normalfont F_{\text{2d}}}^2-F^2}\right)\subset(0,\infty),
\ee 
the $\mu_L$--weighted dispersion relations $\lambda_{j,-}(k;\eta)$, $j=1,2,3$, are stable in the $k^2+\eta^2\to \infty$ limit.
\end{lemma}
\begin{proof}
For $\eta=\sin(\theta)r$ and $k=\cos(\theta)r$ with $r\gg 1$ and $\theta\in [0,2\pi]$, the $\mu_L$--weighted dispersion relations are given by
$$\sigma(-ikA_1-i\eta A_2+E-\mu_LA_1)=r\sigma\left(-i\cos(\theta)A_1-i\sin(\theta)A_2+\eps(E-\mu_LA_1)\right)=:r\sigma(M(\eps))$$
where $\eps:=1/r\ll 1$. The eigenvalues of the leading matrix $-i\cos(\theta)A_1-i\sin(\theta)A_2$ are 
$$
\tilde{\lambda}_{1,2}(0)=\pm\frac{i}{F}+\frac{i\cos(\theta)}{F_{\text{exist}}}\quad \text{and} \quad \tilde{\lambda}_3(0)=\frac{i\cos(\theta)}{F_{\text{exist}}},
$$
hence, distinct and purely imaginary. For sufficiently small $\eps$, let $P_j(\eps)$ be the projection onto the $\tilde{\lambda}_j(\eps)$-eigenspace of $M(\eps)$. Let $U_j(\eps)$ be transformation function for $P_j(\eps)$ and set
$$
L_j(\eps)=e_j^TU_j^{-1}(\eps),\quad R_j(\eps)=U_j(\eps)e_j.
$$
For small $\eps$, there holds, for $j=1,2,3$,
$$
\sigma_{\Big|\text{near $\tilde{\lambda}_j(0)$}}(M(\eps))=\sigma\left(L_j(\eps)M(\eps)R_j(\eps)\right)=L_j(\eps)M(\eps)R_j(\eps).
$$
By rescaling, we may further assume $L_j(0)R_j'(0)=0$ and $L_j'(0)R_j(0)=0$, with which the leading $\mathcal{O}(\eps)$ real part of $L_j(\eps)M(\eps)R_j(\eps)$ is given by $L_j(0)(E-\mu_LA_1)R_j(0)\eps$. Therefore, for $|k|+|\eta|\gg 1$,
$$
\Re\sigma(-ikA_1-i\eta A_2+E-\mu_LA_1)=\Re\; r\sigma(M(\eps))=\big\{L_j(0)(E-\mu_LA_1)R_j(0)+\mathcal{O}(\eps):j=1,2,3\big\},
$$
where 
\ba 
\label{dispersion_real}
&L_1(0)(E-\mu_LA_1)R_1(0)=\frac{\cos(\theta)F_{\text{exist}} + F}{FF_{\text{exist}}}\mu_L- \frac{1}{2}\cos(\theta)^2 - \frac{F}{2}\cos(\theta) - \frac{1}{2},\\
&L_2(0)(E-\mu_LA_1)R_2(0)=-\frac{\cos(\theta)F_{\text{exist}} + F}{FF_{\text{exist}}}\mu_L- \frac{1}{2}\cos(\theta)^2 + \frac{F}{2}\cos(\theta) - \frac{1}{2},\\
&L_3(0)(E-\mu_LA_1)R_3(0)=\frac{\mu_L}{F_{\text{exist}}} - \sin(\theta)^2 - 1.
\ea 
At leading order, a sufficient condition for stabilizability is $L_j(0)(E-\mu_LA_1)R_j(0)< 0$, for all $\theta\in[0,2\pi]$ and $j=1,2,3$ while a sufficient condition for non-stabilizability is $L_j(0)(E-\mu_LA_1)R_j(0)> 0$ for some $\theta$ and $j$ and any $\mu_L\geq 0$. For $j=3$, stabilizability requires $\mu_L<F_{\text{exist}}$. For $j=1,2$, noting that $(L_1(E-\mu_LA_1)R_1)(\theta,\mu_L)=(L_2(E-\mu_LA_1)R_2)(\theta+\pi,\mu_L)$, stabilizability reduces to requiring $(L_1(E-\mu_LA_1)R_1)(\theta,\mu_L)< 0$ for all $\theta\in[0,2\pi]$, yielding: 

(i) when $\cos(\theta)F_{\text{exist}} + F=0$, $F<F_{\text{2d}}$,

(ii) when $\cos(\theta)F_{\text{exist}} + F<0$, 
$$
\frac{FF_{\text{exist}}\left(\cos(\theta)^2+F\cos(\theta)+1\right)}{2\left(\cos(\theta)F_{\text{exist}}+F\right)}< \mu_L, \quad \forall \cos(\theta)\in[-1,-\frac{F}{F_{\text{exist}}})
$$

and, (iii) when $\cos(\theta)F_{\text{exist}} + F>0$,
$$
\mu_L< \frac{FF_{\text{exist}}\left(\cos(\theta)^2+F\cos(\theta)+1\right)}{2\left(\cos(\theta)F_{\text{exist}}+F\right)},\quad \forall \cos(\theta)\in(-\frac{F}{F_{\text{exist}}}, 1 ].
$$
Computation reveals that the derivative
$$
\frac{d}{d\cos(\theta)}\frac{\cos(\theta)^2+F\cos(\theta)+1}{\cos(\theta)F_{\text{exist}}+F}=\frac{\cos(\theta)^2F_{\text{exist}}+2F\cos(\theta)+F^2 -F_{\text{exist}}}{(\cos(\theta)F_{\text{exist}} + F)^2}
$$
is positive at $\cos(\theta)=\pm 1$. In (ii) and (iii), $\cos(\theta)\neq -F/F_{\text{exist}}$, the denominator of the derivative is always positive. By (i), given $F<F_{\text{2d}}$, the numerator of the derivative has two zeros, one in $(-1,-F/F_{\text{exist}})$, where the function achieves a maximum, and the other one in $(-F/F_{\text{exist}},1)$, where the function achieves a minimum, giving \eqref{w_range}. We also verify $$\frac{F^2}{2}-\frac{F^2}{F_{\text{exist}}}-\frac{F}{ F_{\text{2d}}}\sqrt{ F_{\text{2d}}^2-F^2}>0$$  and 
$$
\frac{F^2}{2}-\frac{F^2}{F_{\text{exist}}}+\frac{F}{ F_{\text{2d}}}\sqrt{ F_{\text{2d}}^2-F^2}<F_{\text{exist}}.
$$
As for non-stabilizability, from (i), we see when $F>F_{\text{2d}}$ and $\cos(\theta)=-F/F_{\text{exist}}$, $L_1(0)(E-\mu_LA_1)R_1(0)>0$ for any $\mu_L\in[0,\infty)$.

For the critical case $F=F_{\text{2d}}$, stabilizability is determined upon investigating the next $\mathcal{O}(\eps)$-order. 
\end{proof}

It turns out that the mid-point of the interval \eqref{w_range} ``maximally stabilizes'' the dispersion relations at $-\infty$ in the $k^2+\eta^2\to \infty$ limit.

\begin{lemma}\label{optimal_weight}
For $2<F<{\normalfont F_{\text{2d}}}$, the optimal weight $\mu_L$ such that the dispersion relations $\lambda_{j,-}(k;\eta)$, $j=1,2$, are maximally stabilized in the $k^2+\eta^2\to\infty$ limit is the mid-point of \eqref{w_range} that is 
\be 
\label{muopt}
\mu_{L,opt}=\frac{F^2}{2}-\frac{F^2}{F_{\text{exist}}}.
\ee
For $\mu_L=\mu_{L,opt}$, the leading $\mathcal{O}(1)$ real parts of the dispersion relations \eqref{dispersion_real} in the $k^2+\eta^2\to\infty$ limit satisfy
\ba \label{optimal_weight_inequalities}
&L_1(0)(E-\mu_{L,opt}A_1)R_1(0),L_2(0)(E-\mu_{L,opt}A_1)R_2(0)\leq \frac{F^2-F_{\text{2d}}^2}{2F_{\text{2d}^2}}<0,\\
&L_3(0)(E-\mu_{L,opt}A_1)R_3(0)\le\frac{F_{\text{exist}} - 2}{2F_{\text{exist}}^2}F^2-1<0.
\ea 
\end{lemma}

\begin{proof}
As a quadratic function of $\cos(\theta)$, \eqref{dispersion_real}[i] attains its maximum 
at $\cos(\theta)=\frac{\mu_L}{F} -\frac{F}{2}$ provided that $-1\le\frac{\mu_L}{F} -\frac{F}{2}\leq 1$, which holds for any $\mu$ in \eqref{w_range}. Hence, the maximum is 
$$
\frac{1}{2F^2}\mu^2+\left(\frac{1}{F_{\text{exist}}} - \frac{1}{2}\right)\mu+\frac{F^2}{8} - \frac{1}{2},
$$
which is minimized at the mid point of \eqref{w_range}. Similar analysis follows for \eqref{dispersion_real}[ii] and the optimal $\mu_L$ is the same as before.
\end{proof}

\begin{remark}\label{numericalchecked}
Lemma~\ref{stability_1d2d} and Lemma~\ref{optimal_weight} together prove the stability of $\mu_{L,opt}$--weighted dispersion relations $\lambda_{j,-}(k;\eta)$, $j=1,2,3$, for all sufficiently large $(k,\eta)$. The stability of the optimally weighted dispersion relations for bounded $(k,\eta)$ is checked by (nonrigorous) numerics, which would be very desirable to prove analytically or verify by validated numerics. 
\end{remark}
Based on Remark~\ref{numericalchecked}, we conjecture that for a $(F,H_L,H_R)$ discontinuous planar hydraulic shock in the whole space with $2<F<F_{\text{2d}}$ there exists $(\mu_L,\mu_R)\in[0,\infty)\times(-\infty,0]$ such that the wave has $(\mu_L,\mu_R)$--weightedly stable dispersion relations (essential spectra).

\smallskip

\textbf{General weights and absolute instability.} In the one-dimensional case \cite{FRYZ}, it is proven that for $F_{\text{1d}}<F<F_{\text{exist}}$  the limiting matrix $G_{H_L}(\lambda)$ has identical eigenvalues for a $\lambda$ with positive real part where the eigenvalues lose analyticity and from which ``{\it the resolvent operator cannot be continuously
extended even as an operator from the space of test functions to distributions}". In such a case, ``{\it absolute instabilities cannot be cured in any sensible sense, in particular not
by replacing exponential weights by a more general class of reasonable weights}". 

In the present two dimensional case, it is again of our curiosity to investigate whether the essential spectrum, for $F_{\text{2d}}< F<F_{\text{1d}}$, can be stabilized by a more general class of reasonable weights. Unfortunately, for $0<F-F_{\text{2d}}\ll 1$, as we will prove, the two eigenvalues $\gamma_{1,-}(\lambda,\eta)$ and $\gamma_{2,-}(\lambda,\eta)$ with smaller real parts than that of $\gamma_{3,-}(\lambda,\eta)$ can collide for $\lambda$ with positive real part, resulting in absolute instability when $F$ passes the critical value $F_{\text{2d}}$ from below to above.

At $x=-\infty$, eigenvalues $\gamma_{j,-}$, $j=1,2,3$ of $G(-\infty,\lambda,\eta)$ \eqref{zgeig} are zeros of the characteristic polynomial $\det(E-\lambda \Id -i\eta A_2-\gamma A_1)=:p(\gamma;\lambda,\eta)$ which reads
$$\begin{aligned}
p(\gamma;\lambda,\eta)=&-\frac{F_{\text{exist}}^2-F^2}{F^2F_{\text{exist}}^3}\gamma^3+\frac{(\lambda + 1)(- 3F^2 + F_{\text{exist}}^2)+F^2F_{\text{exist}}}{F^2F_{\text{exist}}^2}\gamma^2\\
&+\frac{-F^2(-3\lambda^2+(F_{\text{exist}} - 6)\lambda+ F_{\text{exist}} - 2)+\eta^2}{F^2F_{\text{exist}}}
\gamma-\frac{(\lambda + 2)(F^2\lambda^2 + F^2\lambda + \eta^2)}{F^2}.
\end{aligned}$$
To locate double zeros of $p(\;\cdot\;;\lambda,\eta)$, taking the resultant of $p(\gamma;\lambda,\eta)=0$ and $p_\gamma(\gamma;\lambda,\eta)=0$ gives a remainder linear in $\gamma$, yielding $\gamma_*(\lambda;\eta)=\frac{\gamma_{top}}{\gamma_{bot}}(\lambda;\eta)$ where
$$
\begin{aligned}
\gamma_{top}&(\lambda;\eta)=F_{\text{exist}}\Big(6F^2F_{\text{exist}}^2\lambda^3+F^2F_{\text{exist}}(- 6F^2 + F_{\text{exist}}^2 + 18F_{\text{exist}})\lambda^2\\
&+\big((-6F^2+8F_{\text{exist}}^2)\eta^2+F^4(F_{\text{exist}}^2 - 12F_{\text{exist}} + 6)+2F^2F_{\text{exist}}^2(F_{\text{exist}} + 5) \big)\lambda\\
&+\eta^2(- F^2(F_{\text{exist}}+15)+17F_{\text{exist}}^2)+F^2(F_{\text{exist}} - 2)(F^2(F_{\text{exist}} - 3)+F_{\text{exist}}^2)\Big)
\end{aligned}
$$
and
$$
\begin{aligned}
\gamma_{bot}&(\lambda;\eta)=2F_{\text{exist}}^2(3F^2+F_{\text{exist}}^2)\lambda ^2+6F_{\text{exist}}(F_{\text{1d}}^2-F^2)(F^2+\frac{F_{\text{exist}}^2}{3})\lambda\\
&+6\eta^2(-F^2+F_{\text{exist}}^2)+F^4(2F_{\text{exist}}^2 - 6F_{\text{exist}} + 6)-2F^2F_{\text{exist}}^3 +2F_{\text{exist}}^4.
\end{aligned}
$$

By the Routh-Hurwitz stability criterion, $\gamma_{bot}(\lambda;\eta)$ does not vanish for $\lambda$ with positive real part, provided that the coefficients of the quadratic function in $\lambda$ are all positive/negative. The coefficient of $\lambda^2$ is surely positive. The coefficient of the linear term is also positive since $F<F_{\text{1d}}$. As for the constant term, the coefficient of $\eta^2$ is positive since $F<F_{\text{1d}}<F_{\text{exist}}$ and the remaining part is also positive, since, as a quadratic function of $F^2$, it opens upwards ($2F_{\text{exist}}^2 - 6F_{\text{exist}} + 6>2\times2^2 - 6\times2 + 6>0$) and has negative discriminant ($\Delta=-12F_{\text{exist}}^4(F_{\text{exist}} - 2)^2<0$). 

The root
$\gamma_*(\lambda;\eta)$ will be a double eigenvalue provided that $p_\gamma(\gamma_*;\lambda,\eta)=0$, which amounts to solving $f(\lambda;\eta)=0$ where $f(\lambda;\eta)=0$ is a sextic polynomial in $\lambda$ with coefficients being polynomials of $\eta$. If $F<F_{\text{2d}}$, we expect that
$\gamma_{1,-}(\lambda,\eta)$ does not collide with $\gamma_{2,-}(\lambda,\eta)$ for $\lambda$ with positive real part. On the other hand, if the collision does occur for $F>F_{\text{2d}}$ and $\lambda$ with positive real part, the collision must also occur in the mid way for some purely imaginary $\lambda$. We then investigate  $f(i\tau;\eta)=0$ for $\tau \in \mathbb{R}$. Separating the real and imaginary part of $f(i\tau;\eta)$, we obtain formulae for $\Re f(i\tau;\eta)$ and $\Im f(i\tau;\eta)=0$ of
$$
\begin{aligned}
\Re f&(i\tau;\eta)=-36F^2F_{\text{exist}}^6\tau ^6+9F_{\text{exist}}^4\big(\eta ^2(4F_{\text{exist}}^2-12F^2)+12F^6+F^4(F_{\text{exist}}^2 -60F_{\text{exist}} + 12)\\&+ 56F^2F_{\text{exist}}^2\big)\tau ^4-9F_{\text{exist}}^2\big(\eta^4(12F^2-8F_{\text{exist}}^2)+\eta^2(-12F^6+F^4(2F_{\text{exist}}^2+126F_{\text{exist}}-24)\\
&-2F^2F_{\text{exist}}^2(15F_{\text{exist}}+82)+36F_{\text{exist}}^4)+F^2(F^6(F_{\text{exist}}^2-12F_{\text{exist}}+12)+F^4(- 6F_{\text{exist}}^3 \\&
+ 70F_{\text{exist}}^2 - 72F_{\text{exist}} + 12)+2F^2F_{\text{exist}}^2(3F_{\text{exist}}^2-48F_{\text{exist}}+32)+36F_{\text{exist}}^4)\big)\tau ^2\\&+36\eta ^6(F_{\text{exist}}-F^2)-9\eta ^4(F^4(F_{\text{exist}}^2-30F_{\text{exist}}-15)+2F^2F_{\text{exist}}^2(13F_{\text{exist}}+45)-71F_{\text{exist}}^4)\\
&+9\eta ^2\big(F^6(6F_{\text{exist}}^3 -32F_{\text{exist}}^2 + 42F_{\text{exist}} - 12)-4F^4F_{\text{exist}}^2(F_{\text{exist}}^2-11F_{\text{exist}}+12)\\
&-2F^2F_{\text{exist}}^4(7F_{\text{exist}}-2)+8F_{\text{exist}}^6\big)+9F^4(F_{\text{exist}}-2)^2(F^2(1-F_{\text{exist}})+F_{\text{exist}}^2)^2
\end{aligned}
$$

$$
\begin{aligned}
\Im f&(i\tau;\eta)=2F_{\text{exist}}\tau \Big(54F^2F_{\text{exist}}^4(2F_{\text{exist}}^2- F^2)\tau ^4-9F_{\text{exist}}^2\big(\eta^2(12F^4-2F^2(2F_{\text{exist}}^2+21F_{\text{exist}}^2)\\&+10F_{\text{exist}}^3)-2F^8+F^6(-F_{\text{exist}}^2+24F_{\text{exist}}-12)-2F^4\nu(- \nu^5 - 3\nu^4 + 25\nu^3 + 55\nu^2 + 16\nu\\& - 12)+32F^2F_{\text{exist}}^3\big)\tau ^2+9\eta^4\big(-6F^4+F^2(4F_{\text{exist}}^2+30F_{\text{exist}})-26F_{\text{exist}}^3\big)+9\eta^2\big(F^6(F_{\text{exist}}^2\\&-21F_{\text{exist}}+12)-F^4F_{\text{exist}}(F_{\text{exist}}^2 - 82F_{\text{exist}} + 42)-2F^2F_{\text{exist}}^3(9F_{\text{exist}}+22)+14F_{\text{exist}}^5\big)\\
&+9F^8(F_{\text{exist}}^3 - 7F_{\text{exist}}^2+ 12F_{\text{exist}} - 6)-9F^6F_{\text{exist}}(3F_{\text{exist}}^3 - 22F_{\text{exist}}^2 + 32F_{\text{exist}} - 12)\\&+18F^4F_{\text{exist}}^3(F_{\text{exist}}^2-9F_{\text{exist}}+8)+36F^2F_{\text{exist}}^5\Big).
\end{aligned}
$$ 
Solving $\Re f(i\tau;\eta),\Im f(i\tau;\eta)=0$ yields (i) when $\tau=0$, the constant term in $\Re f(i\tau;\eta)$, as a cubic polynomial in $\eta^2$, has to vanish, and (ii) when $\tau \neq 0$, taking the resultant of $\Re f(i\tau;\eta)=0$ and $\Im f(i\tau;\eta)/\tau=0$ gives a remainder linear in $\tau^2$, solving which yields
\be \label{tausolution}
\tau_*^2=g(\eta).
\ee 
Evaluating $\Im f(ig(\eta);\eta)/g(\eta)$, we find that the numerator is a sextic polynomial in $\eta^2$. 

To proceed with case (i), the discriminant of the cubic polynomial in $\eta^2$ reads $209952F_{\text{exist}}(F_{\text{1d}}^2-F^2)\#(F,\nu)^3$, where
$$\begin{aligned}
\#&(F,\nu):=(- F_{\text{exist}}^3 - 9F_{\text{exist}}^2 + 27F_{\text{exist}} - 27)(F^2)^3\\
&+3F_{\text{exist}}^2(5F_{\text{exist}}^2-6F_{\text{exist}}+9)(F^2)^2-3F_{\text{exist}}^4(7F_{\text{exist}}+3)(F^2)+17F_{\text{exist}}^6
\end{aligned}
$$
is a cubic polynomial of $F^2$ whose discriminant $-78732F_{\text{exist}}^{12}(F_{\text{exist}} - 2)^6$ is negative. The cubic polynomial $\#(F,\nu)$ in $F^2$ then has a real zero and a pair of complex conjugate zeros. Because $\#(0,\nu)>0$ and $\#(F_{\text{1d}},\nu)=27\nu^3\big((\nu-1)(\nu+2)(\nu+1)\big)^3>0$, the cubic must be positive for  $0<F\leq F_{\text{1d}}$, otherwise the cubic has at least two real zeros. A contradiction. That is we have shown that $\#(F,\nu)>0$ for $\nu>1$, $0<F\leq F_{\text{1d}}$. The discriminant of the cubic polynomial in $\eta^2$ is then positive, yielding that the cubic polynomial has three real roots. To decide the number of positive roots in the region $\nu>1$, $F_{\text{2d}}<F< F_{\text{1d}}$, we find that $\Re f(0,0)$ is positive and 
$$
\Re f(0,\eta)\big|_{F=3,\nu=2}=243\big(4(\eta^2)^3-933(\eta^2)^2-2592\eta^2+3888\big)
$$
has one negative zero and two positive zeros
\ba \label{positiveta2}
&\frac{{\left(31661351-51696\sqrt{1077}i\right)}^{1/3}}{4}+\frac{{\left(31661351+51696\sqrt{1077}i\right)}^{1/3}}{4}+\frac{311}{4}\quad \text{and}\\
&\frac{311}{4}+\frac{\left(-1+\sqrt{3}i\right)\left(31661351-51696\sqrt{1077}i\right)^{1/3}}{8}-\frac{\left(31661351+51696\sqrt{1077}i\right)^{1/3}}{8}\\
&-\frac{\sqrt{3}\left(31661351+51696\sqrt{1077}i\right)^{1/3}i}{8}.
\ea
Therefore, the cubic polynomial in $\eta^2$ always has two positive roots and a negative root for $\nu>1$, $F_{\text{2d}}<F< F_{\text{1d}}$. Correspondingly, there are four real zeros $\eta_j$, $j=1,2,3,4$ to the sextic polynomial in $\eta$. Recall that $p(\gamma;0,\eta_j)$, as a cubic polynomial in $\gamma$ with real coefficients, has a double zero $\gamma_*(0;\eta_j)$, whence it must have three real zeros. Because $p(0;0,\eta_j)=-2\eta_j^2/F^2<0$, the number of positive zeros of $p(\cdot;0,\eta_j)$ must be constant. For $F=3$, $\nu=2$, 
$$
p(\gamma;0,\eta_j)=-\frac{1}{72}\gamma^3+\frac{7}{36}\gamma^2+(\frac{1}{54}\eta_j^2 - \frac{2}{3})\gamma-\frac{2\eta_j^2}{9}.
$$
Substituting $\eta_j^2$ by the two positive numbers in \eqref{positiveta2}, we find that the equation has a double positive zero and a negative zero. The collision of eigenvalues occurs for the two positive ones. Hence, in case (i), $\gamma_{1,-}(\lambda,\eta)$ cannot collide with $\gamma_{2,-}(\lambda,\eta)$ for $\lambda$ with positive real part.

As we shall see in the following theorem, the collision between $\gamma_{1,-}(\lambda,\eta)$ and $\gamma_{2,-}(\lambda,\eta)$ happens in case (ii).

\begin{theorem}
Given a planar hydraulic shock profile $(F,H_L,H_R)$ in the whole space with $0<F-{\normalfont F_{\text{2d}}}\ll 1$, there 
exist $\eta(F;H_L,H_R)\in\mathbb{R}^+$ and $\tau_*(\eta(F;H_L,H_R))\in\mathbb{R}^+$, satisfying $$\eta(F;H_L,H_R),\tau_*(\eta(F;H_L,H_R))\to \infty\quad \text{ as $F\to {\normalfont F_{\text{2d}}}^+$,}$$ such that 
$$\gamma_{1,-}(\pm i\tau_*(\eta(F;H_L,H_R)),\pm\eta(F;H_L,H_R))=\gamma_{2,-}(\pm i\tau_*(\eta(F;H_L,H_R)),\pm\eta(F;H_L,H_R))$$
where $\gamma_{1,2,-}$ is a double eigenvalue of $G(-\infty;\pm i\tau_*(\eta(F;H_L,H_R)),\pm\eta(F;H_L,H_R))$ with smaller real part than that of the 
remaining, third eigenvalue.
\end{theorem}

\begin{proof}
In case (ii), the highest degree term of the sextic polynomial in $\eta^2$ reads
\be
\label{sixthorder}
(F_{\text{2d}}^2-F^2)(\eta^2)^6,
\ee 
whence for $F=F_{\text{2d}}$ the polynomial degenerates to a quintic polynomial in $\eta^2$ whose highest degree term reads
\be\label{fifthorder}
\frac{F_{\text{exist}}^7(F_{\text{exist}} - 2)^2}{4(F_{\text{exist}} - 1)^6}(\eta^2)^5.
\ee
	A polynomial with real coefficients is a product of irreducible polynomials (with real coefficients) of first and second degrees. We claim that, for $F$ near $F_{\text{2d}}$, the factor corresponding to the degeneracy must be of degree one. If not, we can assume the factor takes the form
\be \label{notquadratic}
(F_{\text{2d}}^2-F^2)(\eta^2)^2+b(F;\nu)\eta^2+c(F;\nu),
\ee
where, by \eqref{fifthorder}
\be
\label{b_coeff}
b(F_{\text{2d}};\nu)=\frac{F_{\text{exist}}^7(F_{\text{exist}} - 2)^2}{4(F_{\text{exist}} - 1)^6}>0.
\ee The discriminant of \eqref{notquadratic} is then positive for $F$ near $F_{\text{2d}}$, making the quadratic reducible in $\mathbb{R}$. A contradiction. Therefore, for $F$ near $F_{\text{2d}}$, the degenerate factor must take the form
\be \label{eta2linear}
(F_{\text{2d}}^2-F^2)\eta^2+b(F;\nu),
\ee
where $b(F;\nu)$ again satisfies \eqref{b_coeff} and, whence \eqref{eta2linear} admits a positive $\eta^2$--zero for $F>F_{\text{2d}}$ which blows up as $F\to\left(F_{\text{2d}}\right)^+$. 

Our analysis suggests that, for $F>F_{\text{2d}}$, there is a double eigenvalue of $G(-,i\tau_*(\eta(F;\nu)),\eta(F;\nu))$ where $\eta(F;\nu)=\sqrt{\frac{b(F;\nu)}{F^2-F_{\text{2d}}^2}}\gg 1$ and $\tau_*(\eta)$ is given by \eqref{tausolution}. To compute for $\eta(F;\nu)$ as zeros of a singular perturbation problem, for fixed $\nu$, denote $\eps=\sqrt{F-F_{\text{2d}}}$. In $\eps$ variable, the singular zero $\eta(\eps;\nu)$ expands as
$\eta(\eps;\nu)=a_{-1}\eps^{-1}+a_0+\mathcal{O}(\eps)$ where, by plugging in the sextic polynomial, we find
	that $a_{-1}=\sqrt{b(F_{\text{2d}};\nu)/(2F_{\text{2d}})}$ and $a_0=0$. Substituting $\eta$ and $F$ in \eqref{tausolution} by $\eta(\eps;\nu)$ and $F_{\text{2d}}+\eps^2$ yields $\tau_*(\eta(\eps;\nu))=b_{-1}\eps^{-1}+\mathcal{O}(\eps)$
where $b_{-1}=\frac{\sqrt{F_{\text{exist}} - 2}}{F_{\text{exist}}}a_{-1}$. Substituting $F$ in $EA_1^{-1}$, $A_1^{-1}$, and $A_2A_1^{-1}$ by $F_{\text{2d}}+\eps^2$ yields
$$
EA_1^{-1}=M_{10}+M_{12}\eps^2+\cdots,\quad A_1^{-1}=M_{20}+M_{22}\eps^2+\cdots,\quad\text{and}\quad A_2A_1^{-1}=M_{30}+M_{32}\eps^2+\cdots,
$$
We now turn to compute distinct eigenvalues of the permanently degenerate matrix $$G(\eps):=(EA_1^{-1}-i\tau_*(\eta(\eps;H_R))A_1^{-1}-i\eta(\eps;H_R) A_2A_1^{-1})$$ and consequently compare their real parts. Substituting the asymptotic expansions in $G(\eps)$ yields 
$$
G(\eps)=\left(-ib_{-1}M_{20}-ia_{-1}M_{30}\right)\eps^{-1}+M_{10}+\mathcal{O}(\eps)
$$
Assume now $R(\eps)$ and $L(\eps):=R(\eps)^{-1}$ block-diagonalizes $G(\eps)$ to its Jordan canonical form $J(\eps)$ and they expand as
$$
L(\eps)=L_0+L_1\eps+\mathcal{O}(\eps^2),\quad R(\eps)=R_0+R_1\eps+\mathcal{O}(\eps^2),\quad J(\eps)=J_{-1}\eps^{-1}+\Lambda+\mathcal{O}(\eps)
$$
where $J_{-1}$ is Jordan block-diagonal and $\Lambda$ is diagonal. The leading matrix $\left(-ib_{-1}M_{20}-ia_{-1}M_{30}\right)$ by, for instance,
$$
R_0:=\left[\begin{array}{ccc} 0 & -\frac{\nu(- \nu^3 - 2\nu^2 + \nu + 2)}{\nu^2 + \nu - 1} & \frac{1}{\nu}+1\\\sqrt{\nu^2 + \nu - 2} & -\frac{\nu(- \nu^3 - 2\nu^2 + \nu + 2)}{\nu^2 + \nu - 1} & \frac{\nu^2 + \nu - 1}{\nu^2}\\ 1 & 0 & \frac{\sqrt{\nu^2 + \nu - 2}}{\nu^2} \end{array}\right],\quad L_0:=R_0^{-1},
$$
transforms to its Jordan canonical form
$$
J_{-1}=-ia_{-1}\left[\begin{array}{rrr} \frac{1}{\sqrt{\nu^2 + \nu - 2}}&1&0\\0&\frac{1}{\sqrt{\nu^2 + \nu - 2}}&0\\0&0&-\sqrt{\nu^2 + \nu - 2}\end{array}\right],
$$
which has purely imaginary eigenvalues. The next order terms are subject to linear systems
$$
\begin{aligned}
&L_0R_1+L_1R_0=\mathbf{0}\\ 
&L_0M_{10}R_0+L_1\left(-ib_{-1}M_{20}-ia_{-1}M_{30}\right)R_0+L_0\left(-ib_{-1}M_{20}-ia_{-1}M_{30}\right)R_1=\Lambda,
\end{aligned}
$$
where $\Lambda={\rm diag}(\Lambda_{11},\Lambda_{11},\Lambda_{33})$, i.e., the $(1,1)$ entry is equal to the $(2,2)$ entry and we are interested in comparing $\Re \Lambda_{11}$ with $\Re \Lambda_{33}$. The foregoing system forces $L_1=-L_0R_1L_0$, which makes the latter system amounts to 
$$
L_0M_{10}R_0+[J_{-1},L_0R_1]=\Lambda,
$$
giving $7$ linear equations for entries of $R_1$. Solving them yields
$R_1\in R_{1,*}+{\rm span}\{R_{1,1},R_{1,2},R_{1,3}\}$ where
$$
\begin{aligned}
R_{1,*}&:=\frac{1}{ia_{-1}}\left[\begin{array}{cc} \frac{F_{\text{exist}}^3(F_{\text{exist}} - 2)}{2(F_{\text{exist}} - 1)^2} & -\frac{F_{\text{exist}}^2\sqrt{F_{\text{exist}} - 2}(F_{\text{exist}}^2 -3F_{\text{exist}} + 2)}{(F_{\text{exist}}- 1)^3}\\
\frac{F_{\text{exist}}^3(F_{\text{exist}} - 2)}{2(F_{\text{exist}} - 1)^2} & -\frac{F_{\text{exist}}(F_{\text{exist}} - 2)^{(3/2)}}{F_{\text{exist}} - 1}\\
0 & 0\end{array}\right.\\
&\left.\begin{array}{c}
-\frac{(\nu + 1)^2(F_{\text{exist}}^2 - 2F_{\text{exist}} + 2)}{(F_{\text{exist}} - 1)^2\sqrt{F_{\text{exist}} - 2}}\\   -\frac{(\nu + 1)(\nu^8 + 4\nu^7 + 3\nu^6 - 5\nu^5 - 6\nu^4 + \nu^3 + 4\nu^2 + 2\nu - 2)}{\nu(F_{\text{exist}} - 1)^3\sqrt{F_{\text{exist}} - 2}}\\ 0  \end{array}\right],\quad R_{1,1}:=\left[\begin{array}{ccc} 0 & 0 & 0\\ 0 & \sqrt{F_{\text{exist}} - 2} & 0\\ 0 & 1 & 0 \end{array}\right]\\
R_{1,2}&:=\left[\begin{array}{ccc} 0 & \frac{F_{\text{exist}}(F_{\text{exist}} - 2)}{F_{\text{exist}} - 1} & 0\\ \sqrt{\nu^2 + \nu - 2} & \frac{F_{\text{exist}}(F_{\text{exist}} - 2)}{F_{\text{exist}} - 1} & 0\\ 1 & 0 & 0 \end{array}\right],\quad R_{1,3}:=\left[\begin{array}{ccc} 0 & 0 & \frac{F_{\text{exist}}}{\sqrt{F_{\text{exist}}- 2}}\\ 0 & 0 & \frac{F_{\text{exist}}-1}{\sqrt{F_{\text{exist}}- 2}}\\ 0 & 0 & 1 \end{array}\right].
\end{aligned}
$$
This allows us to compute 
\be\label{Lambda}
\Lambda={\rm diag}\left\{\frac{F_{\text{exist}}^2-F_{\text{1d}}^2}{2(F_{\text{exist}} - 1)},\frac{F_{\text{exist}}^2-F_{\text{1d}}^2}{2(F_{\text{exist}} - 1)},\frac{F_{\text{exist}}^2}{F_{\text{exist}} - 1}\right\},
\ee
from which we compute
$$
\frac{F_{\text{exist}}^2-F_{\text{1d}}^2}{2(F_{\text{exist}} - 1)}-\frac{F_{\text{exist}}^2}{F_{\text{exist}} - 1}=-\frac{F_{\text{exist}}^2+F_{\text{1d}}^2}{2(F_{\text{exist}} - 1)}<0.
$$
Thus the collision happens for $\gamma_{1,-}(i\tau_*(\eta(F;\nu)),\eta(F;\nu))$ and $\gamma_{2,-}(i\tau_*(\eta(F;\nu)),\eta(F;\nu))$ with smaller real parts than that of $\gamma_{3,-}(i\tau_*(\eta(F;\nu)),\eta(F;\nu))$. The proof is completed.
\end{proof}

In figure \ref{analyticity2d}, we set $H_R=0.7$ ($F_{\text{2d}}(\sqrt{\frac{10}{7}})=2.0590\cdots$) and follow a curve $\lambda=\lambda(\eta)$ in the $\lambda$--plane on which $\gamma_{1,-}(\lambda(\eta),\eta)$ and $\gamma_{2,-}(\lambda(\eta),\eta)$ are identical for $F=2.06$, $2.18$, and $2.25$, respectively, showing that 
\begin{itemize}
\item[(a)] the collision of eigenvalues happens at a point on the imaginary axis and continue into the left and right half planes. Also, the collision continues for $F\gg F_{\text{2d}}$.

\item[(b)] for $F=2.06$ close to $F_{\text{2d}}$, the intercept of the curve with imaginary axis is large,

\item[(c)] our numerics also agree with the analytically derived $\Lambda$; we omit the details of this comparison.
\end{itemize}
\begin{figure}[htbp]
\begin{center}
\includegraphics[scale=0.35]{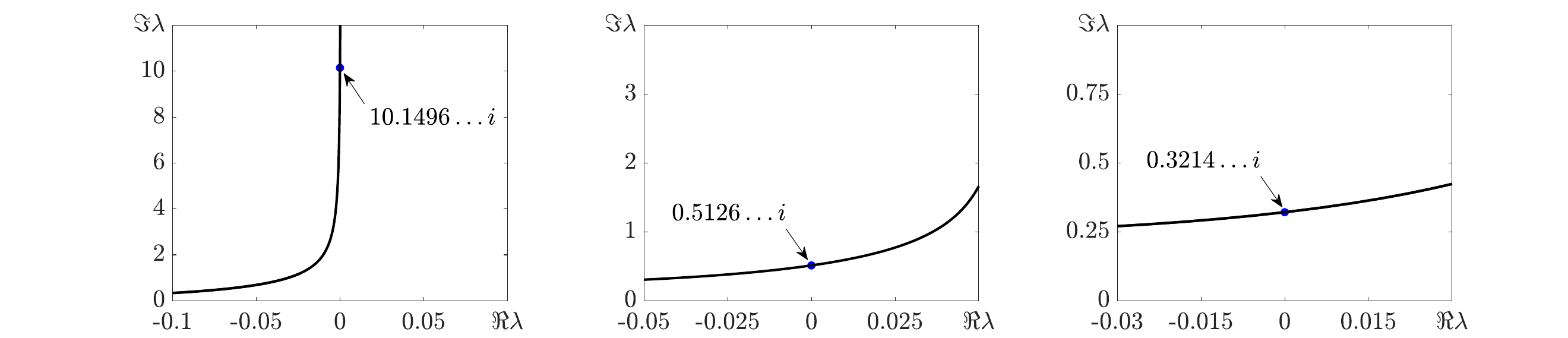}
\end{center}
\caption{The curve $\lambda(\eta)$ on which the matrix $(E-\lambda(\eta)\Id-i\eta A_2)A_1^{-1}$ has a double eigenvalue with smaller real part than that of the other eigenvalue. In these figures, we set $H_R=0.7$ and $F=2.06$, $2.18$, and $2.25$ in the left, middle, and right panel.}
\label{analyticity2d}
\end{figure}

\subsection{Evans functions}\label{s:evans_hydro}
For $2<F<F_{\text{2d}}$, adopting the weighted space \eqref{weighted_space}, the eigenvalue system \eqref{zgeig} becomes
\be 
\label{zgeig_weight}
\left\{\begin{aligned}& W_x=(G(x;\lambda,\eta)-\mu_{L,R}\Id)W=:G_{\mu_{L,R}}W,\quad \text{for  $x\lessgtr 0$}\\
&\hat \psi[\lambda F_0(\bar U) +i\eta  F_2(\bar U) - R(\bar U)]+[W]=0.\end{aligned}\right.
\ee 
Our analysis in the foregoing section shows that there exist weights $\mu_{L,R}$ such that the weighted interior equation \eqref{zgeig_weight}(i) restores consistent splitting at $x=\pm \infty$ for $\Re\lambda>0$. For $F_{\text{char}}<F\leq 2$, consistent splitting holds for the unweighted interior equation \eqref{zgeig}(i) which amounts to setting $\mu_{L,R}=0$ in \eqref{zgeig_weight}(i). By setting $\mu_{L,R}=0$, we identify the eigenvalue system \eqref{zgeig} in the latter case with \eqref{zgeig_weight}.  In either case, there is a basis of $2$ ($0$) solutions $\mathcal{W}_j^-$ ($\mathcal{W}_l^+$) of interior equations 
\eqref{zgeig_weight}(i) on $x<0$ ($x>0$) that decay exponentially as $x\to- \infty$ ($x\to+\infty$), 
whence, evidently, \eqref{zgeig_weight}(i) has an $L^2$
solution  $W=\sum_{j=1}^2 \alpha_j^- \mathcal{W}_j^-$ for $x<0$ ($W\equiv 0$ for $x>0$) satisfying \eqref{zgeig_weight}(ii),
that is, an unstable eigenmode growing as $e^{\Re \lambda t}$,
if and only if the {\it Evans-Lopatinsky determinant}
\be\label{EL}
D(\lambda, \eta):= \det ( [\lambda F_0(\bar U)+ i\eta F_2(\bar U) - R(\bar U)], 
\mathcal{W}_1^-(0^-;\lambda,\eta), \mathcal{W}_2^-(0^-;\lambda,\eta))
\ee
vanishes at frequencies $(\lambda, \eta)$ with $\Re \lambda >0$.
Thus, {\bf spectral stability} amounts to {\it nonvanishing of $D$ for $\Re \lambda>0$.}

Note, in the conservative case $R\equiv 0$, that \eqref{geig}(i) is constant-coefficient,
and \eqref{EL} reduces to the Lopatinsky determinant of Majda \cite{M1,M2}.
For smooth hydraulic shocks $F<F_{\text{char}}$, \eqref{geig}(ii) reduces to continuity across $x=0$,
the sum of dimensions of decaying solutions as $x\to \pm \infty$ becomes $3$, and \eqref{EL} reduces to the Evans function 
\be  \label{Evans}
\Delta(\lambda, \eta):= \det ( \mathcal{W}^+_1(0^+;\lambda,\eta),
\mathcal{W}_{1}^-(0^-;\lambda,\eta),\mathcal{W}_{2}^-(0^-;\lambda,\eta)) 
\ee 
of \cite{AGJ,PW,ZS}.

For both cases, discontinuous or smooth, the Evans-Lopatinsky (Evans) function can be chosen analytically 
in $\eta$ and $\lambda$ on the entire domain $\eta\in \R$ and $\Re \lambda \geq 0$ for any $F< F_{\text{2d}}$.
For, in this case, there is a spectral gap in suitably weighted norm between decaying and growing eigenspaces
of the limiting constant-coefficient systems at $x_1\to \pm \infty$, whence the (unique) 
associated spectral projectors extend analytically. By Kato's Lemma together with the Conjugation Lemma of \cite{MZ3}
one may then define in standard fashion analytic choices of the functions $\mathcal{W}_j^\pm$; see, e.g., \cite{Z0}.

In the 1d case $\eta=0$, the third equation of \eqref{geig} decouples, reducing
to a 2-variable first-order ODE.  In \cite{SYZ} \cite{FRYZ}, this was reduced to a nonlinear (in $\lambda$) eigenvalue problem
for a second-order scalar operator, and stability shown by a generalized Sturm--Liouville argument.
In the present, multi-d case, no such simplification presents itself, and so our strategy as for other complicated
physical systems is to carry out an exhaustive
numerical Evans (Evans-Lopatinsky) function study (See Appendix~\ref{Evans_numeric}), over parameters $F<F_{\text{char}}$ ($F_{\text{char}}<F<F_{\text{2d}}$), $0<H_R/H_L<  1$ characterizing (up to scale-invariance)
the family of hydraulic shock waves of \eqref{sv_intro}, via $\lambda$-winding number computations
indexed by $(F, H_R/H_L, \eta)$.
See \cite{HuZ1,Z4,HLyZ1,HLyZ,BHZ} for the general approach, and \cite{YZ} 
for the specialization to \eqref{sv_intro} in 1d. In the course of our numerical Evans-Lopatinsky study, we formulate a rescaled Evans-Lopatinsky \eqref{rescaledduallopatinsky} whose computation is independent of the choices of weights $\mu_{L,R}$ and necessarily has the same zeros as \eqref{EL}. 


As with any such all-parameters study, a crucial aspect is to truncate or in some cases extend
the parameter/frequency regime to a compact set, allowing uniform control of errors and efficient computation
across the full parameter range.
A first, important step as in the similar 1d study carried out in \cite{YZ}
is to notice that we can adjoin the value $H_R/H_L=0$, closing the parameter set, while
maintaining the basic integration scheme used in \cite{YZ} to evolve $\mathcal{V}_j^\pm$ from $x=\pm \infty$
to $x=0$.

Similarly, away from the origin $(\lambda, \eta)=(0,0)$, values $\Re \lambda=0$ may be adjoined to the frequency set 
without affecting uniformity of estimates.
This leaves the problem of truncating {\it high frequencies} $|(\lambda, \eta)|\geq  R$ and
{\it low frequencies} $0<|(\lambda, \eta)|\leq  1/R$. Both issues turn out to be quite interesting.

\subsection{Low-frequency stability}\label{s:lf}
Because limits $U_\pm$ of a relaxation profile $\bar U$
must be rest points of profile ODE $(F_1(\bar U))_x=R(\bar U)$, we have
$R(U_\pm)=0$, hence $((U_1)_-,(U_{2})_+,s)$ is a shock triple for the associated equilibrium model
(generalizing \eqref{CE})
\be\label{gCE}
(U_1)_t +f_{1}^*(U_1)_x + f_{2}^*(U_1)_y=0,
\ee
where $U=(U_1,U_2)$, $F_j=(F_{j1},F_{j2})$, $f_{j}^*(U_1)=F_{j1}(U_1,U_2^*(U_1))$, and $R(U_1, U_2^*(U_1))\equiv 0$.
Then, we have the following fundamental relation generalizing observations of \cite{GZ,ZS,Go,JLW}, reflecting the
idea that internal structure is not ``seen'' on low-frequency/large space-time scales.

\bl[Discontinuous Zumbrun--Serre Lemma]\label{ZSlem}
Defining $\Delta(\lambda, \eta)$ to be the Lopatinsky determinant for the equilibrium equation \eqref{gCE} 
associated with \eqref{req}, there holds near $(\lambda, \eta)=(0,0)$
\be\label{ZSeq}
D(\lambda, \eta)= \gamma \Delta(\lambda, \eta)+ \mathcal{O}(|(\lambda, \eta|^2),
\ee
where $\gamma$ is a fixed Melnikov-type determinant that is nonvanishing for transversal connections $\bar U$.
\el
\begin{proof}
For $(\lambda, \eta)=(0,0)$, the $U_1$ equation of \eqref{geig}
reduces to $(d/dx)(A_1 V)_1=0$, so that $(A_1\mathcal{V}_j)_1 \equiv 0$ for ``fast,'' or exponentially 
decaying modes, and $(A_1\mathcal{V}_j)_1 \equiv \const$ for ``slow,'' or neutrally decaying modes.
By inspection of the linearized existence problem, which agrees with the eigenvalue equation when
$(\lambda, \eta)=(0,0)$, we see that there are $2$ fast decaying modes, both at $x=-\infty$, and no
slow decaying modes.
The invertible change of coordinates $V\to W:= A_1 V$ thus imposes to lowest order in $\lambda,\eta$ 
a block-determinant reduction on the matrix in \eqref{EL},
for which the $1-1$ block $[\lambda F_0(\bar U)+ i\eta F_2(\bar U) - R(\bar U)]$ 
is independent of $x$ and the $2-2$ block $\gamma$ independent of $\lambda, \eta$.

By translation-invariance, the $x$-derivative of the traveling wave satisfies the interior eigenvalue
equation at the origin, so that we may take without loss of generality
	$\mathcal{V}_1^-(0^-;\lambda,\eta)=\bar U_x$, and, setting
	$(\lambda, \eta)=\rho (\hat \lambda, \hat \eta)$ and evaluatating \eqref{geig} at $\rho=0$, obtain
	$(A_1 \partial_x \bar U)_x = E\partial_x \bar U= \partial_x R(\bar U)$, and 
	and thus, integrating from $-\infty$ to $x$,
	\be\label{unvareq}
	A_1 V_1= A_1 \partial_x \bar U = R(\bar U).
	\ee
	Meanwhile, differentiating \eqref{geig} with respect to $\rho$, we obtain
	$$i
	(A_1 \partial_\rho V_1)_x - E\partial_\rho V_1=
	\hat \lambda A_0 \bar U_x + \hat \eta A_2 \bar U_x=\partial_x(F_0(\bar U)+ F_2(\bar U)),
	$$
	and thus, integrating the first coordinate of this equation (for which $E$ vanishes)
from $-\infty $ to $x$,
	\be\label{vareq}
	(1,0,0)^T A_1 \partial_\rho V_1 =
	(1,0,0^T (F_0(\bar U)+ F_2(\bar U))|_{-\infty}^x
	\ee
	at $\rho=0$,
	Combining \eqref{unvareq} and \eqref{vareq}, we thus obtain
	$$
	W_1^-(0^-)= R(\bar U(0^-)) + \rho \bp (1,0,0^T (F_0(\bar U)+ F_2(\bar U))|_{-\infty}^x
	\\ *\ep + O(\rho^2).
	$$

	Performinging a column operation subtracting the second column $W_1^-$ of \eqref{EL} from 
	the first, we thus have
	$$
	D(\lambda, \eta)=\det \bp \Delta(\lambda, \eta) & 0 & 0\\ * & e_2 W_2^-(0,0) & e_2 W_3^-(0,0) \ep
	=\Delta \gamma,
	$$
	where 
	$$
	\Delta(\lambda,\eta)=(1,0,0)^T ( \lambda F_0\bar U +R(\bar U) + i\eta F_2(\bar U)|_{-\infty}^{+\infty}
	$$
	is the Lopatinsky determinant for the associated equilibrium system and
	$$
	\gamma:= \det \bp e_2 W_2^-(0,0) & e_2 W_3^-(0,0) \ep
	$$
	with $e_2:=\bp 0& 0\\0 & \Id_2\ep$, may be recognized as a Melnikov type determinant whose nonvanishing
	corresponds to transversality of the traveling wave connection in the existence problem.
For similar arguments, see \cite{ZS,BSZ,Go,JLW}.
\end{proof}

The surprising consequence is that equilibrium dynamics are relevant for large-amplitude relaxation shocks 
whose profiles pass far from equilibrium, and may possess subshock discontinuities.
In the present case that \eqref{gCE} is scalar,
there is near the origin an isolated, hence analytically dependent, curve of spectra 
\be\label{lambda_exp}
\lambda(\eta)=i \alpha \eta + \beta \eta^2,
\quad
\alpha= -[f_2^*(U_1)]/[U_1]\in \R
\ee
bifurcating from the ``translational'' eigenvalue $\lambda(0)=0$: $\alpha=0$ in the case \eqref{CE}.

Thus, treatment of low-frequency stability reduces as in the scalar viscous case \cite{HoZ1} to
determination of the sign of $\beta$, which should be do-able either analytically by spectral expansion/Fredholm
solvability conditions as in \cite[\S 5]{HoZ2} or by continuation/numerical root following as in, e.g., \cite{HL},
with $\beta< 0$ corresponding to ``diffusive stability'' and $\beta >0$ presumably indicating
``fingering instabilities'' of the front. 

We do not carry out such an analysis in this work, instead lumping the question of low-frequency diffusive
stability together with determination of intermediate-frequency stability in our numerical experiments described below.  Our results indicate stability, $\Re \lambda \leq 0$ for all spectra $(\lambda, \eta)$
with $0\leq |\eta|\leq C$, implicitly indicating at least neutral diffusive stability $\beta \leq 0$. 

\subsection{High-frequency stability}\label{s:hf}
In 1d, truncation of high frequencies is readily accomplished.
For, a fairly standard WKB type approximation shows that in the
high-frequency limit $\mathcal{V}^\pm_j$ approximately track along stable/unstable subspaces of the
frozen-coefficient version of \eqref{geig}(i), while, in \eqref{geig}(ii), the $\mathcal{O}(1)$ term $[R(\bar U)]$
becomes negligible, yielding convergence to the Majda Lopatinsky determinant for the associated conservative
system with $R=0$, in this case 1d isentropic $\gamma$-law gas dynamics. 
This yields high-frequency stability by the 1d version of {\it Majda's Theorem} \cite{M2} 
on stability of isentropic $\gamma$-law shocks.
In 2d, likewise, this follows in the high-frequency limit
for $\Re \lambda/|\eta|$ bounded uniformly below, from the multi-d
version of Majda's Theorem, as there remains a uniform spectral gap between stable/unstable eigenvalues of
the frozen-coefficient system.

However, for $\Re \lambda\ll |(\Im \lambda, \eta)|$, the situation becomes quite subtle, as noted by Erpenbeck
in a far ahead of its time paper \cite{Er4} 
on detonation stability for the ZND equations, a model of form \eqref{req} with first-order part
consisting in its relevant elements also of compressible gas dynamics.
Viewed as a perturbation of the case $\lambda=i\tau$, the issue, in modern hyperbolic terminology
\cite{K,M1,M2,Met3}, is that there may occur for certain values of $x$, $\tau$, $\eta$ ``glancing points''
where growing and decaying modes coalesce at a Jordan block, and this corresponds generically to an Airy-type
{\it turning point} in the eigenvalue ODE.
The end result
is that if 
a certain quantity \footnote{The quantity $c^2- (u-s)^2$ in the notation of \cite{LWZ}.} decreases along the profile as $x$ decreases toward the endstate $x=-\infty$ 
at which it is initiated, then the limiting decaying mode is oscillating there and to the left of any turning point, and growing/decaying to the right, which, with connection properties of the Airy function, yields again convergence
to Majda's Lopatinsky determinant for the nonreactive (Neumann) shock at the end of the reaction zone, hence
stability. 
In the reverse case, there may appear a lattice of unstable transverse modes, depending on 
a certain ``ratio condition'' that must be calculated. For a detailed discussion, see \cite{LWZ};
for an informal exposition, see \cite[\S 2.2.2-2.2.3]{Z9}.

The structure of (SV) is similar to but somewhat simpler than that of the detonation equations, being a $3\times 3$
relaxation of isentropic gas dynamics, whereas detonation is a $5\times 5$ relaxation of full nonisentropic gas dynamics augmented with an additional reaction variable.
We are thus able to carry out a high-frequency analysis quite parallel to that of \cite{Er4,LWZ}, as we now do.

\smallskip

\noindent\textbf{High-frequency analysis.} 
For $\lambda=\lambda_r+i\lambda_i$, $\lambda_i$, $\eta\in \mathbb{R}$, and $\lambda_r\geq 0$, we analyze solutions of \eqref{zgeig} in the high-frequency region \be\label{domainhf}\sqrt{|\lambda|^2+\eta^2}\sim \mathcal{O}(r),\quad \text{with $r\gg 1$.}\ee 
Considering the order of $\lambda$, $\eta$ with respect to $r$, we may further split the high-frequency region \eqref{domainhf} into case (a) $|\lambda|\sim \mathcal{O}(r)$, $|\eta|\sim \mathcal{O}(1)$; case (b) $|\lambda|\sim \mathcal{O}(1)$, $|\eta|\sim \mathcal{O}(r)$; case (c) $|\lambda|\sim \mathcal{O}(r)$, $|\eta|\sim \mathcal{O}(r)$. As noticed in \cite{Er4,LWZ,Z9}, possible presence of turning points where the $\mathcal{O}(r)$ term of the matrix $G(x;\lambda,\eta)$ \eqref{zgeig} processes generalized eigenvectors can complicate the high-frequency analysis. In search of turning points, the eigenvalues of $-(\lambda A_0+i\eta A_2)A_1^{-1}(x)$ reads
\ba\label{Lead_eigenvalues_allcases}
&\gamma_1(x;\lambda)=\lambda F_{\text{exist}}H(x),\\
&\gamma_2,\gamma_3(x;\lambda,\eta)=-\left(\lambda\frac{F^2H}{F_{\text{exist}}(H^3-H_s^3)}\pm\frac{H\sqrt{H^3F^2\lambda^2 + (H^3-H_s^3)H\eta^2}}{H^3-H_s^3}\right)(x).
\ea
where $H_s=(F/F_{\text{exist}})^{2/3}$ is the singular point of the profile ode \cite[eq. (2.14)]{YZ} and there hold, for discontinuous hydraulic shocks, $H(x)>H_s$, for $x<0$ and $H(x)\equiv H_R=1/\nu^2<H_s$ for $x>0$, and for smooth hydraulic shocks $H(x)>H_s$, for $-\infty<x<+\infty$.

In case (a) and case (b), easy computations show that the leading $\mathcal{O}(r)$ terms of \eqref{Lead_eigenvalues_allcases} are distinct, whence $(\lambda,\eta)$ cannot be turning point frequency. This leaves the search to case (c).

If the leading $\mathcal{O}(r)$ term of $\gamma_2$ is equal to that of $\gamma_3$, there must hold 
\be
\label{turningpoint1}
\lambda=\pm i\tau_1(H)\eta+\mathcal{O}(1)\quad\text{and}\quad \eta=\mathcal{O}(r),\quad \text{where}\quad \tau_1(H):=\frac{\sqrt{H^3-H_s^3}}{HF}.
\ee 
Clearly, $\tau_1(H)$ is monotone increasing.

On the other hand, if the leading $\mathcal{O}(r)$ term of $\gamma_1$ is equal to either that of $\gamma_2$ or that of $\gamma_3$, there must hold 
\be
\label{turningpoint2}
\lambda=\tau_2(H)\eta+\mathcal{O}(1)\quad\text{and}\quad \eta=\mathcal{O}(r),\quad \text{where}\quad \tau_2(H):=\frac{1}{F_{\text{exist}}H}.
\ee 
Clearly, $\tau_2(H)$ is monotone decreasing. 

Without loss of generality, in the lemma below, we take $\eta=r$ and positive sign in \eqref{turningpoint1}[i].
\begin{lemma}[turning points and turning point frequencies; hydraulic shocks]\label{turningpoint:lemma}
${}$

\noindent Consider a non-characteristic discontinuous $($smooth$)$ hydraulic shock profile. 

At the frequencies $(\lambda,\eta)=(i\tau_1 r+\mathcal{O}(1),r)$, where \be\label{tau1range}
\tau_1\in\big(\tau_1(\min(\{1,H_*\})),\tau_1(\max(\{1,H_*\}))\big) \quad \quad \big(\tau_1\in(\tau_1(H_R),\tau_1(1))\big),
\ee
there exists a unique turning point $x_1(\tau_1)\in (-\infty,0)$  $\big(x_1(\tau_1)\in(-\infty,\infty)\big)$, where $\gamma_2(x_1(\tau_1);i\tau_1r,r)=\gamma_3(x_1(\tau_1);i\tau_1r,r)\neq \gamma_1(x_1(\tau_1);i\tau_1r)$ and the geometric multiplicity of the double eigenvalue is $1$.

At the frequencies $(\lambda,\eta)=(\tau_2r+\mathcal{O}(1),r)$, where \be\label{tau2range}
\tau_2\in\big(\tau_2(\max(\{1,H_*\})),\tau_2(\min(\{1,H_*\}))\big) \quad \quad \big(\tau_2\in(\tau_2(1),\tau_2(H_R))\big),
\ee
there exists a unique turning point $x_2(\tau_2)\in (-\infty,0)$  $\big(x_2(\tau_2)\in(-\infty,\infty)\big)$, where $\gamma_1(x_2(\tau_2);\tau_2r)=\gamma_3(x_2(\tau_2);\tau_2r,r)=\eta\neq \gamma_2(x_2(\tau_2);\tau_2r,r)$ and the geometric multiplicity of the double eigenvalue is $1$.
\end{lemma}
\begin{proof}
The lemma follows by direct computations.
\end{proof}
\begin{remark}
Within {\normalfont case (c)}, our foregoing analyses confirm $(\lambda,\eta)$ can be turning point frequencies if either $|\lambda_i|\sim\mathcal{O}(1)$ or $|\lambda_r|\sim\mathcal{O}(1)$ and, in particular, they cannot be turning point frequencies if $|\lambda_i|,|\lambda_r|\sim\mathcal{O}(r)$.
\end{remark}
We now treat the former turning point frequencies discussed in Lemma~\ref{turningpoint:lemma}. Let $\lambda=i\tau_1r+\lambda_r$ and $\eta=r$, where $|\lambda_r|\sim \mathcal{O}(1)$. Away from the turning point, $x\neq x_1(\tau_1)$, the leading matrix 
$-ir(\tau_1 A_1^{-1}+A_2A_1^{-1})(x)$ can be diagonalized as $${\rm diag}\big(\gamma_1(x;ir\tau_1),\gamma_2(x;ir\tau_1,r),\gamma_3(x;ir\tau_1,r)\big),$$
where
\ba \label{gammaturning}
&\gamma_1(x;ir\tau_1)=ir\tau_1 F_{\text{exist}}H(x),\\
&\gamma_2,\gamma_3(x;ir\tau_1,r)=-ir\left(\frac{F^2H\tau_1}{F_{\text{exist}}(H^3-H_s^3)}\pm\frac{FH^{5/2}\sqrt{\tau_1^2-\tau_1(H)^2}}{H^3-H_s^3}\right)(x),
\ea
by similar transform $T^{-1}(\cdot;\tau_1)\left(-ir(\tau_1 A_1^{-1}+A_2A_1^{-1})\right)T(\cdot;\tau_1)$ where, for example
\ba 
\label{matT}
T(\cdot;\tau_1)&=\left[\begin{array}{cc}F_{\text{exist}} H&F F_{\text{exist}} H \sqrt{\tau_1^2-\tau_1(H)^2}\\H + F_{\text{exist}} H - 2&F(H + F_{\text{exist}} H - 1)\sqrt{\tau_1^2-\tau_1(H)^2}+ F_{\text{exist}} H^{3/2}\tau_1 \\F_{\text{exist}} H \tau_1&\sqrt{H}
\end{array}\right.\\
&\quad\quad\left.\begin{array}{c}-F F_{\text{exist}} H \sqrt{\tau_1^2-\tau_1(H)^2}\\F(H + F_{\text{exist}} H - 1)\sqrt{\tau_1^2-\tau_1(H)^2}+ F_{\text{exist}} H^{3/2}\tau_1\\\sqrt{H}
\end{array}\right].
\ea
In $Z_1:=T^{-1}Z$, \eqref{zgeig} becomes
\be\label{eigenZ1}
{Z_1}_x={\rm diag}\big(\gamma_1(\cdot;ir\tau_1),\gamma_2(\cdot;ir\tau_1,r),\gamma_3(\cdot;ir\tau_1,r)\big) Z_1+M(\cdot;\tau_1)Z_1,
\ee
where $M(\cdot;\tau_1,\lambda_r):=\left(T^{-1}(E\,A_1^{-1}-\lambda_r A_1^{-1})T-H_xT^{-1}d_HT\right)(\cdot;\tau_1)$ has diagonal entries
\ba \label{diagonalM}
M_{11}&(\cdot;\tau_1,\lambda_r)=HF_{\text{exist}}\lambda_r+\Big(\tau_1^2HF_{\text{exist}}^2(H^3-H_s^3)(H + F_{\text{exist}} H - 1)\\
&+2(F_{\text{exist}} + 1)H^3+2(F^2 - 1)H^2-2H_s^3(F_{\text{exist}} + 1)^2H\\
&+2H_s^3(F_{\text{exist}} + 1)\Big)\Big((1 + H^2\tau_1^2F_{\text{exist}}^2)(H^3  - H_s^3)\Big)^{-1},\\
M_{22}&(\cdot;\tau_1,\lambda_r)=\left(-\frac{FH_s^{3/2}H}{H^3-H_s^3}-\frac{FH^{5/2}\tau_1}{(H^3 -H_s^3)\sqrt{\tau_1^2-\tau_1(H)^2}}\right)\lambda_r+\#_2(\cdot;\tau_1)+\#_3(\cdot;\tau_1),\\
M_{33}&(\cdot;\tau_1,\lambda_r)=\left(-\frac{FH_s^{3/2}H}{H^3-H_s^3}+\frac{FH^{5/2}\tau_1}{(H^3 -H_s^3)\sqrt{\tau_1^2-\tau_1(H)^2}}\right)\lambda_r-\#_2(\cdot;\tau_1)+\#_3(\cdot;\tau_1),
\ea 
where
\ba \label{pond2}
\#_2(\cdot;\tau_1):=&-\frac{\tau_1(H + F_{\text{exist}} H - 1)}{2FF_{\text{exist}}^3 H^{5/2} (H^3-H_s^3)(F_{\text{exist}}^2 H^2\tau_1^2 + 1)}\\&
\times\Big(\sqrt{\tau_1^2-\tau_1(H)^2}F^2F_{\text{exist}}^2H^2(2F_{\text{exist}}^2H^3-F^2(F_{\text{exist}} + 1)H+F^2)\\
&+\frac{F_{\text{exist}}^2H^3(F_{\text{exist}}^2H^3-F^2(F_{\text{exist}} + 1)H+ 2F^2)}{\sqrt{\tau_1^2-\tau_1(H)^2}}\Big),
\ea 

\ba 
\label{pond3}
\#_3(\cdot;\tau_1):=&\Big(-\tau_1^2F^2 H^2(-3F_{\text{exist}}^2 H^3 +5(F_{\text{exist}} + 1)^2H^2 -14(F_{\text{exist}} +1)H+9)\\&-F_{\text{exist}}^4 H^6+ F_{\text{exist}}^2(F_{\text{exist}} + 1)^2H^5+F_{\text{exist}}^2(F^2 - 1)H^3 -3F^2(F_{\text{exist}} + 1)^2H^2\\
&+8F^2(F_{\text{exist}} + 1)H-5F^2\Big)\Big(4H(H_s^3 - H^3)(F_{\text{exist}}^2H^2\tau_1^2 + 1)\Big)^{-1}\\&+\Big((2H_s^3 +  H^3)(H - 1)(- F_{\text{exist}}^2H^2 + 2F_{\text{exist}}H + H - 1)\Big)\\
&\times\Big(4F^2(H_s^3-H^3)(F_{\text{exist}}^2H^2\tau_1^2 + 1)(\tau_1^2-\tau_1(H)^2)\Big)^{-1}.
\ea
For later usage, we compute by \eqref{diagonalM}--\eqref{pond3} 
\ba \label{Minfty}
M_{11}&(-\infty;\tau_1,\lambda_r)=F_{\text{exist}}(\lambda_r+1)+\frac{F_{\text{exist}}}{F_{\text{exist}}^2\tau_1^2+1}>0,\\
M_{11}&(+\infty;\tau_1,\lambda_r)=(\sqrt{H_R}+1)\lambda_r+\frac{1}{\sqrt{H_R}} + 1+\frac{\sqrt{H_R} +1}{\sqrt{H_R}(\sqrt{H_R} + 1)^2\tau_1^2+ \sqrt{H_R}}>0,\\
M_{22}&(-\infty;\tau_1,\lambda_r)=\left(-\frac{FH_s^{3/2}}{1-H_s^3}-\frac{F\tau_1}{(1-H_s^3)\sqrt{\tau_1^2-\tau_1(1)^2}}\right)\lambda_r+\#_2(-\infty;\tau_1)+\#_3(-\infty;\tau_1),\\
M_{22}&(+\infty;\tau_1,\lambda_r)=\left(-\frac{FH_s^{3/2}H_R}{H_R^3-H_s^3}-\frac{FH_R^{5/2}\tau_1}{(H_R^3-H_s^3)\sqrt{\tau_1^2-\tau_1(H_R)^2}}\right)\lambda_r+\#_2(+\infty;\tau_1)+\#_3(+\infty;\tau_1),\\
M_{33}&(-\infty;\tau_1,\lambda_r)=\left(-\frac{FH_s^{3/2}}{1-H_s^3}+\frac{F\tau_1}{(1-H_s^3)\sqrt{\tau_1^2-\tau_1(1)^2}}\right)\lambda_r-\#_2(-\infty;\tau_1)+\#_3(-\infty;\tau_1),\\
M_{33}&(+\infty;\tau_1,\lambda_r)=\left(-\frac{FH_s^{3/2}H_R}{H_R^3-H_s^3}+\frac{FH_R^{5/2}\tau_1}{(H_R^3-H_s^3)\sqrt{\tau_1^2-\tau_1(H_R)^2}}\right)\lambda_r-\#_2(+\infty;\tau_1)+\#_3(+\infty;\tau_1).
\ea

\begin{lemma}\label{lemma_high2}
Consider turning point frequencies $(\lambda,\eta)=(i\tau_1r+\lambda_r,r)$, where $\tau_1$ satisfies \eqref{tau1range} and $|\lambda_r|\sim\mathcal{O}(1)$, with associated turning point $x_1(\tau_1)\in(-\infty,0)$ $\big(x_1(\tau_1)\in(-\infty,\infty)\big)$ of a non-characteristic discontinuous $($smooth$)$ hydraulic shock profile. If the profile is non-monotone, then for $x\in(-\infty,x_1(\tau_1))$, $\gamma_2(x;ir\tau_1,r)$ and $\gamma_3(x;ir\tau_1,r)$ are purely imaginary, and, for $x\in(x_1(\tau_1),0)$, $\gamma_2(x;ir\tau_1,r)$ and $\gamma_3(x;ir\tau_1,r)$ are complex with non-vanishing opposite real parts of order $\mathcal{O}(r)$ and $\Re M_{22}(x;\tau_1,\lambda_r)=\Re M_{33}(x;\tau_1,\lambda_r)$. If the profile is monotone decreasing, then, for $x\in(-\infty,x_1(\tau_1))$, $\gamma_2(x;ir\tau_1,r)$ and $\gamma_3(x;ir\tau_1,r)$ are complex with non-vanishing opposite real parts of order $\mathcal{O}(r)$ and $\Re M_{22}(x;\tau_1,\lambda_r)=\Re M_{33}(x;\tau_1,\lambda_r)$, and, for $x\in(x_1(\tau_1),0)$ $\big(x\in(x_1(\tau_1),\infty)\big)$, $\gamma_2(x;ir\tau_1,r)$ and $\gamma_3(x;ir\tau_1,r)$ are purely imaginary.
\end{lemma}
\begin{proof}
Clearly $\tau_1(H)$ \eqref{turningpoint1} is strictly increasing on $[H_s,\infty)$. If the profile is non-monotone, it is monotone increasing on $(-\infty,0)$. Then, for $x\in(-\infty,x_1(\tau_1))$, we have $H_s<1<H(x)<H(x_1(\tau_1))$ and $\tau_1^2-\tau_1(H(x))^2>0$, from which the assertions for $\gamma_{2,3}(x)$ and $\Re M_{22}(x)$, $\Re M_{33}(x)$ follow by referring to \eqref{gammaturning}, and \eqref{diagonalM}-\eqref{pond3}. Similarly, the remaining claims also hold.
\end{proof}

\begin{lemma}\label{decayingmode_decreasing}
Under the consideration made in Lemma~\ref{lemma_high2}, if the profile is monotone decreasing, for $x\in(x_1(\tau_1),0)$ the solutions of the system \eqref{eigenZ1} that connect to the unstable manifold at $-\infty$ are linear combinations of
\ba \label{decayingmodesminus}
\mathcal{W}_1^-(x;\tau_1,\lambda_r)=&e^{\int_0^x\gamma_1(y;ir\tau_1)+M_{11}(y;\tau_1,\lambda_r)dy}\left(T_1(x;\tau_1)+\mathcal{O}(1/r)\right)\quad \text{and}\\
\mathcal{W}_2^-(x;\tau_1,\lambda_r)=&e^{\int_{x_1(\tau_1)}^x\gamma_2(y;ir\tau_1,r)+M_{22}(y;\tau_1,\lambda_r)dy}\left(T_2(x;\tau_1)+\mathcal{O}(1/r)\right)\\
&+e^{\int_{x_1(\tau_1)}^x\gamma_3(y;ir\tau_1,r)+M_{33}(y;\tau_1,\lambda_r)dy}\left(T_3(x;\tau_1)+\mathcal{O}(1/r)\right),
\ea 
where $T_j$ denotes the $j^{th}$ column of the matrix $T$ \eqref{matT}. Furthermore, if the profile is smooth, 
	the solution that connects to the stable manifold at $+\infty$ is parallel to 
\ba \label{decayingmodesplus}
&\mathcal{W}_3^-(x;\tau_1,\lambda_r)=e^{\int_0^x\gamma_2(y;ir\tau_1,r)+M_{22}(y;\tau_1,\lambda_r)dy}\left(T_2(x;\tau_1)+\mathcal{O}(1/r)\right).\\
\ea 
\end{lemma}
\begin{proof}
	We sketch the argument given in \cite{LWZ}, directing the reader to the reference for full details.
	For $x\in (-\infty, x_1(\tau_1)-\eps)$, $\eps>0$ fixed but sufficiently small, one finds by standard
	WKB approximation that the solution ``tracks'' the positive real part eigenspace of the 
	frozen-coefficient symbol.
	Indeed, using a block-diagonalization, we may reduce on all of $(-\infty,0)$ to the $2\times 2$
	subspace consisting of the two acoustic modes that collide at $x_1(\tau_1)$.
	On $ (x_1(\tau_1)-\eps, x_1(\tau_1)+\eps)$, this $2\times 2$ block may be conjugated to an Airy
	equation, thus giving a link between the value at $x_1(\tau_1)-\eps$, determined by tracking,
	and the value at $x_1(\tau_1+\eps)$.
	For the Airy equation, it is a standard fact that the decaying solution on the growth/decay side
	connects to an equal sum of both sinusoidal solutions on the oscillatory side.
	Applying a tracking argument now to the interval $(x_1(\tau_1)+\eps,0)$, on which these two modes
	are spectrally separated, we obtain the result.
\end{proof}

\begin{proposition}\label{high_stability_turning;decreasing;discontinuous} 
Under the consideration made in Lemma~\ref{lemma_high2}, for decreasing discontinuous profiles, set
\be  \label{indexhydraulic} 
{\rm ind}(\tau_1):=e^{\int_{x_1(\tau_1)}^02\#_2(y;\tau_1)dy}\left|\frac{\det\big([\tau_1 F_0(\bar{U})+ F_2(\bar{U})],T_1(0^-;\tau_1), T_2(0^-;\tau_1)\big)}{\det\big([\tau F_0(\bar{U})+ F_2(\bar{U})],T_1(0^-;\tau_1), T_3(0^-;\tau_1)\big)}\right|.
\ee  
if ${\rm ind}(\tau_1)>1$, there exists a sequence $\{(\lambda_n,r_n)(\tau_1):\lambda_n(\tau_1)=i\tau_1 r_n(\tau_1)+\lambda_{n,r}(\tau_1),\eta_n(\tau_1)=r_n(\tau_1)\}_{n=1}^\infty$ of zeros of the Evans-Lopatinsky $D(\lambda,\eta)$ \eqref{EL} satisfying
$$
 r_{n+1}(\tau_1)-r_n(\tau_1)\rightarrow \frac{2\pi }{\int_{x_1(\tau_1)}^0\frac{2FH^{5/2}\sqrt{\tau_1^2-\tau_1(H)^2}}{(H^3-H_s^3)}(y)dy},\quad \text{as $n\rightarrow \infty$,}
$$
and
$$
\lambda_{n,r}(\tau_1)\rightarrow \lambda_{r,\infty}:= \frac{{\rm log}\big({\rm ind}(\tau_1)\big)}{e^{\int_{x_1(\tau_1)}^0\frac{2FH^{5/2}\tau_1}{(H^3-H_s^3)\sqrt{\tau_1^2-\tau_1(H)^2}}(y)dy}}>0,\quad \text{as $n\rightarrow +\infty$,}
$$
giving a sequence of unstable point spectra depended continuously on $\tau_1\in(\tau_1(H_*),\tau_1(1))$.

If ${\rm ind}(\tau_1)<1$, $D(\lambda,\eta)$ \eqref{EL} does not vanish at turning point frequencies $(\lambda,\eta)=(i\tau_1r+\lambda_r,r)$, where $\tau_1\in(\tau_1(H_*),\tau_1(1))$ and $|\lambda_r|\sim\mathcal{O}(1)$, $\lambda_r \geq 0$.

If ${\rm ind}(\tau_1)=0$, stability at the turning point frequencies depends on higher order terms in \eqref{decayingmodesminus}, which we do not treat.
\end{proposition}
\begin{proof}
Inserting \eqref{decayingmodesminus} in \eqref{EL}, we obtain
$$ \begin{aligned}
D(i\tau_1 r+\lambda_r,r)=&\det\Big([(i\tau_1 r+\lambda_r)F_0(\bar{U})+ir F_2(\bar{U})-R(\bar{U})],T_1(0^-)+\mathcal{O}(\frac{1}{r}), \\&\quad\quad\;\;\; e^{\int_{x_1}^0\gamma_2(y)+M_{22}(y)dy}T_2(0^-)+e^{\int_{x_1}^0\gamma_3(y)+M_{33}(y)dy}T_3(0^-)+\mathcal{O}(\frac{1}{r})\Big)\\
=&ire^{\int_{x_1}^0\gamma_2(y)+M_{22}(y)dy}\det\big([\tau_1 F_0(\bar{U})+ F_2(\bar{U})],T_1(0^-), T_2(0^-)\big)\\
&+ire^{\int_{x_1}^0\gamma_3(y)+M_{33}(y)dy}\det\big([\tau_1 F_0(\bar{U})+ F_2(\bar{U})],T_1(0^-), T_3(0^-)\big)+\mathcal{O}(1)\\
=&ire^{\int_{x_1}^0\gamma_3(y)+M_{33}(y)dy}\det\big([\tau_1 F_0(\bar{U})+ F_2(\bar{U})],T_1(0^-), T_3(0^-)\big)\\
&\times\big(\widetilde{{\rm ind}}(\tau_1,\lambda_r)+1\big)+\mathcal{O}(1),
\end{aligned}
$$ 
where
$$
\begin{aligned}
\widetilde{{\rm ind}}(\tau_1,\lambda_r):=&e^{\int_{x_1}^0(\gamma_2-\gamma_3)(y)dy}e^{\int_{x_1}^0(M_{22}-M_{33})(y)dy}\frac{\det\big([\tau_1 F_0(\bar{U})+ F_2(\bar{U})],T_1(0^-), T_2(0^-)\big)}{\det\big([\tau_1 F_0(\bar{U})+ F_2(\bar{U})],T_1(0^-), T_3(0^-)\big)}\\
=&e^{\int_{x_1}^0(\gamma_2-\gamma_3)(y)dy}e^{-\lambda_r\int_{x_1}^0\frac{2FH^{5/2}\tau_1}{(H^3-H_s^3)\sqrt{\tau_1^2-\tau_1(H)^2}}(y)dy}e^{\int_{x_1}^02\#_2(y)dy}\\
&\times\frac{\det\big([\tau_1 F_0(\bar{U})+ F_2(\bar{U})],T_1(0^-), T_2(0^-)\big)}{\det\big([\tau_1 F_0(\bar{U})+ F_2(\bar{U})],T_1(0^-), T_3(0^-)\big)}.
\end{aligned}
$$
The factorization above is valid provided that $\det\big([\tau_1 F_0(\bar{U})+ F_2(\bar{U})],T_1(0^-), T_3(0^-)\big)\neq 0$, which we prove now. Direct computation shows
\ba 
\label{ELT1T3}
&\det\big([\tau_1 F_0(\bar{U})+ F_2(\bar{U})],T_1(0^-), T_3(0^-)\big)\\
=&\tfrac{1}{2}F^{-2}\sqrt{H_*}(H_* - H_R)\Big(FF_{\text{exist}}\sqrt{H_*}(2F^2\tau_1^2 + H_R + H_*)\sqrt{\tau_1^2-\tau_1(H)^2}\\
&+2F^2F_{\text{exist}}^2 H_*^2\tau_1\left(\tau_1^2+\frac{H_R(2H_R - H_*)}{H_*^2(\sqrt{H_R}+ 1)^2}\right)\Big)>0,
\ea 
since for $\tau_1\in (\tau_1(H_*),\tau_1(1))$
\ba
\label{tau1inquality}
&\tau_1^2+\frac{H_R(2H_R - H_*)}{H_*^2(\sqrt{H_R}+ 1)^2}> \tau_1(H_*)^2+\frac{H_R(2H_R - H_*)}{H_*^2(\sqrt{H_R}+ 1)^2}\\
=&\frac{H_R(F^2(H_*-H_R) + H_R^2H_* + 2H_R^{3/2}H_* + H_RH_*)}{F^2H_*^2(\sqrt{H_R} + 1)^2}>0.
\ea 

Recall Lemma~\ref{lemma_high2}, $\gamma_2(y)$ and $\gamma_3(y)$ are purely imaginary for $x\in(x_1,0)$. By adjusting $r$, one can make $e^{\int_{x_1}^0(\gamma_2(y)-\gamma_3(y))dy}$ real so that the equation $\widetilde{{\rm ind}}(\tau_1,\lambda_r)+1=0$ amounts to
$$
e^{\lambda_r\int_{x_1}^0-\frac{2FH^{5/2}\tau_1}{(H^3-H_s^3)\sqrt{\tau_1^2-\tau_1(H)^2}}(y)dy}{\rm ind}(\tau_1)=1,\;\;\text{
where}\;\; -\infty<\int_{x_1}^0-\frac{2FH^{5/2}\tau_1}{(H^3-H_s^3)\sqrt{\tau_1^2-\tau_1(H)^2}}(y)dy<0.$$ We pause to note the integrand above is integrable for it $\sim\mathcal{O}(H(x_1)-H(y))^{-1/2}$ as $y\rightarrow x_1^+$. If ${\rm ind}(\tau_1)>1$, then there exists $\lambda_{r,\infty}>0$ such that the equation above holds, from where we conclude the former part of the proposition. If ${\rm ind}(\tau_1)<1$, the equation above holds for no $\lambda_r\geq 0$, justifying the latter part of the proposition.
\end{proof}

\begin{corollary}\label{cor:high_stability_turning;decreasing;discontinuous}
It can be proven that ${\rm ind}(\tau_1)<1$ for all $\tau_1\in(\tau_1(H_*),\tau_1(1))$. Therefore, $D(\lambda,\eta)$ \eqref{EL} does not vanish at any turning point frequencies $(\lambda,\eta)=(i\tau_1r+\lambda_r,r)$.
\end{corollary}
\begin{proof}
Recalling the formula of $\#_2$ \eqref{pond2}, we will show the integrand $\#_2(y;\tau_1)<0$ for $x_1<y<0$, by noting that for $H>H_*>H_R$ there hold 
\ba \label{pond2sign}
&H+F_{\text{exist}}H-1=\frac{H - H_R + H\sqrt{H_R} + HH_R}{H_R}>0, \\
&2F_{\text{exist}}^2H^3-F^2(F_{\text{exist}} + 1)H+F^2>0,\quad \text{and}\\
&F_{\text{exist}}^2H^3-F^2(F_{\text{exist}} + 1)H+ 2F^2>0.
\ea 
To prove \eqref{pond2sign}[iii], note that
$$
\begin{aligned}
&d_H(F_{\text{exist}}^2H^3-F^2(F_{\text{exist}} + 1)H +2F^2)(y) \\
>& d_H(F_{\text{exist}}^2H^3-F^2(F_{\text{exist}} + 1)H +2F^2)(0^-)\\
=&3F_{\text{exist}}^2H_*^2-F^2(F_{\text{exist}} + 1)=\left(3(\sqrt{H_R} + 1)^2H_*^2 -F^2H_R(H_R + \sqrt{H_R} + 1)\right)/H_R^2>0,
\end{aligned}
$$
where the last inequality can be justified by (i) for $H_R=0.25$ and $F=1$, 
$$
3(\sqrt{H_R} + 1)^2H_*^2 -F^2H_R(H_R + \sqrt{H_R} + 1)=\frac{163}{128} - \frac{9\sqrt{137}}{128}>0,$$
and (ii) the inconsistency of the system\footnote{The second equation of the system is given by the Rankine-Hugoniot conditions for the subshock connecting states $H_*$ and $H_R$.}
$$
\left\{\begin{aligned}&3(\sqrt{H_R} + 1)^2H_*^2 -F^2H_R(H_R + \sqrt{H_R} + 1)=0,\\
&(\sqrt{H_R} + 1)^2H_*^2+ H_R(\sqrt{H_R}+ 1)^2H_* -2F^2H_R=0.\end{aligned}\right.
$$
Here, solving the system yields
$$
\begin{aligned}
0=&-(H_R + \sqrt{H_R} - 5)^2F^2+ 3H_R(\sqrt{H_R} + 1)^2(H_R + \sqrt{H_R} + 1)\\
<&-(H_R + \sqrt{H_R} - 5)^2(H_R+\sqrt{H_R})^2+ 3H_R(\sqrt{H_R} + 1)^2(H_R + \sqrt{H_R} + 1)\\
=&-H_R(\sqrt{H_R}+ 1)^2(\sqrt{H_R} + 2) (1-\sqrt{H_R})(11-H_R - \sqrt{H_R})<0, \quad \text{for $0<H_R<1$.}
\end{aligned}
$$
Thus, \eqref{pond2sign}[iii] follows from $$\begin{aligned}&F_{\text{exist}}^2H^3-F^2(F_{\text{exist}} + 1)H +2F^2>F_{\text{exist}}^2H_*^3-F^2(F_{\text{exist}} + 1)H_* +2F^2\\
=&\left((\sqrt{H_R} + 1)^2H_*^3-F^2H_R(H_R + \sqrt{H_R} + 1)H_* +2F^2H_R^2\right)/H_R^2>0,
\end{aligned}$$ where the last inequality can be justified by (i) for $H_R=0.25$ and $F=1$, $$(\sqrt{H_R} + 1)^2H_*^3-F^2H_R(H_R + \sqrt{H_R} + 1)H_* +2F^2H_R^2=\frac{13\sqrt{137}}{1536} -\frac{13}{512}>0,$$
and (ii) the inconsistency of the system
$$
\left\{\begin{aligned}&(\sqrt{H_R} + 1)^2H_*^3-F^2H_R(H_R + \sqrt{H_R} + 1)H_* +2F^2H_R^2=0,\\
&(\sqrt{H_R} + 1)^2H_*^2+ H_R(\sqrt{H_R}+ 1)^2H_* -2F^2H_R=0.\end{aligned}\right.
$$
Here, solving the system yields $0=(1 - \sqrt{H_R} - H_R)F^2+H_R(\sqrt{H_R} + 1)^2$, which, if $(1 - \sqrt{H_R} - H_R)\geq 0$, gives a contradiction, otherwise, a contradiction can be concluded from 
$$
\begin{aligned}
0=&(1 - \sqrt{H_R} - H_R)F^2+H_R(\sqrt{H_R} + 1)^2\\
>&(1 - \sqrt{H_R} - H_R)\frac{(\sqrt{H_R} + 1)^2(H_R + 1)}{2H_R}+H_R(\sqrt{H_R} + 1)^2\\
=&\frac{(H_R - 1)^2(H_R + \sqrt{H_R} + 1)}{2H_R}>0,
\end{aligned}
$$
where, recalling the domain of existence of decreasing discontinuous hydraulic shocks (See \cite[Proposition 2.1]{FRYZ}), the first inequality follows.

Note that
$$
d_H(2F_{\text{exist}}^2H^3-F^2(F_{\text{exist}} + 1)H +F^2)(y)> d_H(F_{\text{exist}}^2H^3-F^2(F_{\text{exist}} + 1)H +2F^2)(y)>0.
$$
Thus, \eqref{pond2sign}[ii] follows from $$\begin{aligned}&2F_{\text{exist}}^2H^3-F^2(F_{\text{exist}} + 1)H +F^2>2F_{\text{exist}}^2H_*^3-F^2(F_{\text{exist}} + 1)H_* +F^2\\
=&\left(2(\sqrt{H_R} + 1)^2H_*^3-F^2H_R(H_R + \sqrt{H_R} + 1)H_* +F^2H_R^2\right)/H_R^2>0,
\end{aligned}$$ where the last inequality can be justified by (i) for $H_R=0.25$ and $F=1$, $$2(\sqrt{H_R} + 1)^2H_*^3-F^2H_R(H_R + \sqrt{H_R} + 1)H_* +F^2H_R^2=\frac{9\sqrt{137}}{256} -\frac{75}{256}>0,$$
and (ii) the inconsistency of the system
$$
\left\{\begin{aligned}&2(\sqrt{H_R} + 1)^2H_*^3-F^2H_R(H_R + \sqrt{H_R} + 1)H_* +F^2H_R^2=0,\\
&(\sqrt{H_R} + 1)^2H_*^2+ H_R(\sqrt{H_R}+ 1)^2H_* -2F^2H_R=0.\end{aligned}\right.
$$
Here, solving the system yields $$2(H_R + \sqrt{H_R} - 3)^2F^4 -H_R(\sqrt{H_R} + 1)^2(5H_R + 5\sqrt{H_R} - 6)F^2+2H_R^2(\sqrt{H_R} + 1)^4=0.$$
Treating the lefthand side of the equation above as a quadratic function of $F^2$, we find that the discriminant 
$9H_R^2(\sqrt{H_R} + 1)^4(\sqrt{H_R} - 1)(\sqrt{H_R} + 2)(H_R + \sqrt{H_R} + 6)$ is negative for $0<H_R<1$, yielding the inconsistency.

Then, we prove that the absolute value factor in ${\rm ind}(\tau_1)$ is less than $1$. Direct computation shows 
$$\begin{aligned}
&\Big(\det\big([\tau_1 F_0(\bar{U})+ F_2(\bar{U})],T_1(0^-), T_2(0^-)\big)\Big)^2-\Big(\det\big([\tau_1 F_0(\bar{U})+ F_2(\bar{U})],T_1(0^-), T_3(0^-)\big)\Big)^2\\
=&-4FF_{\text{exist}}^3H_*^{7/2}\tau_1(H_* - H_R)^2\sqrt{\tau_1^2-\tau_1(H_*)^2}\left(\tau_1^2+\frac{H_R + H_*}{2F^2}\right)\left(\tau_1^2+\frac{H_R(2H_R - H_*)}{H_*^2(\sqrt{H_R}+ 1)^2}\right)<0,
\end{aligned}
$$
where we have used \eqref{tau1inquality}. Combining with \eqref{ELT1T3}, we obtain
$$
\left|\frac{\det\big([\tau_1 F_0(\bar{U})+ F_2(\bar{U})],T_1(0^-), T_2(0^-)\big)}{\det\big([\tau_1 F_0(\bar{U})+ F_2(\bar{U})],T_1(0^-), T_3(0^-)\big)}\right|<1.
$$
\end{proof}

\begin{lemma}\label{high_stability_turning;increasing;discontinuous} The Evans-Lopatinsky determinant $D(\lambda,\eta)$ \eqref{EL} associated to a non-monotone discontinuous hydraulic shock profile does not vanish at turning point frequencies $(\lambda,\eta)=(i\tau_1r+\lambda_r,r)$, where $\tau_1$ satisfies \eqref{tau1range} and $|\lambda_r|\sim\mathcal{O}(1)$, $\lambda_r\geq 0$.
\end{lemma}
\begin{proof}

Recalling the proof of Lemma~\ref{decayingmode_decreasing}, by symmetry,  we know that an equally weighted combination of oscillatory modes at $-\infty$ connects
to the decaying (in forward $x$-direction) mode for $x>x_1$, i.e.,
$$
e^{\int_{x_1}^x\gamma_3(y)+M_{33}(y)dy}\big(T_3(x)+\mathcal{O}(\frac{1}{r})\big).
$$
It follows that the single oscillatory mode that corresponds to the decaying mode at $-\infty$ does not connect to the decaying (in forward $x$-direction) mode for $x>x_1$ alone, and instead must connect to some linear
combination involving also the growing (in forward $x$-direction) mode, i.e.
$$
e^{\int_{x_1}^x\gamma_2(y)+M_{22}(y)dy}\big(T_2(x)+\mathcal{O}(\frac{1}{r})\big).
$$
By \eqref{gammaturning} and Lemma~\ref{lemma_high2}, that for $y\in(x_1,0)$, $\gamma_2(y)$ and $\gamma_3(y)$ are complex with opposite real parts of order $\mathcal{O}(r)$ to conclude 
$$
\left|e^{\int_{x_1}^0\gamma_3(y)+M_{33}(y)dy}\big(T_3(0^-)+\mathcal{O}(\frac{1}{r})\big)\right|\sim e^{-Ar}\ll 1.
$$
Therefore,
$$
D(i\tau_1 r+\lambda_r,r)=ire^{\int_{x_1}^0\gamma_2(y)+M_{22}(y)dy}\left(\det\big([\tau_1 F_0(\bar{U})+ F_2(\bar{U})],T_1(0^-), T_2(0^-)\big)+\mathcal{O}(\frac{1}{r})\right).
$$
Direct computation shows
\ba\label{ELT1T2}
&\Im\det\big([\tau_1 F_0(\bar{U})+ F_2(\bar{U})],T_1(0^-), T_2(0^-)\big)\\
=&-\tfrac{1}{2}F^{-1}F_{\text{exist}} H_*(H_* - H_R)(2F^2\tau_1^2 + H_R + H_*)\sqrt{\tau_1(H_*)^2-\tau_1^2}\neq 0.
\ea 
Hence the leading order term of the Evans-Lopatinsky determinant does not vanish.
\end{proof}
\begin{lemma}\label{high_stability_turning;smooth} The Evans function $\Delta(\lambda,\eta)$ \eqref{Evans} associated to a non-characteristic decreasing smooth hydraulic shock profile does not vanish at turning point frequencies $(\lambda,\eta)=(i\tau_1r+\lambda_r,r)$, where $\tau_1$ satisfies \eqref{tau1range} and $|\lambda_r|\sim\mathcal{O}(1)$, $\lambda_r \geq 0$.
\end{lemma}
\begin{proof}
Inserting \eqref{decayingmodesminus} and \eqref{decayingmodesplus} in \eqref{Evans}, we obtain
\ba 
\Delta(i\tau_1 r+\lambda_r,r)=&\det\Big(T_2(0)+\mathcal{O}(\frac{1}{r}),T_1(0)+\mathcal{O}(\frac{1}{r}), \\&\quad\quad\;\;\; e^{\int_{x_1}^0\gamma_2(y)+M_{22}(y)dy}T_2(0)+e^{\int_{x_1}^0\gamma_3(y)+M_{33}(y)dy}T_3(0)+\mathcal{O}(\frac{1}{r})\Big)\\
=&e^{\int_{x_1}^0\gamma_3(y)+M_{33}(y)dy}\det\big(T_2(0),T_1(0), T_3(0)\big)+\mathcal{O}(\frac{1}{r}),
\ea 
so long as $x=0$ is not the turning point, from which we easily see the leading $\mathcal{O}(1)$ 
	term does not vanish.
	In the case where $x=0$ is the turning point, the asymptotics
\eqref{decayingmodesminus} and \eqref{decayingmodesplus} do not apply, as they
do not give information exactly at the turning point but only nearby.
However, in this case we need only test the stability condition at a point $\tilde x_1$ close to but not equal to $0$
to obtain the same result, again verifying nonvanishing.
	Here, we are using the standard Evans function fact (Abel's Theorem) that vanishing/nonvanishing 
	of the Wronskian by which it is defined is independent of the point at which it is evaluated.
\end{proof}

\begin{theorem}[high-frequency stability; decreasing discontinuous profiles]\label{high_stability;decreasing;discontinuous} Any non-characteristic decreasing discontinuous hydraulic shock profile is high-frequency stable in the sense that there are no zeros $(\lambda,\eta)$ of its Evans-Lopatinsky determinant \eqref{EL} with $\Re\lambda\geq 0$ and $\sqrt{|\lambda|^2+\eta^2} \geq R_0$ for some $R_0>0$ sufficiently large.
\end{theorem}
\begin{proof}
We shall prove that the Evans-Lopatinsky determinant \eqref{EL} achieves no zeros for the 
three cases (a) (b) (c) outlined in the beginning of the section. 
	Corollary~\ref{cor:high_stability_turning;decreasing;discontinuous} has treated a sub-case of case (c) where the frequencies are of turning point frequencies. 

We now prove for the remaining sub-cases within case (c) $|\lambda|\sim \mathcal{O}(r)$, $|\eta|\sim \mathcal{O}(r)$.

\noindent (1) $|\lambda_i|,|\eta|\sim\mathcal{O}(r)$, $|\lambda_r|\sim\mathcal{O}(1)$:

The remaining frequencies in this region are $(\lambda,\eta)=(i\tau_1r+\lambda_r,r)$, with (i) $\tau_1\in(0,\tau_1(H_*))$, (ii) $\tau_1\in(\tau_1(1),+\infty)$, (iii) $\tau_1=\tau_1(H_*)$, and (iv) $\tau_1=\tau_1(1)$.

\noindent(i) $\tau_1\in(0,\tau_1(H_*))$:

In this case, $(\lambda,\eta)$ is not a turning point frequency. By \eqref{gammaturning}, for $x<0$, $\gamma_1(x;ir\tau_1)$ is purely imaginary and $\gamma_2,\gamma_3(x;ir\tau_1;r)$ are complex with non-vanish opposite real parts of order $\mathcal{O}(r)$. And, by \eqref{Minfty}[i], $M_{11}(-\infty;\tau_1,\lambda_r)>0$. We can use standard WKB methods to conclude that solutions of the system \eqref{eigenZ1} that connect to the unstable manifold at $-\infty$ are linear combinations of
$$
\begin{aligned}
\mathcal{W}_1^-(x;\tau_1,\lambda_r)=&e^{\int_0^x\gamma_1(y;ir\tau_1)+M_{11}(y;\tau_1,\lambda_r)dy}\left(T_1(x;\tau_1)+\mathcal{O}(1/r)\right)\quad \text{and}\\
\mathcal{W}_2^-(x;\tau_1,\lambda_r)=&e^{\int_{0}^x\gamma_2(y;ir\tau_1,r)+M_{22}(y;\tau_1,\lambda_r)dy}\left(T_2(x;\tau_1)+\mathcal{O}(1/r)\right).
\end{aligned}
$$
Plugging them into \eqref{EL}, we find that the Evans-Lopatinsky determinant does not vanish since \eqref{ELT1T2} holds.

\noindent(ii) $\tau_1\in(\tau_1(1),+\infty)$:

In this case, $(\lambda,\eta)$ is not a turning point frequency. By \eqref{gammaturning}, for $x<0$, $\gamma_1(x;ir\tau_1)$, $\gamma_2(x;ir\tau_1;r)$, and $\gamma_3(x;ir\tau_1;r)$ are purely imaginary. And, by \eqref{Minfty}[i], $M_{11}(-\infty;\tau_1,\lambda_r)>0$. Direct computation reveals
\ba\label{useofF2d} &F_{\text{exist}}^2H(-\infty)^3-F^2(F_{\text{exist}}+1)H(-\infty)+2F^2
=&(F_{\text{exist}}-1)(F_{\text{2d}}^2-F^2)>0,
\ea 
which gives $\#_2(-\infty;\tau_1)<0$. Hence, by \eqref{Minfty},
$$
\Re(M_{33}-M_{22})(-\infty;\tau_1,\lambda_r)=\frac{2F\tau_1\lambda_r}{(1-H_s^3)\sqrt{\tau_1^2-\tau_1(1)^2}}-2\#_2(-\infty;\tau_1)>0,\quad \text{for $\lambda_r\geq 0$.}
$$
In the weighted space where consistent splitting holds, solutions of the system \eqref{eigenZ1} that connect to the unstable manifold at $-\infty$ are linear combinations of
$$
\begin{aligned}
\mathcal{W}_1^-(x;\tau_1,\lambda_r)=&e^{\int_0^x\gamma_1(y;ir\tau_1)+M_{11}(y;\tau_1,\lambda_r)dy}\left(T_1(x;\tau_1)+\mathcal{O}(1/r)\right)\quad \text{and}\\
\mathcal{W}_2^-(x;\tau_1,\lambda_r)=&e^{\int_{0}^x\gamma_3(y;ir\tau_1,r)+M_{33}(y;\tau_1,\lambda_r)dy}\left(T_3(x;\tau_1)+\mathcal{O}(1/r)\right).
\end{aligned}
$$
Plugging them into \eqref{EL}, we find that the Evans-Lopatinsky determinant does not vanish for \eqref{ELT1T3} holds.

\noindent(iii) $\tau_1=\tau_1(H_*)$:

In this case, $0^-$ is a turning point. This is governed by an Airy equation on half its domain, whence
one finds by Airy equation facts \cite{LWZ} that again the continuation of the decaying solution at $-\infty$
	to $x=0^-$ is approximated by the equal sum of the two sinusoidal solutions to the right of the origin (just
	before they have a chance to oscillate).
	This validates the index computation as before.

\noindent(iv) $\tau_1=\tau_1(1)$:

In this case, $-\infty$ is a turning point. This scenario is not governed by an Airy equation, but is
	a genuinely new situation. Nonetheless, as shown in \cite{LWZ}, it can be transformed
	to Bessel's equation, with the same conclusion that the decaying mode at $-\infty$
	emerges for $x<0$ finite as the equal sum of the two sinusoidal solutions colliding at minus infinity,
	again validating the index computation.

\noindent (2) $|\lambda_i|,|\lambda_r|,|\eta|\sim\mathcal{O}(r)$:

In this case, $(\lambda,\eta)$ is not a turning point frequency. By \eqref{Lead_eigenvalues_allcases}, for $x<0$, $\gamma_1(x;\lambda)$, $\gamma_2(x;\lambda,\eta)$, and $\gamma_3(x;\lambda,\eta)$ achieves spectral gaps of order $\mathcal{O}(r)$. Moreover, $\Re\gamma_1(x;\lambda)=\lambda_rF_{\text{exist}}H(x)>0$ and $\Re(\gamma_2-\gamma_3)(x;\lambda,\eta)>0$. Hence, in the weighted space where consistent splitting holds, solutions of the system \eqref{eigenZ1} that connect to the unstable manifold at $-\infty$ are linear combinations of
$$
\begin{aligned}
\mathcal{W}_1^-(x;\tau_1,\lambda_r)=&e^{\int_0^x\gamma_1(y;\lambda)+M_{11}(y;\lambda,\eta)dy}\left(T_1(x;\lambda)+\mathcal{O}(1/r)\right)\quad \text{and}\\
\mathcal{W}_2^-(x;\lambda,\eta)=&e^{\int_{0}^x\gamma_3(y;\lambda,\eta)+M_{33}(y;\lambda,\eta)dy}\left(T_3(x;\lambda,\eta)+\mathcal{O}(1/r)\right).
\end{aligned}
$$
Plugging them into \eqref{EL}, we find that the Evans-Lopatinsky determinant does not vanish for it converges to the Lopatinsky determinant associated with planar Riemann shock in the whole space connecting $H_*$ and $H_R$.

\noindent (3) $|\lambda_r|,|\eta|\sim\mathcal{O}(r)$, $|\lambda_i|\sim \mathcal{O}(1)$:

In this case, $(\lambda,\eta)$ can be a turning point frequency of the latter case discussed in Lemma~\ref{turningpoint:lemma}.  However, different from the previous case, this concerns collision of two decaying modes and not collision of a decaying and a growing mode. By a block diagonalization argument, their span remains fixed as the total eigenspace for the frozen coefficient symbol of their two associated eigenvalues, even through collision. Thus, this type
of turning point does not at all affect the stability computation, and we arrive to lowest order at a nonvanishing
multiple of the nonvanishing Lopatinsky determinant for the subshock, as usual.

We now prove for case (a) $|\lambda|\sim \mathcal{O}(r)$, $|\eta|\sim \mathcal{O}(1)$. 

Indeed, case (a) corresponds to a perturbation of the 1d case and we have shown in \cite{YZ} that the leading $\mathcal{O}(r)$ order term of the Evans-Lopatinsky does not vanish.

Finally, we prove for case (b) $|\lambda|\sim \mathcal{O}(1)$, $|\eta|\sim \mathcal{O}(r)$.

Indeed, case (b) corresponds to case (c) (1) (i) by allowing $\tau_1=0$. And, we see the derivations hold for $\tau_1=0$.
\end{proof}

\begin{theorem}[high-frequency stability; nonmonotone discontinuous profiles]\label{high_stability;nonmonotone;discontinuous} Any non-monotone discontinuous hydraulic shock profile with $F<F_{\text{2d}}$ is high-frequency stable in the sense that there are no zeros $(\lambda,\eta)$ of its Evans-Lopatinsky determinant \eqref{EL} with $\Re\lambda\geq 0$ and $\sqrt{|\lambda|^2+\eta^2} \geq R_0$ for some $R_0>0$ sufficiently large.
\end{theorem}
\begin{proof}
We shall prove that the Evans-Lopatinsky determinant \eqref{EL} achieves no zeros for the three cases (a) (b) (c) outlined in the beginning of the section. Lemma~\ref{high_stability_turning;increasing;discontinuous} has treated a sub-case of case (c) where the frequencies are of turning point frequencies. 

We now prove for the remaining sub-cases within case (c) $|\lambda|\sim \mathcal{O}(r)$, $|\eta|\sim \mathcal{O}(r)$.

\noindent (1) $|\lambda_i|,|\eta|\sim\mathcal{O}(r)$, $|\lambda_r|\sim\mathcal{O}(1)$:

The remaining frequencies in this region are $(\lambda,\eta)=(i\tau_1r+\lambda_r,r)$, with (i) $\tau_1\in(0,\tau_1(1))$, (ii) $\tau_1\in(\tau_1(H_*),+\infty)$, (iii) $\tau_1=\tau_1(H_*)$, and (iv) $\tau_1=\tau_1(1)$.

\noindent(i) $\tau_1\in(0,\tau_1(1))$:

In this case, $(\lambda,\eta)$ is not a turning point frequency. By \eqref{gammaturning}, for $x<0$, $\gamma_1(x;ir\tau_1)$ is purely imaginary and $\gamma_2,\gamma_3(x;ir\tau_1;r)$ are complex with non-vanish opposite real parts of order $\mathcal{O}(r)$. And, by \eqref{Minfty}[i], $M_{11}(-\infty;\tau_1,\lambda_r)>0$. Solutions of the system \eqref{eigenZ1} that connect to the unstable manifold at $-\infty$ are linear combinations of
$$
\begin{aligned}
\mathcal{W}_1^-(x;\tau_1,\lambda_r)=&e^{\int_0^x\gamma_1(y;ir\tau_1)+M_{11}(y;\tau_1,\lambda_r)dy}\left(T_1(x;\tau_1)+\mathcal{O}(1/r)\right)\quad \text{and}\\
\mathcal{W}_2^-(x;\tau_1,\lambda_r)=&e^{\int_{0}^x\gamma_2(y;ir\tau_1,r)+M_{22}(y;\tau_1,\lambda_r)dy}\left(T_2(x;\tau_1)+\mathcal{O}(1/r)\right).
\end{aligned}
$$
Plugging them into \eqref{EL}, we find that the Evans-Lopatinsky determinant does not vanish for \eqref{ELT1T2} holds.

\noindent(ii) $\tau_1\in(\tau_1(H^*),+\infty)$:

In this case, $(\lambda,\eta)$ is not a turning point frequency. By \eqref{gammaturning}, for $x<0$, $\gamma_1(x;ir\tau_1)$, $\gamma_2(x;ir\tau_1;r)$, and $\gamma_3(x;ir\tau_1;r)$ are purely imaginary. And, by \eqref{Minfty}[i], $M_{11}(-\infty;\tau_1,\lambda_r)>0$. Also, we infer from $\sqrt{\tau_1^2-\tau_1(H)^2}>0$ and \eqref{useofF2d} that $\#_2(-\infty;\tau_1)<0$. Hence, by \eqref{Minfty},
$$
\Re(M_{33}-M_{22})(-\infty;\tau_1,\lambda_r)=\frac{2F\tau_1\lambda_r}{(1-H_s^3)\sqrt{\tau_1^2-\tau_1(1)^2}}-2\#_2(-\infty;\tau_1)>0,\quad \text{for $\lambda_r\geq 0$.}
$$
In the weighted space where consistent splitting holds, solutions of the system \eqref{eigenZ1} that connect to the unstable manifold at $-\infty$ are linear combinations of
$$
\begin{aligned}
\mathcal{W}_1^-(x;\tau_1,\lambda_r)=&e^{\int_0^x\gamma_1(y;ir\tau_1)+M_{11}(y;\tau_1,\lambda_r)dy}\left(T_1(x;\tau_1)+\mathcal{O}(1/r)\right)\quad \text{and}\\
\mathcal{W}_2^-(x;\tau_1,\lambda_r)=&e^{\int_{0}^x\gamma_3(y;ir\tau_1,r)+M_{33}(y;\tau_1,\lambda_r)dy}\left(T_3(x;\tau_1)+\mathcal{O}(1/r)\right).
\end{aligned}
$$
Plugging them into \eqref{EL}, we find that the Evans-Lopatinsky determinant does not vanish for \eqref{ELT1T3} holds.

\noindent(iii) $\tau_1=\tau_1(H_*)$:

	In this case, $0^-$ is a turning point. This is treatable as done in the decreasing case above \cite{LWZ}.

\noindent(iv) $\tau_1=\tau_1(1)$:

In this case, $-\infty$ is a turning point. 
	This is treatable as done in the decreasing case above \cite{LWZ}.

\noindent (2) $|\lambda_i|,|\lambda_r|,|\eta|\sim\mathcal{O}(r)$:

In this case, the Evans-Lopatinsky determinant does not vanish for it converges to the Lopatinsky determinant associated with planar Riemann shock in the whole space connecting $H_*$ and $H_R$.

\noindent (3) $|\lambda_r|,|\eta|\sim\mathcal{O}(r)$, $|\lambda_i|\sim \mathcal{O}(1)$:

In this case, $(\lambda,\eta)$ can be a turning point frequency of the latter case discussed in Lemma~\ref{turningpoint:lemma}. As discussed in the monotone decreasing case above, this type of collision does not affect the stability analysis, yielding convergence to a nonvanishing multiple of the nonvanishing Majda Lopatinsky determinant at the subshock.

We now prove for case (a) $|\lambda|\sim \mathcal{O}(r)$, $|\eta|\sim \mathcal{O}(1)$. 

Indeed, case (a) corresponds to a perturbation of the 1d case and we can make use of the convective spectral stability results of \cite{FRYZ} particular at high frequencies.

Finally, we prove for case (b) $|\lambda|\sim \mathcal{O}(1)$, $|\eta|\sim \mathcal{O}(r)$.

Indeed, case (b) corresponds to case (c) (1) (i) by allowing $\tau_1=0$. And, we see the derivations hold for $\tau_1=0$.

\end{proof}
\begin{theorem}[high-frequency stability; decreasing smooth profiles]\label{high_stability;decreasing;smooth} Any non-characteristic decreasing smooth hydraulic shock profile is high-frequency stable in the sense that there are no zeros $(\lambda,\eta)$ of its Evans function \eqref{Evans} with $\Re\lambda\geq 0$ and $\sqrt{|\lambda|^2+\eta^2} \geq R_0$ for some $R_0>0$ sufficiently large.
\end{theorem}
\begin{proof}
Lemma~\ref{high_stability_turning;smooth} has treated a sub-case of case (c) where the frequencies are of turning point frequencies. The treatment for the remaining sub-cases of case (c) and the other two cases (a) (b) 
are similar to those of Theorems \ref{high_stability;decreasing;discontinuous} and 
\ref{high_stability;nonmonotone;discontinuous}.
\end{proof}

\smallskip

\noindent{\bf Comments on the general case.} Here, and in the detonation case originally treated by
Erpenbeck \cite{Er4,LWZ}, we have been able to rigorously determine stability or instability of
the high-frequency stability index. However, it is important to note that, in more complicated settings
where one cannot compute this index by hand, it should still be computable numerically.
And, if numerical computation too proves problematic, there is a great deal of information carried
by the WKB analysis as a dichotomy of possible behaviors. Namely, if one does not see a clearly-defined
curve of instabilities in the high-frequency region, one may tentatively conclude stability.
(Of course there is always the possibility that one has not tested far enough out.)

\subsection{Medium-frequency and diffusive stability}\label{s:shockMF}
Having reduced to a compact parameter set by demonstration of high-frequency stability, we are now
in position to complete our study numerically.
Determination of diffusive (i.e., second-order low-frequency \cite{Z0,Z3,YZ,JNRYZ}) 
and medium-frequency stability, $1/R\leq  |(\lambda, \eta)|\leq  R$ are 
carried out by numerical Evans computations; see Appendix \ref{Evans_numeric} for further discussion.
The result is that all waves satisfying the convective stability (weighting) condition at the endstates satisfy also
both medium-frequency spectral stability and at least a neutral version $\beta\leq 0$ of the diffusive (second-derivative) stability condition.
The latter conclusion is obtained by contradiction, as $\beta>0$ would imply existence of strictly unstable spectra for $|\eta|\to 0$, which we do not see.


\subsection{Stability of channel flow}\label{s:channel_flow_hydraulic}
So far we have discussed the stability of planar shocks in the whole space.
Some interesting further issues arise in treating the main physical application of flow in a channel
$x\in \R$, $0\leq  y\leq  L$, with wall-type boundary conditions 
\be\label{BC}
\hbox{\rm $v=0$ at $y=0,L$, where $L=$ channel width.}
\ee
Denoting $\hat{V}$ and $\hat{\psi}$ the Laplace transforms in $t$ of $V$ and $\psi$, we reduce \eqref{lineqs}-\eqref{BC} to eigenvalue system
\ba\label{eigen_channel}
&\lambda A_0 \hat V +(A_1\hat V)_x+ (A_2\hat V)_y =E\hat V, \quad
[\lambda\hat \psi F_0(\bar U) + \hat \psi_y F_2(\bar U) - R(\bar U)\hat \psi] +[A_1 \hat V]=0,\\ &\hat{p}(x,0;\lambda)=\hat{p}(x,L;\lambda)=0.
\ea
For $L^2$ stability, we assume $\hat{V}$, $\hat\psi$ in \eqref{eigen_channel} are $L^2$ integrable, i.e., $\hat h,\hat q,\hat p\in L^2(\mathbb{R}\times(0,L))$ and $\hat \psi \in L^2(0,L)$. We infer from Fubini's theorem that, for a.e. $x\in \mathbb{R}$, $\hat h(x,\cdot;\lambda),\hat q(x,\cdot;\lambda),\hat p(x,\cdot;\lambda)\in L^2(0,L)$. Since $\hat{p}(x,0;\lambda)=\hat{p}(x,L;\lambda)=0$, we set in the ``ghost'' mirror channel $\mathbb{R}\times[-L,0]$ the ``ghost'' $\hat{p}(x,y;\lambda)$ by
\be\label{ghostp} 
\hat{p}(x,y;\lambda):=-\hat{p}(x,-y;\lambda),\quad\text{for $y\in[-L,0]$},
\ee
and extend it on the whole line as a $2L$-periodic odd function. We therefore represent $\hat{p}(x,y;\lambda)$ by Fourier series
$
\hat{p}(x,y;\lambda)=\sum_{n=-\infty}^{\infty}p_n(x;\lambda)e^{i\frac{n\pi y}{L}},
$
where 
$$
p_n(x;\lambda):=\frac{1}{2L}\int_{-L}^{L}\hat{p}(x,y;\lambda)e^{-i\frac{n\pi y}{L}}dy=-\frac{1}{L}\int_{0}^L i\hat p(x,y;\lambda)\sin(\frac{n\pi y}{L})dy,
$$
from which we see $p_n(x;\lambda)=-p_{-n}(x;\lambda)$, yielding that $\hat p$ consists of purely sine modes and \eqref{eigen_channel}[iii] are satisfied. As for $\hat h$, $\hat q$,  and $\hat{\psi}$, we set in the ``ghost'' mirror channel $\mathbb{R}\times[-L,0]$ the ``ghost'' functions
\be\label{ghosthqpsi}
\hat h(x,y;\lambda):=\hat h(x,-y;\lambda),\;\; \hat q(x,y;\lambda):=\hat q(x,-y;\lambda),\;\; \hat{\psi}(y;\lambda):=\hat{\psi}(-y;\lambda),\quad \text{for $y\in[-L,0]$,}
\ee
and extend them on the whole line as $2L$-periodic even functions, writing
$$
\hat{h}(x,y;\lambda)=\sum_{n=-\infty}^{\infty}h_n(x;\lambda)e^{i\frac{n\pi y}{L}},\quad  
\hat{q}(x,y;\lambda)=\sum_{n=-\infty}^{\infty}q_n(x;\lambda)e^{i\frac{n\pi y}{L}},\quad 
\hat{\psi}(y;\lambda)=\sum_{n=-\infty}^{\infty}\psi_n(\lambda)e^{i\frac{n\pi y}{L}},$$ where $h_n$, $q_n$, and $\psi_n$ are Fourier coefficients computed by, for instance, $$h_n(x;\lambda)=\frac{1}{2L}\int_{-L}^{L}\hat{h}(x,y;\lambda)e^{-i\frac{n\pi y}{L}}dy=\frac{1}{L}\int_{0}^L \hat h(x,y;\lambda)\cos(\frac{n\pi y}{L})dy.$$ 
Consequently, $h_n(x;\lambda)=h_{-n}(x;\lambda)$, $q_n(x;\lambda)=q_{-n}(x;\lambda)$, and $\psi_n(x;\lambda)=\psi_{-n}(x;\lambda)$.
Combining, there hold
\ba \label{fouriermodes}
\hat{h}(x,y;\lambda)=&2\sum_{n=0}^{\infty} h_n(x;\lambda)\cos(\frac{n\pi y}{L}),&\hat{q}(x,y;\lambda)=&2\sum_{n=0}^{\infty} q_n(x;\lambda)\cos(\frac{n\pi y}{L}),\\
\hat{p}(x,y;\lambda)=&2i\sum_{n=1}^{\infty} p_n(x;\lambda)\sin(\frac{n\pi y}{L}),\quad\text{and}, &\hat{\psi}(y;\lambda)=&2\sum_{n=0}^{\infty} \psi_n(\lambda)\cos(\frac{n\pi y}{L}).
\ea 

Also, it is easy to check that if $(\hat{V}=[\hat h,\hat q,\hat p]^T,\hat\psi)$ 
is a solution of the eigenvalue system \eqref{eigen_channel} on the channel $\mathbb{R}\times[0,L]$, the ``ghost'' functions $([\hat h,\hat q,\hat p]^T,\hat\psi)$ defined by \eqref{ghostp} \eqref{ghosthqpsi} solve \eqref{eigen_channel} on the ``ghost'' mirror channel $\mathbb{R}\times[-L,0]$. Therefore, the extended $2L$-periodic (in $y$) functions $([\hat h,\hat q,\hat p]^T,\hat\psi)$ solve \eqref{eigen_channel} on the extended channel $\mathbb{R}\times[-L,L]$ with periodic boundary conditions
\be\label{periodicbd}
\hat h(x,-L)=\hat h(x,L),\;\; \hat q(x,-L)=\hat q(x,L),\;\;\hat p(x,-L)=\hat p(x,L),\quad \text{and}\quad \hat{\psi}(-L)=\hat{\psi}(L).
\ee

\begin{lemma}\label{lem:periodic} The existence of a normal mode  associated to frequencies $(\lambda,\eta=n\pi/L)$ with $\Re\lambda>0$, $n\in\mathbb{N}$ for channel flow with width $L$ and wall boundary condition is equivalent to existence of a normal mode associated with the same frequencies  for channel flow with width $2L$ and periodic boundary conditions, or, equivalently, for flow in the whole space.
In particular, (i) normal mode instability for channel flow with wall boundary and width $L$ is equivalent to normal mode instability for
channel flow with periodic boundary condition and width $2L$, and (ii) normal mode instability for channel flow with either wall or periodic boundary conditions implies normal mode instability for flow in the whole space. (iii) normal mode stability for flow in the whole space implies normal mode stability for channel flow with either wall or periodic boundary conditions.
Moreover, for normal modes $(\hat h, \hat q, \hat p,\hat \psi)$ of channel flow with wall boundary conditions,
$\hat p(x,\cdot;\lambda)$ has a pure sine expansion and $\hat h(x,\cdot;\lambda), \hat q(x,\cdot;\lambda)$, and $\hat{\psi}(\cdot;\lambda)$ have pure cosine expansions.
\end{lemma}
\begin{proof}
Assuming existence of a non-trivial solution $([h(x)\;q(x)\;p(x)]^T\in L^2,\psi\in \mathbb{C})$ of \eqref{geig} at $\eta=n\pi/L$ for some $\Re \lambda >0$, we can define $\hat h=h(x)\cos(n\pi y/L)$, $\hat q=q(x)\cos(n\pi y/L)$, $\hat p=i p(x)\sin(n\pi y/L)$, and $\hat \psi=\psi\cos(n\pi y/L)$, being $L^2$ solution of \eqref{eigen_channel}. Conversely, a non-trivial $L^2$ solution $(\hat{V}=[\hat h,\hat q,\hat p]^T,\hat\psi)$ of \eqref{eigen_channel} for some $\Re\lambda >0$, as we derived above, has Fourier representation \eqref{fouriermodes}. The Plancherel's identities
$$\begin{aligned}
||\hat h||_{L^2(\mathbb{R}\times(0,L))}^2=&2L\sum_{n=0}^{\infty}||h_n||_{L^{2}(\mathbb{R})}^2,&||\hat q||_{L^2(\mathbb{R}\times(0,L))}^2=&2L\sum_{n=0}^{\infty}||q_n||_{L^{2}(\mathbb{R})}^2,\\
||\hat p||_{L^2(\mathbb{R}\times(0,L))}^2=&2L\sum_{n=1}^{\infty}||p_n||_{L^{2}(\mathbb{R})}^2,\quad \text{and} &||\hat \psi||_{L^2(0,L)}^2=&2L\sum_{n=1}^{\infty}||\psi_n||^2
\end{aligned}
$$ force that there exists $n\in \mathbb{N}$ such that $([h_n(x;\lambda),q_n(x;\lambda),p_n(x;\lambda)]^T,\psi_n)$ makes a nontrivial solution of \eqref{geig} with $\eta=n\pi/L$.
\end{proof}

The conclusion is that stability in a channel may be studied simply by restricting the Evans functions in Section~\ref{s:evans_hydro} to $\eta=n\pi/L$, $n\in \Z$, a striking simplification
of an at first apparently complicated refinement of the whole space problem.

\smallskip

\noindent{\bf Hydraulic shock stability in whole space vs. channel.} Taking as given our partly numerically-based
conclusions on stability of hydraulic shocks-- recall that medium-frequency and low-frequency diffusive
stability were assessed by numerical Evans function computations--
we readily see that {\it stability of a hydraulic shock in a channel is generically} (i.e., for $F\neq F_{\text{2d}}$) 
{\it equivalent to its stability in the whole space}. For, stability in the whole space by restriction evidently implies stability in the channel.
But, on the other hand, instability appears to occur only through ``essential'' instability having to do
with failure of convective stabilizability at the endstates, which in turn appears to be equivalent to
failure along open frequency cones in the high-frequency limit.
And, an open frequency cone intersects any lattice $(k,\eta)$ with $k\in\mathbb{R}$ and $\eta=n\pi/L$, $n\in \Z$.

This is quite different from the scenario in the case of roll waves, for which channel width effectively
selects the onset of channel instability/bifurcation; see Sections \ref{ss:rollwaves} and \ref{s:channel_flow_roll}.

\subsection{Numerical time evolution studies and convergence to herringbone flow}\label{s:hydro_bif}
We complete our study of multi-d hydraulic shock stability 
with numerical time-evolution simulations in a channel using CLAWPACK \cite{Cl}, both validating
our stability conclusions and, in case of instability, demonstrating the
bifurcation and metastability phenomena described in Sections \ref{s:new}-\ref{ss:shocks}.

These experiments were carried out in channels of varying (fixed) finite width, with wall-type boundary
conditions on the upper and lower $y$-boundaries.  The infinite length channel was simulated by a rectangle
of ``long enough'' extent, with extrapolation (``non-reflecting outflow'') boundary conditions at
the left and right $x$-boundaries.  In practice, length $20$ was typically enough, as determined by trial and
error comparing with result computed on longer domains.

Initial data was either a numerically generated approximate hydraulic shock profile, Riemann-type ``dambreak''
data, or perturbations of the above by a smooth compactly supported ``bump function''.
Here, by dam-break, or Riemann-type initial data, 
we mean data that is constant for $x\gtrless 0$, with a jump at $x=0$,
the states on either side being in equilibrium, with left and right heights $H_L> H_R$ given.
As noted earlier, this corresponds with a unique shock wave $(H_L,H_R)$ of the associated equilibrium system,
and, for Froude number $F<F_{\rm exist}$, to a unique hydraulic shock solution, each with the same speed $s$.
In the case of stability, this hydraulic shock solution is expected to described time-asymptotic behavior
for any of the above types of data; we thus compute in a coordinate frame moving with speed $s$, so as to
keep relevant dynamics approximately centered in the frame.

As expected, shocks in the hydrodynamically stable regime $F<2$, or indeed in the larger convective stability
range $F<F_{\text{2d}}$ exhibited robust stability, with convergence for all types of data to translates
of the hydraulic shock solutions predicted by the endstates $H_L$ and $H_R$.
This suggests strongly that the {\it spectral stability properties}
verified here are associated also with {\it linear and nonlinear stability}.

Still more interesting are the results for $F_{\text{2d}}<F< F_{\text{1d}}$, where one might expect instability/bifurcation to
transverse patterns.  
And, indeed, we do find such behavior in this range in the form of transition to herringbone flow.  
A typical example of herringbone formation is illustrated in FIGURE~\ref{fig:herringbone1}.

However, as described in the introduction, the bifurcation point is not at $F_{\text{2d}}$ as one might at first expect,
but at a value $F=\tilde F$ lying strictly between $F_{\text{2d}}$ and $F_{\text{1d}}$.
Though it is a theorem that linear instability holds in any scalar exponentially weighted norm for $F>F_{\text{2d}}$,
this does not imply instability in the corresponding unweighted norm with respect to 
exponentially-decaying perturbations \cite{FHSS}; nor does it give information on superexponentially weighted norms 
or stability with respect to superexponentially-decaying perturbations.

And, under perturbed dambreak data with compactly supported perturbation, with $F$ sufficiently close to $F_{\text{2d}}$,
we in fact see metastable behavior, with herringbone structure initially appearing near the site of the 
perturbation, but eventually convecting into the subshock discontinuity and disappearing.
By carrying out a bisection search in $F$, starting with endpoints $F_{\text{2d}}$ and $F_{\text{1d}}$ we were able to locate
fairly accurately the bifurcation point at which this metastable behavior gives way to asymptotic herringbone
flow.  As may be expected from this discussion, this is exactly the point at which convection rate of the
herringbone pattern arising from the perturbation exactly matches the speed of the subshock discontinuity.

The result for a particular set of data is illustrated in FIGURE~\ref{fig:herringbone2}, for which we obtain
$F_{\text{2d}}= 2.059\ldots$, $F_{\text{1d}}= 2.290\ldots$, and $\tilde F\approx 2.14$.
In this figure we display also what we believe is the mechanism for this transition, namely, behavior
of a perturbed {\it flat solution} $(H,U)\equiv (H_L,U_L)$ corresponding to the upstream state in the dambreak
problem.
In the time evolution of this perturbed flat state, we see the same herringbone and roll wave patterns forming,
separated by a sort of parabolic interface, as are seen in the full dambreak problem.
And, one can see that the transition point $\tilde F\approx 2.14$ is also the point at which the herringbone
pattern forming from this flat solution has speed agreeing with that of the shock wave in the associated dam
break problem. Moreover, though we do not show it here, behavior for the perturbed dambreak problem, on the
upstream side, up to a small interface layer near the subshock, looks essentially identical to that of the
perturbed flat solution on the same domain, with no disturbance emerging from the subshock on the 
downstream side.

Thus, as described in the introduction, it would appear that understanding of behavior of the constant (``flat'')
solution under compact perturbation is the key-- or at least an important first step-- to understanding of
the metastable behavior of dambreak data under similar compact perturbation.

\begin{figure}[htbp]
\begin{center}   
\includegraphics[scale=0.27]{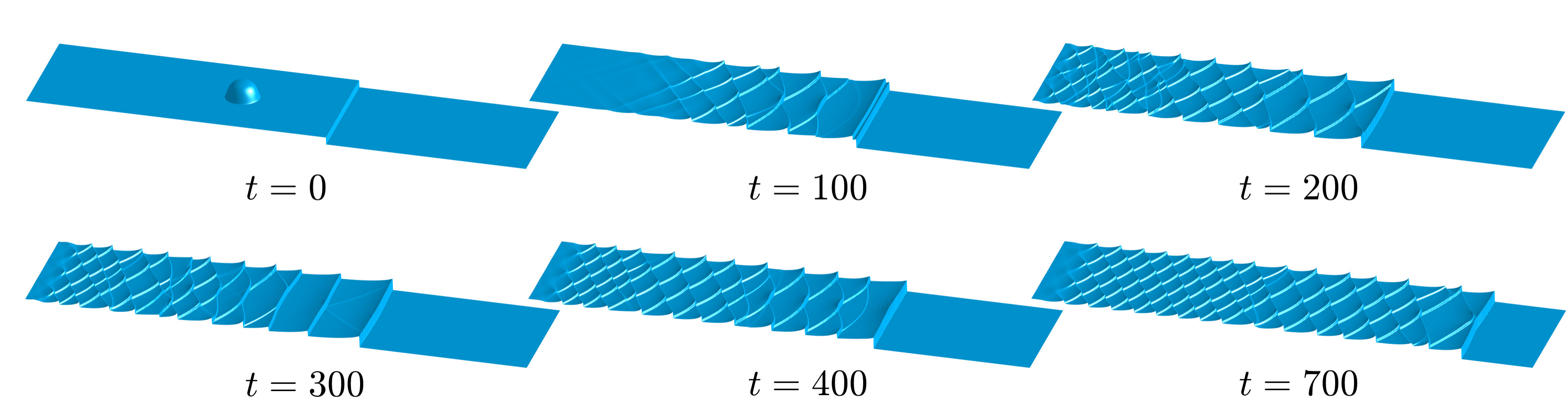}
\end{center}
\caption{See \cite[dambreak\_F\_equals\_2\_point\_25.mp4]{YZmovie} for the full movie. Time-evolution of dam-break initial data with a smooth local perturbation put after the dam-break shock in a channel with width $5$, length $20$, $H_R=0.7$, and $F=2.25$, at time $0$, $100$, $200$, $300$, $400$, and $700$, respectively, showing the development of herringbone from time $0$ to $200$, formation of herringbone to the left, parabolas in the middle, and roll waves between the parabolas and hydraulic shock front at time $300$. Later, for example at time $700$, there are no more parabolas and roll waves, leaving a persistent fully-developed herringbone flow. Also, the hydraulic shock front moves downstream as the herringbone convects toward the front.}
\label{fig:herringbone1}
\end{figure}
\begin{figure}
\begin{center}
\includegraphics[scale=0.5]{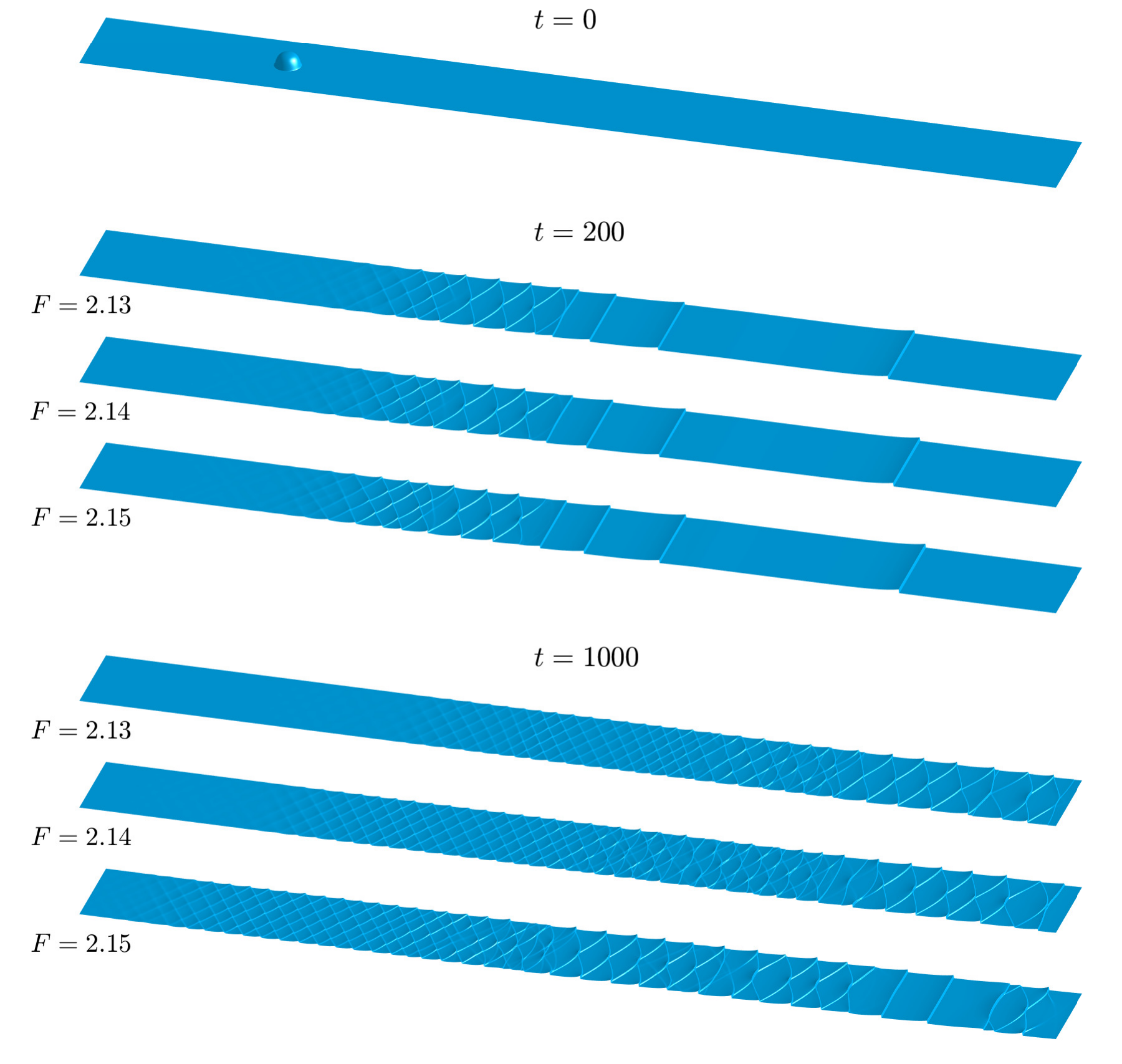}
\end{center}
\caption{See \cite[flat\_comparison.mp4]{YZmovie} for the full movie. Time-evolution of perturbed flat solution in channel, with width $5$, length $50$, and $H_R=0.7$. The first panel displays the common initial condition used for three simulations with $F=2.13$, $F=2.14$, and $F=2.15$. The next three panels show at time $200$ the development of herringbone patterns to left, parabola patterns in the middle, and roll waves to the right. The last three panels suggest that for the corresponding dam break problem, at time $1000$ the herringbone pattern slowly gets dragged into the subshock for $F=2.13$, keeps relatively static with the subshock for $F=2.14$, and deviates away from the subshock for $F=2.15$, indicating a transition from metastability to instability/herringbone flow at $\tilde F \approx 2.14$.}
\label{fig:herringbone2}
\end{figure}


\section{ Multi-d roll wave stability}\label{s:multi_roll}
Next, we turn to the hydrodynamically unstable (``pattern formation'') regime $F>2$
and stability of roll wave solutions depicted in FIGURE \ref{fig_phy}.
In contrast to the universally 1d stable hydraulic shocks for $F<2$, these exhibit interesting stability
transitions already in 1d, so may be expected to display interesting phenomena in multi-d.
Our immediate goal is to add to the longitudinal stability diagram of FIGURE \ref{fig_roll_1d}
transverse stability information, similarly as the 1d Eckhaus diagram \cite{Eck}
is refined by Busse's 2d ``zigzag instability'' criterion \cite{SLB} 
in the classical small-amplitude pattern formation theory.
Indeed, as described in the introduction, our original hope was to see something
like a zigzag instability, initiating multi-dimensional structure connected with herringbone flow.
We do find instabilities, and, different from the shock case, these are low- or mid-frequency rather than
high-frequency in nature. However, we do not for the moment see a connection to zigzag or herringbone flow.

\smallskip

\noindent\textbf{Scale-invariance, roll waves.} In the rest of Section~\ref{s:multi_roll}, we adopt the useful scale-invariance \cite[eq. (2.10)]{JNRYZ}, which reduces all conclusions, up to scale-invariance, to the roll waves with $H_s=1$. Here $H_s$ denotes the fluid height at the sonic point. See FIGURE~\ref{fig_profiles}.

\subsection{Linear stability and periodic Evans-Lopatinsky function}\label{s:linroll}
Changing to comoving coordinates and linearizing about a periodic roll wave solution of period $X$,
containing an infinite array of shock discontinuities at $x=jX$, one in each periodic cell, using ``good unknown",
and taking the Laplace/Fourier transform in $t$/$y$, we obtain by the same derivation as in the hydraulic shock case 
the interior generalized eigenvalue system
\eqref{geig}(i) on each cell $[jX, (j+1)X)$, $j\in \Z$, linked by jump conditions of form \eqref{geig}(ii)
at each boundary point $x=jX$, where now $\hat{\psi}$ is replaced by an infinite array of discontinuities $\hat{\psi}_j$,
and all coefficient matrices are periodic functions of $x$. Performing a further Bloch transform $\check{\cdot}$ and a change of unknowns $e^{i\xi x}\check{\hat{V}}\to w$ and $e^{i\xi x}\check{\hat{\psi}}\to \chi$, we reach a multi-d Floquet eigenvalue system on a single cell $[0, X)$
\be 
\label{Floquet_system}
\lambda A_0 w +(A_1 w)_x+ i\eta A_2w =E w, \qquad
\chi[\lambda F_0(\bar U) +i\eta  F_2(\bar U) - R(\bar U)]=\{A_1 w\}_\xi.
\ee 
where $[f]=f|_0^X$, $\{f\}_\xi=f(X^-)-e^{i\xi X}f(0^+)$, and the Floquet/Bloch number $\xi\in [-\pi/X, \pi/X)$ is real. This generalizes the {\it periodic Evans-Lopatinsky determinant} in the 1d case \cite{JNRYZ} to a multi-d {\it periodic Evans-Lopatinsky determinant}
\be\label{pEL}
E(\lambda, \eta,\xi):= \det ( [\lambda F_0(\bar U)+ i\eta F_2(\bar U) - R(\bar U)], 
\{A_1\mathcal{W}_1\}_\xi,  \{A_1 \mathcal{W}_{2}\}_\xi),
\ee
analogous to \eqref{EL},
vanishing for $\Re \lambda >0$ if and only if $(\lambda, \eta, \xi)$ correspond to an eigenmode $w$,
$\chi$.
Thus, similarly as in the hydraulic shock case {\it spectral stability amounts to nonvanishing of the periodic
Evans-Lopatinsky function $E$ for $\Re \lambda>0$.} 

\noindent{\bf Homogeneous local analytic solvability.} In the 1d Floquet system \cite[eq. (1.10)]{JNRYZ}, the interior equation (i) is degenerate at the sonic point local to which we construct a basis $\{w_1(\,\cdot\,;\lambda)\}$ for the one-dimensional regular solutions that depend analytically on $\lambda$ within an open connected set 
\be \label{Lambda0}
\Lambda_{0}:=\left\{\lambda:\Re\lambda >-\frac{F-2}{2}\right\}\subset \mathbb{C}
\ee 
containing the closed right half plane.  In the 2d case, the Floquet system \eqref{Floquet_system} again carries a single degenerate sonic point, causing again a dimension loss of the space of regular solutions. It turns out that one can construct a basis of two regular solutions depending analytically on $\lambda$ within the same open connected set $\Lambda_0$ \eqref{Lambda0} as the 1d case and also depending analytically on $\eta\in \mathbb{C}$. For convenience of constructing the basis, we carry out a coordinate change to \eqref{Floquet_system}. Setting constant matricies 
$$
T_l=\left[\begin{array}{rrr}\frac{1}{-F^2+F+2} & 0 & 0\\ 0 & 0 & -\frac{1}{F-2}\\ -\frac{F-1}{F^2\,\left(F-2\right)} & \frac{1}{F\,\left(F-2\right)} & 0\end{array}\right],\quad T_r=\left[\begin{array}{rrr}1 & 0 & \frac{F}{F+1}\\ 0 & 0 & 1\\ 0 & 1 & 0\end{array}\right],
$$ we obtain from \eqref{Floquet_system} $H'T_lAT_rW_H=(T_lET_r-\lambda T_lA_0T_r-i\eta T_lA_2T_r-H'T_lA_HT_r)W$ where $W:=T_r^{-1}w$. To make all entries of the coefficient matrices polynomials of $H$, multiplying by a common factor $(H^2 + H + 1)H^2$ yields
$$
\left(\sum_{n=0}^5L^{(n)}(H-1)^n \right)W_H=\left(\sum_{n=0}^3R^{(1,n)}(H-1)^n+\lambda \sum_{n=0}^4R^{(2,n)}(H-1)^n+\eta\sum_{n=0}^5R^{(3,n)}(H-1)^n\right)W,
$$
where
$$
\begin{aligned}
(H^2 + H + 1)H^2H'T_lAT_r=&\sum_{n=0}^5L^{(n)}(H-1)^n,\\
(H^2 + H + 1)H^2(T_lET_r-H'T_lA_HT_r)=&\sum_{n=0}^3R^{(1,n)}(H-1)^n,\\
-(H^2 + H + 1)H^2T_lA_0T_r=&\sum_{n=0}^4R^{(2,n)}(H-1)^n,\\
-i(H^2 + H + 1)H^2T_lA_2T_r=&\sum_{n=0}^5R^{(3,n)}(H-1)^n,\\
\end{aligned}
$$
where $L^{(n)}$, $R^{(1,n)}$, $R^{(2,n)}$, and $R^{(3,n)}$ are constant matrices given in Appendix \eqref{LRmat1}-\eqref{LRmat3}. Expanding a regular solution ansatz as $W=\sum_{n=0}^\infty W^{(n)}(H-1)^n$, presumably it converges, there shall hold the $\mathcal{O}(H-1)^n$-- recurrence relation for $n=0,1,2,3,\ldots$ 
\begin{equation}
\begin{aligned}
\label{reccurence_roll}
&(n+1)L^{(0)}W^{(n+1)}=-\sum_{i=1}^{\min(n,5)}(n+1-i)L^{(i)}W^{(n+1-i)}\\
&+\sum_{i=0}^{\min(n,3)}R^{(1,i)}W^{(n-i)}+\lambda\sum_{i=0}^{\min(n,4)}R^{(2,i)}W^{(n-i)}+\eta \sum_{i=0}^{\min(n,5)}R^{(3,i)}W^{(n-i)}.
\end{aligned}
\end{equation}
Because $L^{(0)}={\rm diag}(1,1,0)$ (see \eqref{LRmat1}) is singular, so is the linear system \eqref{reccurence_roll} for  $W^{(n+1)}$. The last equation of the $\mathcal{O}(H-1)^0$ system makes $W^{(0)}$ lie in the null space of the third row of $R^{(1,0)}+\lambda R^{(2,0)}+\eta R^{(3,0)}$, a two dimensional space spanned by
\be\label{w0}
\left\{\left[\begin{array}{c}-(R^{(1,0)}_{33}+\lambda R^{(2,0)}_{33}+\eta R^{(3,0)}_{33})\\0\\R^{(1,0)}_{31}+\lambda R^{(2,0)}_{31}+\eta R^{(3,0)}_{31}\end{array}\right],\left[\begin{array}{c}-(R^{(1,0)}_{32}+\lambda R^{(2,0)}_{32}+\eta R^{(3,0)}_{32})\\R^{(1,0)}_{31}+\lambda R^{(2,0)}_{31}+\eta R^{(3,0)}_{31}\\0\end{array}\right]\right\}.
\ee 
For $n\geq 0$, assuming that $W^{(m)}$ has been determined for $0\leq m\leq n$, the first two entries of $W^{(n+1)}$ are given by dividing the righthand side of the first two equations of  \eqref{reccurence_roll} by $n+1$ and the third entry of $W^{(n+1)}$ is determined by the last equation of the $\mathcal{O}(H-1)^{n+1}$-- recurrence relation, which reads
\ba
\label{third_row}
&((n+1)L^{(1)}_{33}-R^{(1,0)}_{33}-\lambda R^{(2,0)}_{33}-\eta R^{(3,0)}_{33})W^{(n+1)}_3\\
=&\sum_{j=1}^2(-(n+1)L^{(1)}_{3j}+R^{(1,0)}_{3j}+\lambda R^{(2,0)}_{3j}+\eta R^{(3,0)}_{3j})W^{(n+1)}_j\\
&-\sum_{j=1}^3\sum_{i=2}^{\min(n+1,5)}(n+2-i)L^{(i)}_{3j}W^{(n+2-i)}_j+\sum_{j=1}^3\sum_{i=1}^{\min(n+1,3)}R^{(1,i)}_{3j}W^{(n+1-i)}_j\\&
+\lambda\sum_{j=1}^3\sum_{i=1}^{\min(n+1,4)}R^{(2,i)}_{3j}W^{(n+1-i)}_j+\eta \sum_{j=1}^3\sum_{i=1}^{\min(n+1,5)}R^{(3,i)}_{3j}W^{(n+1-i)}_{j}.
\ea
\eqref{third_row} is solvable provided that
\be 
\label{nonvanish_const}
(n+1)L^{(1)}_{33}-R^{(1,0)}_{33}-\lambda R^{(2,0)}_{33}-\eta R^{(3,0)}_{33}=\frac{3\left(2\lambda+(n+1)(F-2)\right)
}{F\left(F+1\right)\left(F-2\right)}\neq 0.
\ee 
We pause to remark that \eqref{nonvanish_const} holds for all $n\geq 0$ and $\lambda\in \Lambda_0$ \eqref{Lambda0} and the basis \eqref{w0} depends analytically on $\eta\in \mathbb{C}$. $W^{(n+1)}$ is thereby determined. Next, we turn to bounding $||W^{(n+1)}||$, which is needed for the convergence of the solution ansatz $W=\sum_{n=0}^\infty W^{(n)}(H-1)^n$.

For matrices $A\in \mathbb{C}^{m\times p}$ and $B\in \mathbb{C}^{p\times n}$, denote $||A||:=\max_{1\leq i\leq m,1\leq j\leq p}|A_{ij}|$, the maximum norm of $A$. Clearly, we have $||AB||\leq p ||A||\, ||B||$. By \eqref{LRmat1}-\eqref{LRmat3}, $L^{(n)}$, $R^{(1,n)}$, $R^{(2,n)}$ and $R^{(3,n)}$ are constant matrices satisfying 
$$
||L^{(n)}||, ||R^{(1,n)}||, ||R^{(2,n)}||, ||R^{(3,n)}||\leq C,
$$
where $C$ depends solely on $F$.
\begin{lemma}\label{W_vector_analytic}
Let $W^{(0)}(\lambda,\eta)$ be a vector in the two dimensional space spanned by \eqref{w0} that depends analytically on $(\lambda,\eta)\in\Lambda_0\times \mathbb{C}$ and $W^{(n+1)}(\lambda,\eta)$ be the solution of the recurrence relation \eqref{reccurence_roll} for $n\geq 0$ Then, $W^{(n)}(\lambda,\eta)$ $n\geq 0$ are analytic on $\Lambda_0\times \mathbb{C}$. Moreover, for compact subsets $K_1 \subset \Lambda_0$ and $K_2\subset \mathbb{C}$, let $C_1:=\max_{\lambda\in K_1}|\lambda|$, $C_2:=\max_{\eta\in K_2}|\eta|$,
$$
C_3:=\sup_{n\geq 1,\lambda\in K_1}\left|\frac{nF\left(F+1\right)\left(F-2\right)}{3\left(2\lambda+n(F-2)\right)
}\right|=\sup_{n\geq 1}\frac{F(F+1)(F-2)}{6\,\mathrm{dist}(\frac{K_1}{n},\frac{2-F}{2})}\in(0,\infty),
$$
$N:=\max(5,\lceil 24C+30CC_1+36CC_2\rceil)$, and $$M:=\max\left\{\max_{\substack{0\leq n\leq N,j=1,2\\ \lambda\in K_1,\eta\in K_2}}|W^{(n)}_j(\lambda,\eta)|^{\frac{1}{n}},\max_{\substack{0\leq n\leq N,\\\lambda\in K_1,\eta\in K_2}}|W^{(n)}_3(\lambda,\eta)|^{\frac{1}{n+1}},30C,C_3(14C+\frac{1}{2}),1\right\},$$ 
there holds, for all $n\geq 0$,
\be \label{bound_recurrence}
|W^{(n)}_j(\lambda,\eta)|\leq M^n,\quad j=1,2,\quad |W^{(n)}_3(\lambda,\eta)|\leq M^{n+1},
\ee 
for any $\lambda\in K_1$ and $\eta\in K_2$.
\end{lemma}

\begin{proof}
The first analyticity assertion can be easily shown by mathematical induction. Assuming the analyticity of $W^{(m)}(\lambda,\eta)$ for $0\leq m\leq n$, we then have (i) the analyticity of the first two entries of  $W^{(n+1)}(\lambda,\eta)$ can be concluded from that of the righthand side of the first two equations of  \eqref{reccurence_roll} and (ii) the analyticity of the third entry of $W^{(n+1)}(\lambda,\eta)$  can be concluded from that of the righthand side of \eqref{third_row} and the analyticity of $$((n+1)L^{(1)}_{33}-R^{(1,0)}_{33}-\lambda R^{(2,0)}_{33}-\eta R^{(3,0)}_{33})^{-1}=\frac{F\left(F+1\right)\left(F-2\right)}{3\left(2\lambda+(n+1)(F-2)\right)
}$$
on $\Lambda_0$ \eqref{Lambda0}. To prove \eqref{bound_recurrence}, when $n\leq N$, we immediately obtain the bounds by the choices of $M$. When $n>N$, using the first two equations of the recurrence relation \eqref{reccurence_roll}, we obtain, for $j=1,2$,
$$
\begin{aligned}
|W^{(n)}_j(\lambda,\eta)|&\leq \frac{1}{n}\Big(\sum_{i=1}^53(n-i)CM^{n-i}+\sum_{i=0}^33CM^{n-i}+C_1\sum_{i=0}^43CM^{n-i}+C_2\sum_{i=0}^53CM^{n-i}\Big)\\
&< M^n\Big(\frac{15C}{M}+\frac{12C+15CC_1+18CC_2}{n}\Big)<M^n\Big(\frac{15C}{M}+\frac{12C+15CC_1+18CC_2}{N}\Big)\\
&\leq M^n,
\end{aligned}
$$
where, on the righthand side of the first inequality, we have used $L^{(i)}_{13}=L^{(i)}_{23}=0$ for $1\leq i\leq 5$ in bounding $|(n-i)L^{(i)}W^{(n-i)}|$.
By \eqref{nonvanish_const} and the choice of $C_3$, we have 
$$
|(nL^{(1)}_{33}-R^{(1,0)}_{33}-\lambda R^{(2,0)}_{33}-\eta R^{(3,0)}_{33})^{-1}|=\frac{1}{n}\left|\frac{nF\left(F+1\right)\left(F-2\right)}{3\left(2\lambda+(n+1)(F-2)\right)
}\right|\leq \frac{C_3}{n}.
$$
Therefore, using \eqref{third_row}, we obtain the bound
$$
\begin{aligned}
|W^{(n)}_3(\lambda,\eta)|\le& \frac{C_3}{n}\Big(2(nC+C+C_1C+C_2C)M^n+3\sum_{i=2}^5(n+1-i)CM^{n+2-i}\\
&+3\sum_{i=1}^3CM^{n+1-i}+3C_1\sum_{i=1}^4CM^{n+1-i}+3C_2\sum_{i=1}^5CM^{n+1-i}\Big)\\
<& M^{n+1}\Big(\frac{2C_3C}{M}+\frac{2C_3C(1+C_1+C_2)}{nM}+\frac{12C_3C}{M}
+\frac{9C_3C+12C_3CC_1+15C_3CC_2}{nM}\Big)\\
<&M^{n+1}\Big(\frac{14C_3C}{M}+\frac{C_3}{2M}\Big)\leq M^{n+1}.\end{aligned}
$$
\end{proof}
\begin{theorem}\label{analyticity}
The periodic Evans-Lopatinsky determinant $E(\lambda,\eta,\xi)$\eqref{pEL} is analytic on $\Lambda_0\times \mathbb{C}\times\mathbb{C}$.
\end{theorem}
\begin{proof}
It suffices to show that $\mathcal{W}_1(H_\pm;\lambda,\eta)$ and $\mathcal{W}_2(H_\pm;\lambda,\eta)$ are analytic on $\Lambda_0\times \mathbb{C}$. By Lemma \ref{W_vector_analytic}, for $|H-1|$ sufficiently small, the Taylor series $W(H;\lambda,\eta)=\sum_{n=0}^\infty W^{(n)}(\lambda,\eta)(H-1)^n$ converges compactly on $\{\lambda:\Re\lambda>\frac{2-F}{2}\}\times \mathbb{C}$. Hence, by Goursat’s theorem and Morera’s theorem, the series $W(H;\lambda,\eta)$ is analytic on $\Lambda_0 \times \mathbb{C}$ for $|H-1|$ sufficiently small. Moreover, because $W^{(n)}$, $n=0,1,2,3,\ldots$ solve the recurrence relations given by the interior equation local to the sonic point, $W(H;\lambda,\eta)$ makes a regular solution. Away from the sonic point, evolving the interior ode preserves analyticity. By \eqref{w0}, the regular solutions that depend analytically on $\lambda$ are of two dimensional. Finally, by changing to $w$ coordinates via $w=T_rW$,  $\mathcal{W}_1(H;\lambda,\eta_\pm)$ and $\mathcal{W}_2(H;\lambda,\eta_\pm)$ are also analytic on $\Lambda_0\times \mathbb{C}$.
\end{proof}
\begin{lemma}[Transverse symmetry]\label{transverse_sym_roll}
 A solution to the Floquet eigenvalue system \eqref{Floquet_system} is invariant under the coordinate change
$\big(\lambda,\eta,\xi,\chi,[w_1,\;w_2,\;w_3]\big)\rightarrow \big(\lambda,-\eta,\xi,\chi,[w_1,\;w_2,\;-w_3]\big)$.
\end{lemma}
\begin{proof}
Similar to the proof of Lemma~\ref{transverse_sym_hydraulic}. 
\end{proof}
\begin{corollary}[Evenness in $\eta$]\label{evenineta}
The periodic Evans-Lopatinsky determinant $E(\lambda,\eta,\xi)$ \eqref{pEL} is even in $\eta$ that is there holds $E(\lambda,\eta,\xi)=E(\lambda,-\eta,\xi)$.
\end{corollary}
\begin{lemma}\label{lemmaeta0}
There hold for all $\xi\in \mathbb{R}\setminus\{0\}$ \be E(\lambda_r(F)+\frac{i\xi X}{\int_0^XFH(x)dx},0,\xi)=0,\quad \text{where}\quad \label{lambdarF} \lambda_r(F):=\frac{\int_{0}^X(\frac{1}{H}-1-F)(x)dx}{\int_{0}^XFH(x)dx}<0,
\ee 
and 
\be \label{xiequal0}
E(\lambda_r(F),0,0)=0\quad \text{ provided that $2\lambda_r(F)+(n+1)(F -2)\neq 0$ for all $n\geq 0$.}
\ee 
For $\lambda\in \Lambda_0$ satisfying $\Re\lambda>\lambda_r(F)$,  $E(\lambda,0,\xi)$ \eqref{pEL} is, up to multiplying a non-vanishing factor, equal to the {\normalfont 1d} periodic Evans-Lopatinsky determinant $\Delta(\lambda,\xi)$ defined in \cite[eq. (1.12)]{JNRYZ}. 
\end{lemma}
\begin{proof}
For $\eta=0$, the first two equations of \eqref{Floquet_system}(i) decouples from the third and make, together with the first two equations of \eqref{Floquet_system}(ii), the Floquet system of the 1d case. Taking the first two entries of $w$ in \eqref{Floquet_system} to be zeros, we investigate eigenmodes corresponding to the third equation of \eqref{Floquet_system}(i) together with the third equation of \eqref{Floquet_system}(ii).  The third equation of \eqref{Floquet_system}(ii) reads $\{\frac{w_3}{FH}\}_{\xi}=0$ and the third equation of \eqref{Floquet_system}(i) reads $\lambda w_3-\left(\frac{w_3}{FH}\right)_x=-\frac{H + FH - 1}{FH^2}w_3$.
Denoting $\omega=\frac{w_3}{FH}$, these two equations amount to $\{\omega\}_{\xi}=0$ and
$
\omega_x=(\lambda FH+1+F-\frac{1}{H})\omega .
$
Integrating from $0$ to $X$ yields 
$\omega(X)=\omega(0)e^{\int_0^X (\lambda FH+1+F-\frac{1}{H})(x)dx}=\omega(0)e^{i\xi X}$. Hence, for any $\xi\in \mathbb{R}$, $\lambda(\xi)$ with
$$
\Re\lambda(\xi) = \lambda_r(F)\quad\text{and}\quad
\Im \lambda(\xi)=\frac{\xi X}{\int_0^XFH(x)dx}$$ 
is an essential spectrum for which the periodic Evans-Lopatinsky determinant \eqref{pEL} vanishes, provided the function is defined at $(\lambda(\xi),0,\xi)$, which holds given \eqref{nonvanish_const} is met for $n=0,1,2,3\ldots$, proving \eqref{lambdarF} and \eqref{xiequal0}.
We further infer from $H_{hom}<H_-\leq  H$ that
$$
\frac{1}{H}-1-F<\frac{1}{H_{hom}}-1-F=- \frac{1+\sqrt{4F + 1}}{2}<0,
$$
whence $\lambda_r(F)<0$.  For $\lambda\in \Lambda_0$ satisfying $\Re \lambda>\lambda_r(F)$, the third equations of \eqref{Floquet_system}(i)(ii) cannot make an eigenmode, reducing the Floquet eigenvalue system to the 1d case.
\end{proof}

\subsection{Low-frequency analysis}\label{s:rollLF}
We now carry out a low-frequency stability analysis, generalizing that of \cite{JNRYZ} in the 1d case,
by Taylor series expansion of the periodic Evans-Lopatinsky determinant near the origin.

By Theorem~\ref{analyticity}, in the vicinity of the origin $(\lambda,\eta,\xi)=(0,0,0)$, $E(\lambda,\eta,\xi)$ expands as 
\be\label{pEL_expansion}
E(\lambda,\eta,\xi)=\sum_{i+j+k\geq  0}a_{(i,j,k)}\lambda^i\eta^j\xi^k,\quad \text{where $a_{(i,j,k)}\in (\sqrt{-1})^{j+k}\mathbb{R}$.}
\ee
When $\eta=0$ and $|\xi|\ll 1$, we infer from Lemma~\ref{lemmaeta0} that $E(\cdot,0,\xi)$ admits the same zeros including their multiplicity as $\Delta(\cdot,\xi)$. We thus collect from \cite[\S 5]{JNRYZ} that $\Delta(0,0),\partial_\lambda \Delta(0,0),\Delta(0,\xi)=0$, which imply that
\be\label{afrom1d} a_{(1,0,0)}=0\quad\text{and}\quad a_{(0,0,k)}=0\quad \text{for all $k\geq 0$.}
\ee
Furthermore, by Corollary~\ref{evenineta}, the powers of $\eta$ must be even numbers, i.e., in addition to \eqref{afrom1d} there hold \be\label{etaevenpower}
a_{(i,2j+1,k)}=0,\quad \text{for all $i,j,k\geq 0$.}
\ee
In particular, $a_{(0,1,0)},\; a_{(1,1,0)}, \; a_{(0,1,1)},\; a_{(0,3,0)},\; a_{(0,1,2)},\; a_{(1,1,1)},\; a_{(2,1,0)}=0$. Therefore, to the lowest and second lowest order, \eqref{pEL_expansion} becomes 
\ba  \label{expandE}
E(\lambda,\eta,\xi)=&a_{(2,0,0)}\lambda^2+a_{(0,2,0)}\eta^2+a_{(1,0,1)}\lambda\xi+a_{(3,0,0)}\lambda^3+a_{(1,2,0)}\lambda\eta^2\\&+a_{(1,0,2)}\lambda\xi^2+a_{(0,2,1)}\eta^2\xi+a_{(2,0,1)}\lambda^2\xi+\mathcal{O}(|\lambda|+|\eta|+|\xi|)^4,
\ea  
where we check numerically $a_{(2,0,0)}\neq 0$. Here, we note $a_{(2,0,0)}\neq 0$ iff $\partial_\lambda\hat{\Delta}(0,0)\neq 0$. The latter is defined and numerically checked in \cite[\S 5.2]{JNRYZ}. Assuming $a_{(2,0,0)}\neq 0$, the Weierstrass preparation theorem asserts that \eqref{expandE} can be factorized as $E(\lambda,\eta,\xi)=W(\lambda;\eta,\xi)h(\lambda,\eta,\xi)$
for $|\lambda|, |\eta|, |\xi|\ll 1$, where $W(\lambda;\eta,\xi)$ is a Weierstrass polynomial and
\ba
\label{weierstrassE}
W(\lambda;\eta,\xi)=\lambda^2+b(\eta,\xi)\lambda+c(\eta,\xi)
\ea
with $b(\eta,\xi)$ and $c(\eta,\xi)$ analytic at $(\eta,\xi)=(0,0)$ and $h(\lambda,\eta,\xi)$ analytic at $(0,0,0)$ and $h(0,0,0)=a_{(2,0,0)}\neq 0$. Therefore, in the vicinity of $(0,0,0)$, $E(\lambda,\eta,\xi)=0$ if and only if $W(\lambda;\eta,\xi)=0$. 
\begin{lemma}\label{bcexpansion:lemma}Assume $a_{(2,0,0)}\neq 0$ so that \eqref{weierstrassE} holds {\normalfont (numerically checked)}. Local to $(\eta,\xi)=(0,0)$, the powers of $\eta$ in the Taylor series expansions of $b(\eta,\xi)$ and $c(\eta,\xi)$ must be even and $\eta^2|c(\eta,\xi)$. Furthermore, direct computations reveal
\ba\label{bcexpansion}
b(\eta,\xi)=\,&\frac{a_{(1,0,1)}}{a_{(2,0,0)}}\xi+\Big(\frac{a_{(3,0,0)}a_{(1,0,1)}^2}{a_{(2,0,0)}^3}- \frac{a_{(2,0,1)}a_{(1,0,1)}}{a_{(2,0,0)}^2} + \frac{a_{(1,0,2)}}{a_{(2,0,0)}}\Big)\xi^2\\
&+\Big(-\frac{a_{(0,2,0)}a_{(3,0,0)}}{a_{(2,0,0)}^2}+\frac{a_{(1,2,0)}}{a_{(2,0,0)}}\Big)\eta^2+\mathcal{O}(|\eta|+|\xi|)^3\\
=:&b_{01}\xi+b_{02}\xi^2+b_{20}\eta^2+\mathcal{O}(|\eta|+|\xi|)^3,\\
c(\eta,\xi)=\,&\eta^2\Big(\frac{a_{(0,2,0)}}{a_{(2,0,0)}}+\Big(\frac{ a_{(0,2,0)}a_{(1,0,1)}a_{(3,0,0)}}{a_{(2,0,0)}^3} -\frac{ a_{(0,2,0)}a_{(2,0,1)}}{a_{(2,0,0)}^2} +\frac{a_{(0,2,1)}}{a_{(2,0,0)}}\Big)\xi+\mathcal{O}(|\eta|+|\xi|)^2\Big)\\
=:&\eta^2(c_{20}+c_{21}\xi+\mathcal{O}(|\eta|+|\xi|)^2),
\ea 
where, by \eqref{pEL_expansion}, 
\be \label{bcrealimag}
b_{02},b_{20},c_{20}\in \mathbb{R}\quad \text{and} \quad b_{01},c_{21}\in i\mathbb{R}.
\ee 
\end{lemma}
\begin{proof}
Again, by Corollary~\ref{evenineta}, the powers of $\eta$ in \eqref{pEL_expansion} must be even. Let $\eta^2=\tilde{\eta}$ and $\tilde{E}(\lambda,\tilde{\eta},\xi):=E(\lambda,\sqrt{\tilde{\eta}},\xi)=E(\lambda,\eta,\xi)$. $\tilde{E}$ is analytic in $\lambda$, $\tilde{\eta}$, and $\xi$ and, by Weierstrass preparation theorem, can be factorized as $\tilde{E}(\lambda,\tilde{\eta},\xi)=\tilde{W}(\lambda;\tilde{\eta},\xi)\tilde{h}(\lambda;\tilde{\eta},\xi)$ where
$
\tilde{W}(\lambda;\tilde{\eta},\xi)=\lambda^2+\tilde{b}(\tilde{\eta},\xi)\lambda+\tilde{c}(\tilde{\eta},\xi)
$ with $
\tilde{b}(\tilde{\eta},\xi)=b(\eta,\xi)$ and $\tilde{c}(\tilde{\eta},\xi)=c(\eta,\xi)$. Hence, powers of $\eta$ in the expansions of $b(\eta,\xi)$ and $c(\eta,\xi)$ must be even.

Recall \cite[\S 5]{JNRYZ}, $E(0,0,\xi)\equiv 0$, whence $W(0,0,\xi)\equiv 0$ and $c(0,\xi)\equiv 0$, yielding $\eta|c(\eta,\xi)$. Since the powers of $\eta$ in the Taylor series of $c(\eta,\xi)$ are all even, $\eta|c(\eta,\xi)$ implies $\eta^2|c(\eta,\xi)$. The Taylor expansions for $b(\eta,\xi)$ and $c(\eta,\xi)$ \eqref{bcexpansion} are obtained by matching terms in the Weierstrass factorization $E(\lambda,\eta,\xi)=W(\lambda;\eta,\xi)h(\lambda,\eta,\xi)$.
\end{proof}

For $\eta=0$, let us denote, as in \cite{JNRYZ}, $\lambda_1$ to be the spectrum that bifurcates off the origin as $\xi$ is varied and $\lambda_2$ to be the spectrum that stays at the origin for all $\xi$. With Lemma~\ref{bcexpansion:lemma} in hand, we make use of \eqref{weierstrassE} and \eqref{bcexpansion} to recover the expansion of $\lambda_1(0,\xi)$ obtained in \cite{JNRYZ} and, more importantly, to obtain a new expansion of $\lambda_2(\eta,0)$ helpful in examining the stability of $\lambda_2$ as $\eta$ is varied. 
\begin{lemma}[Expansions of $\lambda_{1}(0,\xi)$ and $\lambda_2(\eta,0)$]\label{lambda12} Assuming $a_{(2,0,0)}\neq 0$, $\lambda_1(0,\xi)$ is analytic in $\xi$ and expands in the vicinity of $\xi=0$ as
\be 
\label{lambda1expansions}
\lambda_1(0,\xi)=-b(0,\xi)=-b_{01}\xi-b_{02}\xi^2+\mathcal{O}(|\xi|)^3
\ee
and, with the additional assumption $a_{(0,2,0)}\neq 0$ {\normalfont (numerically checked)},  $\lambda_2(\eta,0)$ is analytic in $\eta$ and expands in the vicinity of $\eta=0$ as
\be 
\label{lambda2expansions}
\lambda_2(\eta,0)=\sqrt{-c_{20}}\,\eta-\frac{1}{2}b_{20}\eta^2+\mathcal{O}(\eta^3).
\ee
For the stability of \eqref{lambda1expansions}, recalling \eqref{bcrealimag}, we recover results, as in \cite[\S 5]{JNRYZ}, that $-b_{01}$, the coefficient of $\xi$, is purely imaginary and $-b_{02}$, the coefficient of $\xi^2$, is real, so that the stability of $\lambda_1(0,\xi)$ is ruled by the diffusive term. As for the stability of \eqref{lambda2expansions}, recalling \eqref{bcrealimag}, $c_{20}=\frac{a_{(0,2,0)}}{a_{(2,0,0)}}\in \mathbb{R}$ and, under the assumptions $a_{(2,0,0)},a_{(0,2,0)}\neq 0$, it must be either positive or negative. We numerically check $c_{20}>0$, yielding $\sqrt{-c_{20}}$, the coefficient of $\eta$, is purely imaginary. By \eqref{bcrealimag}, $-\frac{1}{2}b_{20}$, the coefficient of $\eta^2$ is real, so that the stability of $\lambda_2(\eta,0)$ is also ruled by the diffusive term.
\end{lemma}

\begin{proof}
By the quadratic formula,
$$
\lambda_{1}(\eta,\xi),\lambda_{2}(\eta,\xi)=\frac{-b(\eta,\xi)}{2}\pm\frac{\sqrt{b^2(\eta,\xi)-4c(\eta,\xi)}}{2}.
$$
Recall from the proof of Lemma~\ref{bcexpansion:lemma} that $c(0,\xi)=0$. For the multi-valued square root function, let us take $\sqrt{b^2(0,\xi)}=b(0,\xi)$ and, by our foregoing convention, 
\be
\lambda_{1}(\eta,\xi)=\frac{-b(\eta,\xi)}{2}-\frac{\sqrt{b^2(\eta,\xi)-4c(\eta,\xi)}}{2}\;\; \text{and} \;\; \lambda_{2}(\eta,\xi)=\frac{-b(\eta,\xi)}{2}+\frac{\sqrt{b^2(\eta,\xi)-4c(\eta,\xi)}}{2}.
\ee
Setting $\eta=0$ in $\lambda_1$ yields \eqref{lambda1expansions}. Setting $\xi=0$ in $\lambda_2$, we find that the lowest order term of $b^2(\eta,0)-4c(\eta,0)$ is $-4c_{20}\eta^2$. Given $a_{(0,2,0)}\neq 0$, we then obtain analytic dependence of $\lambda_2(\eta,0)$ on $\eta$ and the expansion \eqref{lambda2expansions}.
\end{proof}

For full low-frequency stability, it amounts to $\Re\lambda_{1}(\eta,\xi),\Re\lambda_{2}(\eta,\xi)< 0$ for all $|\eta|,|\xi|$ sufficiently small and $(\eta,\xi)\neq (0,0)$. Assuming the $\Re\lambda_{1}(0,\xi)<0$ for $0<|\xi|\ll 1$ and  $\Re\lambda_{2}(\eta,0)<0$ for $0<|\eta|\ll 1$, our next theorem states extra conditions needed for full low-frequency stability.

\begin{theorem}[Low-frequency stability criteria]\label{lowfrequencystability}
Under the assumptions made in Lemma~\ref{lambda12}, we further assume $c_{20}>0$ {\normalfont (numerically checked)} so that the stability of $\lambda_2(\eta,0)$ is also ruled by the diffusive term. Provided that
\be\label{ind12}
{\rm ind}_1:=b_{02}>0\quad \text{{\normalfont and}} \quad 
{\rm ind}_2:=b_{20}>0
\ee 
so that $\Re\lambda_{1}(0,\xi)<0$ for $0<|\xi|\ll 1$ {\normalfont and} $\Re\lambda_{2}(\eta,0)<0$ for $0<|\eta|\ll 1$ and that either
\be \label{cond1}
{\rm ind}_3:=c_{21}^2(b_{01}^2b_{20}^2 - 2b_{01}b_{20}c_{21} + 4b_{02}b_{20}c_{20} + c_{21}^2)<0
\ee 
or
\be \label{cond2}
\left\{\begin{aligned}&{\rm ind}_3>0\\
&{\rm ind}_4:= - b_{01}b_{20}c_{21} + 2b_{02}b_{20}c_{20}+c_{21}^2>0\\
&{\rm ind}_5:=-b_{01}c_{21}+b_{02}c_{20} >0,
\end{aligned}\right.
\ee 
there is no spectrum on the imaginary axis in the vicinity of the origin of the $\lambda$ plane for $0<|\eta|+|\xi|\ll 1$. Consequently, $\lambda_{1}(\eta,\xi)$ and $\lambda_{2}(\eta,\xi)$ stay on the left half plane for $0<|\eta|+|\xi|\ll 1$, giving low-frequency stability.
\end{theorem}
\begin{proof}
Substituting $\lambda=i\lambda_i$ with $\lambda_i\in\mathbb{R}$ into \eqref{weierstrassE} yields
\ba \label{realimagW}
&\Re W(i\lambda_i,\eta,\xi)=-\lambda_i^2+(b_{01}\xi+\mathcal{O}(|\eta|+|\xi|)^3)\lambda_i i+(c_{20}+\mathcal{O}(|\eta|+|\xi|)^2)\eta^2\quad \text{and}\\
&\Im W(i\lambda_i,\eta,\xi)=(b_{02}\xi^2+b_{20}\eta^2+\mathcal{O}(|\eta|+|\xi|)^4)\lambda_i+(-ic_{21}\xi+\mathcal{O}(|\eta|+|\xi|)^3)\eta^2.
\ea 
Given \eqref{ind12}, the leading order term $(b_{02}\xi^2+b_{20}\eta^2)$ of the coefficient of $\lambda_i$ on the righthand side of $\Im W(i\lambda_i,\eta,\xi)$ does not vanish for $0<|\eta|+|\xi|\ll 1$, justifying that one can solve for a unique $\lambda_i$ by $\Im W(i\lambda_i,\eta,\xi)=0$. Substituting the $\lambda_i$-solution in $\Re W(i\lambda_i,\eta,\xi)=0$ yields
$$
\eta^2\left(b_{20}^2c_{20}\eta^4+(c_{21}^2 - b_{01}b_{20}c_{21} + 2b_{02}b_{20}c_{20})\eta^2\xi^2+(b_{02}^2c_{20} - b_{01}b_{02}c_{21})\xi^4+\mathcal{O}(|\eta|+|\xi|)^6\right)=0.
$$
As the low frequency behavior with $\eta=0$ is completely understood in \cite{JNRYZ}, we focus on
\be 
\label{tauremoved}
b_{20}^2c_{20}\eta^4+(c_{21}^2 - b_{01}b_{20}c_{21} + 2b_{02}b_{20}c_{20})\eta^2\xi^2+(b_{02}^2c_{20} - b_{01}b_{02}c_{21})\xi^4+\mathcal{O}(|\eta|+|\xi|)^6=0,
\ee 
where $b_{20}^2c_{20}$ is positive by assumptions and the powers of $\eta$ in the remainder $\mathcal{O}(|\eta|+|\xi|)^6$ must be even.
Applying the Weierstrass preparation theorem to the lefthand side of \eqref{tauremoved}, we have that \eqref{tauremoved} amounts to
\be
\label{etaxieq}
b_{20}^2c_{20}\eta^4+((c_{21}^2 - b_{01}b_{20}c_{21} + 2b_{02}b_{20}c_{20})\xi^2+\mathcal{O}(|\xi|^3))\eta^2+((b_{02}^2c_{20} - b_{01}b_{02}c_{21})\xi^4+\mathcal{O}(|\xi|^5))=0.
\ee
The discriminant of \eqref{etaxieq} reads ${\rm ind}_3 \xi^4+\mathcal{O}(|\xi|^5)$. If ${\rm ind}_3\neq 0$, \eqref{etaxieq} achieves two $\eta^2(\xi)$ zeros that depend analytically on $\xi$ and expand as
\be 
\label{etaofxi}
\eta^2(\xi)=\frac{- c_{21}^2 + b_{01}b_{20}c_{21} - 2b_{02}b_{20}c_{20}\pm\sqrt{{\rm ind}_3}}{2b_{20}^2 c_{20}}\xi^2+\mathcal{O}(|\xi|^3).
\ee
If \eqref{cond1} holds, the two $\eta^2(\xi)$ zeros are complex to the leading order term. If instead \eqref{cond2} hold, the two zeros $\eta^2(\xi)$ are both negative to the leading order term. In either case, we reach a contradiction against $\eta,\xi\in \mathbb{R}$. This completes the proof.
\end{proof}
For the sake of instability criteria, the negation of the stability criteria reads
\be \label{negationofstability}
{\rm ind}_1\leq 0\quad \text{or} \quad {\rm ind}_2\leq 0\quad \text{or} \quad {\rm ind}_3= 0 \quad \text{or} \quad 
\left\{\begin{aligned}&{\rm ind}_3\geq 0\\
&{\rm ind}_4\leq 0\\
\end{aligned}\right.\quad \text{or} \quad  
\left\{\begin{aligned}&{\rm ind}_3\geq 0\\
&{\rm ind}_5\leq 0\\
\end{aligned}\right..
\ee
We will not treat the cases for which either ${\rm ind}_1=0$ or ${\rm ind}_2=0$, for the stability of either $\lambda_1(0,\xi)$ \eqref{lambda1expansions} or $\lambda_2(\eta,0)$ \eqref{lambda2expansions} is determined by higher order terms. Also we will not treat the case for which ${\rm ind}_3=0$, since the $\eta^2(\xi)$ roots of \eqref{etaxieq} becomes non analytic in $\xi$. In addition, we will not treat the case for which ${\rm ind}_5=0$ and ${\rm ind}_4>0$, for the positivity of the righthand side of \eqref{etaofxi} is determined by higher order terms. Except for the situations above, we can conclude low-frequency instability provided that \eqref{negationofstability} holds.
\begin{corollary}
[Low-frequency instability criteria] Under the assumptions made in Lemma~\ref{lambda12}, we further assume $c_{20}>0$. Provided that one of the following three conditions is satisfied 
\begin{itemize}
\item[(i)] ${\rm ind}_1<0$ {\normalfont or} ${\rm ind}_2<0$,  
\item[(ii)] ${\rm ind}_1>0$, ${\rm ind}_2>0$, ${\rm ind}_3>0$, {\normalfont and} ${\rm ind}_4\leq 0$,
\item[(iii)] ${\rm ind}_1>0$, ${\rm ind}_2>0$, ${\rm ind}_3>0$, {\normalfont and} ${\rm ind}_5< 0$,
\end{itemize}
\noindent for any $r>0$, there exists spectrum $\lambda(\eta,\xi)$ with $0<|\eta|+|\xi|< r$ and $\Re\lambda(\eta,\xi)>0$, giving low-frequency instability.
\end{corollary}
\begin{proof}
The condition (i) leads to the low-frequency instability of $\lambda_1(0,\xi)$ \eqref{lambda1expansions} or $\lambda_2(\eta,0)$  \eqref{lambda2expansions}. Given either (ii) or (iii) holds, we have that \eqref{ind12} hold, which guarantees that $\Im W(i\lambda_i,\eta,\xi)=0$ \eqref{realimagW}[ii] admits a (unique) solution $\lambda_i=\lambda_i(\eta,\xi)$ for $0<|\eta|+|\xi|\ll 1$ and that the rest assumptions on ${\rm ind}_3$, ${\rm ind}_4$, and ${\rm ind}_5$ ensure that at least one of 
$$
\frac{- c_{21}^2 + b_{01}b_{20}c_{21} - 2b_{02}b_{20}c_{20}\pm\sqrt{{\rm ind}_3}}{2b_{20}^2 c_{20}}
$$
of the coefficient of the leading $\mathcal{O}(|\xi|^2)$ term of \eqref{etaofxi} is positive, which gives analytic solutions $\eta_\pm(\xi)\in\mathbb{R}$ for $|\xi|\ll 1$ that expand as
\be \label{etapm}
\eta_\pm(\xi)=\pm\eta^{(1)}\xi+\mathcal{O}(|\xi|^2),\quad \text{where $\eta^{(1)}> 0$.}
\ee Combining, we justify the existence of spectrum $\lambda(\xi)=i\lambda_i(\eta_+(\xi),\xi)=i\lambda_i(\eta_-(\xi),\xi)\neq 0$, for $0<|\xi|\ll 1$, on the imaginary axis. Indeed, in the course of solving for $\lambda_i$ by $\Im W(i\lambda_i,\eta_+(\xi),\xi)=0$ \eqref{etapm}[ii], it is easy to see that $\xi=0$ is a removable singularity of the analytic function $\lambda_i(\eta_+(\xi),\xi)$ which can be expanded as
\be 
\label{lambda_i}
\lambda_i(\eta_+(\xi),\xi)=\lambda_i^{(1)}\xi+\mathcal{O}(|\xi|^2),\quad \text{where we have $\lambda_i^{(1)}\neq 0$ by $c_{21}\neq 0$ obtained from ${\rm ind}_3\neq 0$.}
\ee 
It remains to justify that the neutral spectrum can cross the imaginary axis transversely when $\eta$ and $\xi$ are varied. To this end, for $\lambda=\lambda_r+i\lambda_i$ with $\lambda_r,\lambda_i\in\mathbb{R}$ and $|\lambda_r|,|\lambda_i|\ll 1$, we compute the Jacobian matrix
$$\begin{aligned}
&\frac{\partial (\Re W(\lambda_r+i\lambda_i,\eta,\xi),\Im W(\lambda_r+i\lambda_i,\eta,\xi))}{\partial (\lambda_r,\lambda_i)}\\
=&\left[\begin{array}{rr}2\lambda_r+b_{20}\eta^2+b_{02}\xi^2+\mathcal{O}(|\eta|+|\xi|)^4&-2\lambda_i+b_{01}\xi i+\mathcal{O}(|\eta|+|\xi|)^3\\2\lambda_i-b_{01}\xi i+\mathcal{O}(|\eta|+|\xi|)^3&2\lambda_r+b_{20}\eta^2+b_{02}\xi^2+\mathcal{O}(|\eta|+|\xi|)^4\end{array}\right],
\end{aligned}
$$
whose determinant at $(0,\lambda_i(\eta_+(\xi),\xi),\eta_+(\xi),\xi)$ reads $(2\lambda_i^{(1)} - b_{01} i)^2\xi^2+\mathcal{O}(|\xi|^3)$. Substituting $\lambda_i$ and $\eta$ in $\Re W(i\lambda_i,\eta,\xi)=0$ \eqref{realimagW}[i] by $\lambda_i(\eta_+(\xi),\xi)$ \eqref{lambda_i} and $\eta_+(\xi)$ \eqref{etapm} respectively yields, at the $\mathcal{O}(|\xi|^2)$ order,
$$
\left(-(\lambda_i^{(1)})^2+b_{01}\lambda_i^{(1)}i+c_{20}(\eta^{(1)})^2\right)\xi^2=0
$$
Since $c_{20}$ by assumption is positive and by \eqref{etapm} $\eta^{(1)}>0$, we obtain $\lambda_{i}^{(1)}\neq b_{01} i/2$. Therefore the Jacobian matrix is invertible for $0<|\xi|\ll 1$. The implicit function theorem then asserts the existence of $\lambda(\eta,\xi)=\lambda_r(\eta,\xi)+\lambda_i(\eta,\xi)i$ satisfying $W(\lambda(\eta,\xi),\eta,\xi)=0$ \eqref{weierstrassE} for $(\eta,\xi)$ in the vicinity of $(\eta_+(\xi),\xi)$ for $0<|\xi|\ll 1$. By the theorem, we make further computation using
$$
\frac{\partial (\lambda_r,\lambda_i)}{\partial (\eta,\xi)}=-\left(\frac{\partial (\Re W,\Im W)}{\partial (\lambda_r,\lambda_i)}\right)^{-1}\frac{\partial (\Re W,\Im W)}{\partial (\eta,\xi)}
$$
to obtain
$$
\left[\begin{array}{rr}\frac{\partial\lambda_r}{\partial\eta}&\frac{\partial\lambda_r}{\partial\xi}\end{array}\right](\eta_+(\xi),\xi)=\left[\begin{array}{rr} \frac{2( b_{20}\lambda_i^{(1)} - c_{21}i)\eta^{(1)}}{b_{01}i - 2\lambda_i^{(1)}}
&\frac{ 2b_{02}\lambda_i^{(1)}- c_{21}(\eta^{(1)})^2i}{b_{01}i - 2\lambda_i^{(1)}}\end{array}\right]\xi+\mathcal{O}(|\xi|^2).
$$
Along the direction $[\eta^{(1)},1]$, we compute the directional derivative
$$
\begin{aligned}
\nabla_{[\eta^{(1)},1]}\lambda_r(\eta_+(\xi),\xi)=&[\eta^{(1)},1]\cdot\left[\begin{array}{rr} \frac{2( b_{20}\lambda_i^{(1)} - c_{21}i)\eta^{(1)}}{b_{01}i - 2\lambda_i^{(1)}}
&\frac{ 2b_{02}\lambda_i^{(1)}- c_{21}(\eta^{(1)})^2i}{b_{01}i - 2\lambda_i^{(1)}}\end{array}\right]\xi+\mathcal{O}(|\xi|^2)\\
=&\frac{2(b_{20}(\eta^{1})^2+b_{02})\lambda_i^{(1)}-3 c_{21}(\eta^{(1)})^2i}{b_{01}i - 2\lambda_i^{(1)}}\xi+\mathcal{O}(|\xi|^2)\\
=&\frac{- c_{21}(\eta^{(1)})^2i}{b_{01}i - 2\lambda_i^{(1)}}\xi+\mathcal{O}(|\xi|^2),
\end{aligned}
$$
where in the last equation we have used $(b_{02}+b_{20}(\eta^{1})^2)\lambda_i^{(1)}=ic_{21}(\eta^{(1)})^2$ obtained from $$\Im W(i\lambda_i(\eta_+(\xi),\xi),\eta_+(\xi),\xi)=0, \quad \text{at the $\mathcal{O}(|\xi|^3)$ order.}$$
Lastly, we infer from \eqref{etapm} and \eqref{lambda_i} that the leading $\mathcal{O}(|\xi|)$ term of the directional derivative $\nabla_{[\eta^{(1)},1]}\lambda_r(\eta_+(\xi),\xi)$ is nonzero for $0<|\xi|\ll 1$, proving that the neutral spectra can cross the imaginary axis transversely so that $\lambda_r(\eta,\xi)$ becomes positive.
\end{proof}

\begin{corollary}\label{boundaryIind}
If $b_{01}=0$, then there satisfy
$${\rm ind}_1,{\rm ind}_5=0,\quad {\rm ind}_3\geq 0,\quad {\rm ind}_4\leq 0.$$ We conclude that, on the low-frequency stability boundary I where $b_{01}=0$ \cite{JNRYZ}, both ${\rm ind}_1$ and ${\rm ind}_5$ are zeros and \eqref{cond2} cannot hold.
\end{corollary}
\begin{proof}
If $b_{01}=0$, we readily see from \cite[the two eqs. below (5.4)]{JNRYZ} that ${\rm ind}_1=b_{02}=\beta=0$. Therefore, we find by \eqref{cond2} that
	${\rm ind}_5=-b_{01}c_{21}+b_{02}c_{20}=0$, ${\rm ind}_3=c_{21}^4\geq 0$, and ${\rm ind}_4=c_{21}^2\leq 0$, since by \eqref{bcexpansion} $c_{21}\in i\mathbb{R}$.
\end{proof}

The coefficients $a_{(i,j,k)}$ in the expansion \eqref{pEL_expansion} are numerically computable by the Cauchy's differential formulas:
$$
a_{(i,j,k)}=\frac{1}{(2\pi \sqrt{-1})^2}\oint_{\partial B(0,r)}\oint_{\partial B(0,r)}\frac{\partial_\xi^k\Delta(\lambda,\eta,0)}{\lambda^{i+1}\eta^{j+1}k!}d\lambda d\eta,\quad \text{for $r$ sufficiently small,}
$$
so are the index function ${\rm ind}_1,\ldots,{\rm ind}_5$. See FIGURE~\ref{fig:2dlowstability} for numerical results. From the top left panel and bottom middle panel of FIGURE~\ref{fig:2dlowstability}, we see the curve ${\rm ind}_4=0$ does not intersect the 1d low-frequency stability boundary I, which holds provided that $c_{21}$ does not vanish on the 1d low-frequency stability boundary I. Hence, evidently there is no wave satisfying \eqref{ind12} and \eqref{cond2} simultaneously and 2d low-frequency stable waves in the whole space shall satisfy both \eqref{ind12} and \eqref{cond1}, giving that the 2d low-frequency stable waves sit in between the 2d low frequency stability boundary (see FIGURE~\ref{fig:2dlowstability} bottom left panel) and the 1d low-frequency stability boundary II (see \cite[Fig. 7. b.]{JNRYZ}). The former stability boundary is then used in FIGURE~\ref{fig:stabdiag_fig}.

\begin{figure}[htbp]
\begin{center}
\includegraphics[scale=0.33]{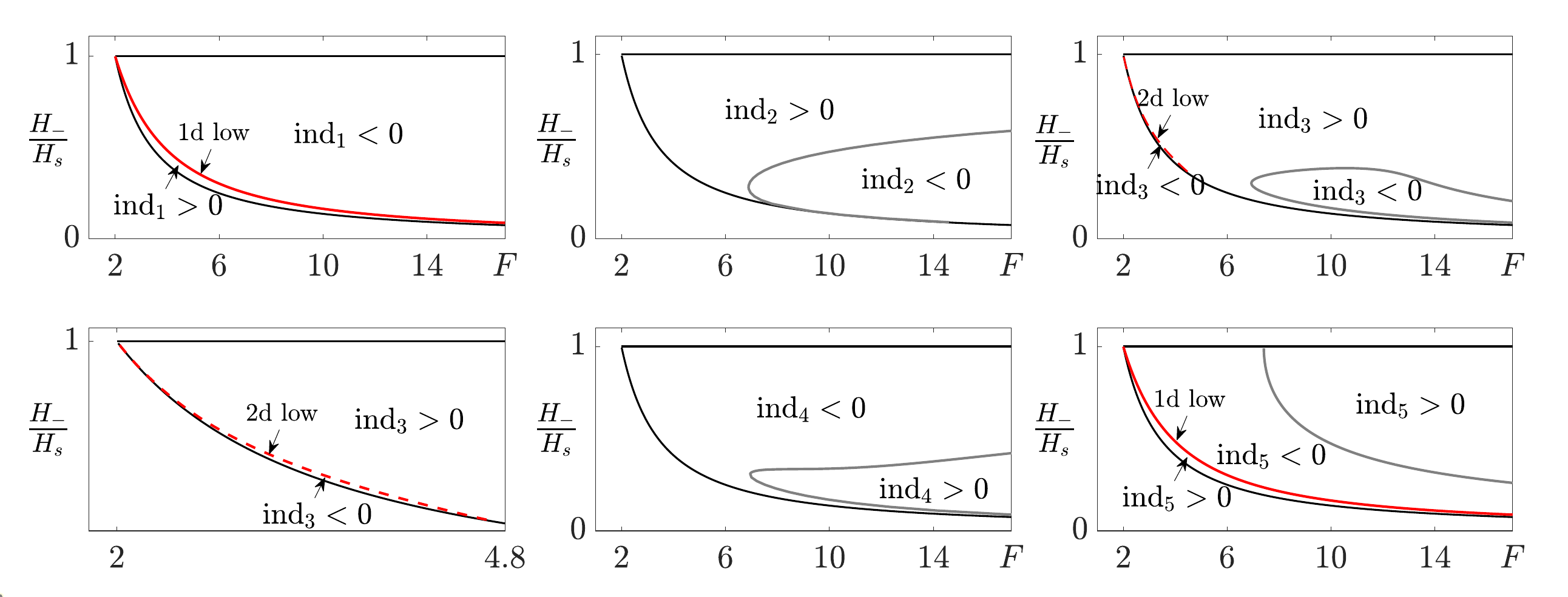}
\end{center}
\caption{Numerical investigations of 2d low-frequency stability indexes ${\rm ind}_{1},\ldots, {\rm ind}_5$ \eqref{ind12} \eqref{cond2}. Top left panel: The ${\rm ind}_1$ by \eqref{lambda1expansions} corresponds to 1d low-frequency stability. That is ${\rm ind}_1>0$ iff $\beta>0$ \cite[eq. (5.4)]{JNRYZ}. There were found two 1d low-frequency stability boundaries: the boundary I (red solid curve with label ``1d low'') and the boundary II (see \cite[Fig. 7. b.]{JNRYZ}). Recall the latter boundary II seemingly asymptotes to the lower existence curve as F increases, making it indistinguishable from the lower existence curve. The ${\rm ind}_1$ is positive in the region between the two curves and negative elsewhere. Top middle panel: There is found a curve (gray solid curve) on which ${\rm ind}_2$ is numerically found vanishing and to the left/right of which ${\rm ind}_2$ is positive/negative. Top right panel: There are found two curves (a red dash curve and a gray solid curve) on which ${\rm ind}_3$ is numerically found vanishing. The ${\rm ind}_3$ is positive in the region between the two curves and negative elsewhere. The red dash curve sits very close to the lower existence curve. See blowup of the red dash curve with label ``2d low'' in the bottom left panel. Bottom middle panel: There is found a curve (gray solid curve) on which ${\rm ind}_4$ is numerically found to be vanishing and to the right/left of which ${\rm ind}_4$ is positive/negative. Bottom right panel: By Corollary~\ref{boundaryIind}, both ${\rm ind}_1$ and ${\rm ind}_5$ vanish on the 1d low-frequency stability boundary I (red solid curve with label ``1d low''). There is found another curve (gray solid curve) on which ${\rm ind}_5$ is numerically found vanishing. The ${\rm ind}_5$ is negative in the region between the two curves and positive elsewhere.}
\label{fig:2dlowstability}
\end{figure}

\subsection{High-frequency stability}\label{s:rollHF} 
Similarly as in \cite[\S 6]{JNRYZ},
one may perform a WKB-type expansion to obtain an asymptotic description of $E$ for high frequencies.
In the roll-wave case, different from the hydraulic shock case, the resulting condition is not automatically 
satisfied in 1d, but, due to the additional Floquet/Bloch frequency $\xi$,
reduces \cite{JNRYZ} to a ``ratio test'' like that of Erpenbeck \cite{Er4} for 2d detonations.  
This ratio test is not associated with an Airy-type turning point, but rather with the ``sonic'' point,
a regular singular point of the eigenvalue ODE.
Fortunately, it is given in terms of explicitly evaluable functions,  hence readily checked;
it was shown numerically in \cite{JNRYZ} that high-frequency longitudinal (1d) instabilities do not occur.

In 2d, on the other hand, one finds along with sonic type regular singular point also
Airy-type turning points as in \cite{Er4,LWZ,Z9}, corresponding to irregular singular points of
the generalized eigenvalue ODE \eqref{geig}.
The 2d high-frequency analysis involves both types of singularity, making behavior potentially quite rich.
A high-frequency stability condition should be computable combining the techniques of \cite{JNRYZ,LWZ,Z9} and
\cite{Z8};
{\it it is a very interesting question whether and when it is violated,}
i.e., high-frequency transverse (2d) instabilities occur.

We will not attempt to carry out a full analysis here, but only to provide a road map indicating the
types of singular points arising and a general idea how such an analysis might be carried out-
the main phenomena of interest for roll waves have to do, rather, with low- and med-frequency instability.

\smallskip

\noindent\textbf{High-frequency analysis.} For $\lambda=\lambda_r+i\lambda_i$, $\lambda_i$, $\eta\in \mathbb{R}$, and $\lambda_r\geq 0$, we analyze solutions of the interior eigenvalue problem \eqref{Floquet_system}[i] in the high-frequency region \eqref{domainhf}. Again, we may further split the high-frequency region \eqref{domainhf} into case (a) $|\lambda|\sim \mathcal{O}(r)$, $|\eta|\sim \mathcal{O}(1)$; case (b) $|\lambda|\sim \mathcal{O}(1)$, $|\eta|\sim \mathcal{O}(r)$; case (c) $|\lambda|\sim \mathcal{O}(r)$, $|\eta|\sim \mathcal{O}(r)$. 

Performing the change of coordinates $w\rightarrow w_1:=A_1w$, we find that $w_1$ satisfies
\be
\label{eigenw1}
\det(A_1)d_x w_1=\left(-\lambda A_0\,{\rm adj}(A_1) -i\eta  A_2\,{\rm adj}(A_1) \right)w_1+E\,{\rm adj}(A_1) w_1.
\ee
The eigenvalues of the leading matrix $\big(-\lambda A_0\,{\rm adj}(A_1) -i\eta A_2\,{\rm adj}(A_1)\big)(x)$ are
\be\label{Leading_eigenvalues}
\gamma_1(x;\lambda)=\frac{\lambda(H^3 - 1)}{F^2H^2}(x),\quad 
\gamma_2,\gamma_3(x;\lambda;\eta)=\frac{- F\lambda\pm \sqrt{H(F^2H^2\lambda^2 + (H^3-1)\eta^2)} }{F^3H^2}(x),
\ee 
from which we immediately see, at the sonic point $x_s$ where $H=1$, $$\gamma_1(x_s;\lambda)=\gamma_2(x_s;\lambda,\eta)=0,\quad \gamma_3(x_s;\lambda,\eta)=-2\lambda/F^2,$$
moreover, the leading matrix can be diagonalized as ${\rm diag}(0,0,-2\lambda/F^2)$ by similar transform $$T^{-1}(x_s)\big(-\lambda A_0\,{\rm adj}(A_1) -i\eta A_2\,{\rm adj}(A_1)\big)(x_s)T(x_s),$$ where, for example

$$
T(x_s):=\left[\begin{array}{rrr}F&0&F^2\lambda\\F-1&0&F(F+1)\lambda\\0&1&i\eta\end{array}\right] 
$$
is non-singular provided that $\det(T(x_s))=-2F^2\lambda\neq 0$. In case (a) and case (c) where $|\lambda|\sim\mathcal{O}(r)$, we obtain $\det(T(x_s))\neq 0$, whence the sonic point is not a turning point. On the other hand, for case (b), we have $\gamma_3=-2\lambda/F^2$ vanishes as well at leading $\mathcal{O}(r)$ order and $\det(T(x_s))=0$, making the sonic point $x_s$ a turning point. Note in 1d that $\gamma_2$ and $\eta$ do not appear, so that the latter scenario does not arise.

In search of turning points other than the sonic point, easy computations show that, for case (a) and case (b), the leading $\mathcal{O}(r)$ terms of \eqref{Leading_eigenvalues} are distinct, whence $(\lambda,\eta)$ cannot be turning point frequencies. This leaves the search to case (c).

If the leading $\mathcal{O}(r)$ term of $\gamma_2$ is equal to that of $\gamma_3$, there must hold $F^2H^2\lambda^2+(H^3-1)\eta^2=0$. Let 
\be 
\tau_3(H):=\frac{\sqrt{H^3-1}}{HF},\quad \text{for}\quad H\in(1,H_+)\quad\text{and}\quad\tau_4(H):=\frac{\sqrt{1-H^3}}{HF},\quad \text{for}\quad H\in(H_-,1).
\ee 
On the other hand, if the leading $\mathcal{O}(r)$ term of $\gamma_1$ is equal to that of $\gamma_2$, there must hold $- F^2H^2\lambda^2 + \eta^2=0$. Let 
\be 
\tau_5(H):=\frac{1}{HF},\quad \text{for}\quad H\in (H_-,1)\cup(1,H_+).
\ee 

\begin{lemma}[Turning points and turning point frequencies; roll waves]\label{turningpoint:lemma;roll}
${}$

\noindent Consider a roll wave in two dimensions. 

At the frequencies $(\lambda,\eta)=(i\tau_3 r+\mathcal{O}(1),r)$, where \be\label{tau3range}
\tau_3\in\big(0,\tau_3(H_+)\big),
\ee
there exists a unique turning point $x_3(\tau_3)\in (x_s,X)$, where $\gamma_2(x_3(\tau_3);i\tau_3r,r)=\gamma_3(x_3(\tau_3);i\tau_3r,r)\neq \gamma_1(x_3(\tau_3);i\tau_3r)$ and the geometric multiplicity of the double eigenvalue is $1$.

At the frequencies $(\lambda,\eta)=(\tau_4r+\mathcal{O}(1),r)$, where \be\label{tau4range}
\tau_4\in\big(0,\tau_4(H_-)\big),
\ee
there exists a unique turning point $x_4(\tau_4)\in (0,x_s)$, where $\gamma_2(x_4(\tau_4);\tau_4r,r)=\gamma_3(x_4(\tau_4);\tau_4r,r)\neq \gamma_1(x_4(\tau_4);\tau_4r)$ and the geometric multiplicity of the double eigenvalue is $1$.

At the frequencies $(\lambda,\eta)=(\tau_5r+\mathcal{O}(1),r)$, where \be\label{tau5range}
\tau_5\in\big(\tau_5(H_+),\frac{1}{F}),(\frac{1}{F},\tau_5(H_-)\big),
\ee
there exists a unique turning point $x_5(\tau_5)\in(x_s,X),(0,x_s)$, where $$\gamma_1(x_5(\tau_5);\tau_5r)=\gamma_2(x_5(\tau_5);\tau_5r,r)\neq \gamma_3(x_5(\tau_5);\tau_5r,r)$$ and the geometric multiplicity of the double eigenvalue is $1$.
\end{lemma}
\begin{proof}
The lemma follows by direct computations. 
\end{proof}

\begin{remark}
Within {\normalfont case (c)}, our foregoing analyses confirm $(\lambda,\eta)$ can be turning point frequencies if either $|\lambda_i|\sim\mathcal{O}(1)$ or $|\lambda_r|\sim\mathcal{O}(1)$ and, in particular, they cannot be turning point frequencies if $|\lambda_i|,|\lambda_r|\sim\mathcal{O}(r)$.
\end{remark}
\begin{remark}
If $\frac{1}{F}<\tau_4(H_-)$, or equivalently, $1>H_-^3+H_-^2$, for $\tau_4,\tau_5=\tau\in(\frac{1}{F},\tau_4(H_-)$, there are two distinct turning points  $0<x_4(\tau)<x_5(\tau)<x_s$, at the former turning point  $\gamma_2$ collides with $\gamma_3$ and at the latter turning point $\gamma_1$ collides with $\gamma_2$.
\end{remark}

\subsubsection{Conclusions and future prospects}\label{s:prospects}
In our analysis of the high-frequency limit, above, we have shown that there appears in all cases a 
sonic point $x_s$ where $H=1$, just as in the 1d case. 
Likewise, we have shown that there appear Airy-type turning points different from the sonic point, of types similar to those arising in the hydraulic shock case in case (c): at most two on $(0,x_s)$ and at most one on $(x_s,X)$. Moreover, the sonic and Airy turning points collide only for case (b). 

Restricting for the moment to the case that sonic and turning points do not collide, i.e., case (c), we may decompose the domain $[0,X)$ into small regions containing 
singular (sonic or turning) points, and complementary regions on which principal eigenvalues remain 
distinct. On the small, singular regions, we may perform either a sonic-type expansion as in the 1d roll wave 
case, or a conjugation to Airy equation as in the 2d hydraulic shock (or detonation \cite{Er4,LWZ}) case, 
to obtain connection information from left to right endpoints. Meanwhile, on complementary regions, the 
generalized eigenvalue ODE may be completely diagonalized, so that, combining these pieces of information, 
we obtain a full description of the spectral problem, yielding stability indices similar to those of the 
1d roll and 2d hydraulic shock case.

The case of colliding sonic and turning points is new to the 2d roll wave case, and requires a separate
handling to treat the associated principal frequencies $\lambda\sim\mathcal{O}(1)$, $|\eta|=r\gg 1$, and nearby
ones as well.  This interesting open problem we leave for future investigation, 
contenting ourselves that low- and med-frequency instabilities appear to be the physically important ones for roll waves.

\subsection{Medium-frequency stability: numerics}\label{s:rollMF}
Medium frequencies $1/R\leq  |(\lambda, \eta, \xi)|\leq  R$ will be studied by numerical computation.
A complicating factor, already present in 1d,
is the appearance independent of frequencies of a singular point in the eigenvalue ODE's.
For, the Lax characteristic condition at the component shock implies that, in comoving coordinates 
moving with the shock, the numbers of positive/negative characteristics of the equation must change 
as the profile traverses from $x=0^+$ to $x=X^-$, since they are different on either side $x=0\pm$ of the shock, 
and by periodicity, equal at $x=0^-$ and $X^-$.  Indeed, there is precisely one ``sonic point''
$x=x_s$ where a characteristic changes sign, i.e., an eigenvalue of $A_1=dF_1/dU$ vanishes in \eqref{Floquet_system},
at which the eigenvalue ODE thus has a {regular singular point}.
This was dealt with in \cite{JNRYZ} by a ``hybrid'' ODE solver using symbolic computation/Frobenius method
to obtain a high-order analytic expansion in a neighborhood of $x_s$, which was then continued out to the boundaries
$x=0,X$ by standard ODE solvers. Recalling results developed in Section~\ref{s:linroll}, we develop a similar hybrid method for computing the periodic Evans-Lopatinsky determinant \eqref{pEL}, which performs quite well. 

\subsection{Medium-frequency stability: analysis}\label{s:medanalysis}
Recalling from \cite[\S 7]{JNRYZ}, an inviscid roll wave with $(F,H_-)$ on the 1d medium frequency stability boundary (see \cite[Fig. 2. and Fig. 3.]{JNRYZ}) achieves a neutral spectrum at $\lambda_*\in i\mathbb{R}$ at a critical Floquet exponent $\xi_*$ and the nearby essential spectrum $\lambda(\xi)$ satisfies $\Re \lambda'(\xi_*)=0>\Re \lambda''(\xi_*)$. The corresponding planar inviscid roll waves, considered in the whole space, must achieve a $\eta$-family of arcs of spectrum $\lambda(\eta,\xi)$ satisfying

\be\label{mid_neutral}
\lambda(0,\xi_*)=\lambda_*\in i\mathbb{R}\quad \text{and}\quad \Re \frac{\partial\lambda}{\partial \xi}(0,\xi_*)=0>\Re \frac{\partial^2\lambda}{\partial \xi^2}\lambda(0,\xi_*).
\ee 

To investigate if the arcs cross the imaginary axis transversely as $\eta$ becomes non-zero, we expand $E(\lambda,\eta,\xi)$ in the vicinity of its zero $(\lambda_*,0,\xi_*)$ as
\be \label{expandEmed}
E(\lambda,\eta,\xi)=a_{(1,0,0)}(\lambda-\lambda_*)+a_{(0,2,0)}\eta^2+a_{(0,0,1)}(\xi-\xi_*)+\mathcal{O}(|\lambda-\lambda_*|+|\eta|^2+|\xi-\xi_*|)^2,
\ee 
where we have used Corollary~\ref{evenineta} to assert that powers of $\eta$ terms must be even. Assuming $a_{(1,0,0)}\neq 0$, the Weierstrass preparation theorem asserts that \eqref{expandEmed} can be factorized as $E(\lambda,\eta,\xi)=W(\lambda-\lambda_*;\eta,\xi-\xi_*)h(\lambda-\lambda_*,\eta,\xi-\xi_*)$
for $|\lambda-\lambda_*|, |\eta|, |\xi-\xi_*|\ll 1$, where $W(\lambda-\lambda_*;\eta,\xi-\xi_*)$ is a Weierstrass polynomial 
\ba
\label{weierstrassEmed}
W(\lambda-\lambda_*;\eta,\xi-\xi_*)=(\lambda-\lambda_*)+\frac{a_{(0,2,0)}}{a_{(1,0,0)}}\eta^2+\frac{a_{(0,0,1)}}{a_{(1,0,0)}}(\xi-\xi_*)+\mathcal{O}(|\eta|^2+|\xi-\xi_*|)^2.
\ea
and $h(\lambda-\lambda_*,\eta,\xi-\xi_*)$ is analytic at $(0,0,0)$ with $h(0,0,0)=a_{(1,0,0)}\neq 0$. Therefore, in the vicinity of $(\lambda_*,0,\xi_*)$, $E(\lambda,\eta,\xi)=0$ if and only if $W(\lambda-\lambda_*;\eta,\xi-\xi_*)=0$.

\begin{proposition}
Considering nearby spectra of a neutral spectrum $\lambda(0,\xi_*)$ satisfying \eqref{mid_neutral}, if $\Re\frac{a_{(0,2,0)}}{a_{(1,0,0)}}<0$ where $a_{(0,2,0)}$ and $a_{(1,0,0)}$ are coefficients in the expansion \eqref{expandEmed}, then the neutral spectrum becomes unstable as $\eta$ becomes non-zero. if $\Re\frac{a_{(0,2,0)}}{a_{(1,0,0)}}>0$, then $\Re\lambda(\eta,\xi)$ achieves a local maximum $0$ at $(0,\xi_*)$. 
\end{proposition}
Note that the coefficients $a_{(i,j,k)}$ in the expansion \eqref{expandEmed} are numerically computable by the Cauchy's differential formulas:
$$
a_{(i,j,k)}=\frac{1}{(2\pi \sqrt{-1})^2}\oint_{\partial B(0,r)}\oint_{\partial B(\lambda_*,r)}\frac{\partial_\xi^k\Delta(\lambda,\eta,\xi_*)}{(\lambda-\lambda_*)^{i+1}\eta^{j+1}k!}d\lambda d\eta,\quad \text{for $r$ sufficiently small.}
$$
Numerically, $\Re\frac{a_{(0,2,0)}}{a_{(1,0,0)}}(F)$ is negative for $F\in(4.28\ldots,4.57\ldots)$ and positive elsewhere on $(2,5)$. See FIGURE~\ref{fig:2dmedlocal} for plot of  $\Re\frac{a_{(0,2,0)}}{a_{(1,0,0)}}$ versus $F$. Therefore, planar roll waves with parameters $(F,H_-)$ on the 1d medium frequency stability boundary and with $F$ between $4.28\ldots$ and $4.57\ldots$ are 2d unstable (in the whole space). On the other hand, one could numerically solve for the spectrum $\lambda(\eta,\xi X)$ continued from the neutral one $\lambda(\eta_*,\xi_* X)$. In FIGURE~\ref{fig:spectrumf4_point4} left panel, we plot the surface $(\eta,\xi X,\Re\lambda(\eta,\xi X))$ for the parameter $(F,H_-)=(4.4,0.36450716\ldots)$ on the 1d medium frequency stability boundary. The spectrum is unstable inside two circular curves that are symmetric about the $\xi X$ axis. See FIGURE~\ref{fig:spectrumf4_point4} right panel.  Furthermore, we make use of the Matlab built-in function fminsearch 
to search for a local maximum point of the function $(\eta,\xi X)\rightarrow \Re\lambda(\eta,\xi X)$ and find 
that the function achieves local maximum $0.002695\ldots$ at $(\eta,\xi X)=(3.334\ldots,-2.100\ldots)$. 
\begin{figure}[htbp]
\begin{center}
\includegraphics[scale=0.4]{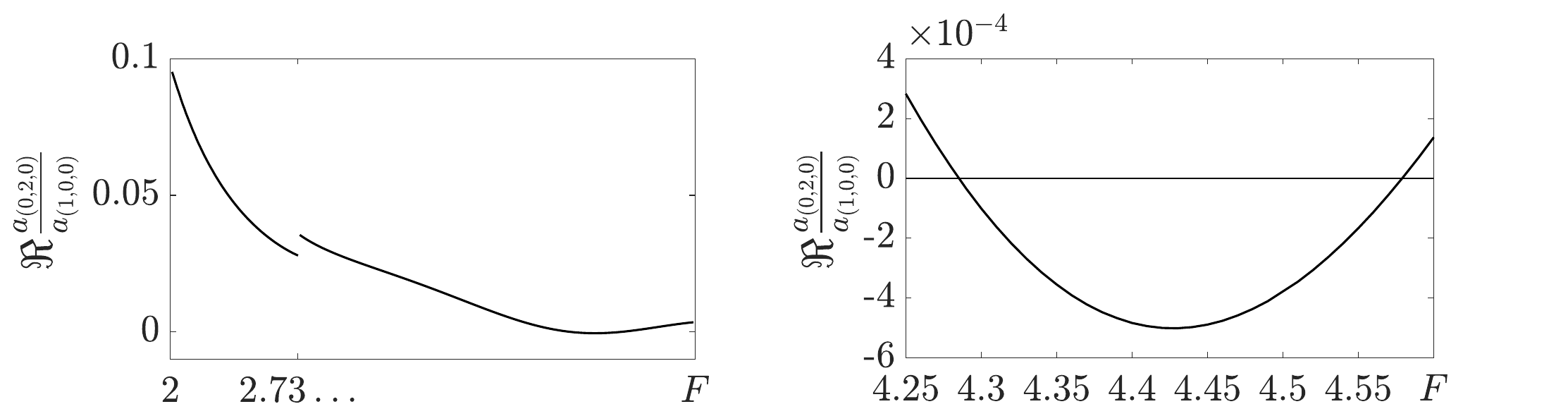}
\end{center}
\label{fig:2dmedlocal} 
\end{figure}
The numerics suggest there should be a 2d medium frequency stability boundary sitting for fixed $F$ above the 1d medium frequency stability boundary. On this 2d mid boundary, the global maximum of the function $\max_{j=1,2,3}\Re\lambda_j(\eta,\xi X)$ must be zero. The global maximum is not necessarily trackable from neutral 1d spectra by a steepest descent method,
and appears difficult to find numerically in a systematic way. Nonetheless, we search for a nearby neutral {\it local} maximum by the method.
Namely, we fix $F=4.4$ and perturb $H_-$ from the 1d mid-boundary to a critical value where the local maximum of $\Re\lambda(\eta,\xi X)$ is zero. 
We then vary $F$ and track a 2d local medium frequency stability boundary where the local maximum of $\Re\lambda(\eta,\xi X)$ is zero. It turns out that the boundary intersects the 1d medium frequency boundary at $F=4.09\ldots$ and $F=4.61\ldots$. Also the boundary sits very close to the 1d medium frequency stability boundary. For $F\in[4.09\ldots,4.61\ldots]$, the difference in $H_-$ is less than $10^{-6}$. We therefore compare the two boundaries in $F$ versus $X/H_s$ coordinates. See bottom left and right panels of FIGURE~\ref{fig:stabdiag_fig}.

\begin{figure}[htbp]
\begin{center}\includegraphics[scale=0.35]{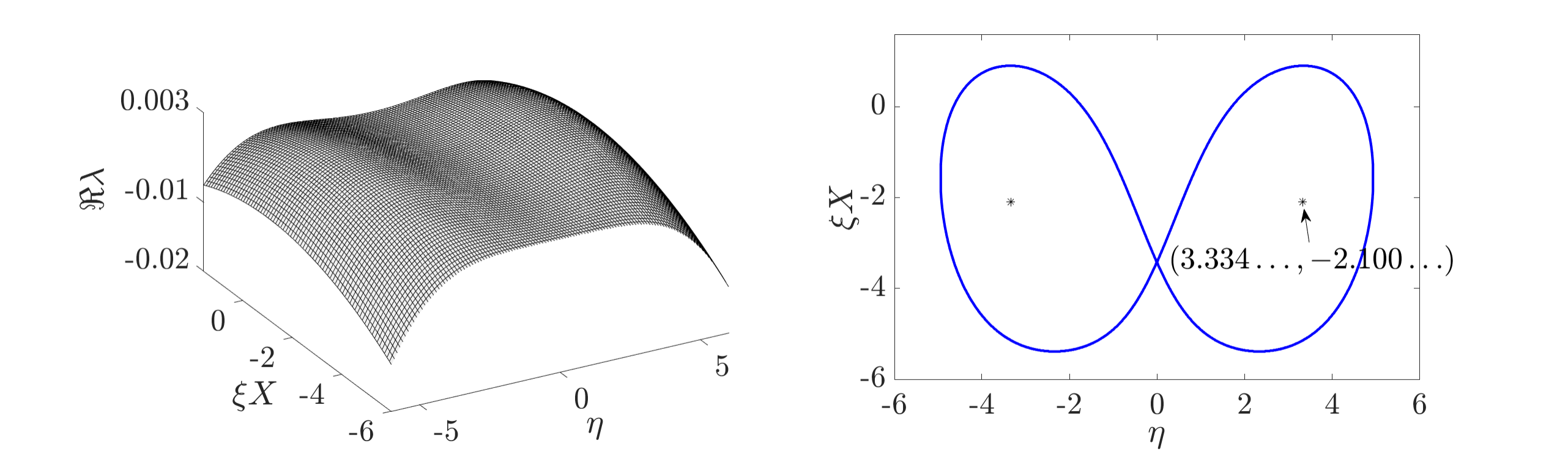}
\end{center}
    \caption{Visualization of spectra of planar roll waves in the whole space with $(F,H_-)=(4.4,0.3645071607\ldots)$ on the 1d medium frequency stability boundary. Left panel: plots of the surface $(\eta,\xi X,\Re\lambda(\eta,\xi X))$; Right panel: the level contour (blue curve) $\{(\eta,\xi X):\Re\lambda(\eta,\xi X)=0\}$. The bounded region inside the blue curve is where $\Re\lambda>0$. At $(\eta,\xi X)=(3.334\ldots,-2.100\ldots)$, $\Re\lambda(\eta,\xi X)$ achieves its local maximum $0.002695\ldots$. }
    \label{fig:spectrumf4_point4}
\end{figure}

\subsection{Channel flow and flapping bifurcation}\label{s:channel_flow_roll}
By the same arguments as in subsection~\ref{s:channel_flow_hydraulic} for the hydraulic shock case,
channel flow with wall-type boundary conditions in a channel with width $L$ leads for roll waves to the restrictions $\eta=n\pi/L$, $n\in \mathbb{N}$. 
Moreover, a complete basis for normal modes in the channel are given by a pure sine expansion for interior
variable $\hat p$ and pure cosine expansions for interior variables $\hat h$, $\hat q$, and front variable 
$\hat \psi$. See \eqref{fouriermodes}.

From this simple observation, we can draw very interesting conclusions regarding the generic form of bifurcations
resulting from transverse low-frequency instability as channel width $L$- here considered as a bifurcation 
parameter- is increased from the 1d limit $L=0$.
This generic behavior, that we have termed ``flapping bifurcation,'' could be considered as the
periodic analog of ``cellular instabilities'' described in \cite{Z7} for viscous traveling fronts.

Namely, consider a fixed roll wave with $(F,H_-)$ sitting in the 1d stable but 2d unstable region as depicted in FIGURE~\ref{fig:stabdiag_fig}, and suppose that high-frequency 
stability holds as we have numerically observed.  Then, there are no unstable normal modes for
the 1d case $\eta=0$, but there do exist unstable normal modes $\Re \lambda>0$ for
frequencies $(\lambda, \eta, \xi)$ with $\eta\neq 0$.  Define $\eta_*$ to be the supremum of $\eta$ among
such unstable frequencies. By boundedness of spectra, coming from the assumed high-frequency stability,
there is a sequence of normal mode frequencies converging to a limit $(i\tau_*, \eta_*, \xi_*)$ for
which the periodic Evans-Lopatinsky function vanishes, and this is thus a neutral normal mode.
Assume the generic situation that this is the unique set of normal mode frequencies involving $\eta_*$

Now, taking width $L$ as a bifurcation parameter consider the channel flow problem as width $L$ is 
increased from $L=0$.
As we have seen, this restricts $\eta$ to the frequencies $\eta_j= \pi j/L$.
In the limit as $L\to 0$, only the 1d case $ \eta=0$ remains; in particular, the wave is
stable for width-$L$ channel flow, by assumed 1d stability.
As $L$ increases, the minimum nonzero $\eta$ frequency $\eta_{1}=  \pi/L$ decreases,
so that eventually the wave becomes unstable as a solution of width-$L$ channel flow.
Evidently, the stability transition $L=L_*$ occurs for $L_*:=\pi/\eta_*$, where
$\eta_*$ is the above-defined supremum value of $\eta$ among nonstable normal mode frequencies,
with $(i\tau_*, \eta_*, \xi_*)$ the associated neutral frequencies.

We expect therefore, a bifurcation as $L$ increases past $L_*$ to a time- and space-oscillating solution
close to the normal mode associated with $(i\tau_*, \eta_*, \xi_*)$.
By \cite[eq. (1.8)]{JNRYZ} regarding structure of normal modes, 
this associated neutral normal mode should have front behavior
\be\label{nflap}
\psi_j(y,t)= \hat \psi e^{i\tau_* t}e^{i\eta_*y}e^{ij\xi_* X} + c.c.= |\hat{\psi}| \cos(\eta_* y + \tau_* t + j \xi_* X+{\rm arg}(\hat{\psi}) ),
\ee
where $c.c.$ denotes complex conjugate and ${\rm arg}$ stands for argument. Evaluating at the endpoints $y=0, L_*$, we have the $2\pi/\tau_*$--periodic function $\psi_j(y,\cdot)$ satisfies
$$
\psi_j(L_*,t)=-\psi_j(0,t),
$$
giving a ``flapping'' appearance to the movement of each front.  Meanwhile, successive fronts are translated
in phase by a constant increment $\xi_* X$, giving the cited appearance of a roll wave ``doing the wave.''
And, this is indeed what we see in our numerical time-evolution experiments.  See FIGURE \ref{fig:flapping} and FIGUREs \ref{fig:flappingf6_width18}--\ref{fig:flapping_time_plot}.

The analyses above were indeed motivated by time-evolution simulations for 2d channel roll waves where we stumbled upon ``flapping'' shock fronts instead of herringbone patterns, while we were initially looking for pattern formations due to transverse instability of roll waves with $(F,H_-)=(6,0.28)$ in the 1d stable but 2d unstable region in a channel with width $1$. See FIGURE~\ref{fig:flapping} for snapshots of the simulation.
\begin{figure}[htbp]
\begin{center}
\includegraphics[scale=0.6]{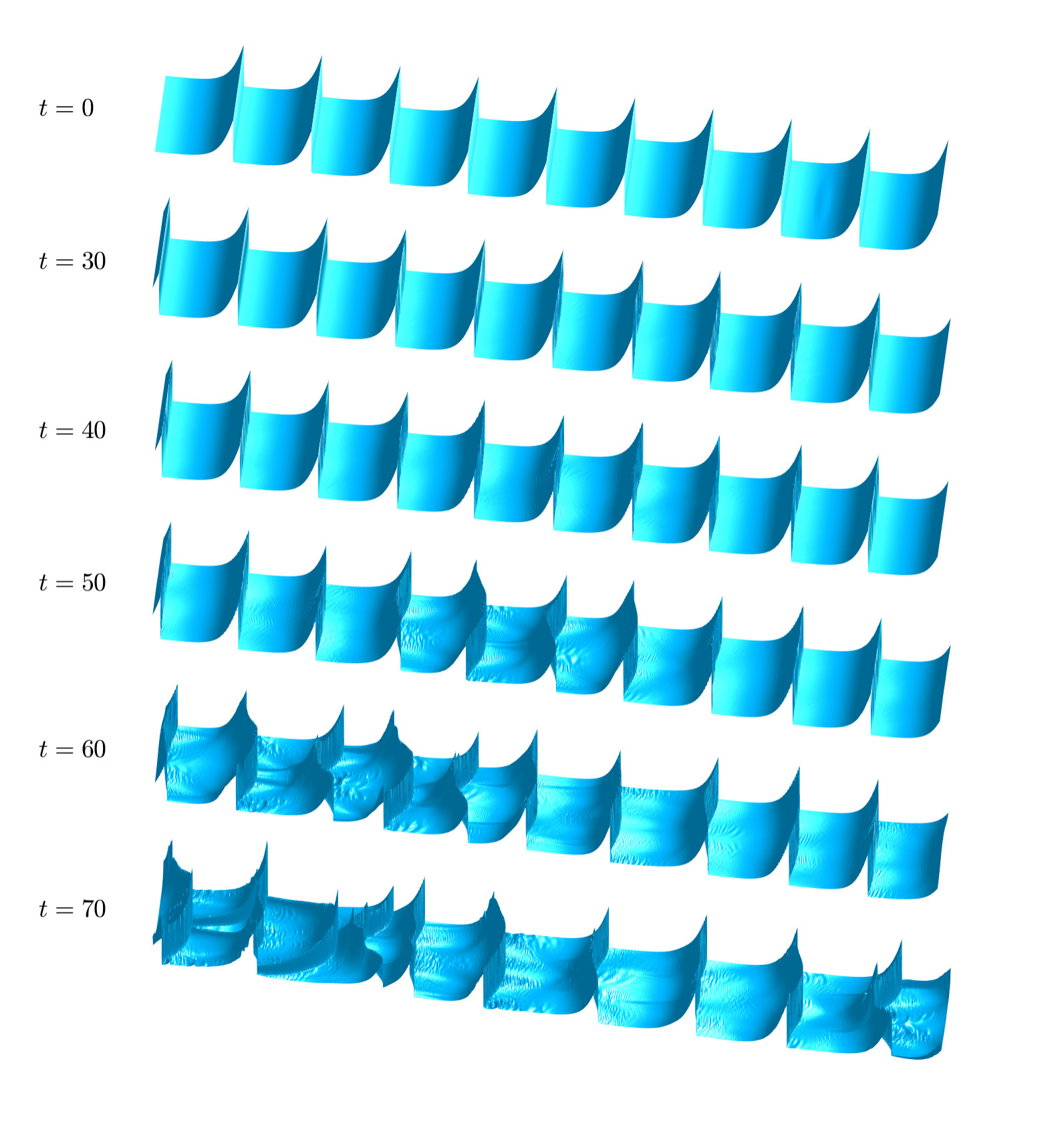}
\end{center}
    \caption{See \cite[roll\_width\_1.mp4]{YZmovie} for the full movie. Time-evolution in a channel with width $1$ and perturbed roll waves with $F=6$ and $H_-=0.28$ as the initial condition. At $t=30$, shock fronts start to vibrate. From $t=30$ to $t=60$, the amplitude of the flapping fronts increases. For $t>70$, waves become chaotic.}
    \label{fig:flapping}
\end{figure}

Numerically, one can use nonlinear system solvers, for instance fsolve available in MATLAB, to solve for roots $\lambda(\eta,\xi)$ of the periodic Evans-Lopatinsky function $E:\mathbb{C}\times \mathbb{R}\times\mathbb{R}\rightarrow \mathbb{C}$, $(\lambda,\eta,\xi)\rightarrow E(\lambda,\eta,\xi)$. Recalling Section~\ref{s:rollLF}, for $|\eta|, |\xi|\ll 1$, the function admits two roots $\lambda_1(\eta,\xi)$ and $\lambda_2(\eta,\xi)$ being also roots of the Weierstrass polynomial \eqref{weierstrassE}-\eqref{bcexpansion}, which help to supply initial guess of the roots. For  $|\eta|,|\xi|$ relatively large, roots of the function can be traced by continuation. See FIGURE~\ref{fig:spectrumf6} for plots of the surface $(\eta,\xi X,\max\{\Re\lambda_{1},\Re\lambda_{2}\}(\eta,\xi X))$ and the level contour $\{(\eta,\xi X):\max\{\Re\lambda_{1},\Re\lambda_{2}\}(\eta,\xi X)=0\}$ for $(F,H_-)=(6,0.28)$, from which we find $$(i\tau_*,\eta_*,\xi_* X)=(i3.201\ldots,19.37\ldots,0.3370\ldots).$$
\begin{figure}[htbp]
\begin{center}
\includegraphics[scale=0.35]{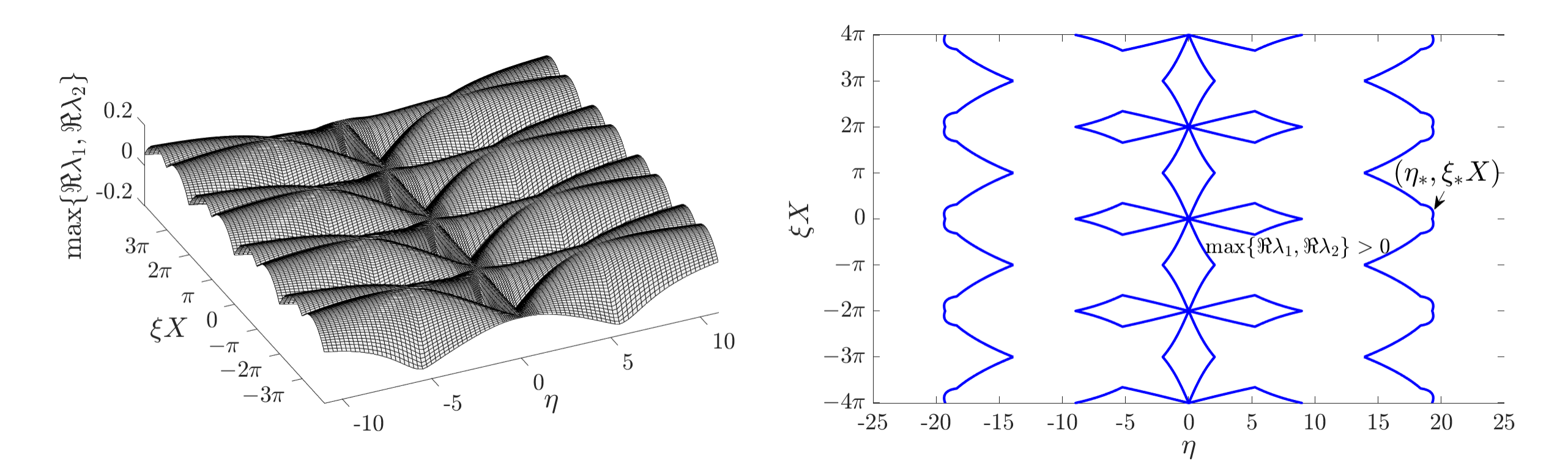}
\end{center}
    \caption{Visualization of spectra of planar roll waves in the whole space with $(F,H_-)=(6,0.28)$. Left panel: plots of the surface $(\eta,\xi X,\max\{\Re\lambda_{1},\Re\lambda_{2}\}(\eta,\xi X))$; Right panel: the level contour $\{(\eta,\xi X):\max\{\Re\lambda_{1},\Re\lambda_{2}\}(\eta,\xi X)=0\}$. The region where $\max\{\Re\lambda_{1},\Re\lambda_{2}\}$ is positive is indicated in the plot. On the blue curves are where $\max\{\Re\lambda_{1},\Re\lambda_{2}\}(\eta,\xi X)=0$, crossing which the function switches sign. }
    \label{fig:spectrumf6}
\end{figure}
We readily compute the critical channel width
$L_*=\pi/\eta_*=0.1621\ldots$ and the period of the flapping $T=2\pi/\tau_*=1.9628\ldots$.
Simulations were then made for $0.15$, $0.16$, $0.17$, $0.18$, $0.19$, and $0.2$ up to $t=1000$ and the movies with $L=0.18$ and greater exhibited flapping shock fronts. See \cite[roll\_width\_point15.mp4, roll\_width\_point16.mp4, roll\_width\_point17.mp4, roll\_width\_point18.mp4, roll\_width\_point18\_refined.mp4,  roll\_width\_point19.mp4, and roll\_width\_point2.mp4]{YZmovie} for the series of full movies. In particular, in \cite[roll\_width\_point18\_refined.mp4]{YZmovie}, shock fronts start flapping at about $t=120$. The flapping waves are persistent and do not become chaotic. See FIGURE~\ref{fig:flapping_time_plot} left panel for $\psi_2(0,t)$, namely the $2^{nd}$ shock front location with $y=0$ versus $t$. On the right panel of  FIGURE~\ref{fig:flapping_time_plot}, we read that the period of the flapping is $2.1$ (close to the prediction $1.9628\ldots$).  See FIGURE~\ref{fig:flappingf6_width18} for the first three flapping shock fronts at $t=350$, $350.2$,  $350.4$,  $350.6$,  $350.8$. Because the small-amplitude flapping shock fronts can be barely seen in the snapshots, we therefore also plot the shock fronts in the $x-y$ plane. See FIGURE~\ref{fig:flappingf6_wid18}. For $L=0.18$ and $L=0.19$, shock fronts start flapping at about $t=220$ \footnote{Compared with  \cite[roll\_width\_point18\_refined.mp4]{YZmovie}, The movie \cite[roll\_width\_point18.mp4]{YZmovie} is simulated in a longer channel with length of ten periods of the roll waves, whence shock fronts start flapping at a different time.} and $100$. The flapping waves are persistent and do not become chaotic. See \cite[roll\_width\_point18.mp4, roll\_width\_point19.mp4]{YZmovie}. For $L=0.2$, the flapping waves are also unstable, transitioning to chaotic flow at $t\approx 350$. See \cite[roll\_width\_point2.mp4]{YZmovie}.

\smallskip

\noindent\textbf{Periodic y-boundary condition.} For roll waves with $(F,H_-)=(6,0.28)$, Lemma~\ref{lem:periodic} together with our former numerics implies that the critical channel width of channels with periodic $y$-boundary condition will be $L_{*,\text{per}}=2L_*=0.3242\ldots$. Changing to periodic $y$-boundary condition, we predict that stability shall be recovered for roll waves in a channel with width $0.18$ and flapping shock fronts shall exhibit again if the channel width is doubled to $0.36$. Simulations \cite[roll\_width\_point18\_periodic.mp4, roll\_width\_point36\_periodic.mp4]{YZmovie} are consistent with our predictions.

\begin{figure}[htbp]
\begin{center}
\includegraphics[scale=0.4]{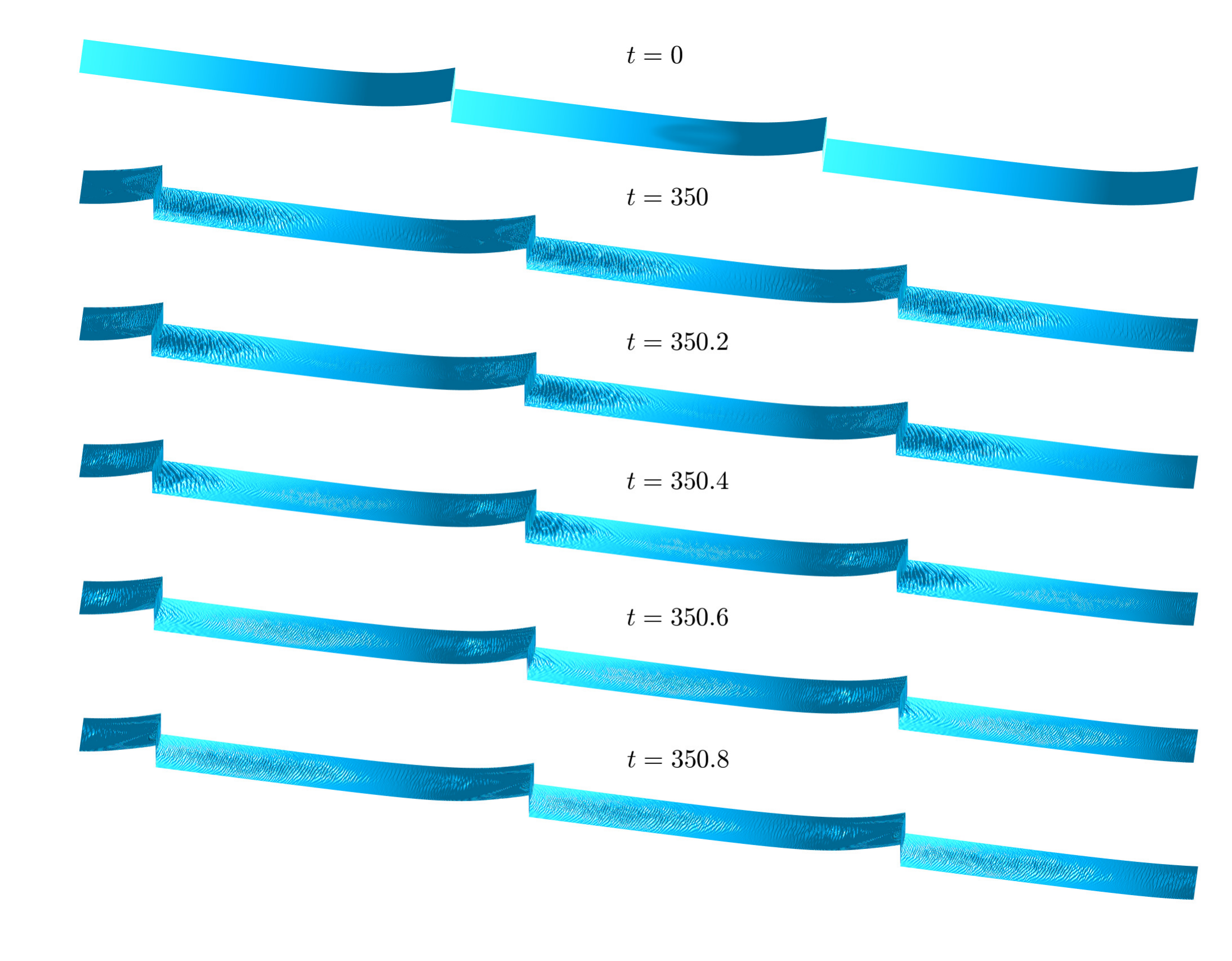}
\end{center}
\caption{See \cite[roll\_width\_point18\_refined.mp4]{YZmovie} for the full movies. Time-evolution in a channel with width $0.18$ and perturbed roll waves with $F=6$ and $H_-=0.28$ as the initial condition. We display the persistent flapping shock fronts at $t=350$, $350.2$, $350.4$, $350.6$, $350.8$.}
\label{fig:flappingf6_width18}
\end{figure}

\begin{figure}[htbp]
\begin{center}
\includegraphics[scale=0.31]{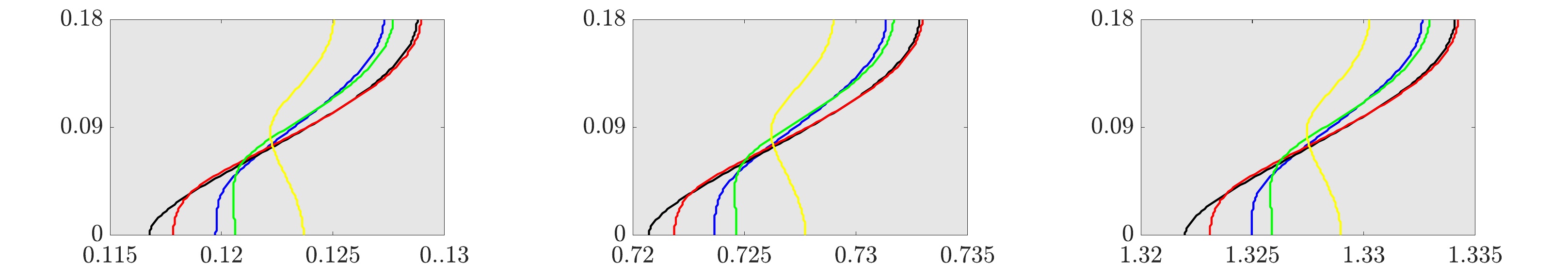}
\end{center}
\caption{Visualization of flapping shock fronts of channel roll waves with $(F,H_-)=(6,0.28)$ and $L=0.18$. On the left, middle, and right panels, there display consecutive three shock fronts at time $t=350$ (blue), $t=350.2$ (black), $t=350.4$ (red), $t=350.6$ (green), and $t=350.8$ (yellow).}
    \label{fig:flappingf6_wid18}
\end{figure}

\begin{figure}[htbp]
\begin{center}
\includegraphics[scale=0.33]{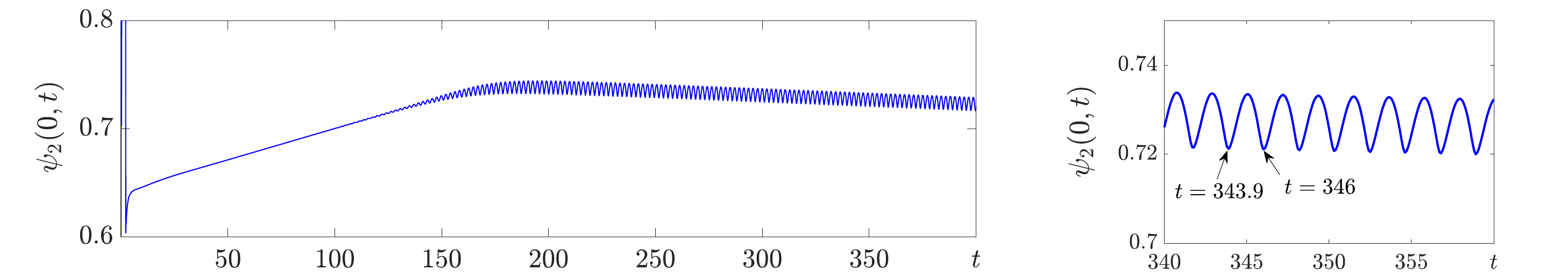}
\end{center}
\caption{Visualization of flapping shock fronts of channel roll waves with $(F,H_-)=(6,0.28)$ and $L=0.18$. Left panel: the $2^{nd}$ shock front with $y=0$ versus $t$; Right panel blowup at the time interval $[340,360]$. }
\label{fig:flapping_time_plot}
\end{figure}
\subsubsection{Stability boundaries for channel roll waves}\label{s:bd_channel_roll}
Associated to a roll wave with $(F,H_-)$ sitting in the 1d stable but 2d unstable region, our foregoing analysis defines a critical frequency $\eta_*$ and a critical channel width $L_*$ by
\be \label{defetahstar}
\eta_*=\sup\{\eta:E(\lambda,\eta,\xi)=0,\Re\lambda>0\}\quad\text{and}\quad L_*=\frac{\pi}{\eta_*}.
\ee 
Treating $L_*$ ($\eta_*$) as function of $(F,H_-)$, which can be numerically computed by the method used for FIGURE~\ref{fig:spectrumf6}, the former consistency between theoretical predictions and numerical time-evolution experiments suggests that one can, without doing time-expensive evolutions, predict the stability and behavior of roll waves in a channel with width $L$. That is if $L<L_*(F,H_-)$, the channel roll waves are stable. As $L$ is increased beyond $L_*(F,H_-)$ the roll waves become unstable, and we see a bifurcation to flapping waves. See FIGURE~\ref{fig:flappingf6_width18}. For $L$ still larger, the flapping waves also become unstable, transitioning to apparently chaotic flow.  See FIGURE~\ref{fig:flapping}. Since the stability for roll waves in the whole space implies the stability for roll waves in a channel with arbitrary channel width, we obtain $L_*(F,H_-)=+\infty$ for $(F,H_-)$ sitting in the 2d stable region. For any $(F,H_-)$ sitting in the interior of the complement of the 2d stable region, $L_*(F,H_-)$ is thus finite and, by assumed high-frequency stability, greater than zero. By fixing $L_*$ to be some constant $c\in \mathbb{R}^+$, there we can compute a curve (curves) where roll waves have associated critical channel width $L_*=c$. Knowing such curves parameterized by $c$ is very useful: for we know not only $c$ is the critical channel width $L_*$ associated to roll waves with parameters on the curve(s), but also the critical channel width $L_*$ must be greater than $c$ for roll waves with parameter on one side of the curve(s) and less than $c$ on the other side. Because we already know where the region $L_*=+\infty$ sits, it will be easy to determine the side on which $L_*>c$ ($L_*<c$). The curve(s) on which $L_*=c$ then sub-divides the 1d stable region into sub-region(s). Consider roll waves in a channel with width $L$. The condition $L<c$ implies stability in the sub-region(s) which contains the region where $L_*=+\infty$ and the condition $L>c$ implies instability in the sub-region(s) which does not contain the region where $L_*=+\infty$. We thus refer to the curve(s) where $L_*=c$ as the stability boundary(ies) for channel roll waves with critical channel width $L_*=c$. In FIGURE~\ref{fig:stabdiag_fig}, we display the stability boundaries for channel roll waves with $L_*=1$, $\frac{1}{3}$, and $0.1$. These curves intersect the 1d low-frequency stability boundary at $F=3.83\ldots$, $4.88\ldots$, and $6.97\ldots$ and the 1d medium-frequency stability boundary at $F=4.96\ldots$, $5.29\ldots$, and $6.72\ldots$, respectively.

\subsection{Whitham modulation equations and the role of oblique waves}\label{s:rollformal} 
Finally, we return briefly to the whole-space setting to discuss the interesting phenomenon
of {\it oblique waves}, and their connection to low-frequency stability.
By {oblique waves}, we mean planar traveling-wave profiles 
$(h,u,v)(x,y,t)= (\bar h, \bar u, \bar v)(\cos \theta x+ \sin \theta y -st )$
in directions other than the $x$-axis, i.e., for $\theta\neq 0,\pi$.
Evidently, these cannot exist in channels unless $v\equiv 0$, which can be seen to force $\theta=0,\pi$,
the standard 1d case.
Thus, neither oblique waves- nor, for that matter, low-frequency stability, since it involves
frequencies excluded by discrete Fourier transform- directly impact channel flow.

However, as discussed in Appendix \ref{s:oblique}, there appears to exist a smooth family of such waves 
in the whole-space setting, and these may be linked to low-frequency stability by formal Whitham
approximation \cite{Wh,JNRZ4} as we now discuss.
Moreover, even though low-frequency stability does not explicitly appear in stability for channel flow,
it can serve as a useful organizing center in the background.

\smallskip
\noindent\textbf{Whitham modulation equations.}
There is in the periodic case a low-frequency approximation analogous to that of Lemma \ref{ZSlem}
in the hydraulic shock case, reflecting the idea that again, fine structure is not ``seen'' at large space-time
scales, but instead leads to homogenized or ``averaged'' dynamics: namely, the modulation expansion of
Whitham \cite{Wh}. 
See \cite{Se,OZ3,JNRZ4} for general derivations and rigorous connection to low-frequency spectra;
in the specific context of roll waves, see the 1d treatment of \cite{JNRYZ}.
These take the form
\ba\label{modeq}
\mathcal{F}_0(M,K)_t + \mathcal{ F}_1(M,K)_x + \mathcal{ F}_2(M,K)_y&=0,\\
K_t + \nabla_{x,y} \Omega(K)=0,
\ea
augmented with constraint ${\rm curl}(K)\equiv 0$, 
where $K$ denotes spatial (vector) wavenumber and $\Omega$ temporal wavenumber of
planar periodic solution $\bar u^{M,K}(K\cdot x + \Omega t)$, parametrized locally by $M\in \R^m$, where $m$
is nullity of $R$; $\Omega=\Omega(K,N)$ denotes the nonlinear dispersion
relation given by the traveling-wave existence problem; and $\mathcal{F}_j$ are the averages over one 
$\mathcal{F}_j$ is the average over one period of the conservative part of flux $F_j(\bar u^{M,K})$;
without loss of generality $\mathcal{F}_0=M$.

In this framework, general solutions are approximated by periodic traveling-waves $\bar u^{(M,K)(x,t)}$
with parameters $(M,K)$ slowly varying compared to the size of the period $X$.
Thus, a single planar roll wave solution corresponds to a {\it constant-parameter} solution
$(M,K)\equiv  (M_0,K_0)$,
and the linearization of \eqref{modeq} about $(M_0,K_0)$ 
describes the approximate evolution of parameters $(M,K)$ under perturbation.
Taking without loss of generality $K_0=(1,0)^T$, and observing by isotropy that 
$\mathcal{F}_j= K_j h(M,r)$, $j=1,2$, and $\Omega=\omega(M,r)$, with $\kappa:=|K|$ gives linearized equations
\ba\label{linmod}
\bp M\\K_1\\K_2\ep_t + 
\bp \overline{h_M} & \overline{h_\kappa}+ \overline{h} & 0\\ \overline{\omega_M} & \overline{\omega_\kappa}& 0\\
 0&0&0 \ep \bp M\\K_1\\K_2\ep_x + 
\bp 0 & 0 & \overline{h} \\ 0&0&0 \\ 
\overline{\omega_M} & \overline{\omega_\kappa}& 0 \ep \bp M\\K_1\\K_2\ep_y= \bp 0\\0\\0\ep,
\ea
where barred entries are evaluated at base state  $(M_0,K_0)$, or $W_t+B_1W_x+B_2W_x=0$.

Defining the characteristic determinant 
\be\label{epsdef}
\eps(\lambda, \eta, \xi):=\lambda^{-1}\det(\lambda \Id + i\eta B_1 + i\eta B_2),
\ee
we have that \eqref{linmod} is well-posed, i.e., {\it hyperbolic}, under the constraint ${\rm curl}(K)=0$
if and only if zeros $\lambda$ of $\eps(\cdot, \xi, \eta)$
are purely imaginary for all $\xi,\eta\in \R$. 
Here, term $\lambda^{-1}$ factors out the spurious root $\lambda=0$ coming from curl equation $(d/dt){\rm curl}(K)_t=0$.
Based on \cite{Se,OZ3,JNRYZ}, we expect the following analog of Lemma \ref{ZSlem}.

\smallskip

\noindent{\bf Conjecture:} For $|(\lambda, \eta,\xi)|$ sufficiently small, and some fixed real $\gamma\neq 0$,
\be\label{pZS}
E(\lambda,\eta,\xi)=\gamma \eps(\lambda, \eta,\xi) + \mathcal{o}(|(\lambda, \eta, \xi)|).
\ee
Under mild assumptions on $\eps(\cdot)$- valid for \eqref{sv_intro}- 
this would give agreement to linear order of zeros of $E$ and $\eps$,
rigorously justifying the formal stability requirement \cite{Wh} of hyperbolicity of \eqref{linmod}.

\smallskip

We leave the proof of this conjecture as an interesting direction fo future investigation.
For the present work, with its emphasis on hydraulic engineering channel flow,
this represents a side direction outside our main scope- as indeed are oblique waves themselves.

\smallskip

\noindent\textbf{Low-frequency stability as organizing center.}
As seen in Section \ref{s:channel_flow_roll}, stability transitions for channel flow, having to do inherently
with medium and large frequencies, are difficult to predict analytically and seem to require numerical 
resolution.
On the other hand, these large/medium-scale instabilities can be thought of as growing from low-frequency
instabilities for the whole-space problem, detectable by Taylor expansion as in Section \ref{s:rollLF}.
In this sense, the low-frequency analysis, and associated whole-space stability transition,
serves as an organizing center for bifurcations in the family of channels with all possible widths $L$.
Recall that whole-space stability transitions were observed numerically to occur precisely through
low-frequency instability at the origin.

Moreover, the analysis of Section \ref{s:channel_flow_roll} for channel flow, based on
numerical computation, could in the whole space be converted in principle to a rigorous description of behavior,
in the limit $(F,H_-)\to $ 2d stable region and $(F,H_-)\in$ the complement of 2d stable region with respect to 1d stable region,
for which the set of unstable spectrum is vanishingly small.
These can therefore be computed approximately from the Taylor series expansion about the origin of the 
periodic Evans-Lopatinsky function, as carried out in Section \ref{s:rollLF}.
This, also, is left to the future.

\appendix
\numberwithin{equation}{section}
\section{Existence of oblique roll waves}\label{s:oblique}
We consider here briefly the question, so far as we know not previously investigated,
of existence of {\it oblique waves} for \eqref{sv_intro}, 
that is, of planar traveling-wave profiles {\it in the whole space}
\be\label{oblique_prof}
(h,u,v)(x,y,t)= (\bar h, \bar u, \bar v)(k_1 x+ k_2 y + \omega t)
\ee
moving in transverse directions, i.e., for $k_2\neq 0$. 
Evidently, these cannot exist in channels unless $v\equiv 0$, which can be seen to force $(k_1,k_2,\omega)=(1,0,-s)$,
the standard normal case. We will focus on periodic (in the direction of wave propagation) solutions, referred to as oblique roll waves, while noting that part of the analyses also apply to oblique hydraulic shock profiles.

\smallskip

\noindent\textbf{Equations.}
By rescaling the temporal and spatial variables, one may, without loss of generality, take $(k_1,k_2,\omega)=(\cos \theta , \sin \theta,  -s)$. Performing change of variables and coordinates
$$
\begin{aligned}
h(t,x,y)\rightarrow &h(t,\cos(\theta)x+\sin(\theta)y,-\sin(\theta)x+\cos(\theta)y),\\
u(t,x,y)\rightarrow& \cos(\theta)u(t,\cos(\theta)x+\sin(\theta)y,-\sin(\theta)x+\cos(\theta)y)
\\&+\sin(\theta)v(t,\cos(\theta)x+\sin(\theta)y,-\sin(\theta)x+\cos(\theta)y),\\
v(t,x,y)\rightarrow& -\sin(\theta)u(t,\cos(\theta)x+\sin(\theta)y,-\sin(\theta)x+\cos(\theta)y)\\&+\cos(\theta)v(t,\cos(\theta)x+\sin(\theta)y,-\sin(\theta)x+\cos(\theta)y),
\end{aligned}
$$
yields
\ba\label{sv_oblique}
\d_th+\d_x (hu) +\d_y(hv) &=0, \\
\d_t (hu) +\d_{x}\left(hu^2 +h^2/2F^2 \right) + \d_y(huv) & =\cos(\theta)h - \sqrt{u^2+v^2} \,u , 
\\
\d_t (hv) + \d_x( huv) + \d_{y}\left(hv^2 +h^2/2F^2 \right) & =-\sin(\theta)h -\sqrt{u^2+v^2} \,v .
\ea

\noindent\textbf{Profile ODEs and Rankine-Hugoniot conditions.}
Substituting $$(h,u,v)(x,y,t)=(\bar{h},\bar{u},\bar{v})(x-s)$$ into \eqref{sv_oblique} and dropping bars for ease in writing, we obtain the profile
ODEs
\ba\label{profeq}
\Big(-sh+hu\Big)' &=0, \\
\Big(  -shu +hu^2 +h^2/2F^2 \Big)' & =\cos(\theta)h - \sqrt{u^2+v^2} \,u , 
\\
\Big(-shv+ huv \Big)'& = -\sin(\theta) h-\sqrt{u^2+v^2} \,v , 
\ea
or, setting
\be\label{q0}
-q_0:= h(u-s),
\ee
the reduced ODEs
\ba \label{uvode_sub}
-q_0\left(1 + \frac{q_0}{F^2(u - s)^3}\right)u'  & =-\cos(\theta)\frac{q_0}{u-s} - \sqrt{u^2+v^2} \,u , \\
-q_0v'& = \sin(\theta)\frac{q_0}{u-s}-\sqrt{u^2+v^2} \,v ,
\ea 
with $h=-q_0/(u-s)$. The Rankine-Hugoniot conditions read
\be \label{uvRH}
\left[\frac{ 2F^2u^3-4F^2su^2+2F^2s^2u - q_0}{(u-s)^2}\right]_x=0\quad \text{and}\quad  [v]_x=0,
\ee
where $[f]_x=f(x^+)-f(x^-)$ denotes the jump at a shock $x$. The latter equation of \eqref{uvRH} implies that $v(x)$ is continuous and periodic on $\mathbb{R}$.

\smallskip

\noindent\textbf{Normal case.} In the normal case ($\theta=0$), the reduced equations \eqref{uvode_sub} become
\ba\label{nsimp}
-q_0\left(1 + \frac{q_0}{F^2(u - s)^3}\right)u' & =-\frac{q_0}{u-s} - \sqrt{u^2+v^2} \,u , \\
-q_0 v'  &= -\sqrt{u^2+v^2} \,v .
\ea
If $v$ is nonzero somewhere, it is nonzero everywhere, whence it is strictly monotone on $\mathbb{R}$, contradicting the its periodicity. Thus, $v\equiv 0$, reducing to the 1d case.

Note that this is what we expect and assume, but not immediately obvious.  And we have not assumed that $v=0$
from the start, so that we see for initial parameters $(h_0,u_0,v_0)$ for the ODE, not all of them connect to
some viable Lax shock, but only a co-dimension one family.

\smallskip

\noindent\textbf{Orthogonal case.} In the orthogonal case ($\theta=\pi/2$), the equations \eqref{uvode_sub} become, rather,
\ba\label{osimp}
-q_0\left(1 + \frac{q_0}{F^2(u - s)^3}\right)u' & = - \sqrt{u^2+v^2} \,u , \\
-q_0 v'  &=\frac{q_0}{u-s} -\sqrt{u^2+v^2} \,v .
\ea
By Lax characteristic conditions, in one period of the periodic wave, there is a point, referred to as the ``sonic point'', where one of the characteristic speeds is equal to the speed of the wave $s$, making the ODE system \eqref{osimp} degenerate. One may readily compute that the coefficient of the left hand side of the $u$ equation vanishes at a unique sonic point
$u=u_s= s-\sqrt[3]{q_0/F^2}$. On the other hand, the right hand side of the equation vanishes only at $u=0$,
so that there is no continuous solution crossing $u_s$, provided that $u_s\neq 0$. From these simple considerations, we see that {\it orthogonal traveling waves do not exist.}

\smallskip

\noindent\textbf{General case.} To analyze the general case, we first perform a rescaling, analogous to \cite[eq. (2.10)]{JNRYZ}, 
$$
u(\cdot)\rightarrow u_s u(\frac{1}{u_s^2}\times \cdot) ,\quad v(\cdot)\rightarrow u_s v(\frac{1}{u_s^2}\times \cdot),\quad q_0\rightarrow u_s^3 q_0, \quad s\rightarrow u_s s,
$$
so that, at the sonic point $x_s=0$, $u_s=1$. At the sonic point, both the left and right hand sides of \eqref{uvode_sub}[i] must vanish, giving 
$$
q_0=F^2(s-1)^3\quad \text{and}\quad  \cos(\theta)\frac{q_0}{s-1}-\sqrt{1+v_s^2}=0,
$$
from which we obtain
\be \label{seq}
s=1\pm \frac{(1+v_s^2)^{\tfrac{1}{4}}}{F\sqrt{\cos(\theta)}}\quad \text{and} \quad q_0=\pm\frac{(1+v_s^2)^{\tfrac{3}{4}}}{F\sqrt{\cos(\theta)^{3}}}.
\ee 
In the normal case, recall $\theta=0$ and $v_s=0$, giving
\be \label{theta0}
s=1\pm \frac{1}{F}\quad\text{and}\quad  q_0=\pm\frac{1}{F}.
\ee
Note that Dressler roll waves take ``+'' sign above, i.e., there always holds $u<s$. When $\theta\neq 0$, we claim that $v\neq 0$ along the profile. To prove the claim, recall that $h=-q_0/(u-s)>0$ and $v(x)$ is periodic and continuous on $\mathbb{R}$. When $v=0$, $v'=-\sin(\theta)/(u-s)$ is sign-definite. The variable $v$ cannot continuously return to $0$ once it leaves $0$. This completes the proof. Also, since $v$ returns to its initial value in one period, the right hand side of \eqref{uvode_sub}[ii] must vanish at some point, yielding that the sign of $\sin(\theta)$ must be opposite to that of $v$. 

\smallskip

\noindent\textbf{Parameters.} Noting that \eqref{uvode_sub} is invariant under 
$$u\rightarrow u, \quad v\rightarrow -v,\quad \text{and} \quad \theta\rightarrow -\theta,$$ we assume, without loss of generality, $\sin(\theta)<0$ and $v>0$. From \eqref{seq}, we see there are three parameters $(F,\theta,v_s)$ with 
\be
\label{thetadomain}
-\frac{\pi}{2}<\theta< 0.
\ee 
The choice of the parameters is subject to some other hidden conditions, e.g., $u'(x_s)$, being a root of the quadratic equation
\be \label{uprime}
3F^2(s - 1)^5 x^2+ \left(\frac{s - 1}{F^2\cos(\theta)} - F^2\cos(\theta)(s - 1)^4 + F^2\cos(\theta)(s - 1)^5\right)x+\frac{v_s(\tan(\theta) + v_s)}{F^2}=0,
\ee 
must be real and others required by Rankine-Hugoniot conditions \eqref{uvRH}, which we study (numerically) below.

\smallskip

\noindent\textbf{Constant oblique roll waves}. For constant oblique roll waves, the profile $(h,u,v)$ holds at constant level $(h,u,v)=(h_s=\sqrt{1+v_s^2}/\cos(\theta),1,v_s)$, yielding 
$$
0=-q_0v'=-\sin(\theta)F^2(s-1)^2-\sqrt{1+v_s^2}v_s=-\sqrt{1+v_s^2}\left(\tan(\theta)+v_s\right), \quad \text{or} \quad v_s=-\tan(\theta).
$$
Perturbing from the constant roll waves, i.e., perturbing $v_s$ away from $-\tan(\theta)$, we now numerically compute oblique roll waves of small amplitude.

\smallskip

\noindent\textbf{Hybrid method}. Local to the sonic point, we take 
$$
u(x)=1+u_1x+u_2x^2+\cdots\quad \text{and}\quad v(x)=v_s+v_1x+v_2x^2+\cdots.
$$
Away from the sonic point, we evolve the system \eqref{uvode_sub} to some $x_+>0$ and $x_-<0$ so that the Rankine-Hugoniot equations \eqref{uvRH} are satisfied.
For $|v_s+\tan(\theta)|$ small and 
$$
s=1+ \frac{(1+v_s^2)^{\tfrac{1}{4}}}{F\sqrt{\cos(\theta)}},
$$
numerically we find if $v_s<-\tan(\theta)$, or equivalently $v_1=v'(0)<0$, there is NO nearby oblique roll waves for there exists no $(x_-,x_+)$ such that \eqref{uvRH} are satisfied. And, for $-\pi/2<\theta_{\text{min}}<\theta<0$, if $-\tan(\theta)<v_s<v_{s,\text{max}}$, or equivalently $0<v_1=v'(0)<v_{1,\text{max}}$, there is a nearby oblique roll wave. To exemplify oblique roll waves, for $F=2.5$, $\theta=-0.1$, we find $v_{s,\text{max}}=\tan(0.1)+0.000526\ldots$. In FIGURE~\ref{oblique_orbits}, we show, in the left panel, the orbit of oblique roll waves with $$(F,\theta,v_s)=(2.5,-0.1,\tan(0.1)+0.0004),$$
and, in the middle panel, a non-example with $(F,\theta,v_s)=(2.5,-0.1,\tan(0.1)-0.0004)$. In the non-example, for $x_-<0$, though one can find $x_+(x_-)>0$ such that the Rankine-Hugoniot condition for $v$ \eqref{uvRH}[ii] $v(x_-)=v(x_+(x_-))$ is satisfied, it is numerically found that the Rankine-Hugoniot condition for $u$ \eqref{uvRH}[i] is not satisfied at any $(x_-<0,x_+(x_-))$. The transition from non-existence to existence happens at $v_s=-\tan(\theta)$ where $v'(0)=v_1=0$. See FIGURE~\ref{oblique_orbits} right panel for blowup of the example and non-example near the sonic point, showing change of sign of $v'(0)$.

\begin{figure}[htbp]
\begin{center}
\includegraphics[scale=0.4]{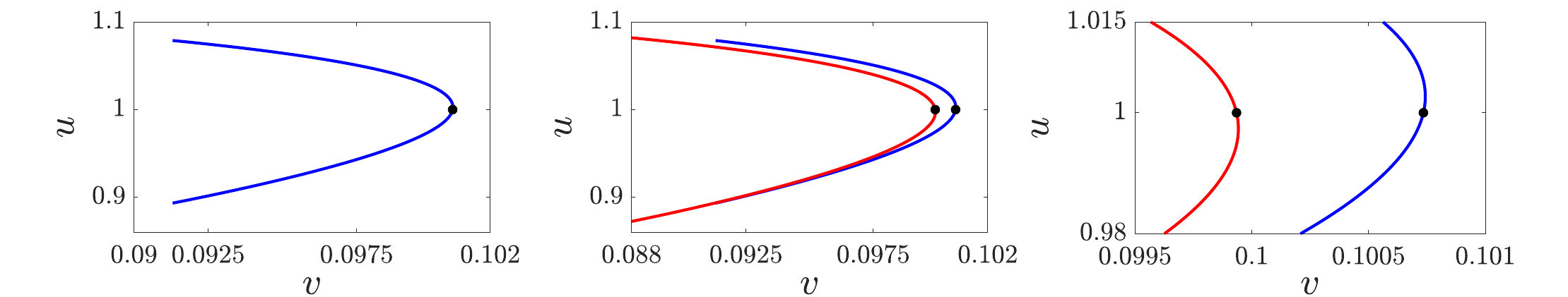}
\end{center}
\caption{Left panel: an example of oblique roll waves with $(F,\theta,v_s)=(2.5,-0.1,\tan(0.1)+0.0004)$. The part of the orbit of \eqref{uvode_sub} satisfying Rankine-Hugoniot conditions \eqref{uvRH} at end points is plotted in blue curve with sonic point marked by a black solid dot. Middle panel: a non-example orbit of \eqref{uvode_sub} with $(F,\theta,v_s)=(2.5,-0.1,\tan(0.1)-0.0004)$ is plotted in red curve with sonic point marked by a black solid dot. Right panel: blowup of the middle panel near the sonic point. At the sonic point, in the non-example $v'$ is negative and in the example $v'$ is positive.}
\label{oblique_orbits}
\end{figure}
\br\label{existrmk}
One should be able to carry out a rigorous analytic proof of existence for small (nearly constant) or
approximately normal ($\theta$ near zero) waves by a more detailed Melnikov analysis of the profile ODEs.
However, we shall not pursue this direction here.
\er

\smallskip

\noindent\textbf{Numerical time-evolution.} Setting the RHS of \eqref{sv_oblique} zeros, we obtain that an equilibrium state takes the form $(h,u,v)=(h,\sqrt{h}\cos(\theta),-\sqrt{h}\sin(\theta))$. We make a simulation of \eqref{sv_oblique} with $(F,\theta)=(2.5,-0.1)$ and the initial data \be \label{oblique_initial}(h,u,v)=(1+e^{-\frac{0.2}{0.25-(x-4)^2-(y-0.5)^2}}\mathbbm{1}_{(x-4)^2+(y-0.5)^2<0.25},\cos(0.1),\sin(0.1)).
\ee 
See FIGURE~\ref{oblique_profiles} bottom panels for a snapshot of the simulation (at $t=800$, at $y=0$ for $x\in[124.7,137.2]$). The period of the simulated waves is approximately $4.1$. By adjusting $v_s$, we find that oblique roll waves with $(F,\theta)=(2.5,-0.1)$ achieves a period of $4.1$ for $v_s=\tan(0.1)+0.0005255$. See FIGURE~\ref{oblique_profiles} top panels for oblique roll waves with $(F,\theta,v_s)=(2.5,-0.1,\tan(0.1)+0.0005255)$. Comparing top and bottom panels, we find the simulated waves resemble the profile of oblique roll waves in shape.
\begin{figure}[htbp]
\begin{center}
\includegraphics[scale=0.4]{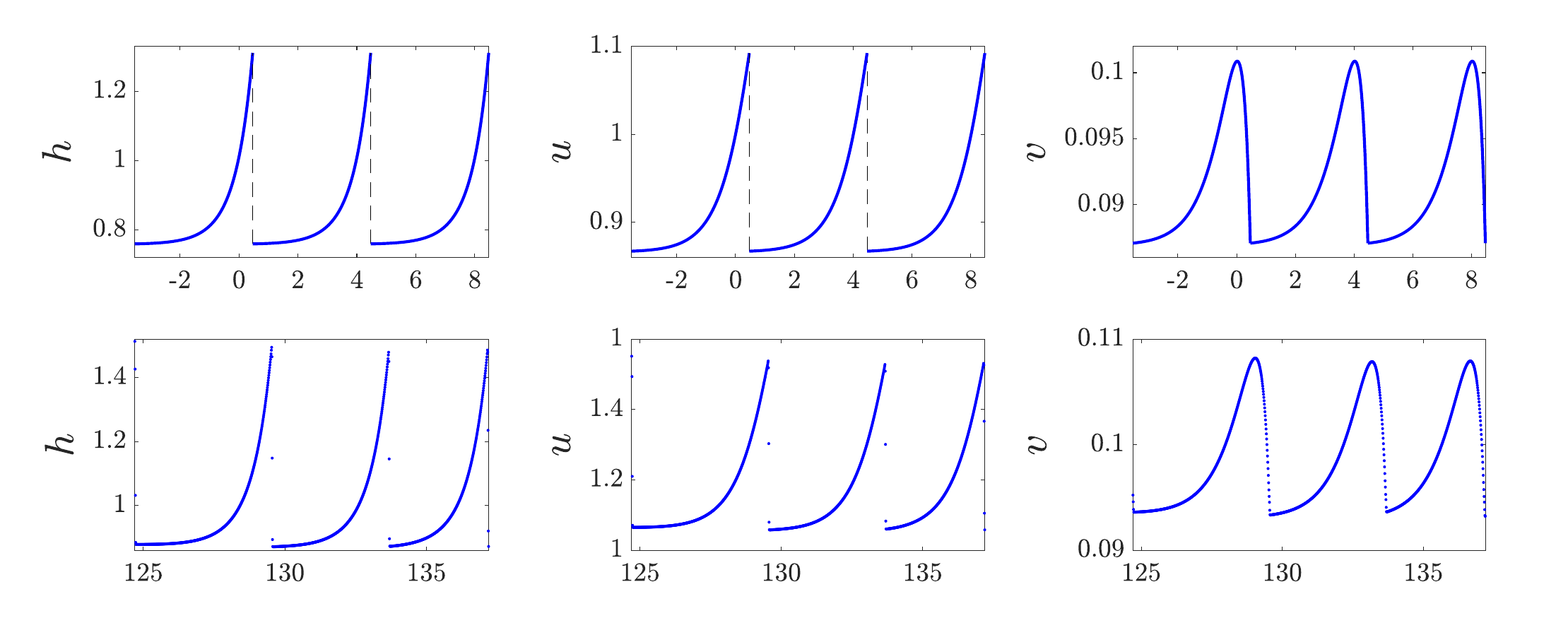}
\end{center}
\caption{Top three panels: the profile of oblique roll waves with $(F,\theta,v_s)=(2.5,-0.1,\tan(0.1)+0.0005255)$. Bottom three panels: simulation with $(F,\theta)=(2.5,-0.1)$ and the initial data \eqref{oblique_initial} in a channel of length $200$ and width $1$ with periodic $x$- and $y$- boundary conditions.}
\label{oblique_profiles}
\end{figure}

\section{Computation of the Evans functions}\label{Evans_numeric}
\noindent{\bf Discontinuous hydraulic shock profile.} To compute the Evans-Lopatinsky determinant \eqref{EL} effectively, we reformulate it by a {\bf dual formulation} \cite{PW,BSZ,HuZ1}. Write the two solutions $\mathcal{W}_j^-$, $j=1,2$ of the interior equations \eqref{zgeig_weight}[i] that decay exponentially as $x\to -\infty$ as $\mathcal{W}_j^-=[w_{1j}\;w_{2j}\;w_{3j}]^T$ and denote their wedge product as $\mathcal{W}_1^-\wedge\mathcal{W}^-_2=c_1e_1+c_2e_2+c_3e_3$ 
where

$$
\begin{aligned}
&e_1=\left[\begin{array}{c}1\\0\\0\end{array}\right]\wedge \left[\begin{array}{c}0\\1\\0\end{array}\right],& e_2=&\left[\begin{array}{c}1\\0\\0\end{array}\right]\wedge \left[\begin{array}{c}0\\0\\1\end{array}\right],& e_3=&\left[\begin{array}{c}0\\1\\0\end{array}\right]\wedge \left[\begin{array}{c}0\\0\\1\end{array}\right],\\
&c_1=\det\left(\left[\begin{array}{rr}w_{11} &w_{12}\\w_{21}&w_{22}\end{array}\right]\right),& c_2=&\det\left(\left[\begin{array}{rr}w_{11} &w_{12}\\w_{31}&w_{32}\end{array}\right]\right),& c_3=&\det\left(\left[\begin{array}{rr}w_{21} &w_{22}\\w_{31}&w_{32}\end{array}\right]\right).
\end{aligned}
$$ The idea is that $\mathcal{W}:=[c_3,\,-c_2,\,c_1]^T$ by direct computation satisfies the dual system
\be 
\label{W_ode}
\mathcal{W}'=\left(\mathrm{trace}(G_{\mu_{L}})\Id-G_{\mu_{L}}^T\right)\mathcal{W},
\ee
and its decaying rate toward $\mathbf{0}$, as $x\to -\infty$, is of order $e^{\mu(\lambda,\eta,\mu_L)x}$, 
where $\mu(\lambda,\eta,\mu_L)$ is the eigenvalue of $\left(\mathrm{trace}(G_{\mu_{L}})\Id-G_{\mu_{L}}^T\right)(-\infty;\lambda,\eta,\mu_{L})$ with the largest real part. Setting $\tilde{\mathcal{W}}=e^{-\mu x}\mathcal{W}$ yields \be
\label{tildeW_ode} \tilde{\mathcal{W}}'=\left(\mathrm{trace}(G_{\mu_{L}})\Id-G_{\mu_{L}}^T-\mu\Id\right)\tilde{W},\quad\text{or}\quad \tilde{\mathcal{W}}'=\left(-G^T-\tilde{\mu}\Id\right)\tilde{\mathcal{W}},
\ee  where $\tilde{\mu}=\mu-2\mu_{L}-\mathrm{trace}(G)$ is the eigenvalue of $-G^T(-\infty)$ with largest real part and $\tilde{W}(-\infty)$ is an eigenvector of $-G^T(-\infty)$ associated with $\tilde{u}$. 
The Evans-Lopatinsky determinant \eqref{EL} may then be expressed as
\be 
\label{duallopatinsky}
D(\lambda,\eta)=\sum_{j=1}^3[\lambda F_0+i\eta F_2-R]_j\tilde{\mathcal{W}}_j(0;\lambda,\eta),
\ee 
where the subindex $j$ corresponds to the $j^{th}$ entry of the corresponding column vector.  
From the form of the second equation of \eqref{tildeW_ode}, we see that 
computation of the Evans-Lopatinsky determinant can be made independent of the choices of weights $\mu_{L,R}$.
\br\label{indrmk}
Independence of the Evans function on choice of weighted norms is a more general phenomenon, reflecting the
distinction between ``inner layer dynamics'' captured by the Evans function- 
analogous to point spectrum/resonant poles- and ``outer layer dynamics'' captured by weighted norm computations-
analogous to essential spectrum in the 1d case \cite{ZH,ZH1,Z3}.
We note that an Evans function may be defined and efficiently computed without reference to weights
so long as there exists a spectral gap between fastest decaying and remaining modes \cite{PW,Z4}, 
whether or not these correspond to stable and unstable manifolds; and, indeed, existence of such a gap
is equivalent to existence of stabilizing scalar weights at $x\to \pm \infty$.
For further discussion, see \cite{FRYZ,BJRZ}.
\er

\smallskip

\noindent\textbf{Numerical schemes}. To solve $\tilde{\mathcal{W}}'=\left(-G^T-\tilde{\mu}\Id\right)\tilde{\mathcal{W}}$ effectively, we integrate numerical schemes used in \cite{JNRYZ,YZ,RYZ1}, combining hybrid method and working in $H$-coordinates and polar non-radial coordinates. Switching to $H$-coordinates $\tilde{\mathcal{W}}(x)\rightarrow \tilde{\mathcal{W}}(H(x))$ yields $H'd_H\tilde{\mathcal{W}}=\left(-G^T-\tilde{\mu}\Id\right)\tilde{\mathcal{W}}$. Without of loss of generality, we assume that $\tilde{\mathcal{W}}(H=1)$ is a unit eigenvector associated with the eigenvalue of $-G^T(H=1)$ with largest real part. Also, by working in polar coordinates $\tilde{\mathcal{W}}(H)=\Omega(H)\alpha(H)$ with $||\Omega(H)||\equiv 1$ and $\alpha(H=1)=1$, we instead solve 
\be \label{polar_interior}
H'd_H\Omega=(-G^T-\tilde{\mu}\Id)\Omega+\Omega\Omega^*(G^T+\tilde{\mu}\Id)\Omega
\ee 
and formulate a rescaled Evans-Lopatinsky determinant

\be 
\label{rescaledduallopatinsky}
\tilde{\Delta}(\lambda,\eta)=\sum_{j=1}^3[\lambda F_0+i\eta F_2-R]_j\Omega_j(H_*;\lambda,\eta),
\ee 
which, although it loses the analytic dependence on variables $\lambda$ and $\eta$, necessarily has the same zeros as \eqref{duallopatinsky}. 
This is sometimes called the ``polar non-radial'' method \cite{BHZ}.

To solve \eqref{polar_interior}
 near $H=1$, we use instead analytic expansion techniques from singular ODE theory; 
 such a combination of numerical ODE solution and singuar analytic expansion we refer to in general
 as a ``hybrid'' method.
 Namely, we make the solution ansatz $\Omega(H)=\sum_{i=0}^\infty \Omega_i(H-1)^i$ where $\Omega_0=\tilde{\mathcal{W}}(H=1)$ and $\Omega_i$ for $i\geq 1$ can be computed by the recurrence relation given by \eqref{polar_interior}. To make all entries of the coefficient matrices polynomials of $H$, multiply the equation \eqref{polar_interior} by the 
 common factor $(H^3H_R + H^3 - F^2H_R^2 + 2H^3\sqrt{H_R})H^2$ so that it becomes
 \be 
 Ld_H\Omega=(R_1+\lambda R_2+\eta R_3+\tilde{\mu}R_4)\Omega-\Omega\Omega^*(R_1+\lambda R_2+\eta R_3+\tilde{\mu}R_4)\Omega
 \ee 
 where
 $$
 \begin{aligned}
 (H^3H_R + H^3 - F^2H_R^2 + 2H^3\sqrt{H_R})H^2H'=&\sum_{i=1}^5L_i(H-1)^i=:L,\\
 (H^3H_R + H^3 - F^2H_R^2 + 2H^3\sqrt{H_R})H^2 (-EA^{-1})^T=&\sum_{i=0}^5R_{1i}(H-1)^i=:R_1,\\
 (H^3H_R + H^3 - F^2H_R^2 + 2H^3\sqrt{H_R})H^2 (A_0A^{-1})^T=&\sum_{i=0}^6R_{2i}(H-1)^i=:R_2,\\
 (H^3H_R + H^3 - F^2H_R^2 + 2H^3\sqrt{H_R})H^2 (iA_2A^{-1})^T=&\sum_{i=0}^6R_{3i}(H-1)^i=:R_3,\\
 -(H^3H_R + H^3 - F^2H_R^2 + 2H^3\sqrt{H_R})H^2=&\sum_{i=0}^5R_{4i}(H-1)^i=:R_4,
 \end{aligned}
 $$
 where $L_i$, $R_{1i}$, $R_{2i}$, $R_{3i}$, and $R_{4i}$ are given in Appendix \eqref{LRmat_hydraulic1}-\eqref{LRmat_hydraulic4}.
 
 At $\mathcal{O}(1)$ order, we obtain
 $$
 0=(R_{10}+\lambda R_{20}+\eta R_{30}+\tilde{\mu}R_{40})\Omega_0-\Omega_0\Omega_0^*(R_{10}+\lambda R_{20}+\eta R_{30}+\tilde{\mu}R_{40})\Omega_0
 $$
 Because $\Omega_0\in\ker((-G^T-\tilde{\mu}\Id)(H=1))$, $\Omega_0$ is then in $\ker(R_{10}+\lambda R_{20}+\eta R_{30}+\tilde{\mu}R_{40})$. The $\mathcal{O}(1)$ equation is satisfied.
 
For $n\geq 1$, at $\mathcal{O}((H-1)^n)$ order, we obtain
$$
\begin{aligned}
&nL_1\Omega_n-(R_{10}+\lambda R_{20}+\eta R_{30}+\tilde{\mu}R_{40})\Omega_n+\Omega_n\Omega_0^*(R_{10}+\lambda R_{20}+\eta R_{30}+\tilde{\mu}R_{40})\Omega_0\\
&+\Omega_0\Omega_n^*(R_{10}+\lambda R_{20}+\eta R_{30}+\tilde{\mu}R_{40})\Omega_0+\Omega_0\Omega_0^*(R_{10}+\lambda R_{20}+\eta R_{30}+\tilde{\mu}R_{40})\Omega_n=L.O.T.
\end{aligned}
$$
where $L.O.T.$ consists of lower order terms involving $\Omega_i$ for $i< n$. Having the term $\Omega_0\Omega_n^*(R_{10}+\lambda R_{20}+\eta R_{30}+\tilde{\mu}R_{40})\Omega_0$, the left hand side of the equation above seems to depend nonlinearly on $\Omega_n$. However, because $\Omega_0\in \ker(R_{10}+\lambda R_{20}+\eta R_{30}+\tilde{\mu}R_{40})$, the equation simplifies to
\be \label{Omega_n}
\left(nL_1\Id-(\Id-\Omega_0\Omega_0^*)(R_{10}+\lambda R_{20}+\eta R_{30}+\tilde{\mu}R_{40})\right)\Omega_n=L.O.T.
\ee 
For solvability of the linear system above, we readily compute by Sylvester's determinant theorem
\ba  \label{solvability_det}
&\det\left(nL_1\Id-(\Id-\Omega_0\Omega_0^*)(R_{10}+\lambda R_{20}+\eta R_{30}+\tilde{\mu}R_{40})\right)\\
=&\det\left(nL_1\Id-(R_{10}+\lambda R_{20}+\eta R_{30}+\tilde{\mu}R_{40})(\Id-\Omega_0\Omega_0^*)\right)\\
=&\det\left(nL_1\Id-(R_{10}+\lambda R_{20}+\eta R_{30}+\tilde{\mu}R_{40})\right).
\ea 
By direct computation, $L_1=F^2(\sqrt{H_R} + 1)(\sqrt{H_R} - 2H_R + 1)>0$, yielding that $nL_1$ is also positive. Because $\tilde{\mu}=\mu-2\mu_L-\mathrm{trace}(G)$ is the eigenvalue of $-G^T(H=1)$ with largest real part, the matrix $(R_{10}+\lambda R_{20}+\eta R_{30}+\tilde{\mu}R_{40})=H_R^2(F_{\text{exist}}^2-F^2)(-G^T(H=1)-\tilde{\mu}\Id)$ then has a zero eigenvalue and two eigenvalues with negative real part. In particular, it cannot have the positive number $nL_1$ as an eigenvalue. Therefore, the linear equation is solvable. Next, we study the convergence of the Taylor series $\Omega(H)=\sum_{i=0}^\infty\Omega_i(H-1)^i$, which validates that the Taylor series is a solution of \eqref{polar_interior} near $H=1$. To this end, the $L.O.T.$ explicitly read for $n\geq 6$,
\ba 
\label{LOT}
L.O.T=&-\sum_{i=2}^{5}(n+1-i)L_i\Omega_{n+1-i}+\sum_{i=1}^{5}(R_{1i}+\tilde{\mu}R_{4i})\Omega_{n-i}+\sum_{i=1}^{6}(\lambda R_{2i}+\eta R_{3i})\Omega_{n-i}\\
&-\sum_{i=0}^{5}\sum_{\substack{j+k+l=n,\\j<n,k<n,l<n}}\Omega_j\Omega_k^*(R_{1i}+\tilde{\mu}R_{4i})\Omega_{l}-\sum_{i=0}^{6}\sum_{\substack{j+k+l=n,\\j<n,k<n,l<n}}\Omega_j\Omega_k^*(\lambda R_{2i}+\eta R_{3i})\Omega_{l}
\ea 
For matrices $A\in \mathbb{C}^{m\times p}$ and $B\in \mathbb{C}^{p\times n}$, denote $||A||:=\max_{1\leq i\leq m,1\leq j\leq p}|A_{ij}|$, the maximum norm of $A$. Clearly, we have $||AB||\leq p ||A||\, ||B||$ Also, because $L_{i}$, $R_{1i}$, $R_{2i}$, $R_{3i}$, and $R_{4i}$ are constant matrices, we may assume
$$
||L_{i}||, ||R_{1i}||, ||R_{2i}||, ||R_{3i}||, ||R_{4i}||\leq C.
$$

\begin{lemma} \label{lemma:rig_lemma2}
Fix $q\in\mathbb{N}$, $q\geq  2$. there holds that $$\sum_{j=0}^{m}\frac{1}{(j+1)^q(m-j+1)^q}<\frac{2^{1+q}\zeta(q)}{(m+2)^q}\quad\text{and}\quad \sum_{j=1}^{m-1}\frac{1}{(j+1)^q(m-j+1)^q}<\frac{2^{1+q}\zeta(q)-2}{(m+2)^q}.$$ Here $\zeta(q)$ is the Riemann zeta function.
\end{lemma}
\begin{proof}
The inequalities follow from
$$\begin{aligned}
&\sum_{j=0}^{m}\frac{1}{(j+1)^q(m-j+1)^q}\\
\leq & 2\sum_{j=0}^{\lfloor m/2\rfloor} \frac{1}{(j+1)^q(m/2+1)^q}=\frac{2^{1+q}}{(m+2)^q} \sum_{j=0}^{\lfloor m/2\rfloor} \frac{1}{(j+1)^q}< \frac{2^{1+q}}{(m+2)^q}\sum_{j=0}^{\infty} \frac{1}{(j+1)^q}
\end{aligned}
$$
and
$$\begin{aligned}
&\sum_{j=1}^{m-1}\frac{1}{(j+1)^q(m-j+1)^q}\\
=&\sum_{j=0}^{m}\frac{1}{(j+1)^q(m-j+1)^q}-\frac{2}{(m+1)^q}<\frac{2^{1+q}\zeta(q)}{(m+2)^q}-\frac{2}{(m+1)^q}<\frac{2^{1+q}\zeta(q)-2}{(m+2)^q}.
\end{aligned}
$$
\end{proof}
\begin{theorem}
There exists a sufficiently large $\alpha>0$, such that $||\Omega_n||\leq \frac{\alpha^n}{(n+1)^2}$, for all $n\geq 1$. The radius of convergence of the Taylor series $\Omega(H)=\sum_{i=0}^\infty \Omega_i(H-1)^i$ is then at least $\frac{1}{\alpha}>0$. The Taylor series then defines a solution of the ODE \eqref{polar_interior} locally near $H=1$.
\end{theorem}
\begin{proof}
Let $N$ be a large integer to be specified later. For $n<N$, we can definitely choose $\alpha$ large enough so that $||\Omega_n||\leq \frac{\alpha^n}{(n+1)^2}$ holds. For $n\geq N$,
by \eqref{Omega_n},
$$
||\Omega_n||\leq 3||\left(nL_1\Id-(\Id-\Omega_0\Omega_0^*)(R_{10}+\lambda R_{20}+\eta R_{30}+\tilde{\mu}R_{40})\right)^{-1}||\,||L.O.T.||.
$$
We first obtain a bound for $||L.O.T.||$.
For the first three summations in \eqref{LOT}, the following estimate holds.
$$
\begin{aligned}
&||-\sum_{i=2}^{5}(n+1-i)L_i\Omega_{n+1-i}+\sum_{i=1}^{5}(R_{1i}+\tilde{\mu}R_{4i})\Omega_{n-i}+\sum_{i=1}^{6}(\lambda R_{2i}+\eta R_{3i})\Omega_{n-i}||\\
\le& \sum_{i=2}^5\frac{(n-1-i)C\alpha^{n+1-i}}{(n+2-i)^2}+\sum_{i=1}^5\frac{3C(1+|\tilde{\mu}|)\alpha^{n-i}}{(n+1-i)^2}+\sum_{i=1}^6\frac{3C(|\lambda|+|\eta|)\alpha^{n-i}}{(n+1-i)^2}\\
=&\frac{\alpha^{n}}{(n+1)^2}\left(\sum_{i=2}^5\frac{(n-1-i)C(n+1)^2}{\alpha^{i-1}(n+2-i)^2}+\sum_{i=1}^5\frac{3C(1+|\tilde{\mu}|)(n+1)^2}{\alpha^i(n+1-i)^2}+\sum_{i=1}^6\frac{3C(|\lambda|+|\eta|)(n+1)^2}{\alpha^i(n+1-i)^2}\right)\\
=:&\frac{\alpha^n}{(n+1)^2}\#_1.
\end{aligned}
$$
For the remaining two summations, we have 
$$
\begin{aligned}
&||\sum_{i=0}^{5}\sum_{\substack{j+k+l=n,\\j<n,k<n,l<n}}\Omega_j\Omega_k^*(R_{1i}+\tilde{\mu}R_{4i})\Omega_{l}||\\
\le&||\sum_{\substack{j+k+l=n,\\j<n,k<n,l<n}}\Omega_j\Omega_k^*(R_{10}+\tilde{\mu}R_{40})\Omega_{l}||+||\sum_{i=1}^{5}\sum_{j+k+l=n-i}\Omega_j\Omega_k^*(R_{1i}+\tilde{\mu}R_{4i})\Omega_{l}||\\
\leq &\sum_{\substack{j+k+l=n,\\j<n,k<n,l<n}} \frac{9C(1+|\tilde{\mu}|)\alpha^n}{(j+1)^2(k+1)^2(l+1)^2}+\sum_{i=1}^{5}\sum_{j+k+l=n-i}\frac{9C(1+|\tilde{\mu}|)\alpha^{n-i}}{(j+1)^2(k+1)^2(l+1)^2}\\
\leq & 9C(1+|\tilde{\mu}|)\alpha^n\left(\sum_{k=1}^{n-1}\frac{1}{(k+1)^2(n-k+1)^2}+\sum_{j=1}^{n-1}\frac{1}{(j+1)^2}\sum_{k+l=n-j}\frac{1}{(k+1)^2(l+1)^2}\right)\\
&+9C(1+|\tilde{\mu}|)\sum_{i=1}^{5}\alpha^{n-i}\sum_{j=0}^{n-i}\frac{1}{(j+1)^2}\sum_{k+l=n-i-j}\frac{1}{(k+1)^2(l+1)^2}
\end{aligned}
$$

$$
\begin{aligned}
\leq & 9C(1+|\tilde{\mu}|)\alpha^n\left(\frac{8\zeta(2)-2}{(n+2)^2}+\sum_{j=1}^{n-1}\frac{8\zeta(2)}{(j+1)^2(n-j+2)^2}\right)\\
&+9C(1+|\tilde{\mu}|)\sum_{i=1}^{5}\alpha^{n-i}\sum_{j=0}^{n-i}\frac{8\zeta(2)}{(j+1)^2(n-i-j+2)^2}\\
\leq &9C(1+|\tilde{\mu}|)\alpha^n\left(\frac{8\zeta(2)-2}{(n+2)^2}+8\zeta(2)(\frac{8\zeta(2)-2}{(n+3)^2}-\frac{1}{4(n+1)^2})\right)\\
&+72C(1+|\tilde{\mu}|)\zeta(2)\sum_{i=1}^{5}\alpha^{n-i}(\frac{8\zeta(2)}{(n-i+3)^2}-\frac{1}{(n-i+2)^2})\\
<&\frac{\alpha^n}{(n+1)^2}9C(1+|\tilde{\mu}|)\left((8\zeta(2)-2)+8\zeta(2)(8\zeta(2)-2)+8\zeta(2)\sum_{i=1}^5\frac{8\zeta(2)(n+1)^2}{\alpha^i(n-i+3)^2}\right)\\
=:&\frac{\alpha^n}{(n+1)^2}\#_2,
\end{aligned}
$$
and
$$\begin{aligned}
&||\sum_{i=0}^{6}\sum_{\substack{j+k+l=n,\\j<n,k<n,l<n}}\Omega_j\Omega_k^*(\lambda R_{2i}+\eta R_{3i})\Omega_{l}||\\
<&\frac{\alpha^n}{(n+1)^2}9C(|\lambda|+|\eta|)\left((8\zeta(2)-2)+8\zeta(2)(8\zeta(2)-2)+8\zeta(2)\sum_{i=1}^6\frac{8\zeta(2)(n+1)^2}{\alpha^i(n-i+3)^2}\right)\\
=:&\frac{\alpha^n}{(n+1)^2}\#_3.
\end{aligned}
$$
Secondly, to bound the maximum norm of $\left(nL_1\Id-(\Id-\Omega_0\Omega_0^*)(R_{10}+\lambda R_{20}+\eta R_{30}+\tilde{\mu}R_{40})\right)^{-1}$, we have
$$\begin{aligned}
&||\left(nL_1\Id-(\Id-\Omega_0\Omega_0^*)(R_{10}+\lambda R_{20}+\eta R_{30}+\tilde{\mu}R_{40})\right)^{-1}||\\
=&\frac{||\mathrm{adj}(nL_1\Id-(\Id-\Omega_0\Omega_0^*)(R_{10}+\lambda R_{20}+\eta R_{30}+\tilde{\mu}R_{40}))||}{|\det\left(nL_1\Id-(\Id-\Omega_0\Omega_0^*)(R_{10}+\lambda R_{20}+\eta R_{30}+\tilde{\mu}R_{40})\right)|}.
\end{aligned}
$$

The denominator, by \eqref{solvability_det}, is equal to the absolute value of the product of eigenvalues of $nL_1\Id-(R_{10}+\lambda R_{20}+\eta R_{30}+\tilde{\mu}R_{40})$. Again because $(R_{10}+\lambda R_{20}+\eta R_{30}+\tilde{\mu}R_{40})$ has a zero eigenvalue and two eigenvalues with negative real part and $L_1>0$, we have the denominator is greater than $(nL_1)^3$. To bound the numerator, for arbitrary $3$ by $3$ matrix $A$, we have $||\mathrm{adj}(A)||$ is the maximum of $|A_{ij}|$, taking over $1\leq i,j\leq 3$. Here, $A_{ij}$ denotes the $(i,j)$ cofactor of $A$. Obviously, $|A_{ij}|\leq 2||A||^2$ for $1\leq i,j\leq 3$, yielding $||\mathrm{adj}(A)||\leq 2||A||^2$. Therefore,
$$
\begin{aligned}
&||\mathrm{adj}(nL_1\Id-(\Id-\Omega_0\Omega_0^*)(R_{10}+\lambda R_{20}+\eta R_{30}+\tilde{\mu}R_{40}))||\\
\le& 2\big(n|L_1|+3(1+||\Omega_0||^2)C(1+|\lambda|+|\eta|+|\tilde{\mu}|)\big)^2
\end{aligned}
$$
and
$$\begin{aligned}
&||\left(nL_1\Id-(\Id-\Omega_0\Omega_0^*)(R_{10}+\lambda R_{20}+\eta R_{30}+\tilde{\mu}R_{40})\right)^{-1}||\\
\leq & \frac{2\big(n|L_1|+3(1+||\Omega_0||^2)C(1+|\lambda|+|\eta|+|\tilde{\mu}|)\big)^2}{(nL_1)^3}\\
\leq & \frac{2\big(n|L_1|+6C(1+|\lambda|+|\eta|+|\tilde{\mu}|)\big)^2}{(nL_1)^3}\end{aligned}
$$
in which the last inequality follows from $||\Omega_0||\leq 1$ given $\Omega_0^*\Omega_0=1$.
Combining, we have 
$$
||\Omega_n||\leq \frac{\alpha^n}{(n+1)^2}\frac{6\big(n|L_1|+6C(1+|\lambda|+|\eta|+|\tilde{\mu}|)\big)^2}{(nL_1)^3}\left(\#_1+\#_2+\#_3\right).
$$
It is then easy to see that we can choose $N$ large enough and $\alpha$ large enough so that when $n\geq N$  $$\frac{6\big(n|L_1|+6C(1+|\lambda|+|\eta|+|\tilde{\mu}|)\big)^2}{(nL_1)^3}\left(\#_1+\#_2+\#_3\right)\leq 1.$$
\end{proof}
Next, for a fixed $N$, we compute $\Omega(H_m):=\sum_{i=0}^N\Omega_i(H_m-1)^i$ for some $H_m$ between $1$ and $H_*$ so that a specified truncation error bound is met. Then, we continue to solve the ODE \eqref{polar_interior} from $H_m$ to $H_*$ with initial condition $\Omega(H_m)$ using ODE solvers, such as $\mathrm{ode45}$ in matlab. Finally, we compute the rescaled Evans-Lopatinsky determinant \eqref{rescaledduallopatinsky}.

For $\eta=0$, we count the number of point spectra in half annulus $\Omega(r,R):=\{\Re\lambda>0:r<|\lambda|<R\}$ by computing the winding number of mapped contours $\tilde{\Delta}(\partial{\Omega}(r,R),0)$. For $\eta\neq 0$, we count the number of point spectra in half disc $\Omega(R):=\{\Re\lambda>0:|\lambda|<R\}$ by computing the winding number of mapped contours $\tilde{\Delta}(\partial{\Omega}(R),\eta)$. See FIGURE~\ref{winding} for typical mapped contours. We verify numerically that there is no point spectrum in $\Omega(0.1,30)$ for $\eta=0$ and in $\Omega(30)$ for $\eta\in 0.2:0.2:6$ and for $7333$ combinations of parameters in the discretized existence domain $H_R\in 0.01:0.01:0.99$ and $F\in  0.02\times\mathbb{N}^+ \cap [\sqrt{H_R}+H_R,F_{\text{2d}})$. 
\begin{figure}[htbp]
    \begin{center}
    \includegraphics[scale=0.42]{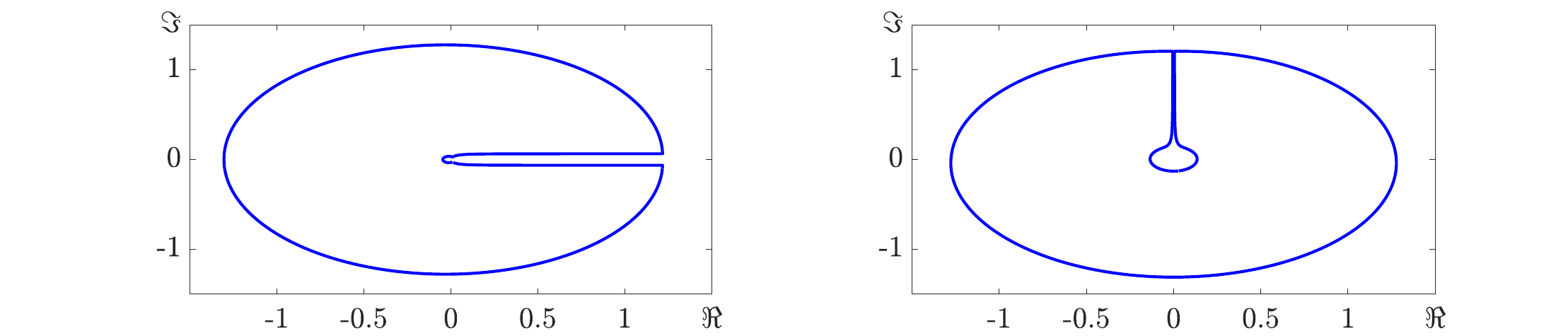}
    \end{center}
    \caption{Images of contours under rescaled Evans Lopatinsky determinant $\tilde{\Delta}(\lambda,\eta)$. Left panel: $F=2.05$, $H_R=0.7$, $\lambda\in \partial\Omega(0.1,5)$, and $\eta=0$  Right panel: $F=2.05$, $H_R=0.7$, $\lambda\in \partial\Omega(5)$, and $\eta=3$. Winding numbers are both zeros.}
    \label{winding}
\end{figure}

\noindent{\bf Smooth hydraulic shocks.} For smooth hydraulic shocks, we shall verify nonvanishing of the Evans function $\Delta(\lambda,\eta)$ \eqref{Evans} for $\eta\in \mathbb{R}$ and $\Re \lambda\geq 0$ in the right half plane $\{\Re\lambda\geq 0\}$, except for $(\lambda,\eta)=(0,0)$. 
For the computation of $\mathcal{W}_1^-(0^-;\lambda,\eta)\wedge \mathcal{W}_2^-(0^-;\lambda,\eta)$, we can again make use of the dual formulation, switch to $H$-coordinates, and apply the polar non-radial method and hybrid method. For the computation of $\mathcal{W}^+_1(0^+;\lambda,\eta)$, we can switch to $H$-coordinates and use polar non-radial method and hybrid method. Indeed, what we compute in the end is a function necessarily having the same zeros as the Evans function. The function is
\be 
\tilde{\Delta}(\lambda,\eta):=\sum_{j=1}^3\Omega^+_j(\frac{1+H_R}{2};\lambda,\eta)\Omega^-_j(\frac{1+H_R}{2};\lambda,\eta),
\ee 
where $\Omega^+(H;\lambda,\eta)=[\Omega^+_1,\,\Omega^+_2\,\Omega^+_3]^T(H;\lambda,\eta)$ solves $$H'd_H\Omega=(G(H,\lambda,\eta)-\mu\Id)\Omega-\Omega\Omega^*(G(H,\lambda,\eta)-\mu\Id)\Omega,$$ where $\mu$ is the smallest eigenvalue of $G(H_R,\lambda,\eta)$ and $\Omega^+(H_R;\lambda,\eta)$ is a corresponding unit eigenvector and $\Omega^-(H;\lambda,\eta)=\left[\begin{array}{rrr}\Omega^-_1&\Omega^-_2&\Omega^-_3\end{array}\right]^T(H;\lambda,\eta)$ solves the dual system $$H'd_H\Omega=(-G^T(H,\lambda,\eta)-\tilde{\mu}\Id)\Omega+\Omega\Omega^*(G^T(H,\lambda,\eta)+\tilde{\mu}\Id)\Omega,$$ where $\tilde{\mu}$ is the largest eigenvalue of $-G^T(H=H_L=1,\lambda,\eta)$ and $\Omega^-(H=H_L=1;\lambda,\eta)$ is a corresponding unit eigenvector. Here for the convenience, we examine the connection at $H=\frac{1+H_R}{2}$. 

We verify numerically that there is no point spectrum in $\Omega(0.1,30)$ for $\eta=0$ and in $\Omega(30)$ for $\eta\in 0.2:0.2:6$ and for $5732$ combinations of parameters in the discretized existence domain $H_R\in 0.01:0.01:0.99$ and $F\in 0.02\times\mathbb{N}^+ \cap (0,\sqrt{H_R}+H_R)$ except for the parameters $(H_R=0.98,F=0.02)$, $(H_R=0.99,F=0.04)$, and $(H_R=0.99,F=0.04)$. Here, note that we decide not to investigate for these three waves for it is extremely time-consuming to compute the Evans function for small amplitude smooth profile with small $F$ value as also noted in the study of multi-dimensional Navier-Stokes shocks by Humpherys, Lyng, and Zumbrun \cite{HLyZ}.

\smallskip

\noindent{\bf Computational environment.} In carrying out our numerical investigations, we have used MacBook Pro 2017 with Intel Core i7-7700HQ processor and a workstation with Intel Core i9-13900K for coding and debugging. The main parallelized computation is done in the compute nodes of IU Quartz, a high-throughput computing cluster featuring 92 compute nodes, each equipped with two 64-core AMD EPYC 7742 2.25 GHz CPUs and 512 GB of RAM. 

\smallskip

\noindent{\bf Computational time.} In the rescaled Evans-Lopatinsky condition study, for $\eta=0$, we discretize the smaller half circle with radius $0.1$ equally into 200 pieces, the bigger half circle with radius $30$ equally into $1000$ pieces, and the two straight line segments $[0.1,30]\times i$ and $[-30,0.1]\times i$ each equally into 300 pieces and for $\eta\in 0.2:0.2:6$, we discretize the half circle with radius $30$ equally into 1000 pieces and the part $[0.1,30]\times i$ and $[-30,0.1]\times i$ on the straight line each equally into 300 pieces and the part $[-0.1,0.1]\times i$ on the straight line equally into 200 pieces. Here the centered line segment $[-0.1,0.1]$ is discretized into finer pieces to reduce argument changes when $(\lambda,\eta)$ is close to the zeros $(0,0)$. Therefore, for a single combination of parameter $(F,H_R)$, we evaluate the rescaled Evans-Lopatinsky determinant $31\times(200+1000+300\times 2)=55800$ times. For the total $7333$ combinations of parameters, we evaluate the rescaled Evans-Lopatinsky determinant $409181400$ times. To complete the computation task, we used Slurm to submit and manage jobs on IU's Quartz. Each job was run with a single Quartz's AMD EPYC 7742 2.25 GHz CPU and took all of its 64 cores. The total computation time on a single Quartz's 64-core AMD EPYC 7742 2.25 GHz CPU is 2506 hours and 35 mins or about 105 days. We submited 15 jobs running at the same time, which reduced our waiting time to 7 days.

In the rescaled Evans condition study, for $\eta=0$, we discretize the smaller half circle with radius $0.1$ equally into 200 pieces, the bigger half circle with radius $30$ equally into $1000$ pieces, and the two straight line segments $[0.1,30]\times i$ and $[-30,0.1]\times i$ each equally into 200 pieces and for $\eta\in 0.2:0.2:6$, we discretize the half circle with radius $30$ equally into 1000 pieces and the part $[0.1,30]\times i$ and $[-30,0.1]\times i$ on the straight line each equally into 200 pieces and the part $[-0.1,0.1]\times i$ on the straight line equally into 200 pieces. Here the centered line segment $[-0.1,0.1]$ is discretized into finer pieces to reduce argument changes when $(\lambda,\eta)$ is close to the zeros $(0,0)$. Therefore, for a single combination of parameter $(F,H_R)$, we evaluate the rescaled Evans function $31\times(200+1000+200\times 2)=49600$ times. For the total $5732$ combinations of parameters, we evaluate the rescaled Evans function $284307200$ times. The total computation time on a single Quartz's 64-core AMD EPYC 7742 2.25 GHz CPU is 4870 hours and 40 mins or about 203 days. We submited 23 jobs running at the same time, which reduced our waiting time to less than 9 days. As we mention above, we decide not to verify for three parameters due to a time-consumption concern. Indeed, in the $23^{nd}$ job, it took more than 4 days to examine for the parameter $H_R=0.98$ and $F=0.02$. To compare, the computation time for the parameter $H_R=0.97$ and $F=0.02$ is 47.5 hours. In the left panel of FIGURE~\ref{computation_time}, we plot computation time versus $H_R\in 0.01:0.01:0.97$ for $F=0.02$. In the right panel of FIGURE~\ref{computation_time}, we plot computation time versus $F\in 0.06:0.02:1.98$ for $H_R=0.99$. Clearly, the figure shows that it is time-expensive to verify stability for small amplitude smooth hydraulic shocks with small $F$ values.
\begin{figure}
\begin{center}
\includegraphics[scale=0.42]{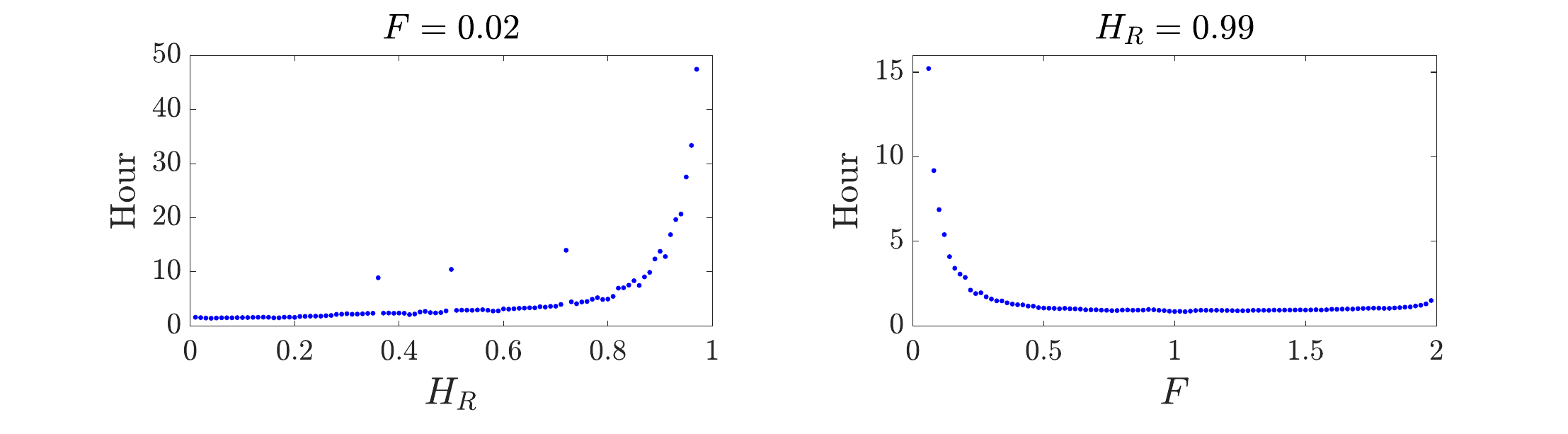}
\end{center}
\vspace{-0.3cm}
\caption{Left: computation time versus $H_R\in 0.01:0.01:0.97$ for $F=0.02$. Right: computation time versus $F\in 0.06:0.02:1.98$ for $H_R=0.99$.}
\label{computation_time}
\end{figure}
\section{Additional formulae}
Formulae of $L^{(n)}$, $R^{(1,n)}$, $R^{(2,n)}$, and $R^{(3,n)}$ are given by
\ba \label{LRmat1}
&L^{(0)}=\left[\begin{array}{rrr}1 & 0 & 0\\ 0 & 1 & 0\\ 0 & 0 & 0\end{array}\right],\quad L^{(1)}=\left[\begin{array}{rrr} \frac{4F^2-6F-1}{F \left(F-2\right)} & 0 & 0\\ 0 & -\frac{-3 F^2+4 F+1}{F \left(F-2\right)} & 0\\ -\frac{2 F-1}{F^2} & 0 & \frac{3}{F \left(F+1\right)} \end{array}\right],\\
&L^{(2)}=\left[\begin{array}{rrr} -\frac{2 \left(-3 F^2+3 F+1\right)}{F \left(F-2\right)} & 0 & 0\\ 0 & -\frac{-3 F^2+2 F+1}{F \left(F-2\right)} & 0\\ -\frac{12 F^3-22 F^2+4 F+2}{2 F^3 \left(F-2\right)} & 0 & -\frac{-18 F^2+24 F+6}{2 F^2 \left(F+1\right) \left(F-2\right)} \end{array}\right],\\
&L^{(3)}=\left[\begin{array}{rrr} -\frac{-4 F^2+2 F+1}{F \left(F-2\right)} & 0 & 0\\ 0 & \frac{F}{F-2} & 0\\ -\frac{36 F^3-48 F^2+12 F+6}{6 F^3 \left(F-2\right)} & 0 & -\frac{-60 F^2+48 F+18}{6 F^2 \left(F+1\right) \left(F-2\right)} \end{array}\right],\\
&L^{(4)}=\left[\begin{array}{rrr} \frac{F}{F-2} & 0 & 0\\ 0 & 0 & 0\\ -\frac{48 F^3-72 F^2+48 F+24}{24 F^3 \left(F-2\right)} & 0 & -\frac{-120 F^2+48 F+24}{24 F^2 \left(F+1\right) \left(F-2\right)} \end{array}\right],\\
&L^{(5)}=\left[\begin{array}{rrr} 0 & 0 & 0\\ 0 & 0 & 0\\ \frac{1}{F \left(F-2\right)} & 0 & \frac{1}{\left(F+1\right) \left(F-2\right)} \end{array}\right],\quad R^{(1,0)}=\left[\begin{array}{rrr} 0 & 0 & 0\\ 0 & \frac{F+1}{F-2} & 0\\ \frac{2 \left(F^2+2 F+1\right)}{F^2 \left(F-2\right)} & 0 & 0 \end{array}\right],\\
&R^{(1,1)}=\left[\begin{array}{rrr} 0 & 0 & 0\\ 0 & \frac{2 \left(F^2+2 F+1\right)}{F \left(F-2\right)} & 0\\ \frac{4 F^3+15 F^2+12 F+1}{F^3 \left(F-2\right)} & 0 & \frac{3 \left(F-1\right)}{F^2 \left(F-2\right)} \end{array}\right],\\
&R^{(1,2)}=\left[\begin{array}{rrr} 0 & 0 & 0\\ 0 & \frac{F^2+4 F+3}{F \left(F-2\right)} & 0\\ \frac{2 \left(F^3+7 F^2+9 F+3\right)}{F^3 \left(F-2\right)} & 0 & \frac{4}{F \left(F-2\right)} \end{array}\right],\\
&R^{(1,3)}=\left[\begin{array}{rrr} 0 & 0 & 0\\ 0 & \frac{F+1}{F \left(F-2\right)} & 0\\ \frac{3 \left(F^2+2 F+1\right)}{F^3 \left(F-2\right)} & 0 & \frac{F+1}{F^2 \left(F-2\right)} \end{array}\right],
\ea 
\ba\label{LRmat2}
&R^{(2,0)}=\left[\begin{array}{ccc} \frac{3}{(F-2)(F+1)} & 0 & \frac{3 F}{(F-2)(F+1)^2}\\ 0 & \frac{3}{F-2} & 0\\ \frac{3 \left(F-1\right)}{F^2 \left(F-2\right)} & 0 & -\frac{6}{F \left(F+1\right) \left(F-2\right)} \end{array}\right],\\
&R^{(2,1)}=\left[\begin{array}{ccc} \frac{9}{(F-2)(F+1)} & 0 & \frac{9 F}{(F-2)(F+1)^2}\\ 0 & \frac{9}{F-2} & 0\\ \frac{9 \left(F-1\right)}{F^2 \left(F-2\right)} & 0 & -\frac{18}{F \left(F+1\right) \left(F-2\right)} \end{array}\right],\\
&R^{(2,2)}=\left[\begin{array}{ccc} \frac{10}{(F-2)(F+1)} & 0 & \frac{10 F}{(F-2)(F+1)^2}\\ 0 & \frac{10}{F-2} & 0\\ \frac{10 \left(F-1\right)}{F^2 \left(F-2\right)} & 0 & -\frac{20}{F \left(F+1\right) \left(F-2\right)} \end{array}\right],\\
&R^{(2,3)}=\left[\begin{array}{ccc} \frac{5}{(F-2)(F+1)} & 0 & \frac{5 F}{(F-2)(F+1)^2}\\ 0 & \frac{5}{F-2} & 0\\ \frac{5 \left(F-1\right)}{F^2 \left(F-2\right)} & 0 & -\frac{10}{F \left(F+1\right) \left(F-2\right)} \end{array}\right],\\
&R^{(2,4)}=\left[\begin{array}{ccc} \frac{1}{(F-2)(F+1)} & 0 & \frac{F}{(F-2)(F+1)^2}\\ 0 & \frac{1}{F-2} & 0\\ \frac{F-1}{F^2 \left(F-2\right)} & 0 & -\frac{2}{F \left(F+1\right) \left(F-2\right)} \end{array}\right],\\
&R^{(3,0)}=i\left[\begin{array}{ccc} 0 & \frac{3}{(F-2)(F+1)} & 0\\ \frac{3}{F^2 \left(F-2\right)} & 0 & \frac{3 }{F \left(F+1\right) \left(F-2\right)}\\ 0 & -\frac{3 }{F^2 \left(F-2\right)} & 0 \end{array}\right],
\ea 
and
\ba
\label{LRmat3}
&R^{(3,1)}=i\left[\begin{array}{ccc} 0 & \frac{9 }{(F-2)(F+1)} & 0\\ \frac{12 }{F^2 \left(F-2\right)} & 0 & \frac{12 }{F \left(F+1\right) \left(F-2\right)}\\ 0 & -\frac{12 }{F^2 \left(F-2\right)} & 0 \end{array}\right],\\
&R^{(3,2)}=i\left[\begin{array}{ccc} 0 & \frac{10 }{(F-2)(F+1)} & 0\\ \frac{19 }{F^2 \left(F-2\right)} & 0 & \frac{19 }{F \left(F+1\right) \left(F-2\right)}\\ 0 & -\frac{16 }{F^2 \left(F-2\right)} & 0 \end{array}\right],\\
&R^{(3,3)}=i\left[\begin{array}{ccc} 0 & \frac{5 }{(F-2)(F+1)} & 0\\ \frac{15 }{F^2 \left(F-2\right)} & 0 & \frac{15 }{F \left(F+1\right) \left(F-2\right)}\\ 0 & -\frac{9 }{F^2 \left(F-2\right)} & 0 \end{array}\right],\\
& R^{(3,4)}=i\left[\begin{array}{ccc} 0 & \frac{1 }{(F-2)(F+1)} & 0\\ \frac{6 }{F^2 \left(F-2\right)} & 0 & \frac{6 }{F \left(F+1\right) \left(F-2\right)}\\ 0 & -\frac{2 }{F^2 \left(F-2\right)} & 0 \end{array}\right],\\
&R^{(3,5)}=i\left[\begin{array}{ccc} 0 & 0 & 0\\ \frac{1 }{F^2 \left(F-2\right)} & 0 & \frac{1 }{F \left(F+1\right) \left(F-2\right)}\\ 0 & 0 & 0 \end{array}\right].
\ea 

Formulae of $L_i$, $R_{4i}$, $R_{1i}$, $R_{2i}$, and $R_{3i}$ are given by
\ba \label{LRmat_hydraulic1}
&L_1=\frac{F^2(\nu + 1)(\nu^2 + \nu - 2)}{\nu^3}, \quad L_2=-\frac{F^2(- 4\nu^4 - 8\nu^3 + 2\nu^2 + 6\nu + 1)}{\nu^4},\\
&L_3=-\frac{2F^2(- 3\nu^4 - 6\nu^3 + 3\nu + 1)}{\nu^4},\quad L_4=\frac{F^2( 4\nu^4 +8\nu^3 + 2\nu^2 - 2\nu - 1)}{\nu^4},\\
&L_5=\frac{F^2(\nu + 1)^2}{\nu^2}, \quad R_{40}=-\frac{- F^2 + \nu^4 + 2\nu^3 + \nu^2}{\nu^4},\\
&R_{41}=-\frac{- 2F^2 + 5\nu^4 + 10\nu^3 + 5\nu^2}{\nu^4},\quad R_{42}=-\frac{- F^2 + 10\nu^4 + 20\nu^3 + 10\nu^2}{\nu^4},\\
&R_{43}=-\frac{10(\nu^2 + 2\nu + 1)}{\nu^2},\quad R_{44}=-\frac{5(\nu^2 + 2\nu + 1)}{\nu^2},\quad R_{45}=-\frac{\nu^2 + 2\nu + 1}{\nu^2},\\
&R_{10}=\left[\begin{array}{ccc} 0 & \frac{\left(\nu +1\right) \left(F^2(\nu^2+\nu -3)+2 \nu ^2+2 \nu \right)}{\nu ^3} & 0\\ 0 & \frac{F^2 \left(-\nu ^3-2 \nu ^2+\nu +2\right)}{\nu ^3} & 0\\ 0 & 0 & -\frac{\left(\nu +1\right) \left(-F^2+\nu ^4+2 \nu ^3+\nu ^2\right)}{\nu ^3} \end{array}\right],\\
&R_{11}=\left[\begin{array}{ccc} 0 & \frac{2 \left(F^2 (2\nu ^4+4\nu ^3- \nu ^2-3\nu -2)+4 \nu ^4+8 \nu ^3+5 \nu ^2+\nu \right)}{\nu ^4} & 0\\ 0 & \frac{2 F^2 \left(-2 \nu ^4-4 \nu ^3+2 \nu +1\right)}{\nu ^4} & 0\\ 0 & 0 & \frac{F^2(2\nu^2+2\nu+1)-\nu^2(\nu + 1)^2(5\nu^2 + 5\nu + 1)}{\nu ^4} \end{array}\right],\\
&R_{12}=\left[\begin{array}{ccc} 0 & \frac{2 \left(F^2(3 \nu ^6+9 \nu ^5+8 \nu ^4+\nu ^3-3 \nu ^2-2\nu -1)+3\nu^2(\nu + 1)^2(2\nu^2 + 2\nu + 1)\right)}{\nu ^5 \left(\nu +1\right)} \\ 0 & -\frac{2 F^2 \left(3 \nu ^4+6 \nu ^3+2 \nu ^2-\nu -1\right)}{\nu ^4} \\ 0 & 0 \end{array}\right.\\
&\quad\quad\quad \left.\begin{array}{c}0\\0\\ \frac{F^2(\nu ^2+\nu +1)-2\nu^2(\nu + 1)^2(5\nu^2 + 5\nu + 2)}{\nu ^4} \end{array}\right],\\
&R_{13}=\left[\begin{array}{ccc} 0 & \frac{2 \left(\nu +1\right) \left(F^2(2 \nu ^2+2  \nu +1)+4 \nu ^2+4 \nu +3\right)}{\nu ^3} & 0\\ 0 & -\frac{4 F^2 {\left(\nu +1\right)}^2}{\nu ^2} & 0\\ 0 & 0 & -\frac{2 {\left(\nu +1\right)}^2 \left(5 \nu ^2+5 \nu +3\right)}{\nu ^2} \end{array}\right],\\
&R_{14}=\left[\begin{array}{ccc} 0 & \frac{\left(F^2+2\right) \left(\nu ^3+2 \nu ^2+2 \nu +1\right)}{\nu ^3} & 0\\ 0 & -\frac{F^2 {\left(\nu +1\right)}^2}{\nu ^2} & 0\\ 0 & 0 & -\frac{{\left(\nu +1\right)}^2 \left(5 \nu ^2+5 \nu +4\right)}{\nu ^2} \end{array}\right],\\
&R_{15}=\left[\begin{array}{ccc} 0 & 0 & 0\\ 0 & 0 & 0\\ 0 & 0 & -\frac{{\left(\nu +1\right)}^2 \left(\nu ^2+\nu +1\right)}{\nu ^2} \end{array}\right],\\
&R_{20}=\left[\begin{array}{ccc} -\frac{F^2 \left(\nu +1\right) \left(\nu ^2+\nu -1\right)}{\nu ^3} & -\frac{\left(F^2-1\right) {\left(\nu +1\right)}^2}{\nu ^2} & 0\\ \frac{F^2 {\left(\nu +1\right)}^2}{\nu ^2} & \frac{F^2 \left(\nu +1\right) \left(\nu ^2+\nu +1\right)}{\nu ^3} & 0\\ 0 & 0 & -\frac{\left(\nu +1\right) \left(-F^2+\nu ^4+2 \nu ^3+\nu ^2\right)}{\nu ^3} \end{array}\right],
\ea 

\ba \label{LRmat_hydraulic2}
&R_{21}=\left[\begin{array}{ccc} -\frac{2 F^2 \left(2 \nu ^3+4 \nu ^2+\nu -1\right)}{\nu ^3} & -\frac{\left(\nu +1\right) \left(F^2 (4\nu ^2+4 \nu +2)-5 \nu ^2-5 \nu \right)}{\nu ^3} & 0\\ \frac{4 F^2 {\left(\nu +1\right)}^2}{\nu ^2} & \frac{4 F^2 \left(\nu +1\right) \left(\nu ^2+\nu +1\right)}{\nu ^3} & 0\\ 0 & 0 & -\frac{3 \left(\nu +1\right) \left(-F^2+2 \nu ^4+4 \nu ^3+2 \nu ^2\right)}{\nu ^3} \end{array}\right],\\
&R_{22}=\left[\begin{array}{ccc} -\frac{6 F^2 {\left(\nu +1\right)}^2}{\nu ^2} & -\frac{F^2 (6\nu ^4+12 \nu ^3+12 \nu ^2+6  \nu +1)-10 \nu ^4-20 \nu ^3-10 \nu ^2}{\nu ^4} & 0\\ \frac{6 F^2 {\left(\nu +1\right)}^2}{\nu ^2} & \frac{6 F^2 \left(\nu +1\right) \left(\nu ^2+\nu +1\right)}{\nu ^3} & 0\\ 0 & 0 & -\frac{3 \left(\nu +1\right) \left(-F^2+5 \nu ^4+10 \nu ^3+5 \nu ^2\right)}{\nu ^3} \end{array}\right],\\
\\&R_{23}=\left[\begin{array}{ccc} -\frac{2 F^2 \left(2 \nu ^3+4 \nu ^2+3 \nu +1\right)}{\nu ^3} & -\frac{2 \left(F^2(2 \nu ^4+4 \nu ^3+5 \nu ^2+3 \nu +1)-5 \nu ^2(\nu+1)^2\right)}{\nu ^4} & 0\\ \frac{4 F^2 {\left(\nu +1\right)}^2}{\nu ^2} & \frac{4 F^2 \left(\nu +1\right) \left(\nu ^2+\nu +1\right)}{\nu ^3} & 0\\ 0 & 0 & -\frac{\left(\nu +1\right) \left(-F^2+20 \nu^2(\nu+1)^2\right)}{\nu ^3} \end{array}\right],\\
\\&R_{24}=\left[\begin{array}{ccc} -\frac{F^2 \left(\nu +1\right) \left(\nu ^2+\nu +1\right)}{\nu ^3} & -\frac{F^2 (\nu^2+\nu+1)^2-5 \nu ^2(\nu+1)^2}{\nu ^4} & 0\\ \frac{F^2 {\left(\nu +1\right)}^2}{\nu ^2} & \frac{F^2 \left(\nu +1\right) \left(\nu ^2+\nu +1\right)}{\nu ^3} & 0\\ 0 & 0 & -\frac{15 {\left(\nu +1\right)}^3}{\nu } \end{array}\right],\\
&R_{25}=\left[\begin{array}{ccc} 0 & \frac{{\left(\nu +1\right)}^2}{\nu ^2} & 0\\ 0 & 0 & 0\\ 0 & 0 & -\frac{6 {\left(\nu +1\right)}^3}{\nu } \end{array}\right],\quad R_{26}=\left[\begin{array}{ccc} 0 & 0 & 0\\ 0 & 0 & 0\\ 0 & 0 & -\frac{{\left(\nu +1\right)}^3}{\nu } \end{array}\right],
\ea 
and
\ba \label{LRmat_hydraulic3}
&R_{30}=i\left[\begin{array}{ccc} 0 & 0 & -\frac{\left(\nu +1\right) \left(\nu ^2+\nu -1\right)}{\nu ^3}\\ 0 & 0 & \frac{{\left(\nu +1\right)}^2}{\nu ^2}\\ \frac{\left(\nu +1\right) \left(F^2-\nu ^2(\nu+1)^2\right)}{\nu ^3} & \frac{\left(\nu +1\right) \left(F^2-\nu ^2(\nu+1)^2\right)}{\nu ^3} & 0 \end{array}\right],\\
&R_{31}=i\left[\begin{array}{ccc} 0 & 0 & -\frac{5 \nu ^3+10 \nu ^2+2 \nu -3}{\nu ^3}\\ 0 & 0 & \frac{5 {\left(\nu +1\right)}^2}{\nu ^2}\\ \frac{3 \left(\nu +1\right) \left(F^2-2 \nu ^2(\nu+1)^2\right)}{\nu ^3} & \frac{F^2 (3\nu ^2+3 \nu +1)-\nu^2(\nu + 1)^2(6\nu^2 + 6\nu + 1)}{\nu ^4} & 0 \end{array}\right],\\
&R_{32}=i\left[\begin{array}{ccc} 0 & 0 & -\frac{10 \nu ^3+20 \nu ^2+8 \nu -2}{\nu ^3}\\ 0 & 0 & \frac{10 {\left(\nu +1\right)}^2}{\nu ^2}\\ \frac{3 \left(\nu +1\right) \left(F^2-5\nu^2( \nu+1)^2\right)}{\nu ^3} & \frac{F^2(3 \nu ^2+3 \nu +2 )-5\nu^2(\nu + 1)^2(3\nu^2 + 3\nu + 1)}{\nu ^4} & 0 \end{array}\right],\\
&R_{33}=i\left[\begin{array}{ccc} 0 & 0 & -\frac{10 \nu ^3+20 \nu ^2+12 \nu +2}{\nu ^3}\\ 0 & 0 & \frac{10 {\left(\nu +1\right)}^2}{\nu ^2}\\ \frac{\left(\nu +1\right) \left(F^2-20 \nu ^2(\nu+1)^2\right)}{\nu ^3} & \frac{F^2 (\nu ^2+\nu +1)-10\nu^2(\nu + 1)^2(2\nu^2 + 2\nu + 1)}{\nu ^4} & 0 \end{array}\right],
\ea 

\ba \label{LRmat_hydraulic4}
&R_{34}=i\left[\begin{array}{ccc} 0 & 0 & -\frac{5 \nu ^3+10 \nu ^2+8 \nu +3}{\nu ^3}\\ 0 & 0 & \frac{5 {\left(\nu +1\right)}^2}{\nu ^2}\\ -\frac{15 {\left(\nu +1\right)}^3}{\nu } & -\frac{5 {\left(\nu +1\right)}^2 \left(3 \nu ^2+3 \nu +2\right)}{\nu ^2} & 0 \end{array}\right],\\
&R_{35}=i\left[\begin{array}{ccc} 0 & 0 & -\frac{\left(\nu +1\right) \left(\nu ^2+\nu +1\right)}{\nu ^3}\\ 0 & 0 & \frac{{\left(\nu +1\right)}^2}{\nu ^2}\\ -\frac{6 {\left(\nu +1\right)}^3}{\nu } & -\frac{{\left(\nu +1\right)}^2 \left(6 \nu ^2+6 \nu +5\right)}{\nu ^2} & 0 \end{array}\right],\\
&R_{36}=i\left[\begin{array}{ccc} 0 & 0 & 0\\ 0 & 0 & 0\\ -\frac{{\left(\nu +1\right)}^3}{\nu } & -\frac{{\left(\nu +1\right)}^2 \left(\nu ^2+\nu +1\right)}{\nu ^2} & 0 \end{array}\right].
\ea

\section{Comparison with Brock's experimental results}\label{s:brock}

In this section, we make a comparison between our analytical spectral stability results and Brock's experimental results. We begin with finding a scaling such that the shallow water wave model used in Brock's dissertation \cite{Brock} becomes the Saint-Venant equations considered in \cite{JNRYZ} and this piece of work. 

\smallskip

\textbf{Shallow water model used in Brock's dissertation.} In chapters II and III of Brock's dissertation, Brock discussed about the basic differential equations that had been used in the analytical study of shallow water flow. Casting the variables in dimensionless form, Brock obtained, in his dissertation,  equations (3.8) and (3.9), which read
\ba
\label{brock-eq}
H_{t'}+(UH)_{x'}&=0,\\
U_{t'}+UU_{x'}+\frac{1}{F^2}H_{x'}&=\frac{S_0\lambda}{F^2h_n}\left(1-\frac{U^2}{H}\right),
\ea
where 
\ba  
\label{dimensionless_brock}
&x'=\frac{x}{\lambda}, \quad \lambda=\text{wave length,}\quad t'=\frac{u_n t}{\lambda}, \quad u_n=\text{normal velocity $q/h_n$,}\\
&H=\frac{h}{h_n}, \quad h_n=\text{normal depth,}\quad U=\frac{u}{u_n}, \\
&F=\frac{u_n}{\sqrt{gh_n}} \;\;\text{is the Froude number of Brock's,}\\
&g=32.16 \;\mathrm{ft/sec^2}\;\;\text{ is the gravitational constant, and} \\
&S_0=\sin(\theta), \theta=\text{angle of inclination of channel.}
\ea     

In addition to these symbols, Brock used $c$ for wave velocity and $T$ for time period of the wave with a dimensionless version 
\be \label{tprime}
T'=S_0T\sqrt{\frac{g}{h_n}}.
\ee
We pause to remark that the $T'$ above, recorded in Brock's Table 6 (see also Table \ref{brock_table} below), does not stand for the wave period in \eqref{brock-eq}. The dimensionless time period, by $t'=\frac{u_n t}{\lambda}$, is
\be \label{dimensionless_period}
T'_{period}=\frac{u_n T}{\lambda}=\frac{F\sqrt{gh_n} T}{\lambda}.
\ee

\smallskip

\textbf{The Saint Venant equations used in \cite{JNRYZ}.} Our convention is to use the dimensionless Saint Venant equations \eqref{sv_intro} and also their 1d version \cite[eq. (1.1)]{JNRYZ}. To distinguish variables in the two systems, we add tildes to the dimensionless variables in the 1d Saint Venant equations as shown below. 
\ba
\label{sv-1d}
\tilde{h}_{\tilde{t}}+(\tilde{u}\tilde{h})_{\tilde{x}}&=0,\\
(\tilde{h}\tilde{u})_{\tilde{t}}+\left(\tilde{h}\tilde{u}^2+\frac{\tilde{h}^2}{2\tilde{F}^2}\right)_{\tilde{x}}&=\tilde{h}-\tilde{u}^2,
\ea
where $\tilde{F}=\frac{U_0}{\sqrt{H_0 g\cos(\theta)}}$ is the Froude number of ours. Here, $U_0$ and $H_0$ are some reference speed and height, $g$ is gravitation constant, and $\theta$ is the angle of inclination of the channel. In the classical sense of solutions, the second equation of \eqref{sv-1d} can be put as 
\be
\label{sv_1d_2}
\tilde{u}_{\tilde{t}}+\tilde{u}\tilde{u}_{\tilde{x}}+\frac{1}{\tilde{F}^2}\tilde{h}_{\tilde{x}}=1-\frac{\tilde{u}^2}{\tilde{h}}.
\ee
The equation \eqref{sv_1d_2} looks the same as \eqref{brock-eq}[ii], except that the right hand side of \eqref{brock-eq} is multiplied by an additional number $\frac{S_0\lambda}{F^2h_n}$. We now look for a re-scaling 
\be  
\label{scaling}
H=h_0\tilde{h},\quad U=u_0\tilde{u},\quad t'=t_0\tilde{t},\quad\text{and}\quad x'=x_0\tilde{x},
\ee  
under which the system \eqref{brock-eq} is equivalent to \eqref{sv-1d}[i]-\eqref{sv_1d_2}. Substituting \eqref{scaling} into \eqref{brock-eq} yields
\ba 
\label{brock-eq2}
\tilde{h}_{\tilde{t}}+\frac{u_0t_0}{x_0}(\tilde{u}\tilde{h})_{\tilde{x}}&=0,\\
\tilde{u}_{\tilde{t}}+\frac{u_0t_0}{x_0}\tilde{u}\tilde{u}_{\tilde{x}}+\frac{h_0t_0}{u_0x_0}\frac{1}{F^2}\tilde{h}_{\tilde{x}}&=\frac{t_0\sin(\theta)\lambda}{u_0F^2h_n}\left(1-\frac{u_0^2}{h_0}\frac{\tilde{u}^2}{\tilde{h}}\right).
\ea 

To make \eqref{brock-eq2} exactly the same as \eqref{sv-1d}[i]-\eqref{sv_1d_2}, we shall set $$ 
x_0=u_0t_0,\quad h_0=u_0^2,\quad F=\tilde{F},\quad u_0=\frac{t_0\sin(\theta)\lambda}{F^2h_n},
$$
or, without loss of generality $t_0:=1$, 
\be  
u_0=\frac{\sin(\theta)\lambda}{F^2h_n},\quad x_0=u_0,\quad h_0=u_0^2,\quad F=\tilde{F}.
\ee 
From $F=\tilde{F}$, we also obtain a relation between the normal depth $h_n$ and velocity $u_n$ of Brock's and the reference height $H_0$ and speed $U_0$ of ours, i.e, 
$U_0^2h_n=u_n^2H_0\cos(\theta)$, which plays no role in the discussion below. Also, it is no longer necessary to distinguish Brock's Froude number and ours. Hence, we drop the tilde on $\tilde{F}$ from now on. We conclude the re-scaling from Brock's dimensionless variables to ours below.
\be 
\label{rescaling}
\tilde{h}=\frac{F^4h_n^2}{\sin^2(\theta)\lambda^2}H,\quad \tilde{u}=\frac{F^2h_n}{\sin(\theta)\lambda}U,\quad  \tilde{t}=t',\quad \text{and} \quad \tilde{x}=\frac{F^2h_n}{\sin(\theta)\lambda} x'.
\ee 
In addition to these, it is reasonable to adopt the re-scaling 
\be 
\label{rescalingy}
 \tilde{y}=\frac{F^2h_n}{\sin(\theta)\lambda} y' 
\ee 
where $y'=\frac{y}{\lambda}$, for the transverse coordinate $\tilde{y}$ of the two-dimensional Saint Venant equations of ours \eqref{sv-1d} \eqref{sv_intro}.

\smallskip

\textbf{Brock's experimental data}. In Brock's lab, he designed and experimented with two laboratory channel setups. The first was a 130-ft tiltable laboratory channel of $3.61$ ft wide with a maximum slope of $S_0=\sin(\theta)=0.02$. Brock stated ``{\it because of insufficient length, roll waves were not formed in this channel, even at the maximum Froude number of 2.65} ''. The second, referred by him as the ``{\it steep channel} '', was an improved setup when compared with the first one. It was 128 ft long and $4\frac{5}{8}$ in. wide with a slope $S_0=\sin(\theta)$ of $0.05011$, $0.08429$, and $0.1192$. To expedite the formation of roll wave, ``{\it A plastic paddle hinged on the upstream wall of the inlet box and driven by a variable speed fractional horsepower motor was used ...} ". Then, there were seen periodic permanent roll waves. Four properties were measured for these waves: minimum depths ($h_{min}$), time periods ($T$), wave velocities ($c$), and maximum depths ($h_{max}$) For details of Brock's experiments, see \cite[Chs. IV and V]{Brock}. 

We now start to translate Brock's basic data obtained for periodic permanent roll waves into our scaling. TABLE~\ref{brock_table} is a duplicated one of \cite[Table 6 p. 109]{Brock}. We remove the redundant column $\frac{h_{max}}{h_{min}}$ for it can be computed by the ratio of the last two columns. The nondimensionalization \eqref{dimensionless_brock} also suggests that we denote $\frac{h_{max}}{h_n}$ as $H_{max}$ and $\frac{h_{min}}{h_n}$ as $H_{min}$.

To make the wave properties in Table \ref{brock_table} useful for \eqref{sv_intro} (\eqref{sv-1d}), we refer to \eqref{dimensionless_brock}, \eqref{tprime}, \eqref{dimensionless_period}, \eqref{rescaling}, and \eqref{rescalingy} and obtain that
\be  
\label{wave_lengthpeiod}
\lambda=cT=c\frac{T'}{S_0}\sqrt{\frac{h_n}{g}}=\frac{h_nT'}{S_0}\frac{c}{\sqrt{gh_n}} \quad \text{and} \quad T'_{period}=\frac{F\sqrt{gh_n}T}{\lambda}=\frac{F\sqrt{gh_n}T}{cT}=F/\frac{c}{\sqrt{gh_n}}.
\ee
Using the above together with \eqref{rescaling} and \eqref{rescalingy}, we obtain the following key wave properties to be used in \eqref{sv-1d}: $\tilde{h}_{min}$ (minimum fluid height), $\tilde{h}_{max}$ (maximum fluid height), $\tilde{\lambda}$ (wave length), and $\tilde{c}$ (wave speed), which are computable by
\be 
\label{rescaling_compute}
\tilde{h}_{min}=\frac{F^4h_n^2}{\sin^2(\theta)\lambda^2}H_{min},\quad\tilde{h}_{max}=\frac{F^4h_n^2}{\sin^2(\theta)\lambda^2}H_{max}, \quad\tilde{\lambda}=\frac{F^2h_n}{\sin(\theta)\lambda},\quad  \tilde{c}=\frac{\tilde{\lambda}}{T'_{period}}.
\ee
And, using the re-scaling in $y$ direction \eqref{rescalingy}, for channel flow of the 2d Saint-Venant equations \eqref{sv_intro}, we have that the channel width $\tilde{L}$ is given by
\be \label{channel_width}
\tilde{L}=\frac{F^2h_n}{\sin(\theta)\lambda}L'=\frac{F^2h_n}{\sin(\theta)\lambda^2}L
\ee 
where $L=4\frac{5}{8}$ inches.

TABLE~\ref{converted-data} is based on \eqref{wave_lengthpeiod}-\eqref{channel_width} and, for convenience, we sort the dataset first by the Froude number $F$ in increasing order and then by the channel width $\tilde{L}$ in increasing order.

To determine the stability of the waves in TABLE~\ref{converted-data}, we immediately find that the Channel width $\tilde{L}$ and Froude number $F$ can be read off from the second and third columns of the table, besides which the rescaled minimum fluid height ($H_-/H_s$) remains to be determined from the last four columns for the table.  As we will see later, there are $C(4,2)=6$ ways to determine $H_s$ and $H_-$. By \cite[\S 2]{JNRYZ}, there hold the following equations for inviscid Dressler's roll wave profiles 
\ba
\label{eq-JNRYZ}
 &\frac{H_s}{H_-}+\frac{H_-^2}{2H_s^2}=\frac{H_s}{H_+}+\frac{H_+^2}{2H_s^2},\\
&X=\int_{H_-}^{H_+}\frac{H^2+HH_s+H_s^2}{F^2H^2-2FHH_s-HH_s+H_s^2}dH,\\
&c=\sqrt{H_s}\left(1+\frac{1}{F}\right).
\ea 

The six ways to determine $H_-$ and $H_s$, hence $H_-/H_s$, are
\begin{itemize}
\item[1.] take $H_-=\tilde{h}_{min}$ and $H_+=\tilde{h}_{max}$ and obtain $H_s$ via \eqref{eq-JNRYZ}[i];
\item[2.] take $H_-=\tilde{h}_{min}$ and $X=\tilde{\lambda}$ and obtain $H_s$ via \eqref{eq-JNRYZ}[i,ii];
\item[3.] take $H_-=\tilde{h}_{min}$ and $c=\tilde{c}$ and obtain $H_s$ via \eqref{eq-JNRYZ}[iii];
\item[4.] take $H_+=\tilde{h}_{max}$ and $X=\tilde{\lambda}$ and obtain $H_s$ and $H_-$ via \eqref{eq-JNRYZ}[i,ii];
\item[5.] take $H_+=\tilde{h}_{max}$ and $c=\tilde{c}$ and obtain $H_s$ and $H_-$ via \eqref{eq-JNRYZ}[i,iii];
\item[6.] take $X=\tilde{\lambda}$ and $c=\tilde{c}$ and obtain $H_s$ and $H_-$ via \eqref{eq-JNRYZ}.
\end{itemize}

For each of the ways above, we can rely on \eqref{eq-JNRYZ} to compute for the other two profile properties and compare them with the corresponding experimental data in TABLE~\ref{converted-data}. We display the results in TABLE~\ref{tway1} -- TABLE~\ref{tway6} below with relative errors in parentheses. While we record the $H_-/H_s$ values in the tables, we also determine its 1d stability by referring to the 1d stability diagrams \cite[Fig. 3.]{JNRYZ} or FIGURE~\ref{fig_roll_1d}, i.e., by seeing if it sits between the low-frequency stability boundary I and med-frequency stability boundary. We color 1d-stable parameters in blue. For the parameters that are 1d stable, we rely on the numerical method used for FIGURE~\ref{fig:spectrumf6} to find and record the critical channel width $L_*$ in TABLE~\ref{tway2cri} -- TABLE~\ref{tway6cri}, by which we then color the $\tilde{L}$ value in red if $\tilde{L}>L_*$, namely the corresponding channel roll wave is predicted to be unstable. It turns out that only the channel roll waves in Way 4 case 9 is found to be 1d stable but 2d unstable. Before presenting the tables, we also would like to remark that, as noted by Brock \cite[p. 111 ``{\it For $h_{max}$ the measurements ...} '']{Brock}, the shallow-water model can overestimate the maximum fluid heights. Richard and Gavrilyuk \cite{RG1,RG2} match the experimental fluid minimum heights and wave lengths with theoretical ones. So the second way has been preferred authors in the community. Also, for a fixed $F$, $H_-/H_s$ needs to be small enough in order to be stable. The overestimating effect then results in more stable cases in way 4 where the experimental maximum fluid heights are used for computing $H_-/H_s$.

\smallskip

\textbf{Periodic permanent  vs. natural roll waves.}
The above discussion concerns what Brock terms ``periodic permanent'' roll waves,
meaning waves initiated through periodic forcing by a plastic paddle at the inlet of the ramp
and eventually forming a persistent periodic pattern downstream.
It is worth mentioning a second flight of experiments performed by Brock \cite[Table 5, pp. 82-90]{Brock} with rougher bottom, in which he was able to observe formation of roll
waves by ordinary flow, without periodic forcing. However, these waves did not form a coherent periodic
array either in wavelength or waveform, and continued to interact through overtaking and merging behavior
throughout the length of the ramp.
Thus, they have mostly been ignored in discussion of (periodic) Dressler waves, as we do here.

However, we note that our numerical experiments did indicate eventual formation of periodic roll
waves for ``dambreak'' initial data analogous to Brock's natural wave conditions.
However, this process takes considerable time, involving a sometimes lengthy transition period of
wave rearrangement by overtaking/merging like that seen by Brock.
(Note: our figures and movies of this process display behavior in a frame moving with the periodic waves,
keeping them in frame for much longer time than on a fixed ramp.)
It appears that the length of the ramp in Brock's experiments was insufficient for this sorting process,
requiring periodic forcing to achieve convergence to periodic Dressler waves within the confines
of the experimental apparatus.

One might conjecture that persistent wave forms emerging from either (forced periodic or natural)
scenario should be stable ones.  And, indeed, most of the waves under the ``way 2'' parametrization
correspond to stable waves of the Saint-Venant equations, and almost all under the ``way 4'' parametrization
(the two most attractive parametrizations from our point of view).
However, a good many are not.  This apparent paradox seems worthy of further exploration, both
from the analytical and modeling points of view.
In particular, it would be interesting to compare with analytical stability results obtained from the
more sophisticated Richard-Gavrilyuk model \cite{RG1,RG2}.

\begin{table}[htbp]
\begin{tabular}{|c|c|c|c|c|c|c|c|c|}
\hline case&
\begin{tabular}[c]{@{}c@{}}$S_0$\\ ($\sin(\theta)$)\end{tabular}    & \begin{tabular}[c]{@{}c@{}}Channel\\  finished\end{tabular} & \begin{tabular}[c]{@{}c@{}}$h_n$\\ inch\end{tabular} & $F$  & $T'$ & $\frac{c}{\sqrt{gh_n}}$ & $H_{max}$ & $H_{min}$ \\ \hline
1&0.05011 & Smooth  & 0.314                                                & 3.71 & 1.08 & 5.18                    & 1.30      & 0.76      \\ \hline 2&
0.05011 & Smooth  & 0.314                                                & 3.71 & 1.64 & 5.27                    & 1.46      & 0.68      \\ \hline 3&
0.05011 & Smooth  & 0.314                                                & 3.71 & 2.14 & 5.36                    & 1.63      & 0.63      \\ \hline 4&
0.08429 & Smooth  & 0.208                                                & 4.63 & 1.63 & 6.46                    & 1.54      & 0.66      \\ \hline 5&
0.08429 & Smooth  & 0.314                                                & 4.96 & 2.50 & 7.01                    & 1.91      & -         \\ \hline 6&
0.08429 & Smooth  & 0.208                                                & 4.63 & 2.89 & 6.79                    & 2.00      & 0.56      \\ \hline 7&
0.08429 & Smooth  & 0.208                                                & 4.63 & 4.07 & 7.15                    & 2.35      & 0.54      \\ \hline 8&
0.08429 & Smooth  & 0.208                                                & 4.63 & 4.53 & 7.24                    & 2.49      & 0.53      \\ \hline 9&
0.1192  & Smooth  & 0.210                                                & 5.60 & 2.25 & 7.74                    & 1.78      & 0.52      \\ \hline 10&
0.1192  & Smooth  & 0.210                                                & 5.60 & 3.55 & 8.24                    & 2.31      & 0.47      \\ \hline 11&
0.1192  & Smooth  & 0.308                                                & 5.90 & 4.25 & 8.82                    & 2.65      & 0.44      \\ \hline 12&
0.1192  & Smooth  & 0.210                                                & 5.60 & 5.19 & 8.76                    & 2.82      & 0.45      \\ \hline 13&
0.1192  & Rough   & 0.199                                                & 3.74 & 1.98 & 5.43                    & 1.34      & 0.70      \\ \hline 14&
0.1192  & Rough   & 0.199                                                & 3.74 & 3.73 & 5.73                    & 1.54      & 0.67      \\ \hline 15&
0.1192  & Rough   & 0.375                                                & 4.04 & 4.19 & 6.14                    & 1.55      & 0.75      \\ \hline 16&
0.1192  & Rough   & 0.199                                                & 3.74 & 5.64 & 5.95                    & 1.68      & 0.62      \\ \hline 
\end{tabular}
\caption{Table 6 of Brock's dissertation, page 109}
\label{brock_table}
\end{table}

\begin{table}[htbp]
\begin{tabular}{|c|c|c|c|c|c|c|}
\hline case
& \begin{tabular}[c]{@{}c@{}} Channel width\\  $\tilde{L}$\end{tabular} & $F$  & $\tilde{h}_{min}$ & $\tilde{h}_{max}$ & $\tilde{\lambda}$ & $\tilde{c}$ \\ \hline
3&0.077214202&3.71&0.90714991&2.347070402&1.199966871&1.73364486\\ \hline
2&0.136002028&3.71&1.724628762&3.702879401&1.592551025&2.262195122\\ \hline
1&0.324599077&3.71&4.600469971&7.869224951&2.460335335&3.435185185\\ \hline
16&0.034410104&3.74&0.107717424&0.291879471&0.416818642&0.663120567\\ \hline
14&0.084830341&3.74&0.28696855&0.659599355&0.654454941&1.002680965\\ \hline
13&0.335234017&3.74&1.184825545&2.268094615&1.30100266&1.888888889\\ \hline
15&0.036253895&4.04&0.301871377&0.623867512&0.63442507&0.964200477\\ \hline
8&0.037351917&4.63&0.226426001&1.063774988&0.653619821&1.022075055\\ \hline
7&0.047444463&4.63&0.293033196&1.275237057&0.736650573&1.137592138\\ \hline
6&0.104340089&4.63&0.668308149&2.386814819&1.092431879&1.602076125\\ \hline
4&0.362364958&4.63&2.735442907&6.382700117&2.03583164&2.840490798\\ \hline
5&0.09945015&4.96&-&3.76395151&1.403800285&1.984\\ \hline
12&0.039829222&5.60&0.214102478&1.341708862&0.689770458&1.078998073\\ \hline
10&0.09621297&5.60&0.540180418&2.654929288&1.072063449&1.577464789\\ \hline
9&0.271454968&5.60&1.686197706&5.771984454&1.800746483&2.488888889\\ \hline
11&0.044343122&5.90&0.379442255&2.285277218&0.928638122&1.388235294\\ \hline
\end{tabular}
\caption{Converted wave properties for the Saint-Venant equations}
\label{converted-data}
\end{table}

\begin{table}[htbp]
\begin{tabular}{|c|c|c|c|c|c|}
\hline
case
& \begin{tabular}[c]{@{}c@{}} Channel width\\  $\tilde{L}$\end{tabular} & $F$  & $\frac{H_-}{H_s}$ & $X$ ($\left|\frac{X-\tilde{\lambda}}{\tilde{\lambda}}\right|$) & $c$ ($\left|\frac{c-\tilde{c}}{\tilde{c}}\right|$) \\ \hline
3&0.077214202&3.71&0.599521852&0.840290022(29.97\%)&1.56165113(56.17\%)\\ \hline
2&0.136002028&3.71&0.66643899&1.058501341(33.53\%)&2.042276835(104.23\%)\\ \hline
1&0.324599077&3.71&0.755583042&1.635740924(33.52\%)&3.132614807(213.26\%)\\ \hline
16&0.034410104&3.74&0.583792825&0.107039497(74.32\%)&0.54440279(45.56\%)\\ \hline
14&0.084830341&3.74&0.641340799&0.199058717(69.58\%)&0.847772758(15.22\%)\\ \hline
13&0.335234017&3.74&0.710390323&0.542933404(58.27\%)&1.636761593(63.68\%)\\ \hline
15&0.036253895&4.04&0.680817948&0.127907397(79.84\%)&0.830700824(16.93\%)\\ \hline
8&0.037351917&4.63&0.421171021&0.315449193(51.74\%)&0.89158204(10.84\%)\\ \hline
7&0.047444463&4.63&0.441181477&0.347028835(52.89\%)&0.991008044(0.90\%)\\ \hline
6&0.104340089&4.63&0.496644194&0.537999408(50.75\%)&1.410564991(41.06\%)\\ \hline
4&0.362364958&4.63&0.6359039&0.984442331(51.64\%)&2.522002351(152.20\%)\\ \hline
5&0.09945015&4.96&-&-(-)&-(-)\\ \hline
12&0.039829222&5.60&0.352819614&0.261487716(62.09\%)&0.918100862(8.19\%)\\ \hline
10&0.09621297&5.60&0.409753789&0.420482775(60.78\%)&1.353205759(35.32\%)\\ \hline
9&0.271454968&5.60&0.509287419&0.70779532(60.69\%)&2.144512991(114.45\%)\\ \hline
11&0.044343122&5.90&0.361612692&0.359170529(61.32\%)&1.197975932(19.80\%)\\ \hline
\end{tabular}
\caption{Way 1}
\label{tway1}
\end{table}

\begin{table}[htbp]
\begin{tabular}{|c|c|c|c|c|c|}
\hline case & \begin{tabular}[c]{@{}c@{}} Channel width\\  $\tilde{L}$\end{tabular} & $F$  & $\frac{H_-}{H_s}$ & $H_+$ ($\left|\frac{H_+-\tilde{h}_{max}}{\tilde{h}_{max}}\right|$) & $c$ ($\left|\frac{c-\tilde{c}}{\tilde{c}}\right|$) \\ \hline
3&0.077214202&3.71&0.547393362&2.746438647(17.02\%)&1.634318757(63.43\%)\\ \hline
2&0.136002028&3.71&0.600008197&4.455832195(20.33\%)&2.152366471(115.24\%)\\ \hline
1&0.324599077&3.71&0.68935572&9.297329275(18.15\%)&3.279641747(227.96\%)\\ \hline
16&0.034410104&3.74&\textcolor{blue}{0.453195141}&0.448350887(53.61\%)&0.61788421(38.21\%)\\ \hline
14&0.084830341&3.74&\textcolor{blue}{0.482457603}&1.07603189(63.13\%)&0.977449466(2.26\%)\\ \hline
13&0.335234017&3.74&0.568814219&3.358092163(48.06\%)&1.829147475(82.91\%)\\ \hline
15&0.036253895&4.04&\textcolor{blue}{0.447006063}&1.285472197(106.05\%)&1.025187807(2.52\%)\\ \hline
8&0.037351917&4.63&\textcolor{blue}{0.360129671}&1.37278089(29.05\%)&0.964186619(3.58\%)\\ \hline
7&0.047444463&4.63&\textcolor{blue}{0.36794323}&1.71603579(34.57\%)&1.085164072(8.52\%)\\ \hline
6&0.104340089&4.63&\textcolor{blue}{0.403647424}&3.366396241(41.04\%)&1.56464094(56.46\%)\\ \hline
4&0.362364958&4.63&0.50910293&9.369370552(46.79\%)&2.818632356(181.86\%)\\ \hline
5&0.09945015&4.96&-&-(-)&-(-)\\ \hline
12&0.039829222&5.60&\textcolor{blue}{0.283187534}&1.905009351(41.98\%)&1.024777755(2.48\%)\\ \hline
10&0.09621297&5.60&\textcolor{blue}{0.308775479}&4.190447827(57.84\%)&1.558848556(55.88\%)\\ \hline
9&0.271454968&5.60&0.368221179&9.862508112(70.87\%)&2.522062087(152.21\%)\\ \hline
11&0.044343122&5.90&\textcolor{blue}{0.278048669}&3.475179127(52.07\%)&1.366185009(36.62\%)\\ \hline
\end{tabular}
\caption{Way 2}
\label{tway2}
\end{table}

\begin{table}[htbp]
\begin{tabular}{|c|c|c|c|c|c|}
\hline case& \begin{tabular}[c]{@{}c@{}} Channel width\\  $\tilde{L}$\end{tabular} & $F$  & $\frac{H_-}{H_s}$ & $H_+$ ($\left|\frac{H_+-\tilde{h}_{max}}{\tilde{h}_{max}}\right|$) & $X$ ($\left|\frac{X-\tilde{\lambda}}{\tilde{\lambda}}\right|$) \\ \hline
3&0.077214202&3.71&\textcolor{blue}{0.486466327}&3.354604596	(42.93\%)&	2.084790017	(73.74\%)	\\ \hline
2&0.136002028&3.71&0.543162147&5.291197203	(42.89\%)&	2.354803681	(47.86\%)	\\ \hline
1&0.324599077&3.71&0.628341706&10.96316127	(39.32\%)&	3.557404141	(44.59\%)	\\ \hline
16&0.034410104&3.74&0.393472552&0.565691548	(93.81\%)&	NI(-)	\\ \hline
14&0.084830341&3.74&\textcolor{blue}{0.458481949}&1.171638403	(77.63\%)&	0.958130799	(46.40\%)	\\ \hline
13&0.335234017&3.74&0.533402521&3.749364921	(65.31\%)&	1.680119887	(29.14\%)	\\ \hline
15&0.036253895&4.04&0.505342232&1.046999814	(67.82\%)&	0.378027761	(40.41\%)	\\ \hline
8&0.037351917&4.63&0.32049077&1.655305899	(55.61\%)&	NI(-)	\\ \hline
7&0.047444463&4.63&0.334810034&1.997610228	(56.65\%)&	NI(-)	\\ \hline
6&0.104340089&4.63&\textcolor{blue}{0.385004023}&3.636274477	(52.35\%)&	1.333022191	(22.02\%)	\\ \hline
4&0.362364958&4.63&0.501297674&9.617071798	(50.67\%)&	2.137894741	(5.01\%)	\\ \hline
5&0.09945015&4.96&-&-	(-)&	-	(-)	\\ \hline
12&0.039829222&5.60&0.255441922&2.240690823	(67.00\%)&	NI	(-)	\\ \hline
10&0.09621297&5.60&\textcolor{blue}{0.30153054}&4.351588492	(63.91\%)&	1.193836562	(11.36\%)	\\ \hline
9&0.271454968&5.60&0.378102278&9.448238793	(63.69\%)&	1.661541491	(7.73\%)	\\ \hline
11&0.044343122&5.90&\textcolor{blue}{0.269285946}&3.655041651	(59.94\%)&	1.12561615	(21.21\%)	\\ \hline
\end{tabular}
\caption{Way 3, NI stands for ``Not Integrable", meaning that the $H_-/H_s$ computed is not in the domain of existence of inviscid Dressler's roll waves.}
\label{tway3}
\end{table}

\begin{table}[htbp]
\begin{tabular}{|c|c|c|c|c|c|}
\hline case& \begin{tabular}[c]{@{}c@{}} Channel width\\  $\tilde{L}$\end{tabular} & $F$  & $\frac{H_-}{H_s}$ & $H_-$ ($\left|\frac{H_--\tilde{h}_{min}}{\tilde{h}_{min}}\right|$) & $c$ ($\left|\frac{c-\tilde{c}}{\tilde{c}}\right|$) \\ \hline
3&0.077214202&3.71&\textcolor{blue}{0.515042798}&0.698828813(22.96\%)&1.478804419(47.88\%)\\ \hline
2&0.136002028&3.71&0.551102106&1.237297521(28.26\%)&1.902251931(90.23\%)\\ \hline
1&0.324599077&3.71&0.639478478&3.406072477(25.96\%)&2.929955158(193.00\%)\\ \hline
16&0.034410104&3.74&\textcolor{blue}{0.446620762}&0.068445051(36.46\%)&0.496144647(50.39\%)\\ \hline
14&0.084830341&3.74&\textcolor{blue}{0.451306698}&0.157375501(45.16\%)&0.74840971(25.16\%)\\ \hline
13&0.335234017&3.74&\textcolor{blue}{0.489844901}&0.620507982(47.63\%)&1.426432186(42.64\%)\\ \hline
15&0.036253895&4.04&\textcolor{blue}{0.407140001}&0.125610611(58.39\%)&0.69293173(30.71\%)\\ \hline
8&0.037351917&4.63&\textcolor{blue}{0.348756533}&0.166615294(26.42\%)&0.840472793(15.95\%)\\ \hline
7&0.047444463&4.63&\textcolor{blue}{0.350705106}&0.201534926(31.22\%)&0.921788996(7.82\%)\\ \hline
6&0.104340089&4.63&\textcolor{blue}{0.362838954}&0.398474348(40.38\%)&1.274296925(27.43\%)\\ \hline
4&0.362364958&4.63&\textcolor{blue}{0.406739044}&1.283033122(53.10\%)&2.159675026(115.97\%)\\ \hline
5&0.09945015&4.96&\textcolor{blue}{0.339916419}&0.565707692(-)&1.550152958(55.02\%)\\ \hline
12&0.039829222&5.60&\textcolor{blue}{0.274216464}&0.14332523(33.06\%)&0.852060572(14.79\%)\\ \hline
10&0.09621297&5.60&\textcolor{blue}{0.279199153}&0.291776997(45.99\%)&1.204826152(20.48\%)\\ \hline
9&\textcolor{red}{0.271454968}&5.60&\textcolor{blue}{0.291260501}&0.678197112(59.78\%)&1.798428104(79.84\%)\\ \hline
11&0.044343122&5.90&\textcolor{blue}{0.259908299}&0.22438418(40.86\%)&1.086633581(8.66\%)\\ \hline
\end{tabular}
\caption{Way 4}
\label{tway4}
\end{table}

\begin{table}[htbp]
\begin{tabular}{|c|c|c|c|c|c|}
\hline case& \begin{tabular}[c]{@{}c@{}} Channel width\\  $\tilde{L}$\end{tabular} & $F$  & $\frac{H_-}{H_s}$ & $H_-$ ($\left|\frac{H_--\tilde{h}_{min}}{\tilde{h}_{min}}\right|$) & $X$ ($\left|\frac{X-\tilde{\lambda}}{\tilde{\lambda}}\right|$) \\ \hline
3&0.077214202&3.71&0.779605258&2.347070402	(158.73\%)&	0.441618725	(63.20\%)	\\ \hline
2&0.136002028&3.71&0.850418331&3.702879401	(114.71\%)&	0.483123596	(69.66\%)	\\ \hline
1&0.324599077&3.71&0.928759374&7.869224951	(71.05\%)&	0.507611132	(79.37\%)	\\ \hline
16&0.034410104&3.74&0.936614347&0.291879471	(170.97\%)&	0.016401943	(96.06\%)	\\ \hline
14&0.084830341&3.74&0.948040836&0.659599355	(129.85\%)&	0.030585061	(95.33\%)	\\ \hline
13&0.335234017&3.74&0.979207459&2.268094615	(91.43\%)&	0.042895991	(96.70\%)	\\ \hline
15&0.036253895&4.04&0.956901298&0.623867512	(106.67\%)&	0.019042323	(97.00\%)	\\ \hline
8&0.037351917&4.63&0.62376399&1.063774988	(369.81\%)&	0.169581612	(74.06\%)	\\ \hline
7&0.047444463&4.63&0.651110914&1.275237057	(335.19\%)&	0.18857095	(74.40\%)	\\ \hline
6&0.104340089&4.63&0.700727804&2.386814819	(257.14\%)&	0.304882567	(72.09\%)	\\ \hline
4&0.362364958&4.63&0.847596225&6.382700117	(133.33\%)&	0.438316405	(78.47\%)	\\ \hline
5&0.09945015&4.96&0.697157231&3.76395151	(-)&	0.400599376	(71.46\%)	\\ \hline
12&0.039829222&5.60&0.574393426&1.341708862	(526.67\%)&	0.141540289	(79.48\%)	\\ \hline
10&0.09621297&5.60&0.636905412&2.654929288	(391.49\%)&	0.240664539	(77.55\%)	\\ \hline
9&0.271454968&5.60&0.754310731&5.771984454	(242.31\%)&	0.367779596	(79.58\%)	\\ \hline
11&0.044343122&5.90&0.564130568&2.285277218	(502.27\%)&	0.214637288	(76.89\%)	\\ \hline
\end{tabular}
\caption{Way 5, NI stands for ``Not Integrable", meaning that the $H_-/H_s$ computed is not in the domain of existence of inviscid Dressler's roll waves.}
\label{tway5}
\end{table}

\begin{table}[htbp]
\begin{tabular}{|c|c|c|c|c|c|}
\hline case& \begin{tabular}[c]{@{}c@{}} Channel width\\  $\tilde{L}$\end{tabular} & $F$  & $\frac{H_-}{H_s}$ & $H_-$ ($\left|\frac{H_--\tilde{h}_{min}}{\tilde{h}_{min}}\right|$) & $H_+$ ($\left|\frac{H_+-\tilde{h}_{max}}{\tilde{h}_{max}}\right|$) \\ \hline
3&0.077214202&3.71&0.569495208&1.061980033	(17.07\%)&	3.003710841	(200.37\%)	\\ \hline
2&0.136002028&3.71&0.621393489&1.973026084	(14.40\%)&	4.794644049	(379.46\%)	\\ \hline
1&0.324599077&3.71&0.709449463&5.194308953	(12.91\%)&	9.967276964	(896.73\%)	\\ \hline
16&0.034410104&3.74&\textcolor{blue}{0.458738257}&0.125584626	(16.59\%)&	0.512261782	(48.77\%)	\\ \hline
14&0.084830341&3.74&\textcolor{blue}{0.48783861}&0.305343185	(6.40\%)&	1.12381909	(12.38\%)	\\ \hline
13&0.335234017&3.74&0.581815121&1.292362503	(9.08\%)&	3.522541175	(252.25\%)	\\ \hline
15&0.036253895&4.04&\textcolor{blue}{0.434940906}&0.259816421	(13.93\%)&	1.15762413	(15.76\%)	\\ \hline
8&0.037351917&4.63&\textcolor{blue}{0.367927701}&0.259940085	(14.80\%)&	1.522342715	(52.23\%)	\\ \hline
7&0.047444463&4.63&\textcolor{blue}{0.375840887}&0.328944312	(12.25\%)&	1.861192835	(86.12\%)	\\ \hline
6&0.104340089&4.63&0.410379655&0.712356367	(6.59\%)&	3.492410261	(249.24\%)	\\ \hline
4&0.362364958&4.63&0.512752363&2.797947983	(2.29\%)&	9.468330835	(846.83\%)	\\ \hline
5&0.09945015&4.96&0.397748575&1.084333317	(-)&	5.594975419	(459.50\%)	\\ \hline
12&0.039829222&5.60&\textcolor{blue}{0.287849554}&0.241265421	(12.69\%)&	2.091994038	(109.20\%)	\\ \hline
10&0.09621297&5.60&\textcolor{blue}{0.311035154}&0.55720757	(3.15\%)&	4.272670558	(327.27\%)	\\ \hline
9&0.271454968&5.60&0.36358852&1.621471662	(3.84\%)&	9.680094573	(868.01\%)	\\ \hline
11&0.044343122&5.90&\textcolor{blue}{0.280187596}&0.394803423	(4.05\%)&	3.57240238	(257.24\%)	\\ \hline
\end{tabular}
\caption{Way 6}
\label{tway6}
\end{table}

\begin{table}[htbp]
\begin{tabular}{|c|c|c|c|c|c|c|c|c|c|}
\hline
case                         & 16        & 14          & 15           & 8           & 7            & 6            & 12           & 10           & 11           \\ \hline
$L_*$ & $+\infty$ & $2.3\ldots$ & $1.44\ldots$ & $1.0\ldots$ & $0.84\ldots$ & $0.49\ldots$ & $0.23\ldots$ & $0.19\ldots$ & $0.17\ldots$ \\ \hline
\end{tabular}
\caption{Way 2, critical channel width $L_*$; The symbol ``$+\infty$'' means the parameter sit in the 2d stable region.}
\label{tway2cri}
\end{table}

\begin{table}[htbp]
\begin{tabular}{|c|c|c|c|c|c|}
\hline
case  & 3           & 14        & 6   & 10           & 11          \\ \hline
$L_*$ & $2.4\ldots$ & $+\infty$ & $0.62\ldots$ & $0.20\ldots$ & $0.18\dots$ \\ \hline
\end{tabular}
\caption{Way 3, critical channel width $L_*$; The symbol ``$+\infty$'' means the parameter sit in the 2d stable region.}
\label{tway3cri}
\end{table}

\begin{table}[htbp]
\begin{tabular}{|c|c|c|c|c|llllll}
\hline
case  & 3            & 16           & 14           & 13           & \multicolumn{1}{c|}{15}        & \multicolumn{1}{c|}{8}           & \multicolumn{1}{c|}{7}           & \multicolumn{1}{c|}{6}            & \multicolumn{1}{c|}{4}            & \multicolumn{1}{c|}{5}           \\ \hline
$L_*$ & $1.3\ldots$  & $+\infty$    & $+\infty$    & $1.8\ldots$  & \multicolumn{1}{c|}{$+\infty$} & \multicolumn{1}{c|}{$2.1\ldots$} & \multicolumn{1}{c|}{$1.7\ldots$} & \multicolumn{1}{c|}{$0.96\ldots$} & \multicolumn{1}{c|}{$0.47\ldots$} & \multicolumn{1}{c|}{$0.43\dots$} \\ \hline
case  & 12           & 10           & 9            & 11           &                                &                                  &                                  &                                   &                                   &                                  \\ \cline{1-5}
$L_*$ & $0.23\ldots$ & $0.23\ldots$ & $0.21\ldots$ & $0.18\ldots$ &                                &                                  &                                  &                                   &                                   &                                  \\ \cline{1-5}
\end{tabular}
\caption{Way 4, critical channel width $L_*$; The symbol ``$+\infty$'' means the parameter sit in the 2d stable region.}
\label{tway4cri}
\end{table}

\begin{table}[htbp]
\begin{tabular}{|c|c|c|c|c|c|c|c|c|}
\hline
case  & 16        & 14          & 15          & 8   & 7            & 12           & 10           & 11          \\ \hline
$L_*$ & $+\infty$ & $1.9\ldots$ & $2.0\ldots$ & $0.84\ldots$ & $0.71\ldots$ & $0.22\ldots$ & $0.19\ldots$ & $0.17\dots$ \\ \hline
\end{tabular}
\caption{Way 6, critical channel width $L_*$; The symbol ``$+\infty$'' means the parameter sit in the 2d stable region.}
\label{tway6cri}
\end{table}

\bibliographystyle{plain}
\newpage
\bibliography{Ref-Multid}

\end{document}